\documentclass[oneside,a4paper]{amsart}

\usepackage[T1]{fontenc}
\usepackage[utf8]{inputenc}
\usepackage{amsmath, amsthm, amsfonts, amssymb}
\usepackage{nicefrac}

\usepackage[colorlinks=true,hyperindex, linkcolor=magenta, pagebackref=false, citecolor=cyan,pdfpagelabels]{hyperref}
\usepackage{url}
\usepackage[all]{xy}
\usepackage{mathabx}
\usepackage{colonequals}
\usepackage{placeins} 
\usepackage[usenames, dvipsnames]{color} 
\usepackage{euscript} 
\usepackage{bbm} 
\usepackage{enumitem} 
\usepackage[nameinlink,capitalise,noabbrev]{cleveref} 
\usepackage{todonotes}

\usepackage{tikz-cd}
\usepackage{tikz}
\usetikzlibrary{arrows,backgrounds,
    decorations.pathreplacing,
    decorations.pathmorphing
}
\usetikzlibrary{matrix,arrows}

\newcommand{\euscr}[1]{\EuScript{#1}} 
\newcommand{\acat}{\euscr{A}} 
\newcommand{\bcat}{\euscr{B}} 
\newcommand{\ccat}{\euscr{C}} 
\newcommand{\dcat}{\euscr{D}} 
\newcommand{\ecat}{\euscr{E}} 
\newcommand{\hcat}{\euscr{H}} 
\newcommand{\Fun}{\textnormal{Fun}} 
\newcommand{\Hom}{\textnormal{Hom}} 
\newcommand{\Ext}{\textnormal{Ext}} 
\newcommand{\Ind}{\textnormal{Ind}} 
\newcommand{\map}{\textnormal{map}} 
\newcommand{\Map}{\textnormal{Map}} 
\newcommand{\Ev}{\textnormal{Ev}} 
\newcommand{\spaces}{\euscr{S}} 
\newcommand{\abeliangroups}{\euscr{A}b} 
\newcommand{\Tot}{\textnormal{Tot}} 
\newcommand{\spectra}{\euscr{S}p} 
\newcommand{\sets}{\euscr{S}et} 

\newcommand{\synspectra}{\euscr{S}yn} 
\newcommand{\ComodE}{\euscr{C}omod_{E_{*}E}} 
\newcommand{\Comod}{\euscr{C}omod} 
\newcommand{\Mod}{\euscr{M}od} 
\newcommand{\monunit}{\mathbbm{1}} 

\newcommand{\Fil}{\mathrm{Fil}}

\newcommand{\triplerightarrow}{%
\tikz[minimum height=0ex]
  \path[->]
   node (a)            {}
   node (b) at (1em,0) {}
  (a.north)  edge (b.north)
  (a.center) edge (b.center)
  (a.south)  edge (b.south);%
}

\theoremstyle{plain}

\newtheorem{theorem}{Theorem}[section]
\newtheorem{lemma}[theorem]{Lemma}
\newtheorem{proposition}[theorem]{Proposition}
\newtheorem{corollary}[theorem]{Corollary}

\newtheorem*{theorem*}{Theorem}

\theoremstyle{definition}
\newtheorem{example}[theorem]{Example}
\newtheorem{warning}[theorem]{Warning}

\newtheorem{definition}[theorem]{Definition}
\newtheorem{remark}[theorem]{Remark}
\newtheorem{notation}[theorem]{Notation}
\newtheorem{construction}[theorem]{Construction}

\newtheorem*{remark*}{Remark}
\newtheorem*{interpretation*}{Interpretation}
\newtheorem*{definition*}{Definition}
\newtheorem*{conjecture*}{Conjecture}
\newtheorem*{notation*}{Notation}
\newtheorem*{convention*}{Convention}

\theoremstyle{remark}

\numberwithin{equation}{section}

\makeatletter
  \def\subsection{\@startsection{subsection}{1}%
  \z@{.7\linespacing\@plus\linespacing}{.5\linespacing}%
  {\normalfont\bfseries\centering}}
\makeatother

\let\oldtocsection=\tocsection
\let\oldtocsubsection=\tocsubsection
\let\oldtocsubsubsection=\tocsubsubsection
\renewcommand{\tocsection}[2]{\hspace{0em}\oldtocsection{#1}{#2}}
\renewcommand{\tocsubsection}[2]{\hspace{1em}\oldtocsubsection{#1}{#2}}
\renewcommand{\tocsubsubsection}[2]{\hspace{2em}\oldtocsubsubsection{#1}{#2}}

\calclayout

\title{Adams spectral sequences and Franke's algebraicity conjecture} 

\author{Irakli Patchkoria}
\address{University of Aberdeen, Scotland, UK}
\email{irakli.patchkoria@abdn.ac.uk}

\author{Piotr Pstr\k{a}gowski}
\address{Harvard University, USA}
\email{pstragowski.piotr@gmail.com}

\begin{document}

\begin{abstract}
To any well-behaved homology theory we associate a derived $\infty$-category which encodes its Adams spectral sequence. As applications, we prove a conjecture of Franke on algebraicity of certain homotopy categories and establish homotopy-coherent monoidality of the Adams filtration. 
\end{abstract}

\maketitle

{\hypersetup{linkcolor=black}
\tableofcontents
}

\section{Introduction}

In this paper, we characterize and classify homology theories which admit an Adams spectral sequence based on injectives. To each such homology theory, we associate an appropriate derived $\infty$-category, which we prove is an $\infty$-categorical deformation in a precise sense, and a suitable context for Goerss-Hopkins theory. As two particular applications of this construction,  we prove 
\begin{enumerate}
    \item a conjecture of Franke that all stable $\infty$-categories which admit a homology theory with a sufficiently simple target have equivalent homotopy categories and 
    \item that the Adams filtration associated to a spectrum $R$ with a right unital multiplication is lax symmetric monoidal, providing numerous new examples of filtered $\mathbf{E}_{n}$-ring spectra. 
\end{enumerate}
We discuss our results at more length below. 

If $\ccat$ is a stable $\infty$-category, it is common to study it using homology theories; that is, additive functors $H\colon \ccat \rightarrow \acat$ valued in an abelian category which take cofibre sequences in $\ccat$ to triangles exact in the middle and which take the suspension to a distinguished autoequivalence of $\acat$, which we think of as an internal grading shift. 

Many important homology theories have the property of being what we will call \emph{adapted}; that is, $\acat$ has enough injectives and for every injective $i \in \acat$ there exists a $c_{i} \in \ccat$ such that $H(c_{i}) \simeq i$ and such that the induced map 
\[
[c, c_{i}]_{\ccat} \rightarrow \Hom_{\acat}(H(c), i)
\]
from homotopy classes is an isomorphism for any $c \in \ccat$. Here, one can think of $c_{i}$ as an appropriate Eilenberg-MacLane object.

An insight due to Adams is that this condition allows one to lift injective resolutions along $H$, leading for any $c, d \in \ccat$ to a spectral sequence of signature 
\[
\Ext^{s, t}_{\acat}(H(d), H(c)) \Rightarrow [\Sigma^{t-s} d, c].
\]
This is arguably \emph{the} most important general method of computing homotopy classes in the stable context, in particular stable homotopy groups of spheres. 

As the second page of the Adams spectral sequence is given by $\Ext$-groups, one expects that the more simple the homological algebra of $\acat$, the greater control it exerts over the stable $\infty$-category. A conjecture of Franke, which we prove in this work in complete generality, asserts that if the abelian category in question is sufficiently simple, then it uniquely determines the homotopy category of $\ccat$. 

If $\acat$ is an abelian category equipped with a local grading; that is, a distinguished autoequivalence $[1]$, then a \emph{splitting of order $q+1$} is a collection of Serre subcategories $\acat_{\phi}$ indexed by $\phi \in \mathbb{Z}/(q+1)$ such that $[1](\acat_{\phi}) \subseteq \acat_{\phi+1}$ and $\acat \simeq \prod_{\phi \in \mathbb{Z}/(q+1)} \acat_{\phi}$.

\begin{theorem}[{Franke's algebraicity conjecture, \ref{theorem:frankes_algebraicity}}]
\label{theorem:franke_algebraicity_in_introduction}
Let $H\colon \ccat \rightarrow \acat$ be a conservative, adapted homology theory such that $\acat$ admits a splitting of order $q+1$ and is of cohomological dimension $d \leq q$. Then, there exists a canonical equivalence
\[
h_{q+1-d} (\ccat) \simeq h_{q+1-d} (\dcat^{per}(\acat))
\]
of homotopy $(q+1-d)$-categories of $\ccat$ and the periodic derived $\infty$-category of $\acat$.
\end{theorem}
The periodic derived $\infty$-category $\dcat^{per}(\acat)$ is of purely algebraic nature, obtained from the category of differential objects in $\acat$ by inverting quasi-isomorphisms, so that \cref{theorem:franke_algebraicity_in_introduction} should be interpreted as giving an algebraic description of the homotopy category $h \ccat$. The statement and the first attempted proof of this conjecture are due to Franke \cite{franke1996uniqueness}; we discuss the history of this and problem in more detail in \S\ref{subsection:history_of_algebraicity_results} below.

While quite general,  \cref{theorem:franke_algebraicity_in_introduction} applies to homotopy categories of many familiar stable $\infty$-categories, such as the following examples.

\begin{example}[{Modules over ring spectra, \ref{theorem:frankes_algebraicity_for_modules}}] \label{modspintro}
Let $R$ be an $\mathbf{E}_{1}$-ring spectrum such that the homotopy ring $R_{*} \colonequals \pi_{*}R$ 
\begin{enumerate}
    \item is concentrated in degrees divisible by $(q+1)$ and 
    \item is of global dimension $d \leq q$ as a graded ring. 
\end{enumerate}
Then, $\pi_{*}\colon \Mod_{R}(\spectra) \rightarrow \Mod_{R_{*}}(\mathrm{gr}\abeliangroups)$ is an adapted homology theory satisfying the conditions of \cref{theorem:franke_algebraicity_in_introduction} and thus we have a canonical equivalence
\[
h_{q+1-d} \Mod_{R}(\spectra) \simeq h_{q+1-d} \dcat(R_{*}),
\]
where on the right hand side we treat $R_{*}$ as a dg-algebra with the zero differential. 

When $d \leq 3$, such as for $KU$, $ko_{(p)}$ at an odd prime $p$, or Morava $K$-theories and their connective covers, several examples of an equivalence of this type at the level of homotopy $1$-categories were known due to work of many authors \cite{GreenleesS1}, \cite{Wolbert}, \cite{DugShiEx}  \cite{patchkoria2012}, \cite{patchkoriaKU}. In most of these examples, our result strengthens these equivalences to the level of higher homotopy categories. 

When $d \geq 4$, such as for $\mathbf{E}_{1}$-forms of $BP\langle n \rangle$ and Johnson-Wilson theory $E(n)$ at sufficiently large primes, these equivalences are completely new. 
\end{example}

\begin{example}[{Chromatic algebraicity, \ref{theorem:chromatic_algebraicity}}]
\label{example:chromatic_algebraicity_in_introduction}
It is classical that  Johnson-Wilson homology 
\[
E(n)_{*}\colon \spectra_{E(n)} \rightarrow \Comod_{E(n)_{*}E(n)}.
\]
satisfies the conditions of Franke's conjecture when $2p-2 > n^{2}+n$. Thus, \cref{theorem:franke_algebraicity_in_introduction} implies that we have
\[
h_{k} (\spectra_{E(n)}) \simeq h_{k} (\dcat(E(n)_{*}E(n))).
\]
where $k = 2p-2-n^{2}-n$. This is an extension of a result from second author's thesis, which proved an equivalence of homotopy categories under the weaker bound $2p-2 > 2(n^{2}+n)$, as well as the classical work of Bousfield at $n=1$ \cite{bousfield1990classification},  \cite{pstragowski_chromatic_homotopy_algebraic}. The current work improves the range in which equivalences exist to the one predicted by Franke.
\end{example}

More generally, in \S\ref{subsection:diagram_infty_categories}, we produce many further examples by passing to diagram $\infty$-categories and bounding the cohomological dimension of the resulting abelian category. For example, we prove an equivariant analogue of the chromatic algebraicity of \cref{example:chromatic_algebraicity_in_introduction} which applies to the $\infty$-category of genuine $E(n)$-local $G$-spectra, where $G$ is a finite group of order coprime to $p$.

As the above examples highlight, the equivalences produced by \cref{theorem:franke_algebraicity_in_introduction} can almost never be lifted to the level of stable $\infty$-categories, and so their construction must necessarily be quite involved. Let us discuss some of the methods we employ. 

\subsection{Adams spectral sequences} 

Our general approach to Franke's conjecture follows the strategy from the second's author thesis \cite{pstragowski_chromatic_homotopy_algebraic}: one wants to embed $\ccat$ into a suitable ``derived $\infty$-category'' and apply Goerss-Hopkins theory \cite{goerss2014hopkins}, \cite{abstract_gh_theory}. In the case of Adams-type homology theories on spectra, a suitable derived $\infty$-category is given by synthetic spectra \cite{pstrkagowski2018synthetic}. To perform an analogous construction in the general case requires a much deeper understanding of Adams spectral sequences and adapted homology theories,  so let us first describe our results in this direction. 

Associated to a stable $\infty$-category $\ccat$ we have the category $A(\ccat) \colonequals P_{\Sigma}^{fp}(\ccat, \abeliangroups)$ of finitely presented presheaves of abelian groups, known as the Freyd envelope \cite{freyd1966stable}. This can be shown to be abelian, and the truncated Yoneda embedding $y\colon \ccat \rightarrow A(\ccat)$ is naturally a homology theory, in fact the universal one \cite{neeman2001triangulated}.

\begin{theorem}[Characterization of adapted homology theories, \ref{theorem:characterization_of_adapted_homology_theories}]
\label{theorem:adapted_ht_same_as_exact_localizations_in_introduction}
Let $H\colon \ccat \rightarrow \acat$ be a homology theory and assume that $\acat$ has enough injectives. Then, the following conditions are equivalent: 
\begin{enumerate}
    \item $H$ is adapted, 
    \item the induced exact functor $L_{H}\colon A(\ccat) \rightarrow \acat$ is a localization; that is, it identifies $\acat$ with a Gabriel quotient of the Freyd envelope by a localizing subcategory. 
\end{enumerate}
\end{theorem}

The study of quotients of the Freyd envelope has recently gained traction in the setting when $\ccat$ is rigid monoidal; starting with the work of Balmer, Krause and Stevenson who call certain such quotients the homological residue fields \cite{balmer2019tensor}, \cite{balmer2020nilpotence} \cite{barthel2021stratification}. Our work clarifies the relationship between these residue fields and the Adams spectral sequence. 

In particular, we show that if $R$ a homotopy ring spectrum, not necessarily Adams-type, then there exists a canonical comonad $C$ on the abelian category $\Mod_{R_{*}}$ such that $R_{*}\colon \spectra \rightarrow \Mod_{R_{*}}$ has a unique, \emph{adapted} lift to $C$-comodules. It follows formally that the $E_{2}$-page of the $R_{*}$-based Adams spectral sequence is computed by $\Ext$-groups in $\Comod_{C}(\Mod_{R_{*}})$. If $R$ is Adams-type, then $C \colonequals R_{*}R \otimes_{R_{*}} -$ and the result is classical. In the general case, we believe these comonads and the resulting description of the Adams $E_{2}$-page are completely novel.

There are further spectral sequences which are often referred to as an Adams spectral sequence, and which a priori do not fit into the framework of homology theories explained above. For example, if $R$ is an $\mathbf{E}_{1}$-ring spectrum, then given $X \in \spectra$ we can construct the Amitsur resolution
\[
X \rightarrow R \otimes X \rightrightarrows R \otimes R \otimes  X \triplerightarrow \ldots
\]
and applying $[Y, -]$ for any spectrum leads to the descent spectral sequence, which in general not isomorphic to the Adams spectral sequence based on comodules. However, there is a more general notion of an Adams spectral sequence, due to Miller and later expanded on by Christensen, which includes the descent spectral sequence as an example \cite{miller1975some}, \cite{Christensen}, \cite{hopkins1999complex}.

We say a class of arrows $\mathcal{M} \subseteq \Fun(\Delta^{1}, \ccat)$ in a stable $\infty$-category is a \emph{monomorphism class} if it satisfies simple axioms analogous to those satisfies by monomorphisms in an ordinary category, such as closure under composition and pushouts, see \cref{definition:epimorphism_class} for a precise statement. An object $c \in \ccat$ is $\mathcal{M}$-injective if it has the right lifting property with respect to $\mathcal{M}$. 

If $\ccat$ has enough $\mathcal{M}$-injectives, then any object admits an appropriate $\mathcal{M}$-injective resolution, leading to a spectral sequence. For example, to recover the descent spectral sequence, we take $\mathcal{M}$ to be the class of maps which are split inclusions of spectra after applying $R \otimes -$. 

To our great surprise, this general framework turns out to be completely equivalent to the one defined by abelian categories. 

\begin{theorem}[Classification of Adams spectral sequences, \ref{theorem:adapted_homology_theories_correspond_to_injective_epimorphism_classes}]
\label{theorem:classification_of_asss_in_introduction}
If $\ccat$ is an idempotent-complete stable $\infty$-category, then the following three collections of data are equivalent
\begin{enumerate}
    \item adapted homology theories $H\colon \ccat \rightarrow \acat$,
    \item localizing subcategories $K \subseteq A(\ccat)$ compatible with the local grading such that the quotient $A(\ccat)/K$ has enough injectives and
    \item monomorphisms classes $\mathcal{M} \subseteq \Fun(\Delta^{1}, \ccat)$ such that $\ccat$ has enough $\mathcal{M}$-injectives. 
\end{enumerate}
\end{theorem}
The above result is an extension of the work of Beligiannis \cite{beligiannis2000relative}, who proved an equivalence between $(2)$ and $(3)$; we directly connect both to the Adams spectral sequence. As one surprising consequence, we see that adapted homology theories form a poset. 

In particular, \cref{theorem:classification_of_asss_in_introduction} implies that for any $\mathbf{E}_{1}$-ring spectrum $R$, there is a uniquely defined abelian category such that the $E_{2}$-page of the descent spectral sequence is \emph{always} given by the corresponding $\Ext$-groups. Even for an Adams-type ring spectrum, this abelian category is usually not equivalent to  $\Comod_{R_{*}R}$, as the classes of $R \otimes -$-split and $R_{*}$-injective maps do not necessarily coincide. 

Adams resolutions based on a non-Adams-type ring spectrum have been a powerful tool in stable homotopy theory, perhaps most famously used in resolution of the height one telescope conjecture at $p = 2$ by Mahowald \cite{mahowald1981bo}, \cite{mahowald1982image}. Describing the Adams $E_{2}$-page in this context is delicate and an active area of research \cite{gonzalez2000vanishing}, \cite{beaudry2019tmf}, \cite{beaudry20202}; the methods developed here provide a new avenue to these calculations, placing the problem firmly in the realm of abelian categories.

Miller's idea of rephrasing an Adams spectral sequence as being determined by a class of morphisms is at the core of our construction of a suitable derived $\infty$-category associated to a homology theory, which we discuss in more detail below. Before we do so, let us describe the following striking application of our methods.

\begin{theorem}[{Monoidality of the $R$-Adams filtration, \ref{theorem:monoidality_of_the_classical_adams_filtration}}]
\label{theorem:monoidality_of_r_adams_filtration_in_the_introduction}
Let $R$ be a spectrum which admits a right unital multiplication. Then, the functor 
\[
F^{R}\colon \spectra \rightarrow \Fil(\spectra)
\]
which associates to any spectrum its $R$-Adams filtration is lax symmetric monoidal. In particular, if $X$ is an $\mathbf{E}_{n}$-algebra, then $F^{R}(X)$ is a filtered $\mathbf{E}_{n}$-algebra. 
\end{theorem}
Note that if $R$ is itself $\mathbf{E}_{\infty}$, then a lax symmetric monoidal Adams filtration can be produced by taking the d\'{e}calage of the Amitsur resolution, as in the work of Gheorghe, Ricka, Krause and Isaksen \cite{gheorghe2018c}. The surprising part of the above result is that no such structure on $R$ is necessary.

\cref{theorem:monoidality_of_r_adams_filtration_in_the_introduction} provides numerous new families of $\mathbf{E}_{n}$-algebras in filtered spectra; these are intricate objects and thus difficult to construct by hand. They are important in computations, as they establish compatibility with power operations, which in the case of the classical $H \mathbb{F}_{p}$-based Adams were the subject of the extensive work of Bruner \cite{bruner2009adams}, \cite{bruner2006h}. 

Our work extends the reach of these methods to a wide class of Adams spectral sequences. For example, $S^{0} / p$ is known to admit a right unital multiplication when $p$ is odd, and we deduce that the induced filtration is compatible with power operations. In the algebraic setting, a similar observation was used by May to produce various formulas involving the Bockstein spectral sequence \cite{may1970general}.

Informally, \cref{theorem:monoidality_of_r_adams_filtration_in_the_introduction} parallels the observation that the Bousfield localization at an arbitrary spectrum is lax symmetric monoidal, with no structure on the spectrum in question required. The core input needed for the proof is that the class of $R$-split maps is stable under tensor products with an arbitrary spectrum and a construction of an appropriate derived $\infty$-category, which is our next point of focus.

\subsection{Derived $\infty$-categories and Goerss-Hopkins theory}

To construct an appropriate derived $\infty$-category associated to an adapted homology theory $H\colon \ccat \rightarrow \acat$, we consider certain classes of sheaves on $\ccat$. This requires care, as in practice $\ccat$ is almost never small. 

We say an additive presheaf $X\colon \ccat \rightarrow \spaces$ is \emph{perfect} if it belongs to the smallest subcategory of $\Fun_{\Sigma}(\ccat^{op}, \spaces)$ closed under finite colimits. We show that perfect presheaves are also closed under finite limits, and form a prestable  $\infty$-category whose heart we can identify with the classical Freyd envelope.

As a consequence of \cref{theorem:classification_of_asss_in_introduction}, an adapted homology theory is completely determined by the class of those maps in $\ccat$ which are taken to epimorphisms, and these determine a Grothendieck pretopology. The connective \emph{perfect derived $\infty$-category} associated to an adapted homology theory $H\colon \ccat \rightarrow \acat$, which we denote by
\[
\dcat^{\omega}(\ccat),
\]
is the $\infty$-category of perfect sheaves with respect to the $H$-epimorphism topology. 

We show that $\dcat^{\omega}(\ccat)$ is a localization of the $\infty$-category of perfect presheaves so that it is itself prestable; moreover, there is a canonical equivalence $\dcat^{\omega}(\ccat)^{\heartsuit} \simeq \acat$. The Yoneda embedding of $\ccat$ factors through sheaves and determines a fully faithful embedding $\nu\colon \ccat \rightarrow \dcat^{\omega}(\ccat)$, which following tradition we will call the \emph{synthetic analogue}.  

If we write $[1]\colon \dcat^{\omega}(\ccat) \rightarrow \dcat^{\omega}(\ccat)$ for the autoequivalence induced by $\Sigma\colon \ccat \rightarrow \ccat$, then for any perfect sheaf $X$ on $\ccat$ there is a canonical comparison map 
\[
\tau\colon \Sigma X[-1] \rightarrow X
\]
which measures the extent to which $X$ takes suspensions in $\ccat$ to loops. We write $C\tau \otimes X$ for the cofibre of this map.

\begin{theorem}[Special fibre, \ref{theorem:finite_ctau_modules_same_as_derived_category}]
\label{theorem:description_of_special_fibre_in_introduction}
The endofunctor $C\tau \otimes -\colon \dcat^{\omega}(\ccat) \rightarrow \dcat^{\omega}(\ccat)$ admits a canonical structure of a monad. Moreover, the inclusion of the heart induces a canonical equivalence
\[
\Mod_{C\tau}(\dcat^{\omega}(\ccat)) \simeq \dcat^{b}(\acat)
\]
between modules over this monad and the connective bounded derived $\infty$-category of $\acat$. 
\end{theorem}
Note that the notation $C\tau \otimes X$ is abusive, if only slightly: we do not assume that $\ccat$ is monoidal and consequently there is no natural monoidal structure on its derived $\infty$-category. In practice, the monad structure on this endofunctor is enough to use the calculus of cofibres of $\tau$ familiar from synthetic spectra and motivic homotopy theory. 

\begin{theorem}[The generic fibre, \ref{remark:generic_fibre_of_derived_category_of_homology_theory}]
\label{theorem:description_of_generic_fibre_in_introduction}
Let $\tau^{-1}\colon \dcat^{\omega}(\ccat) \rightarrow \ccat$ be the left Kan extension of the identity of $\ccat$ along the synthetic analogue $\nu\colon \ccat \rightarrow \dcat^{\omega}(\ccat)$. Then, 
\begin{enumerate}
    \item $\tau^{-1}$ is both left and right exact, 
    \item sends all of the maps $\tau\colon \Sigma X[-1] \rightarrow X$ to equivalences and 
    \item is universal with respect to the above two properties. 
\end{enumerate}
\end{theorem}

Taken together, these two results show that $\dcat^{\omega}(\ccat)$ can be thought of as a  ``$\infty$-categorical deformation'', similar to synthetic spectra, stable motives after $p$-completion, and Brantner's and Balderrama's $\infty$-categories associated to algebraic theories \cite{balderrama2021deformations}, \cite{brantner2017lubin}, \cite{pstrkagowski2018synthetic}, \cite{abstract_gh_theory}.

Through the synthetic analogue functor, the $H$-Adams spectral sequence for $c \in \ccat$ gets identified with the $\tau$-Bockstein spectral sequence associated to $\nu(c)$. This ability to ``tilt'' an Adams spectral sequence, first obtained for the Adams-Novikov spectral sequence using motivic methods by Gheorghe, Isaksen, Wang and Xu \cite{gheorghe2018special}, \cite{isaksen_stable_stems}, \cite{more_stable_stems}, has important computational consequences. For example, Burklund, Hahn and Senger use $C\tau$ methods to prove strong results about the classical Adams spectral sequence, solving - among other things - several conjectures in geometric topology \cite{burklund2019boundaries}, \cite{burklund2020inertia}, \cite{burklund2020high}, \cite{burklund2021extension}.

The current work shows that large parts of the powerful $C\tau$-formalism, known previously only for homology theories based on Adams-type ring spectra, can be set up in the generality of an arbitrary adapted homology theory, making it applicable to a much larger variety of problems. In particular, we are able to construct a general form of the Goerss-Hopkins tower. 

\begin{theorem}[Goerss-Hopkins tower, \ref{example:potential_infty_stages_are_objects_of_c}, \ref{example:potential_o_stages_are_discrete_objects}, \ref{prop: general obstruction for objects}, \ref{remark:obstructions_to_lifting_morphisms}]
Let $H\colon \ccat \rightarrow \acat$ be an adapted homology theory. Then, there exists a tower of additive $\infty$-categories 
\[
\mathcal{M}_{\infty}(\ccat) \rightarrow \ldots \rightarrow \mathcal{M}_{1}(\ccat) \rightarrow \mathcal{M}_{0}(\ccat)
\]
with the following properties: 
\begin{enumerate}
    \item for each $n \geq 0$, $\mathcal{M}_{n}(\ccat)$ is an $(n+1)$-category, 
    \item there is a canonical equivalence $\mathcal{M}_{0}(\ccat) \simeq \acat$,
    \item there is a canonical equivalence $\mathcal{M}_{\infty}(\ccat) \simeq \ccat$ and 
    \item for each of $\mathcal{M}_{n+1}(\ccat) \rightarrow \mathcal{M}_{n}(\ccat)$, there are obstructions in $\Ext^{n+3} _{\acat}$ (respectively $\Ext^{n+2} _{\acat}$, $\Ext^{n+1} _{\acat}$,  etc.) to lifting objects (respectively morphisms, homotopies, etc.)
\end{enumerate}
\end{theorem}
The convergence of the Goerss-Hopkins tower, that is, the extent to which it is a limit diagram of $\infty$-categories, is more subtle and closely related to the convergence of the $H$-based Adams spectral sequence \cite{abstract_gh_theory}. In \S\ref{subsection:goerss_hopkins_theory}, we show that it converges if $\acat$ is of finite cohomological dimension and $H$ is conservative, which is enough for the purpose of Franke's conjecture.

As there is a variety of sources of ``$\infty$-categorical deformations'', it is natural to ask if $\dcat^{\omega}(\ccat)$ is in a certain sense universal. As $\dcat^{\omega}(\ccat)^{\heartsuit} \simeq \acat$, we can think of the perfect derived $\infty$-category as a twisted version of the derived $\infty$-category of $\acat$, and expect that their universal properties will be related; this is indeed the case. 

\begin{theorem}[Universal property of the perfect derived $\infty$-category, \ref{theorem:universal_property_of_finite_derived_category}]
\label{theorem:universal_property_of_domegaccat_in_introduction}
Let $H\colon \ccat \rightarrow \acat$ be an adapted homology theory and $\dcat$ be a prestable $\infty$-category with finite limits. Suppose that we are given 
\begin{enumerate}
    \item an exact functor $G_{0}\colon \acat \rightarrow \dcat^{\heartsuit}$ of abelian categories,
    \item a left exact functor $\euscr{G}\colon \ccat \rightarrow \dcat$ together with a natural isomorphism $\euscr{G}(c)_{\leq 0} \simeq G_0(H(c))$.
\end{enumerate}
Then, there is a unique exact $G\colon \dcat^{\omega}(\ccat) \rightarrow \dcat$ such that $G \circ \nu \simeq \euscr{G}$ and $G_{0} \simeq G_{| \acat}$.
\end{theorem}
Using this result, we obtain a comparison with the theory of synthetic spectra associated to an Adams-type ring spectrum $R$, showing that there is a canonical exact, fully faithful embedding 
\[
\dcat^{\omega}(\spectra) \rightarrow \synspectra_{R}
\]
whose image is generated under finite colimits by the synthetic analogues of ordinary spectra. 

More generally, in \S\ref{subsection:digression_derived_cat_in_grothendieck_abelian_case} we describe a variant of our construction which takes as an input a particularly nice type of a homology theory and which outputs a Grothendieck prestable $\infty$-category in the sense of Lurie \cite{lurie_spectral_algebraic_geometry}[\S C]. Applied to an Adams-type homology theory; this outputs \emph{exactly} hypercomplete synthetic spectra, providing the latter with a universal property of their own. 

\subsection{Further applications}

Apart from proving Franke's algebraicity conjecture and providing a coherent monoidality result for Adams filtrations, we believe that several constructions in this paper are of independent interest. Their applications go beyond the ones given in this paper, which we hope at least partially justifies its length. To provide some motivation, as well as a guide to the reader, we mention here some of these recent developments.

The derived infinity category associated to a homology theory constructed in the current work is used in an essential way in Burklund's breakthrough work on multiplications on Moore spectra \cite{burklund2022moore}. The latter paper shows that $S^{0}/2^{\lceil \nicefrac{3}{2}k+1 \rceil}$ and $S^{0}/p^{k+1}$ for $p$ odd can be made into $\mathbf{E}_{k}$-algebras, solving a long standing problem in algebraic topology. To obtain the improved bound at odd primes, Burklund uses the derived $\infty$-category of spectra associated to the class of $S^{0}/p$-split epimorphisms, which can be made symmetric monoidal by \cref{proposition:perfect_derived_infty_cat_of_r_split_maps_acquires_a_monoidal_structure}. More generally, Burklund proves a result about $\mathbf{E}_k$-structures on $\mathbb{I}/v^n$, where $\mathbb{I}$ is the monoidal unit of an $\mathbf{E}_m$-monoidal $\infty$-category and $v$ a self-map of $\mathbb{I}$, which rests on the fact that the constructions given in this paper apply to an arbitrary stable $\infty$-category. 

The adapted factorization of \cref{theorem:adapted_factorization_of_homology_theories} is used by Burklund and the second author to describe the $E_2$-term of various Adams spectral sequences in terms of quiver representations \cite{burklundpstragowski2023}. In particular, one obtains an elegant description of the $E_2$-term of the Adams spectral sequence sequence based on integral homology in terms of the classical Adams spectral sequence. 

The appendix of \cite{burklundpstragowski2023} develops further functoriality and methods of comparing the derived $\infty$-categories introduced in the current work. 

Several of algebraicity equivalences between homotopy categories constructed in the current work, in particular chromatic algebraicity of \cref{example:chromatic_algebraicity_in_introduction}, were shown to be symmetric monoidal in the work of Barkan \cite{barkan2023}.

\subsection{Summary of contents} 

In \S\ref{subsection:history_of_algebraicity_results} right below, we provide a brief background on algebraicity results for homotopy categories of stable $\infty$-categories. The sections \S\ref{subsection:conventions} and  \S\ref{subsection:acknowledgements} are devoted to conventions and acknowledgements. 

In \S\ref{section:homology_theories_and_adams_spectral_sequences}, we review the notion of a homology theory and the associated Adams spectral sequence. We characterize adapted homology theories in terms of the Freyd envelope. 

In \S\ref{section:classification_of_adams_spectral_sequences}, we discuss Miller's approach to the Adams spectral sequence, and classify adapted homology theories in terms of monomorphism classes. 

In \S\ref{section:prestable_freyd_envelope}, we introduce the notion of perfect and almost perfect presheaves on an additive $\infty$-category $\ccat$. Their $\infty$-categories, which we denote by $A_{\infty}^{\omega}(\ccat)$ and $A_{\infty}(\ccat)$ and call the (perfect) prestable Freyd envelope, are direct $\infty$-categorical analogues of the classical Freyd envelope. We show that these $\infty$-categories inherit favourable categorical properties from $\ccat$, such as the existence of limits. 

If $\ccat$ is stable, we prove in \S\ref{subsection:prestable_enhancements_to_homological_functor} that its prestable envelope can be characterized completely in terms of classifying homology theories together with certain left exact homotopy-coherent lifts, which we call prestable enhancements. In \S\ref{subsection:thread_structure_on_prestable_freyd_envelope}, we introduce the map $\tau$ and prove that its cofibres admit canonical monad structures, developing a non-monoidal version of the calculus of cofibres of $\tau$.  

\S\ref{section:derived_infty_cats} is devoted to the central construction of the paper, the derived $\infty$-category $\dcat^{\omega}(\ccat)$ associated to an adapted homology theory $H\colon \ccat \rightarrow \acat$. In \S\ref{subsection:bounded_derived_cat_of_abelian_cat}, we describe the bounded derived $\infty$-category of an abelian category with enough injectives as an $\infty$-category of perfect sheaves, and prove its universal property. In \S\ref{subsection:perfect_derived_infty_cat_of_a_stable_infty_cat}, we establish basic properties of $\dcat^{\omega}(\ccat)$ by showing that sheafification with respect to the $H$-epimorphism topology preserves perfect sheaves. In \S\ref{subsection:homology_adjunction_and_thread_structure}, we describe an adjunction between the two derived $\infty$-categories and relate the algebraic one to $C\tau$-modules. 

In \S\ref{subsection:monoidality_of_the_adams_filtration}, we relate the Postnikov towers in the derived $\infty$-category $\dcat^{\omega}(\ccat)$ to the $H$-Adams spectral sequence. As an application, we construct functorial Adams filtrations for an arbitrary homology theory, and in \S\ref{subsection:postnikov_towers_and_adams_filtration} we show that the Adams filtration associated to a spectrum with a right unital multiplication is lax symmetric monoidal. 

\S\ref{section:further_topics_in_derived_infty_cats} is devoted to deeper topics in the study of the derived $\infty$-category associated to a homology theory. In \S\ref{subsection:completing_derived_infty_cats}, we show that under certain assumptions, prestable completions of the derived $\infty$-category an be described using almost perfect hypersheaves. In \S\ref{subsection:hypercompletion_and_locality}, we describe the interaction between hypercompletion of sheaves and locality in the sense of Bousfield. In \S\ref{subsection:goerss_hopkins_theory}, we construct the Goerss-Hopkins tower associated to an adapted homology theory and prove its basic properties. In \S\ref{subsection:digression_derived_cat_in_grothendieck_abelian_case}, we describe a variant of our construction which always produces a Grothendieck prestable $\infty$-category. In \S\ref{subsection:case_of_modules_and_synthetic_spectra}, we compare our construction to the ones already appearing in the literature. 

\S\ref{section:proving_algebraicity} is devoted to the proof of Franke's conjecture. In \S\ref{subsection:algebraic_model} we describe our algebraic model in terms of differential objects. In \S\ref{subsection:splitting_of_abelian_categories_and_the_bousfield_functor}, we show that if $H\colon \ccat \rightarrow \acat$ is an adapted homology theory such that $\acat$ admits a splitting of order $(q+1)$, then $H$ admits a partial inverse known as the Bousfield functor. In \S\ref{subsection:bousfield_adjunction}, we show that this induces a monadic adjunction between $C\tau^{q+1}$-modules and the derived $\infty$-category of $\acat$, and identify the monad explicitly in \S\ref{subsection:the_truncated_thread_monad}. We combine this with a Goerss-Hopkins obstruction theory argument to complete the proof of the conjecture in 
\S\ref{subsection:complete_proof_of_frankes_conjecture}. 

In \S\ref{section:applications_to_diagram_categories_and_equivariant_objects}, we give examples of adapted homology theories to which algebraicity conjecture applies. We discuss module $\infty$-categories and $E(n)$-local spectra in, respectively, \S\ref{subsection:modules_over_ring_spectra} and \S\ref{subsection:chromatic_algebraicity}. We extend these results to diagram $\infty$-categories through the study of their cohomological dimension in  \S\ref{subsection:diagram_infty_categories}.

\subsection{Background on algebraicity results in stable homotopy theory} 
\label{subsection:history_of_algebraicity_results}

The prototypical example of a algebracity result in homotopy theory is Quillen's description of rational homotopy theory \cite{quillen1969rational}. In the language of $\infty$-categories, Quillen proves that there is an equivalence
\[
\spaces_{*, \geq 1}^{\mathbb{Q}} \simeq \mathrm{dgLie}_{\geq 1}^{\mathbb{Q}}
\]
between the $\infty$-categories of pointed, simply-connected rational spaces and of connected differential graded Lie algebras over $\mathbb{Q}$. Passing to the stable setting, this reduces to the (much more simple) equivalence
\[
\spectra_{\mathbb{Q}} \simeq \dcat(\mathbb{Q})
\]
between rational spectra and the stable derived $\infty$-category of the rationals. 

The idea of algebraic description of localizations of spectra was taken up by Bousfield \cite{bousfield1985homotopy}, who studied the homotopy category of $KU_{(p)}$-local spectra, where $KU_{(p)}$ is complex topological $K$-theory localized at an odd prime. Due to the Adams splitting 
\[
KU_{(p)} \simeq \bigoplus_{i=0}^{p-2} \Sigma^{2i} E(1)
\]
this is equivalent to studying the localization at the first Johnson-Wilson homology. 

Bousfield's methods to understand $h_1(\spectra_{E(1)})$ go back to Adams' idea of using $e$ and $d$-invariants to understand the $J$-homomorphism \cite{AdamsJIV}. More precisely, Bousfield uses the $E(1)$-based Adams spectral sequence of signature
\[
E_2^{p, q}=\Ext_{E(1)_*E(1)}^p(E(1)_*X, E(1)_*Y[q]) \Rightarrow [L_{E(1)}X, L_{E(1)}Y]_{p+q}.
\]
Here the Ext-groups are taken in the abelian category of $E(1)_*E(1)$-comodules, $L_{E(1)}$ is the Bousfield localization and brackets denote homotopy classes of maps of spectra. 

The key observation is that the if $p$ is an odd prime,  then the abelian category $\Comod_{E(1)_*E(1)}$ has cohomological dimension $2$. Thus, the above Adams spectral sequence has a horizontal vanishing line and the only possible non-trivial differential is
\[
d_2 \colon \Hom_{E(1)_*E(1)}(E(1)_*X, E(1)_*Y[q]) \to \Ext_{E(1)_*E(1)}^2(E(1)_*X, E(1)_*Y[q-1]).
\]
The second crucial observation of Bousfield was that since $E(1)_{*}E(1)$ is concentrated in degrees divisible by $2p-2$, the abelian category of comodules has a splitting 
\[
\Comod_{E(1)_*E(1)} \simeq \bcat_{0} \times \bcat_{1} \times \cdots \times \bcat_{2p-3}
\]
as a product of its Serre subcategories of comodules concentrated in a single degree modulo $2p-2$\footnote{This is a small lie. Instead of $E(1)_{*}E(1)$-comodules, Bousfield works with $KU_{*}$-modules equipped with suitable Adams operations and obtains the splitting using weight decomposition associated to the latter. This argument is equivalent to the one using the grading on $E(1)_{*}E(1)$.}.

Using the $d_2$ differential in the Adams spectral sequence, Bousfield defines for any $E(1)$-local $X$ a class 
\[
\kappa_{X} \in \Ext_{E(1)_*E(1)}^2(E(1)_*X, E(1)_*X[-1]).
\] 
He then uses the splitting and the Adams spectral sequence to prove that equivalence classes of $E(1)$-local spectra are in bijection with the isomorphism classes of pairs $(M, \kappa)$, where $M$ is a comodule and $\kappa \in \Ext_{E(1)_*E(1)}^2(M, M[-1])$ \cite[ 9.1]{bousfield1985homotopy}. Intuitively, Bousfield's class $\kappa$ is a higher analogue of $d$ and $e$-invariants of Adams, since $d$ and $e$-invariants live in $\Hom$ and $\Ext^1$-groups, respectively, which instead of classifying morphisms classifies objects. 

Bousfield also gives a partial descriptions of homotopy classes between $E(1)$-local spectra. More precisely, using the Adams filtration, Bousfield defines a bigraded category associated to $h_1\spectra_{E(1)}$ and a bigraded category where the objects are the pairs $(M, \kappa)$ and shows that they are equivalent. This gives a complete algebraic description of the homotopy category up to extension problems.

Note that Bousfield comes very close, but does not provide an algebraic category which is equivalent to $h_1\spectra_{E(1)}$. This was one of the main motivations for searching for a category which would improve Bousfield's classification. 

Inspired by the ideas of Bousfield as well as algebraic geometry, Jens Franke created a general setup to attack such algebraic classification problems \cite{franke1996uniqueness}. His motivation was the realization functor of Beilinson, Berstein and Deligne \cite{BBD}. 

Franke works with stable model categories and derivators rather than stable $\infty$-categories, but his set-up is virtually identical to the one we use in the current work; in particular, he defines the notion of an adapted homology theory $H\colon \ccat \rightarrow \acat$.  Assuming that $\acat$ admits a splitting as a product of $(q+1)$ Serre subcategories (in a way compatible with its local grading) and that it is of cohomological dimension less than $q+1$, Franke attempts to construct an algebraic model for the homotopy category $h_1\ccat$. 

The algebraic model in question is given by the category of differential objects; that is, of pairs $(a,\partial)$, where $a \in \acat$ and $\partial \colon a \to a[1]$ is a differential satisfying $\partial^{2} = 0$. These have a natural notion of homology and hence of quasi-isomorphisms.

Franke constructs a model structure on differential objects with quasi-isomorphisms as weak equivalences, and this has an underlying stable $\infty$-category which we denote in the current work by $\dcat^{per}(\acat)$. Using the aforementioned ideas of Beilinson, Berstein and Deligne, Franke constructs a realization functor 
\[
\mathcal{R} \colon h_1 \dcat^{per}(\acat) \to h_1\ccat,
\]
which he then attempts to show is an equivalence. Unfortunately, the latter argument has a gap. 

We were not able to resolve the question of whether Franke's $\mathcal{R}$ is an equivalence, and our arguments in the main text proceed using very different methods. However, in the interest of history, let us sketch the ingenious construction of Franke's realization functor.

Using the splitting 
\[
\acat \simeq \prod_{\phi \in \mathbb{Z}/(q+1)} \acat_{\phi}
\]
Franke observes that to give a differential object is equivalent to give a chain complex of objects $A_{\bullet}$ in $\acat$ together with a periodicity isomorphism
\begin{equation}
\label{equation:periodic_chain_complex_in_history}
A_{\bullet}[q+1] \simeq A_{\bullet-q-1}.
\end{equation}
To any such periodic chain complex, Franke associates a certain diagram in $\ccat$ of the form  
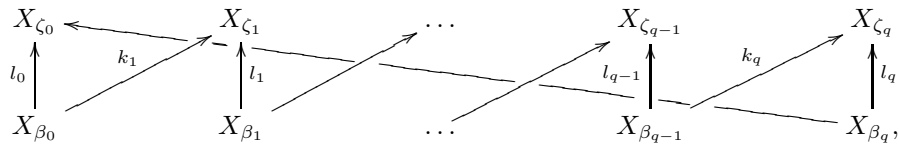
\[
\xymatrix{ X_{\zeta_0}
    \ar@{<-}[drrrrrrrr]%
    |<<<<<<<<<<<<<<<<<<{\text{\Large \textcolor{white}{$\blacksquare$}}}%
    |<<<<<<<<<<<<<<<<<<<<<<<{\text{\Large \textcolor{white}{$\blacksquare$}}}
    |<<<<<<<<<<<<<<<<<<<<<<<<<<<<<<<<<<<<<<<{\text{\Large \textcolor{white}{$\blacksquare$}}}
    |>>>>>>>>>>>>>>>>>>>>>>>>>>>>>>>>>>>>>>>>{\text{\Large \textcolor{white}{$\blacksquare$}}}
    |>>>>>>>>>>>>>>>>>>>>>>>>>{\text{\Large \textcolor{white}{$\blacksquare$}}}
    |>>>>>>>>>>>>>>>>>>{\text{\Large \textcolor{white}{$\blacksquare$}}}
& & X_{\zeta_1} & & \ldots & & X_{\zeta_{q-1}} & & X_{\zeta_{q}} \\   X_{\beta_0} \ar[u]^{l_0} \ar[urr]^{k_1} & & X_{\beta_1} \ar[u]_{l_1} \ar[urr] & & \ldots \ar[urr] & & X_{\beta_{q-1}} \ar[u]^{l_{q-1}} \ar[urr]^{k_{q}} & & X_{\beta_{q}}, \ar[u]_{l_{q}} }
\]
which he calls \emph{crowned diagrams}. These are subject to conditions which are a homotopy-theoretic incarnation of how kernels and images of maps in the given chain complex  are related to one another. 

More precisely, the conditions are that after applying the homology $H \colon \ccat \to \acat$, the objects $X_{\zeta_{i}}$ give the cycles of $A_{\bullet}$ and $X_{\beta_{i}}$ give the boundaries. Moreover, the maps $l_i$ induce the inclusion of the boundaries into cycles. Finally, the mapping cone sequences of $k_i$ yield the extensions relating the chains with the boundaries and cycles. 

Franke shows that for any periodic chain complex such a crowned diagram exists; in fact that it is essentially unique in a suitable homotopy category of diagrams. Taking the homotopy colimit of this unique crowned diagram defines the realization functor 
\[
\mathcal{R}\colon d\acat \rightarrow h_{1} \ccat 
\]
whose source is the category of differential objects. Franke shows that this inverts quasi-isomorphisms and hence defines a functor 
\[
\mathcal{R}\colon h_{1} \dcat^{per}(\acat) \rightarrow h_{1} \ccat,
\]
he then tries to show that this functor is an equivalence  \cite[Section 2, Page 56, Theorem 5]{franke1996uniqueness}.

An error in the proof shows up in \cite[Page 62, Proposition 2]{franke1996uniqueness}. The strategy is to prove that the functor $\mathcal{R}$ preserves certain distinguished triangles which appear in the construction of the corresponding Adams spectral sequence in both the source and target categories. If one could show that $\mathcal{R}$ preserves such distinguished triangles as claimed in \cite[Page 62, Proposition 2]{franke1996uniqueness}, then one would be able to compare the Adams spectral sequences and show that $\mathcal{R}$ is fully faithful. The essential surjectivity would then easily follow, too. 

However, it is not clear at all that $\mathcal{R}$ preserves such distinguished triangles. More precisely, such triangles are of the form
\[
A \xrightarrow{f} B \xrightarrow{g} C \xrightarrow{d} \Sigma A,
\]
where the connecting morphism $d$ induces zero on the homology theory $H$. The argument given in the proof of \cite[Page 62, Proposition 2]{franke1996uniqueness} shows that 
\[\mathcal{R}(A) \xrightarrow{\mathcal{R}(f)} \mathcal{R}(B) \xrightarrow{\mathcal{R}(g)} \mathcal{R}(C)\]
is a part of a distinguished triangle, but does not justify why $\mathcal{R}(d)$ is the correct connecting morphism. 

This is a subtle technical issue; it is difficult to understand maps which induce zero on homology as one cannot approach them using the underlying algebra. Experts in triangulated categories will immediately notice that such questions can be very hard to solve; this is directly related to the non-functoriality of mapping cones.

The above gap in the argument of \cite[Page 62, Proposition 2]{franke1996uniqueness} was first discovered in the paper \cite{patchkoria2012}. By putting further conditions on $\acat$; this gap can be fixed: the papers \cite{GreenleesS1, patchkoriaKU, patchkoria2017exotic} fix this gap in the case when the abelian category $\acat$ has cohomological dimension at most 3. In particular, \cite{patchkoria2017exotic} provided an algebraic model for $h_1\spectra_{E(1)}$ at primes $p \geq 5$, strengthening the result of Bousfield. 

To fix the gap and show that Franke's realization functor $\mathcal{R}$ preserves appropriate distinguished triangles, one has to find crowned diagram models for mapping cones of differential objects. Unfortunately, even in the case of cohomological dimension $2$ such models become combinatorially involved, see \cite[Proof of Lemma 6.2.1]{patchkoria2012}.

There have also been results in the negative direction, where one shows that the homotopy category $h \ccat$ of a given stable $\infty$-category \emph{cannot} be equivalent to a derived category of an abelian category \cite{schwede2008algebraic}. 

An even stronger form of a negative answer would be to show that whenever we have $h \ccat \simeq h \dcat$ as triangulated categories, then $\ccat \simeq \dcat$ as stable $\infty$-categories. For example, Schwede's Annals paper shows that the $\infty$-category of spectra has this property \cite{schwede2007stable}. This was extended by Roitzheim to the $E(1)$-local category when $p=2$, providing a complement to Bousfield's work at odd primes \cite{Roitzheim}.

A solution to the algebraicity question of a slightly different flavour appears in the work of Barthel, Schlank and Stapleton, who consider the ultraproduct of the $\infty$-categories of $E(n)$-local spectra, indexed by the primes, and show that it is equivalent to the ultraproduct of Franke's models in the strongest possible sense, as a symmetric monoidal $\infty$-category \cite{barthel2017chromatic}. Their work was extended to cover the $K(n)$-local case in \cite{bss_monochromatic}.

Our approach to the question of algebraic models based on a homology theory, which we decided to call \emph{Franke's algebraicity conjecture} to highlight our intellectual debt\footnote{We would like to thank Jens Franke for agreeing on our use of this terminology.}, uses similar ideas but is not directly based on the realization functor $\mathcal{R}$. Instead, we base our arguments on Goerss-Hopkins theory. The latter technology, based on Dwyer-Kan-Stover and Bousfield's work on homotopy-theoretic resolutions, was used in the second author's thesis to give algebraic models for the homotopy categories of $E(n)$-local spectra at large primes \cite{pstragowski_chromatic_homotopy_algebraic}. 

\subsection{Conventions}
\label{subsection:conventions}

Our set-theoretic conventions are the same as those in Lurie's book \cite{higher_topos_theory}. To be more precise, we assume that there is a fixed universe $\spaces$ of small spaces, and unless stated otherwise we will assume that our $\infty$-categories have small mapping spaces (but need not be themselves necessarily small). There are places where we pass to a larger universe $\widehat{\spaces}$ and hence encounter $\infty$-categories with large mapping spaces, but we will be very explicit about this. 

We will generally denote derived $\infty$-categories (for example, of an abelian category) using $\dcat(-)$; however, unlike most sources, we reserve an unadorned letter to denote the \emph{connective} derived $\infty$-category, which is prestable rather than stable. In symbols, our convention is that 
\[
\dcat(\ccat) \colonequals \dcat_{\geq 0}(\ccat).
\]
This non-standard choice comes down to practicality; the stable derived $\infty$-categories appear here very little, while their prestable variants are our main focus. 

\subsection{Acknowledgements}
\label{subsection:acknowledgements}

As we hope the title makes clear, we  have an incredible intellectual debt to the work of Jens Franke, and would like to thank him for his inspirational paper.  

The first named author would like to thank Emanuele Dotto, Rune Haugseng, Drew Heard, Julie Rasmusen, Constanze Roitzheim and Stefan Schwede for useful conversations and comments.  

I (the second named author) would like to thank Paul Goerss, who first suggested to me to look into algebraicity problems, and to whom an a early version of this project was described around the time I was graduating in 2019. I would like to thank Jesper Grodal for proof-reading a fellowship application about this project, helping many of these ideas crystallize. I would like to thank Robert Burklund, whose many ideas and whose viewpoint on ``synthetic'' things were a huge inspiration for the direction taken here\footnote{We hope at least parts of this paper will be useful in Robert's upcoming book on synthetic spectra.}. I would like to thank Lars Hesselholt and Jeremy Hahn for reading an earlier version of this work and for the helpful comments they provided. 

\section{Homology theories and Adams spectral sequences}
\label{section:homology_theories_and_adams_spectral_sequences}

In this section we study homological functors on stable $\infty$-categories with values in abelian categories. We introduce the notion of an adapted homology theory which gives rise to a version of the Adams spectral sequence. We discuss the universal homology theory and its target, the Freyd envelope, and use it to characterize adapted homology theories.

\subsection{Homological functors} 

In this very short section, we fix our terminology concerning homological functors and homology theories. Everything here is very classical; an excellent resource is Neeman's book \cite{neeman2001triangulated}.

\begin{definition}
\label{definition:homological_functor}
Let $\ccat$ be a stable $\infty$-category and $\acat$ be an abelian category. We say a functor $H\colon \ccat \rightarrow \acat$ is is \emph{homological} if

\begin{enumerate}
    \item  $H$ is additive and 
    \item if
\[
c \rightarrow d \rightarrow e
\]
is a cofibre sequence in $\ccat$, then 
\[
H(c) \rightarrow H(d) \rightarrow H(e)
\]
is exact.
\end{enumerate}

\end{definition}

\begin{example}
For any object $c \in \ccat$, the functor $[c, -]\colon \ccat \rightarrow \abeliangroups$ is homological. 
\end{example}
In practice, we usually have a little bit more structure on homological functors, namely certain compatibility with the grading. This will be important in the current work and thus we make it explicit.

\begin{definition}
A \emph{local grading} on an $\infty$-category $\ccat$ is an autoequivalence 
\[
[1]_{\ccat}\colon \ccat \rightarrow \ccat.
\]
A \emph{locally graded} $\infty$-category is a pair of an $\infty$-category and a local grading.
\end{definition}

\begin{notation}
\label{notation:local_grading_same_as_action}
Since $\mathbb{Z}$ is the free group object in spaces generated by a single point, the forgetful functor 
\[
\Fun(\mathrm{B}\mathbb{Z}, \euscr{C}at_{\infty}) \rightarrow \{ \ \textnormal{locally graded $\infty$-categories} \  \}
\]
is an equivalence of $\infty$-categories. Informally, this is saying that given an equivalence on an $\infty$-category $\ccat$, there is an essentially unique way to coherently choose its composites and their inverses. If $\ccat$ is locally graded and $n \in \mathbb{Z}$, we will write
\[
[n]_{\ccat}\colon \ccat \rightarrow \ccat
\]
for the $n$-th composite of the local grading
\end{notation}

\begin{remark}
Besides the definition and the observation made in \cref{notation:local_grading_same_as_action}, there are many other equivalent ways of expressing the data of a local grading, see \cite{lurierotation}[2.4.4].
\end{remark}

The following are some examples of local gradings which will be of interest in this paper. 
\begin{example}
Let $\ccat$ be a stable $\infty$-category. Then, the suspension functor $\Sigma\colon \ccat \rightarrow \ccat$ is a local grading.  
\end{example}

\begin{example}
Let $R_{*}$ be a graded ring. Then, the category $\Mod(R_{*})$ of graded $R$-modules has a local grading determined by $(M[1])_{k} = M_{k-1}$.
\end{example}

In what follows, we will assume that $\ccat$ is a stable $\infty$-category, always considered as a locally graded $\infty$-category using the suspension functor, and that $\acat$ is an abelian category with some choice of a local grading. 

\begin{definition}
\label{definition:homology_theory}
We say a functor $H\colon \ccat \rightarrow \acat$ of locally graded $\infty$-categories is a \emph{homology theory} if its underlying functor is homological. 
\end{definition}

\begin{remark}
Note that to promote a homological functor $H\colon \ccat \rightarrow \acat$ to a homology theory is the same as to choose an isomorphism
\[
H(\Sigma c) \simeq H(c)[1]_{\acat}
\]
natural in $c \in \ccat$. Note that this is additional, non-trivial data; in particular a homology theory is not given by a datum of a functor alone.
\end{remark}

\begin{example}
Let $\acat = \textnormal{Vect}(\mathbb{F}_{p})$ be the abelian category of graded vector spaces. Then, the mod $p$ homology functor 
\[
H_{*}(-, \mathbb{F}_{p})\colon \spectra \rightarrow \textnormal{Vect}(\mathbb{F}_{p})
\]
is canonically a homology theory. 
\end{example}

\subsection{Adams spectral sequence and adaptedness} 

In practice, among all homology theories $H\colon \ccat \rightarrow \acat$, the well-behaved are the ones where the injectives of $\acat$ can be compatibly lifted to objects of $\ccat$, giving rise to the Adams spectral sequence. To make this precise, we make the following definition.  

\begin{definition}
\label{definition:injectives_in_a_stable_category}
Let $H\colon \ccat \rightarrow \acat$ be a homology theory and let $i \in \acat$ be an injective. An \emph{injective
lift $i_{\ccat}$ associated to $i$} is an object representing the functor 
\[
\Hom_{\acat}(H(-), i)\colon \ccat^{op} \rightarrow \abeliangroups.
\]
in the homotopy category $h \ccat$. Explicitly, it is an object $i_{\ccat} \in \ccat$ together with a map $H(i_{\ccat}) \rightarrow i$ such that for any $c \in \ccat$, the induced composite
\[
[c, i_{\ccat}] \rightarrow \Hom_{\acat}(H(c), H(i_{\ccat})) \rightarrow \Hom_{\acat}(H(c), i)
\]
is an isomorphism of abelian groups.
\end{definition}

\begin{example}
\label{example:em_spectrum_associated_injective_for_mod_p_homology}
Consider $H_{*}(-, \mathbb{F}_{p})\colon \spectra \rightarrow \textnormal{Vect}(\mathbb{F}_{p})$. In the latter category, the degree-zero one-dimensional vector space $\mathbb{F}_{p}$ is injective, and we have 
\[
\Hom_{\mathbb{F}_{p}}(H_{*}(X, \mathbb{F}_{p}), \mathbb{F}_{p}) \simeq H^{*}(X, \mathbb{F}_{p}) \simeq [X, H\mathbb{F}_{p}],
\] 
Thus, the Eilenberg-MacLane spectrum $H \mathbb{F}_{p}$ is the associated injective lift. 
\end{example}

\begin{definition}
\label{definition:lifting_injective}
We say a homology theory $H\colon \ccat \rightarrow \acat$ has \emph{lifts for injectives} if 
\begin{enumerate}
    \item $\acat$ has enough injectives and 
    \item any injective $i \in \acat$ lifts to some $i_{\ccat}$.
\end{enumerate}
\end{definition}
The fact of having lifts for injectives already has some pleasant consequences. 

\begin{lemma}
\label{lemma:homology_theories_with_lifts_of_injectives_preserve_coproducts}
Let $H\colon \ccat \rightarrow \acat$ be a homology theory which has lifts for injectives. Then, $H$ preserves all, possibly infinite, direct sums which exist in $\ccat$. 
\end{lemma}

\begin{proof}
Let $c_{\alpha}$ be a family of objects of $\ccat$ which admits a direct sum. Since $\acat$ has enough injectives, it is enough to show that for any injective $i \in \acat$ we have 
\[
\Hom_{\acat}(H(\bigoplus c_{\alpha}), i) \simeq \prod \Hom_{\acat}(H(c_{\alpha}), i). 
\]
Both sides can be rewritten in terms of homotopy classes of maps into the associated injective $i_{\ccat}$, and the result follows from the fact that the projection $\ccat \rightarrow h\ccat$ onto the homotopy category preserves all direct sums which exist in $\ccat$. 
\end{proof}

Note that the condition of having lifts for all injectives  might seem quite restrictive at first glance; this is decidedly not the case. In fact, we will now show that \cref{lemma:homology_theories_with_lifts_of_injectives_preserve_coproducts} already characterizes such homology theories whenever $\ccat$ is presentable. Since plenty of stable $\infty$-categories occurring in practice are presentable, this provides a whole zoo of examples.

\begin{proposition}[Brown representability]
\label{proposition:brown_representability}
Let $\ccat$ be a presentable stable $\infty$-category and let $E\colon \ccat^{op} \rightarrow \abeliangroups$ be a homological functor which takes arbitrary direct sums in $\ccat$ to products of abelian groups. Then, $E$ is representable in the homotopy category $h \ccat$. 
\end{proposition}

\begin{proof}
When $\ccat$ is compactly generated, this is a classical result of Neeman \cite{neeman1996grothendieck}. For arbitrary presentable $\ccat$, choose a regular cardinal $\kappa$ such that $\ccat$ is generated by its subcategory $\ccat^{\kappa}$ of $\kappa$-compact objects \cite[5.5.1]{higher_topos_theory}. Then the functor $\Ind(\ccat^{\kappa}) \rightarrow \ccat$ is a localization of stable $\infty$-categories at some saturated class $S$ of maps in $\Ind(\ccat^{\kappa})$. 

By Neeman's representability, the composite
\[
(\Ind(\ccat^{\kappa}))^{op} \rightarrow (\ccat)^{op} \xrightarrow{E} \abeliangroups
\]
is representable by some $X \in \Ind(\ccat^{\kappa})$. However, since the above composite factors through $\ccat$, we see that $X$ is in fact $S$-local and so equivalent to an object of $\ccat$. This ends the argument. 
\end{proof}

\begin{remark}
We were very surprised that this simple argument given above does not seem to appear in the literature. The above reduction to the compactly-generated case is analogous to the reduction to presheaves appearing in characterization of representable functors on presentable $\infty$-categories \cite[5.5.2.2]{higher_topos_theory}.
\end{remark}

\begin{corollary}
\label{corollary:coproduct_preserving_homology_theory_on_presentable_cat_has_lifts_of_injectives}
Suppose that $\ccat$ is presentable and let $H\colon \ccat \rightarrow \acat$ be a homology theory which preserves arbitrary direct sums. Then, $H$ has lifts for injectives. 
\end{corollary}

\begin{proof}
This is immediate from \cref{proposition:brown_representability}, as the functor $\Hom_{\acat}(H(-), i)\colon \ccat^{op} \rightarrow \abeliangroups$ satisfies the conditions of Brown representability. 
\end{proof}

A more subtle point is that having lifts for injectives is still not enough to have a well-behaved Adams spectral sequence, even when both $\ccat$ and $\acat$ are very nice. The issue is that if $i_{\ccat}$ is an associated injective, then we have a comparison map $H(i_{\ccat}) \rightarrow i$, but this need not be an isomorphism, as the following example shows. 

\begin{warning}
\label{warning:homology_theory_with_non_compatible_injectives}
We have seen in \cref{example:em_spectrum_associated_injective_for_mod_p_homology} that the injective associated to $\mathbb{F}_{p}$ via the homology theory $H_{*}(-, \mathbb{F}_{p})\colon \spectra \rightarrow \textnormal{Vect}(\mathbb{F}_{p})$ is given by the Eilenberg-MacLane spectrum $H \mathbb{F}_{p}$.
However, $H_{*}(H \mathbb{F}_{p}, \mathbb{F}_{p}) \simeq A_{*}$ is the dual Steenrod algebra, and the comparison map $A_{*} \rightarrow \mathbb{F}_{p}$ is given by the counit of the Hopf algebra structure, which is not an isomorphism.
\end{warning}
In light of the above warning, we make the following definition. 
\begin{definition} \label{definition:adapted_homology_theory}
We say a homology theory $H\colon \ccat \rightarrow \acat$ is \emph{adapted} if: 
\begin{enumerate}
    \item $\acat$ has enough injectives,
    \item any injective $i \in \acat$ lifts to an associated injective $i_{\ccat}$ (i.e. $H$ has lifts for injectives) and,
    \item the structure map $H(i_{\ccat}) \rightarrow i$ is an isomorphism.
\end{enumerate}
\end{definition}

\begin{remark}
\label{remark:adapted_homology_theories_under_other_names_and_comparison_with_franke}
The concept of adapted homology theories has been studied by different authors under different names. In Franke's manuscript, these homology theories are said to \emph{possess Eilenberg-MacLane objects for injectives} \cite{franke1996uniqueness}. The above condition, spelled out in the setting of an arbitrary triangulated category, appears as early as the work of Moss \cite{moss1968composition}.

To be more precise, in \cite{franke1996uniqueness}[p.~52, Definition 10] Franke only asks that every object $a \in \acat$ can be embedded into an injective $i$ which admits an injective lift, not that all injectives do. This is very close to asking for $H$ to be adapted, as under Franke's condition any injective of $\acat$ is necessarily a direct summand of one that admits a lift. Thus, these two conditions are equivalent whenever $\ccat$ is idempotent-complete; for example, if it admits either countable coproducts or products. 
\end{remark}

As we observed in \cref{lemma:homology_theories_with_lifts_of_injectives_preserve_coproducts}, a homology theory with lifts of injectives, in particular an adapted one, preserves all direct sums which exist in $\ccat$. The corresponding statement for infinite products is slightly more subtle, but let us observe products of injective objects do get preserved. 

\begin{lemma}
\label{lemma:an_injective_associated_to_a_product_is_a_product_associated_injectives}
Let $H\colon \ccat \to \acat$ be an adapted homology theory and $i = \prod i^{\alpha}$ be an arbitrary product of injectives of $\acat$. Then, any choice of maps $i_{\ccat} \rightarrow i^{\alpha}_{\ccat}$ of associated injectives inducing the projections $i \rightarrow i_{\alpha}$ on homology induces an equivalence
\[
i_{\ccat} \simeq \prod i_{\ccat}^{\alpha}.
\]
\end{lemma}

\begin{proof}
In a stable $\infty$-category, products are detected in the homotopy category, where we have 
\[
[c, i_{\ccat}] \simeq \Hom_{\acat}(H(c), i) \simeq \prod \Hom_{\acat}(H(c), i^{\alpha}) \simeq \prod [c, i^{\alpha}_{\ccat}]
\]
as needed. 
\end{proof}

\begin{corollary}
\label{corollary:an_adapted_homology_theory_commutes_with_producs_of_injectives}
Let $H\colon \ccat \rightarrow \acat$ be an adapted homology theory and let $i^{\alpha} \in \acat$ be a family of injectives. If the product $\prod i^{\alpha}$ exists in $\acat$, then the product $\prod i^{\alpha}_{\ccat}$ exists in $\ccat$ and moreover 
\[
H(\prod i^{\alpha}_{\ccat}) \simeq \prod H(i^{\alpha}_{\ccat}).
\]
In other words, an adapted homology theory commutes with taking homology of injectives. 
\end{corollary}

\begin{proof}
This is immediate from \cref{lemma:an_injective_associated_to_a_product_is_a_product_associated_injectives}, since the injective of $\ccat$ associated to $\prod i^{\alpha}$ is the needed product. 
\end{proof}

Let us observe that it if we already know that $H$ is adapted, it is easy to recognize injectives. 

\begin{lemma}
\label{lemma:injectives_recognized_by_homology_and_locality}
Let $H\colon \ccat \rightarrow \acat$ be an adapted homology theory and suppose that $c \in \ccat$ is such that $H(c)$ is injective. Then, the following are equivalent: 
\begin{enumerate}
    \item $c$ is an injective lift associated to $H(c)$ or 
    \item $c$ is $H$-local.
\end{enumerate}
\end{lemma}

\begin{proof}
It is clear that any injective lift is $H$-local, so we only have to show that $(2) \Rightarrow (1)$.

Let $j$ be an injective lift associated to $H(c)$, which exists by adaptedness. Then, the identity of $H(c)$ determines a homotopy class of maps $c \rightarrow j$ which by construction is an isomorphism on homology. As both the source and target are $H$-local, it must be an equivalence. 
\end{proof}

The importance of adapted homology theories is that they allow an Adams spectral sequence, as we now recall. 

\begin{construction}[Adams spectral sequence]
\label{construction:adams_spectral_sequence_associated_to_a_homology_theory}
Let $H\colon \ccat \rightarrow \acat$ be an adapted homology theory and $d \in \ccat$ be an object. 
Writing $d = d^{0}$, we can find an embedding $H(d^{0}) \hookrightarrow i^{0}$ into an injective, which determines a map $d \rightarrow i^{0}_{\ccat}$ into a corresponding associated injective object of $\ccat$. Proceeding inductively by setting $d^{i+1} = cofib(d^{i} \rightarrow i^{i}_{\ccat})$, we construct an \emph{Adams resolution} of the form 
\[
\begin{tikzcd}
	{d^{0}} & {i^{0}_{\ccat}} & {i^{1}_{\ccat}} & {i^{2}_{\ccat}} & {\ldots} \\
	& {d^{1}} & {d^{2}} & {\ldots}
	\arrow[from=1-1, to=1-2]
	\arrow[from=1-2, to=2-2]
	\arrow[from=1-2, to=1-3]
	\arrow[from=2-2, to=1-3]
	\arrow[from=1-3, to=2-3]
	\arrow[from=2-3, to=2-2, dotted]
	\arrow[from=2-3, to=1-4]
	\arrow[from=1-3, to=1-4]
	\arrow[from=1-4, to=2-4]
	\arrow[from=2-4, to=1-5]
	\arrow[from=1-4, to=1-5]
	\arrow[from=2-2, to=1-1, dotted]
	\arrow[from=2-4, to=2-3, dotted]
\end{tikzcd}
\]
This tower has the property that applying $H$ to the top row, we obtain an injective resolution of $H(d) = H(d^{0})$. Applying $[c, -]_{*}$ for some other object $c \in \ccat$ to the above tower we obtain an exact couple which leads to a spectral sequence with 
\[
E_{2}^{s, t} \colonequals \Ext_{\acat}^{s, t}(H(c), H(d)),
\]
using that the above tower was encoding the injective resolution of $H(d)$. This is the \emph{$H$-Adams spectral sequence}, and in favourable cases it converges to $[c, d]_{*}$, the group of homotopy classes of maps, or a suitable completion thereof. 
\end{construction}

\begin{remark}
Observe that if the homology of the target has finite injective dimension; that is, admits a finite injective resolution, then the $H$-Adams resolution can also be chosen to be finite, and the spectral sequence is automatically convergent to $[c,d]_*$. This is in particular always the case when $\acat$ has finite cohomological dimension.
\end{remark}

\begin{remark}
\label{remark:millers_adams_spectral_sequence}
As was first observed by Miller, the $H$-Adams tower is determined by remarkably less information than the homology theory $H\colon \ccat \rightarrow \acat$. Namely, it is enough to know which maps in $\ccat$ are $H$-monic. We will return to and expand on these ideas in \S\ref{subsection:epimorphism_classes}.
\end{remark}

\begin{remark}
\label{remark:cosimplicial_adams_resolutions}
It is sometimes convenient to use a cosimplicial resolution rather than a tower; namely, by repeatedly embedding $d$ into an $H$-injective construct a cosimplicial object
\[
d \rightarrow i^{0}_{\ccat} \rightrightarrows i^{1}_{\ccat} \triplerightarrow \ldots
\]
with the property that after applying $H$ and taking alternating sums of coboundary morphisms, it becomes an injective resolution of $H(d)$ in $\acat$. After applying $[c, -]$ this leads to a totalization spectral sequence which is isomorphic to the one of \cref{construction:adams_spectral_sequence_associated_to_a_homology_theory}.

Such a cosimplicial injective resolution can always be constructed by the dual of \cite{higher_algebra}[7.2.1.4], where we take $S$ to be the class of $H$-injectives (note that presentability is not needed in the proof; the assumption that any object admits an $H$-monic map into an $H$-injective suffices).  
\end{remark}
A reader familiar with the classical Adams spectral sequence and its $E_{2}$-term can now suspect as to why exactly the homology theory $H_{*}(-, \mathbb{F}_{p})\colon \spectra \rightarrow \textnormal{Vect}(\mathbb{F}_{p})$ fails to be adapted, as we observed in \cref{warning:homology_theory_with_non_compatible_injectives}: there simply is not enough information in the target category to get an approximation to homotopy classes of spectra. If we want an adapted homology theory, we have to consider $H_{*}(-, \mathbb{F}_{p})$ with its full structure, as in the following example. 
\begin{example}
\label{example:mod_p_homology_has_compatible_injectives_with_comodules}
By standard arguments, the mod $p$ homology $H_{*}(X, \mathbb{F}_{p})$ of a spectrum $X$ has a canonical structure of a comodule over the dual Steenrod algebra $A_{*}$. Thus, the previously considered homology theory factors as 
\[
\spectra \rightarrow \Comod(A_{*}) \rightarrow \textnormal{Vect}(\mathbb{F}_{p}),
\]
where $\Comod(A_{*})$ is the category of comodules over $A_{*}$ and the second functor is the forgetful one. 

In the category of comodules, $A_{*}$ itself is injective. Moreover, we have
\[
\Hom_{\Comod(A_{*})}(H_{*}(X, \mathbb{F}_{p}), A_{*}) \simeq H^{*}(X, \mathbb{F}_{p}) \simeq [X, H\mathbb{F}_{p}],
\]
so that the associated injective is again the Eilenberg-MacLane spectrum. Moreover, the comparison map $H_{*}(H\mathbb{F}_{p}, \mathbb{F}_{p}) \rightarrow A_{*}$ is now an isomorphism of comodules. In fact, one can show that 
\[
H_{*}(-, \mathbb{F}_{p})\colon \spectra \rightarrow \Comod(A_{*})
\]
is an adapted homology theory.
\end{example}

\begin{remark}
We will show later, in \S\ref{subsection:digression_adapted_factorization}, that for a homology theory admitting lifts of injectives to fail to be adapted, the situation of \cref{example:mod_p_homology_has_compatible_injectives_with_comodules} is essentially the only thing that can happen. More precisely, we will show that any homology theory which admits lifts of injectives canonically factors through an adapted one followed by what is essentially a forgetful functor, like the one from $A_{*}$-comodules into graded vector spaces considered above.

As \cref{corollary:coproduct_preserving_homology_theory_on_presentable_cat_has_lifts_of_injectives} assures us that homology theories with lifts of injectives are abundant, this will produce a variety of examples of adapted homology theories. 
\end{remark}

\subsection{Recollections on the Freyd envelope}
\label{subsection:recollections_on_freyd_envelope}

In this section we will recall the construction of the Freyd envelope, which is a certain kind of abelian category associated to an additive $\infty$-categories $\ccat$ which admits finite limits. 

\begin{remark}
Of particular interest is the case of a Freyd envelope of a \emph{stable} $\infty$-category. For an extended discussion in this context, and proofs of certain properties which we only state, we refer the reader to the excellent monograph by Neeman \cite{neeman2001triangulated}.
\end{remark}
If $\ccat$ is an additive $\infty$-category, recall that a discrete additive presheaf is a functor $X\colon \ccat^{op} \rightarrow \abeliangroups$ which preserves finite sums. These naturally form a category which we denote by $\Fun_{\Sigma}(\ccat^{op}, \abeliangroups)$. 

\begin{example}
Let $\ccat$ be an additive $\infty$-category. If $c \in \ccat$, the associated \emph{representable} discrete additive presheaf $y(c)\colon \ccat^{op} \rightarrow \abeliangroups$
is defined by  
\[
y(c)(d) = [d, c] \colonequals \pi_{0} \Map_{\ccat}(d, c)
\]
\end{example}

\begin{definition}
We say a discrete additive presheaf $X\colon \ccat^{op} \rightarrow \abeliangroups$ is \emph{finitely presented} if there is a cokernel sequence of additive presheaves
\[
y(c) \rightarrow y(d) \rightarrow X \rightarrow 0
\]
where $y(c), y(d)$ are representable.
\end{definition}

\begin{definition}
\label{definition:classical_freyd_envelope}
The \emph{Freyd envelope} of an additive $\infty$-category $\ccat$, denoted by $A(\ccat)$, is the full subcategory of $\Fun_{\Sigma}(\ccat^{op}, \abeliangroups)$ spanned by finitely presented presheaves. 
\end{definition}

\begin{remark}
\label{remark:freyd_envelope_determined_by_homotopy_category}
Notice that since $\abeliangroups$ form an ordinary category, any additive presheaf canonically factors through the homotopy category $\ccat$. Thus the natural projection $\ccat \rightarrow h\ccat$ induces an equivalence $A(\ccat) \simeq A(h \ccat)$. In particular, $A(\ccat)$ only depends on the homotopy category.
\end{remark}

\begin{remark}
Formation of representables assembles into discrete Yoneda functor $y\colon \ccat \rightarrow A(\ccat)$. Note that this is in general far from being fully faithful, since $A(\ccat)$ is only an ordinary category, while $\ccat$ need not be. However, $y$ factors through the homotopy category, in which case the classical Yoneda lemma implies that $y\colon h \ccat \rightarrow A(\ccat)$ is fully faithful. 
\end{remark}

\begin{remark}
\label{remark:freyd_envelope_obtained_by_adjoining_cokernels}
By definition, the Freyd envelope is the smallest subcategory of the category $\Fun_{\Sigma}(\ccat^{op}, \abeliangroups)$ of all additive presheaves which contains all representables and is closed under isomorphisms and under finite colimits. 

Using this description, one can show that the Freyd envelope $A(\ccat)$ enjoys the following universal property: if $\bcat$ is an (ordinary) additive category which admits finite colimits, any additive functor $\ccat \rightarrow \bcat$ uniquely extends to a \emph{right exact} additive functor $A(\ccat) \rightarrow \bcat$.
\end{remark}

\begin{lemma}
\label{lemma:coproducts_in_the_freyd_envelope}
Let $\kappa$ be a cardinal and suppose that $\ccat$ admits $\kappa$-small coproducts. Then, $A(\ccat)$ admits $\kappa$-small coproducts and $y\colon \ccat \rightarrow A(\ccat)$ preserves these. 
\end{lemma}

\begin{proof}
We first check that $y$ preserves $\kappa$-small coproducts; by the Yoneda lemma this is the same as saying that each $X \in A(\ccat)$, considered as an additive presheaf, takes $\kappa$-small coproducts in $\ccat$ to products of abelian groups. This property clearly holds for each representable, and the general case follows from the fact that each $X \in A(\ccat)$ is a finite colimit of representables and that products of abelian groups commute with finite colimits. 

Now suppose that $X_{\alpha}$ is a $\kappa$-small family of objects of the Freyd envelope, and for each choose a cokernel presentation 
\[
y(c_{\alpha}) \rightarrow y(d_{\alpha}) \rightarrow X_{\alpha} \rightarrow 0.
\]
Then an immediate calculation verifies that the cokernel of 
\[
y(\oplus _{\alpha} c_{\alpha}) \rightarrow y(\oplus_{\alpha} d_{\alpha}) 
\]
is the needed coproduct $\oplus_{\alpha} X_{\alpha}$ in the Freyd envelope. 
\end{proof}

\begin{lemma}
\label{lemma:freyd_envelope_has_products_when_c_does}
Let $\kappa$ be a cardinal and suppose that $\ccat$ admits $\kappa$-small products. Then, $A(\ccat)$ is closed under $\kappa$-small products in $\Fun_{\Sigma}(\ccat^{op}, \abeliangroups)$ and the Yoneda functor $\ccat \rightarrow A(\ccat)$ preserves products. 
\end{lemma}

\begin{proof}
If $y(c_{\alpha})$ is a $\kappa$-small family of representables, then since 
\[
\prod_{\alpha} [d, c_{\alpha}] \simeq [d, \prod_{\alpha} c_{\alpha}]
\]
we see that products of representables are again representable, in particular finitely presented. This also shows the second claim, that $y$ preserves all products which exist in $\ccat$. 

Going back to the first claim, suppose that $X_{\alpha}$ is a $\kappa$-small family of finitely presented presheaves, each with a presentation
\[
y(c_{\alpha}) \rightarrow y(d_{\alpha}) \rightarrow X_{\alpha} \rightarrow 0.
\]
Then, since products of abelian groups are exact, we have a cokernel sequence
\[
y(\prod_{\alpha} c_{\alpha}) \rightarrow y(\prod_{\alpha} d_{\alpha}) \rightarrow \prod_{\alpha} X_{\alpha} \rightarrow 0,
\]
showing that the product is again finitely presented, as needed. 
\end{proof}

\begin{remark}
There is an important distinction between \cref{lemma:coproducts_in_the_freyd_envelope} and \cref{lemma:freyd_envelope_has_products_when_c_does}. Namely, the inclusion $A(\ccat) \hookrightarrow \Fun_{\Sigma}(\ccat^{op}, \abeliangroups)$ preserves all products, but not necessarily coproducts. This is by design, since we want $\ccat \rightarrow A(\ccat)$ to preserves both products and coproducts - the issue is that the usual Yoneda embedding valued in the whole presheaf category preserves all products, but very rarely preserves coproducts. 
\end{remark}

As \cref{remark:freyd_envelope_obtained_by_adjoining_cokernels} remark shows, the Freyd envelope can be uniquely characterized as an additive category with finite colimits. The question of the existence of finite limits is more subtle, but follows from the work of Freyd which we now recall. 

\begin{definition}
\label{definition:weak_kernel}
Let $\acat$ be an additive category. We say a map $k \rightarrow a$ is a \emph{weak kernel} of $a \rightarrow b$ if the composite 
\[
k \rightarrow a \rightarrow b
\]
is zero, and for every map $x \rightarrow a$ such that 
\[
x \rightarrow a \rightarrow b 
\]
also vanishes, there exists a map $x \rightarrow k$, not necessarily unique, such that the entire diagram commutes. 
\end{definition}

\begin{remark}
A weak kernel diagram $k \rightarrow a \rightarrow b$ is a kernel diagram if and only if the first map is a monomorphism. Thus, weak kernels are a generalization of kernels. 
\end{remark}

\begin{lemma}
\label{lemma:homotopy_cat_of_additive_infty_cat_with_finite_limits_has_weak_cokers}
Let $\ccat$ be an additive $\infty$-category with finite limits. Then, the homotopy category $h \ccat$ has weak kernels. 
\end{lemma}

\begin{proof}
Suppose that $c \rightarrow d$ is a morphism in $\ccat$, and complete it to a fibre sequence
\[
k \rightarrow c \rightarrow d.
\]
Then, long exact sequence of homotopy of the associated mapping spaces implies that 
\[
[x, k] \rightarrow [x, c] \rightarrow [x, d]
\]
is exact for any $x \in \ccat$ and we deduce that $k$ is a weak kernel of $c \rightarrow d$ in the homotopy category $h \ccat$.
\end{proof}

\begin{proposition}
\label{proposition:if_c_has_finite_limits_freyd_envelope_is_an_abelian_subcat}
Let $\ccat$ be an additive $\infty$-category such that the homotopy category $h \ccat$ has weak kernels; for example, an additive $\infty$-category with finite limits. Then, the Freyd envelope $A(\ccat)$ is an abelian subcategory of $\Fun_{\Sigma}(\ccat^{op}, \abeliangroups)$; that is, it is closed under extensions, finite limits and colimits. 
\end{proposition}

\begin{proof}
By \cref{remark:freyd_envelope_determined_by_homotopy_category}, we have a canonical equivalence $A(\ccat) \simeq A(h \ccat)$, so that we can replace $\ccat$ by its homotopy category, which is an ordinary category with weak kernels. 

Under this assumption, it is a classical result of Freyd that $A(\ccat)$ is closed under kernels and cokernels in the whole additive presheaf category \cite{freyd1966representations}[1.4], and closure under extensions follows from Neeman's argument given in the proof of \cite{neeman2001triangulated}[5.1.9].
\end{proof}

\begin{lemma}
Let $\ccat$ be an additive $\infty$-category such that the homotopy category $h\ccat$ has weak kernels. Then, the Freyd envelope $A(\ccat)$ has enough projectives, and every representable is projective. 
\end{lemma}

\begin{proof}
The first part follows from the second, as by definition any object of $A(\ccat)$ is a cokernel of a map of representables. It is clear that the latter are projective as objects of the Freyd envelope, as they are even projective in the presheaf category. 
\end{proof}

Of particular interest is the case when $\ccat$ is a stable $\infty$-category. In this case, there is a strong relationship between finite limits and colimits, and a triangle is fibre if and only if it is cofibre. This gives $h \ccat$ a peculiar property that every map is both a weak kernel and a weak cokernel, giving the Freyd envelope a host of additional properties. 

\begin{theorem}[Freyd]
\label{theorem:for_stable_cat_representables_are_injective_in_freyd_envelope}
Suppose that $\ccat$ is a stable $\infty$-category. Then, in addition to being projective as objects of the Freyd envelope, the representables $y(c)$ are also injective, and they generate it under finite limits. 
\end{theorem}

\begin{proof}
This is \cite{neeman2001triangulated}[5.1.10, 5.1.11, 5.1.23].
\end{proof}

\begin{remark}
\label{remark:injectives_are_representable_when_c_idemcomplete}
If $\ccat$ is a stable $\infty$-category which is idempotent complete, then any injective of $A(\ccat)$ is representable. To see this, notice that if $i$ is injective, then by \cref{theorem:for_stable_cat_representables_are_injective_in_freyd_envelope} we have a monomorphism $i \rightarrow y(c)$ for some $c \in \ccat$. It follows that $i$ is a direct summand of $y(c)$, so that $i$ is representable by the corresponding direct summand of $c$, which exists by idempotent completeness. 
\end{remark}

As we have observed in \cref{remark:freyd_envelope_obtained_by_adjoining_cokernels}, the Freyd envelope of an additive $\infty$-category can be characterized as the ordinary category obtained from $h \ccat$ by freely adjoining cokernels. If $\ccat$ is in fact stable, the Freyd envelope enjoys \emph{another universal property}, related to homology theories. This explains the central importance of this construction to the current work. 

To speak of homology theories, we will need to speak of local gradings. 

\begin{definition}
\label{definition:local_grading_on_classical_freyd_envelope}
Let $\ccat$ be a locally graded additive $\infty$-category. Then, the \emph{induced local grading} on the Freyd envelope $A(\ccat)$ is defined by 
\[
(X[1])(c) \colonequals X(c[-1]).
\]
\end{definition}

\begin{remark}
Note that the $[-1]$ appearing in the definition above is necessary for the Yoneda functor $y\colon \ccat \rightarrow A(\ccat)$ to become a homomorphism of locally graded $\infty$-categories, since 
\[
y(c[1])(d) \colonequals [d, c[1]] \simeq [d[-1], c] \equalscolon (y(c)[1])(d).
\]
\end{remark}

\begin{remark}
By \cref{remark:freyd_envelope_obtained_by_adjoining_cokernels}, the Freyd envelope, considered as an additive category with finite colimits, is covariantly functorial in $\ccat$. The induced local grading given above is just the local grading induced by this covariant functoriality. 
\end{remark}

\begin{example}
If $\ccat$ is stable, then it has a canonical local grading given by the suspension functors. In this case, the induced local grading on the Freyd envelope is given by 
\[
(X[1])(c) \colonequals X(\Sigma^{-1} c).
\]
\end{example}

\begin{theorem}[Freyd]
\label{theorem:homological_universal_property_of_freyd_envelope}
Let $\ccat$ be a stable $\infty$-category. Then, the the functor $y\colon \ccat \rightarrow A(\ccat)$ is homological. Moreover, it is universal in the following sense: for any homological functor $H\colon \ccat \rightarrow \acat$, there is an essentially unique exact functor $L\colon A(\ccat) \rightarrow \acat$ of abelian categories such that the following diagram commutes
\begin{center}
\begin{tikzcd}
	{\ccat} && {\acat} \\
	& {\acat(\ccat)}
	\arrow["{H}", from=1-1, to=1-3]
	\arrow["{y}"', from=1-1, to=2-2]
	\arrow["{\exists_{!} L}"', from=2-2, to=1-3, dashed]
\end{tikzcd}.
\end{center}
\end{theorem}

\begin{proof}
This is \cite{neeman2001triangulated}[5.1.18]. 
\end{proof}

\begin{remark}[Universal property and homology theories]
\label{remark:universal_property_of_classical_freyd_envelope_and_local_gradings}
Using the local grading on $A(\ccat)$ of \cref{definition:local_grading_on_classical_freyd_envelope}, the discrete Yoneda embedding $y\colon \ccat \rightarrow A(\ccat)$ acquires a structure of a locally graded functor so that $y$ is a homology theory in the sense of \cref{definition:homology_theory}. We claim that this functor is also universal in this sense; that is, that for any homology theory $H\colon \ccat \rightarrow \acat$ there is a unique \emph{locally graded}, exact functor $A(\ccat) \rightarrow \acat$ such that the diagram as above commutes. 

To see this, observe that if $H\colon \ccat \rightarrow \acat$ is a homology theory, then \cref{theorem:homological_universal_property_of_freyd_envelope} shows that there is a unique exact functor $L\colon A(\ccat) \rightarrow \acat$ which makes the above diagram commute. To show the claim, we have to promote $L$ to a homomorphism of locally graded $\infty$-categories, which is the same as choosing a natural isomorphism 
\begin{equation}
\label{equation:natural_iso_making_unique_exact_extension_a_locally_graded_functor}
[1]_{\acat} \circ L \simeq L \circ [1]_{A(\ccat)}. 
\end{equation}

The two sides above are the unique exact functors which yield, respectively, $[1]_{\acat} \circ H$ and $H \circ \Sigma$ after precomposing with $y$. The latter two are canonically isomorphic when $H$ is a homology theory, thus the universal property of $A(\ccat)$ implies that we have a canonical isomorphism of (\ref{equation:natural_iso_making_unique_exact_extension_a_locally_graded_functor}), as needed. 
\end{remark}

\begin{remark}
One can show that that the given homology theory $H$ preserves $\kappa$-small direct sums if and only if the corresponding exact functor out of $A(\ccat)$ does, see \cite{neeman2001triangulated}[5.1.24]. 
\end{remark}

\begin{warning}
\label{warning:freyd_envelope_usually_not_presentable}
Beware that even if $\ccat$ is a presentable $\infty$-category, $A(\ccat)$ will not be presentable except in the most trivial cases. In fact, it is not well-powered, hence not presentable, already when $\ccat = \dcat(\mathbb{Z})$ is the derived $\infty$-category of the integers \cite{neeman2001triangulated}[C.3].
\end{warning}

\subsection{Characterization of adapted homology theories} \label{sec: Homology theories and Freyd}

In \cref{definition:adapted_homology_theory}, we introduced the notion of an adapted homology theory $H\colon \ccat \rightarrow \acat$; that is one, where the injectives of $\acat$ can be compatibly lifted to objects of $\ccat$, giving rise to the Adams spectral sequence. In this section, we completely classify such homology theories by showing that they can be characterized either in terms of localizing subcategories of the Freyd envelope or epimorphism classes. Moreover, we show that any homology theory with target an abelian category with enough injectives factorizes uniquely through an adapted one in an explicit way. 

Recall that by \cref{theorem:homological_universal_property_of_freyd_envelope}, any homology theory $H\colon \ccat \rightarrow \acat$ factors uniquely through a exact functor $L\colon A(\ccat) \rightarrow \acat$. In some ways, it is easier to study the induced functor $L$, as it is an exact functor between abelian categories, rather than $H$ itself, whose source is a stable $\infty$-category. Our goal is to relate the properties of the homology theory to the properties of the induced exact functor.
\begin{lemma}
\label{lemma:functor_induced_by_homology_theory_a_left_adjoint_when_injectives_lift}
Let $\ccat$ be an idempotent complete stable $\infty$-category and $H\colon \ccat \rightarrow \acat$ a homology theory such that $\acat$ has enough injectives. Then, the following are equivalent:

\begin{enumerate}
    \item every injective of $\acat$ has a lift in $\ccat$,
    \item induced functor $L\colon A(\ccat) \rightarrow \acat$ has a right adjoint. 
\end{enumerate}
\end{lemma}

\begin{proof}
$(1 \Rightarrow 2)$ If $L$ has a right adjoint $R$, then the latter satisfies 
\[
\Hom_{\acat(\ccat)}(x, Ra) \simeq \Hom_{\acat}(Lx, a)
\]
for all $x \in \acat(\ccat)$, $a \in \acat$. Notice that since $L$ is right exact and $\acat(\ccat)$ is generated by the image of $\ccat$ under finite colimits, it is enough that we have 
\[
(Ra)(c) \simeq \Hom_{\acat(\ccat)}(y(c), Ra) \simeq \Hom_{\acat}(H(c), a) 
\]
for all $c \in \ccat$. Thus, the definition of $R$ is forced on us, and it will exist precisely when the additive presheaf $\Hom_{\acat}(H(-), a)\colon \ccat^{op} \rightarrow \abeliangroups$ is finitely presented for every $a \in \acat$, so that it is an element of the Freyd envelope. 

If $i \in \acat$ is injective, then
\[
[c, i_{\ccat}] \simeq \Hom_{\acat(\ccat)}(y(c), y(i_{\ccat})) \simeq \Hom_{\acat}(H(c), i),
\]
where $i_{\ccat}$ is the injective lift of $i$ from \cref{definition:injectives_in_a_stable_category}, so that $R(i) \simeq y(i_{\ccat})$ in this case, which is finitely presented. However, since $\acat$ has enough injectives, any $a \in \acat$ can be written as a kernel
\[
0 \rightarrow a \rightarrow i \rightarrow j
\]
of a map between injectives. Then, $Ra \simeq \mathrm{ker}(y(i_{\ccat}) \rightarrow y(j_{\ccat}))$ in additive presheaves, which is also finitely presented. 

$(2 \Rightarrow 1)$ Conversely, suppose that $R$ exists. Since $L$ is exact, it follows that for every injective $i \in \acat$, $R(i)$ is an injective object of the Freyd envelope. Since $\ccat$ is idempotent complete, we deduce from \cref{remark:injectives_are_representable_when_c_idemcomplete} that $R(i) \simeq y(j)$ is representable. It follows that $j$ is an injective lift in $\ccat$ associated to $i$. 
\end{proof}

\begin{theorem}[Characterization of adapted homology theories]
\label{theorem:characterization_of_adapted_homology_theories}
Let $\ccat$ be an idempotent complete stable $\infty$-category and $H\colon \ccat \rightarrow \acat$ a homology theory such that $\acat$ has enough injectives. Then, the following three conditions are equivalent: 
\begin{enumerate}
    \item $H$ is adapted,
    \item the induced functor $L\colon \acat(\ccat) \rightarrow \acat$ has a fully faithful right adjoint,
    \item the kernel of the induced functor $L\colon \acat(\ccat) \rightarrow \acat$ is a localizing subcategory of the Freyd envelope and $L$ identifies $\acat$ with the Gabriel quotient; that is, any exact functor out of $A(\ccat)$ annihilating $\mathrm{ker}(L)$ factors uniquely through $\acat$ 
\end{enumerate}
\end{theorem}

\begin{proof}
($1 \Rightarrow 2$) By \cref{lemma:functor_induced_by_homology_theory_a_left_adjoint_when_injectives_lift}, the right adjoint $R$ exists, we only have to show it is fully faithful. We have seen in the proof of the above lemma that if $i \in \acat$ is an injective, then $R(i) = y(i_{\ccat})$ is represented by the associated injective of $\ccat$. Moreover, under the identification 
\[
LRi \simeq L(y(i_{\ccat})) \simeq H(i_{\ccat}),
\]
the counit map of an injective $i \in \acat$ corresponds to the structure map $H(i_{\ccat}) \rightarrow i$. Thus, if $H$ is adapted, then the counit map is an isomorphism for all injectives. Since $LR$ is left exact, we deduce that it is always an isomorphism, so that $R$ is fully faithful. 

$(2 \Rightarrow 1)$ This is the same as the argument above, with the reasoning reversed. 

($2 \Rightarrow 3$) When $R$ is fully faithful, $LR \simeq id_{\acat}$ so that $\acat$ coincides with the category of comodules over the comonad $LR$. Then, the needed result is \cref{proposition:comodules_form_a_gabriel_quotient}.

($3 \Rightarrow 2$) It follows again from \cref{proposition:comodules_form_a_gabriel_quotient} that $\acat$ can be identified with the category of $LR$-comodules, since they are both Gabriel quotients by the same Serre subcategory. The latter has a fully faithful right adjoint by \cref{theorem:properties_of_comonadic_factorization_of_an_exact_left_adjoint}.
\end{proof}

\begin{corollary}
\label{corollary:localizing_class_determines_an_adapted_homology_theory}
Let $\ccat$ be an idempotent complete stable $\infty$-category. An adapted homology theory $H\colon \ccat \rightarrow \acat$ is completely determined by the kernel of the induced functor $L\colon A(\ccat) \rightarrow \acat$.
\end{corollary}

\begin{proof}
This is immediate from the last equivalent property in the statement of \cref{theorem:characterization_of_adapted_homology_theories}.
\end{proof}

\begin{remark}
\label{remark:adapted_homology_theory_determined_by_homology_epimorphisms}
Note that if $L$ is induced by a homology theory $H\colon \ccat \rightarrow \acat$, we have  
\[
L(\mathrm{coker}(y(c) \rightarrow y(d)) = \mathrm{coker}(H(c) \rightarrow H(d)).
\]
It follows that $\ker(L)$ can be identified with cokernels of \emph{$H$-epimorphisms} in $\ccat$; that is, those $c \rightarrow d$ such that $H(c) \rightarrow H(d)$ is an epimorphism. Thus, \cref{corollary:localizing_class_determines_an_adapted_homology_theory} can be interpreted as saying that $H$ is completely determined by its class of epimorphisms. This point of view will be expanded upon in the next chapter. 
\end{remark}

\section{Classification of Adams spectral sequences}
\label{section:classification_of_adams_spectral_sequences}

In this section we classify adapted homology theories in terms of epimorphism classes and localizing subcategories of the Freyd envelope. This gives a classification of Adams spectral sequences. 

\subsection{Epimorphism classes and Miller's Adams spectral sequence}
\label{subsection:epimorphism_classes}

As we have observed in \cref{remark:adapted_homology_theory_determined_by_homology_epimorphisms}, an adapted homology theory is completely determined by the class of homology epimorphisms. This mirrors an earlier observation due to Miller that an $R$-based Adams spectral sequence, where $R$ is a homotopy ring spectrum, is completely determined by the class of those maps of spectra which are split after applying $R $.  

In this section, we relate these two phenomena by axiomatizing the notion of an epimorphism class and recalling Miller's construction of the Adams spectral sequence. 

\begin{definition}
\label{definition:epimorphism_class}
Let $\ccat$ be a stable $\infty$-category. We say a class of arrows $\ecat \subseteq \Fun(\Delta^{1}, \ccat)$ is an \emph{epimorphism class} if:
\begin{enumerate}
    \item all equivalences are in $\ecat$,
    \item for any pair $f, g$ of composable arrows, if $f, g \in \ecat$ then also $g \circ f \in \ecat$  and if $g \circ f \in \ecat$ then also $g \in \ecat$,
    \item $\ecat$ is stable under pullbacks along arbitrary maps in $\ccat$,
    \item an arrow $f\colon c \rightarrow d$ belongs to $\ecat$ if and only if $\Sigma f\colon \Sigma c \rightarrow \Sigma d$ does. 
\end{enumerate}
\end{definition}

\begin{remark}
As a consequence of (2), if $f, g$ are homotopic, then one belongs to an epimorphism class if and only if the other does. Thus, epimorphism classes can be alternatively defined as classes of maps in the homotopy category $h \ccat$. 
\end{remark}

A reader can observe that the axioms satisfied by an epimorphism class are the same as those enjoyed by the class of epimorphisms in any ordinary category, and consequently by the class of effective epimorphisms in any $\infty$-topos or a Grothendieck prestable $\infty$-category \cite{higher_topos_theory}[6.2.3]. It is natural to ask if there is a similarly natural choice in the stable setting, to which we offer the following warning. 

\begin{warning}
\label{warning:all_maps_are_effective_epi_in_a_stable_infty_category}
In a stable $\infty$-category $\ccat$, \emph{any} arrow $c \rightarrow d$ is an effective epimorphism; that is, the \v{C}ech nerve 
\[
\ldots \triplerightarrow c \times_{d} c \rightrightarrows c \rightarrow d
\]
is a colimit diagram. This is at the same time a source of great simplification in the context of stable $\infty$-categories, but also of the inability of the latter to distinguish between more subtle notions of exactness.
\end{warning}

Intuitively, in the stable context there is no natural non-trivial notion of an epimorphism, and epimorphism classes offer varying choice of such. Let us first give a few examples. 

\begin{example}
\label{example:split_surjections_an_epi_class}
The class of \emph{split surjections}; that is, those maps $c \rightarrow d$ which are equivalent to a projection onto a direct summand, is an epimorphism class. This is the smallest epimorphism class, as any epimorphism class contains this one as a consequence of $(2)$ and $(1)$.
\end{example}
\begin{example} \label{example:homology_surjections}
Let $H\colon \ccat \rightarrow \acat$ be a homology theory. Then, the class of $H$-epimorphisms; that is, those maps $c \rightarrow d$ such that $H(c) \rightarrow H(d)$ is an epimorphism, is an epimorphism class on $\ccat$. 
\end{example}

\begin{example}
\label{example:preimage_of_epi_class_is_epi}
Let $L\colon \ccat \rightarrow \dcat$ be an exact functor between stable $\infty$-categories and let $\ecat$ be an epimorphism class on $\dcat$. Then, $L^{-1}(\ecat)$ is an epimorphism class in $\ccat$. 
\end{example}

\begin{example}
\label{example:r_split_maps}
A particular combination of \cref{example:split_surjections_an_epi_class} and  \cref{example:preimage_of_epi_class_is_epi} is the class of \emph{$R$-split maps}, where $R$ is a spectrum; that is, those maps $X \rightarrow Y$  such that $R \otimes X \rightarrow R \otimes Y$ is a split epimorphism of spectra. 
\end{example}

Dually, we can define monomorphisms. 

\begin{definition}
Let $\ecat$ be an epimorphism class on a stable $\infty$-category $\ccat$. We will say an arrow $c \rightarrow d$ is \emph{$\ecat$-monic} if the canonical map $ d \rightarrow \mathrm{cofib}(c \rightarrow d)$ is in $\ecat$. 
\end{definition}

\begin{example}
Let $H\colon \ccat \rightarrow \acat$ be a homology theory and $\ecat$ the corresponding class of $H$-epimorphisms. Then, an arrow $c \rightarrow d$ is $\ecat$-monic if and only if $H(c) \rightarrow H(d)$ is a monomorphism, as a consequence of the long exact sequence of homology. 
\end{example}

\begin{example}
If $\ecat$ is the class of split surjections, then a map is $\ecat$-monic if and only if it is an inclusion of a direct summand. 
\end{example}

\begin{remark}
If $\ecat$ is an epimorphism class, then the class of $\ecat$-monics satisfies the (almost) dual version of the axioms of \cref{definition:epimorphism_class}: it contains equivalences, satisfies an appropriate 2-out-of-3, and is closed under pushouts. Moreover, it is not hard to see that the corresponding classes of epimorphism and monomorphisms determine each other. 

Thus, our choice of privileging epimorphisms is somewhat arbitrary, as all of the results in this chapter can be alternatively phrased in terms of monomorphism classes. We decided to use epimorphisms for two reasons: 

\begin{enumerate}
    \item the connection to the Freyd envelope which we will explore later is easier to phrase in terms of epimorphisms and 
    \item for wide classes of $\infty$-categories, such as $\infty$-topoi, epimorphisms are more interesting than monomorphisms. 
\end{enumerate}
\end{remark}

\begin{definition}
\label{definition:epi_class_injective}
Let $\ecat$ be an epimorphism class. We say an object $i \in \ccat$ is \emph{$\ecat$-injective} if it has the right lifting property with respect to $\ecat$-monics. That is, if for any $\ecat$-monic $c \rightarrow d$, the induced map 
\[
[d, i] \rightarrow [c, i]
\]
is an epimorphism of abelian groups. Equivalently, $i$ is $\ecat$-injective if for every $c \rightarrow d$ in $\ecat$, the induced map 
\[
[d, i] \rightarrow [c, i]
\]
is a monomorphism of abelian groups.
\end{definition}

\begin{definition}
Let $\ecat$ be an epimorphism class. We say $\ccat$ \emph{has enough $\ecat$-injectives} if for every $c \in \ccat$ there exists an $\ecat$-monic map $c \rightarrow i$ into an $\ecat$-injective. 
\end{definition}
Before we give examples, let us observe that this property is exactly what is needed to be able to construct the Adams spectral sequence.
\begin{construction}[Miller's Adams spectral sequence] \label{Miller}
Let $\ecat$ be an epimorphism class such that $\ccat$ has enough $\ecat$-injectives and let $d \in \ccat$. Then, writing $d = d^{0}$, by assumption we can choose an $\ecat$-monic map $d^{0} \rightarrow i^{0}$ into an $\ecat$-injective. Proceeding inductively by setting $d^{k+1} = \mathrm{cofib}(d^{k} \rightarrow i^{k})$, we construct an \emph{$\ecat$-Adams resolution} of the form 
\[
\begin{tikzcd}
	{d^{0}} & {i^{0}} & {i^{1}} & {i^{2}} & {\ldots} \\
	& {d^{1}} & {d^{2}} & {\ldots}
	\arrow[from=1-1, to=1-2]
	\arrow[from=1-2, to=2-2]
	\arrow[from=1-2, to=1-3]
	\arrow[from=2-2, to=1-3]
	\arrow[from=1-3, to=2-3]
	\arrow[from=2-3, to=2-2, dotted]
	\arrow[from=2-3, to=1-4]
	\arrow[from=1-3, to=1-4]
	\arrow[from=1-4, to=2-4]
	\arrow[from=2-4, to=1-5]
	\arrow[from=1-4, to=1-5]
	\arrow[from=2-2, to=1-1, dotted]
	\arrow[from=2-4, to=2-3, dotted]
\end{tikzcd}
\]
Applying $[c, -]$ for any other object $c$ leads to a spectral sequence, the $\ecat$-Adams spectral sequence. As with injective resolutions in abelian categories, one can show that any two $\ecat$-Adams towers are related by an appropriate notion of homotopy equivalence and so all such spectral sequences are isomorphic from the second page on. 
\end{construction}

\begin{remark} A similar axiomatic construction of the Adams spectral sequence in general triangulated categories is given by Christensen in \cite{Christensen}. 

\end{remark}

\begin{example}
\label{example:split_surjections_form_an_epimorphism_class}
Let $\ecat$ consists of split surjections. Then $\ecat$-monics are given by split monomorphisms and every object of $\ccat$ is $\ecat$-injective. The corresponding Adams spectral sequence will, for any pair of objects, have an $E_{2}$-page which vanishes outside of $s = 0$ and hence collapse immediately.  
\end{example}

\begin{example}
\label{example:millers_adams_associated_to_an_adjunction}
Suppose that $L \dashv R\colon \ccat \rightleftarrows \dcat$ is an adjunction between stable $\infty$-categories. As we have seen in \cref{example:preimage_of_epi_class_is_epi}, the class of these maps $c_{1} \rightarrow c_{2}$ such that $Lc_{1} \rightarrow Lc_{2}$ is a split surjection in $\dcat$ forms an epimorphism class on $\ccat$. 

It is not difficult to verify that any object of the form $Rd$ for $d \in \dcat$ is injective. In this case, any $c \in \ccat$ admits a canonical Adams resolution of the form 
\[
c \rightarrow RL(c) \rightarrow RLRL(c) \rightarrow \ldots
\]
obtained by taking the cobar resolution and taking alternating sums of face maps.
\end{example}

\begin{example}
\label{example:epi_class_of_r_split_maps_has_enough_injectives_when_r_is_e1}
A particularly important case of \cref{example:millers_adams_associated_to_an_adjunction} is given by letting $R$ be an $\mathbf{E}_{1}$-ring spectrum, and considering the left adjoint $R \otimes -\colon \spectra \rightarrow \Mod_{R}(\spectra)$. Then, all spectra underlying $R$-modules are injective, and we have canonical Adams resolutions of the form 
\[
X \rightarrow R \otimes X \rightarrow R \otimes R \otimes X \rightarrow \ldots.
\]
This can be also thought as the Amitsur complex associated to the map $S^{0} \rightarrow R$ tensored with $X$. After normalizing the cosimplicial object, we can obtain the original resolution due to Adams \cite{adams1995stable}:
\[
X \rightarrow R \otimes X \rightarrow R \otimes \overline{R} \otimes X \rightarrow \ldots,
\]
where $\overline{R} \colonequals \mathrm{fib}(S^{0} \rightarrow R)$. The corresponding Miller's Adams spectral sequence is just the descent spectral sequence associated to the map $S^{0} \rightarrow R$ of ring spectra.  
\end{example}
Now suppose that $H\colon \ccat \rightarrow \acat$ is an adapted homology theory, which determines and is determined by the class of $H$-epimorphisms by 
\cref{remark:adapted_homology_theory_determined_by_homology_epimorphisms}. In this context, a priori we have two notions of $H$-injectives, and two notions of Adams spectral sequences, one given in \cref{construction:adams_spectral_sequence_associated_to_a_homology_theory} and the other one due to Miller given above. In fact, these two coincide, as the consequence of the following. 

\begin{proposition}
\label{proposition:injectives_associated_to_epis_and_homology_theory_are_the_same}
Let $\ccat$ be an idempotent complete stable $\infty$-category, $H\colon \ccat \rightarrow \acat$ an adapted homology theory and $\ecat$ the class of $H$-epimorphisms. Then, an object $c \in \ccat$ is $\ecat$-injective if and only if it is equivalent to an injective lift $i_{\ccat}$ for some  injective $i \in \acat$.
\end{proposition}

\begin{proof}
Since $\ecat$-monomorphisms are the same as $H$-monomorphisms, it is clear that all injective lifts $i_{\ccat}$ associated to $i \in \acat$ are $\ecat$-injective in the sense of \cref{definition:epi_class_injective}. 

Conversely, if $c \in \ccat$ is $\ecat$-injective, we can choose an $H$-monomorphism $c \rightarrow i_{\ccat}$ into an injective lift $i_{\ccat}$ for some injective $i$, which will split by $\ecat$-injectivity of $c$. Thus, $c$ is a direct summand of $i_{\ccat}$, and so is equivalent to $j_{\ccat}$ for the corresponding direct summand $j$ of $i$. 
\end{proof}

\begin{corollary}
Miller's Adams spectral sequence associated to $H$-epimorphisms of an adapted homology theory $H\colon \ccat \rightarrow \acat$ coincides with the spectral sequence obtained by lifting an injective resolution in $\acat$ to $\ccat$. 
\end{corollary}

\begin{proof}
By \cref{proposition:injectives_associated_to_epis_and_homology_theory_are_the_same}, both spectral sequences are obtained by resolving the source by repeatedly choosing an $H$-monomorphism into the same class of objects. 
\end{proof}

\begin{warning}[Modified and classical Adams spectral sequence]
The following used to confuse the authors and thus we thought it might be worthwhile to make it explicit. 

If $E$ is a ring spectrum,  then there are \emph{two} natural epimorphism classes associated to $E$, the first one given by $E_{*}$-surjective maps. If $E$ is Adams-type, for example if it is  Landweber exact, then it is known that for any $X \in \spectra$, $E_{*}X$ has a canonical structure of an $E_{*}E$-comodule. This leads to an adapted homology theory
\[
E_{*}\colon \spectra \rightarrow \ComodE
\]
determining the same class of epimorphisms. The corresponding injectives are those $E$-local spectra $I$ such that $E_{*}I$ is injective as a comodule, and the resulting Adams spectral sequence is sometimes called a \emph{modified Adams spectral sequence}. It has the curious property that the $E^{2}$-term is \emph{always} given by $\Ext$-groups in comodules.

Alternatively, as an epimorphism class we can take $E$-split surjective maps, in which case the Adams spectral sequence is the one based on the Amitsur resolution
\[
X \rightarrow E \otimes X \rightarrow E \otimes E \otimes X \rightarrow \ldots
\]

There's always a comparison map from the Amitsur resolution into the comodule one, because any $E$-split map is in particular $E_{*}$-surjective.
Applying $[Y, -]$ to the above two different resolutions gives in general two non-isomorphic spectral sequences. It is well-known that these two coincide when $E_{*}Y$ is projective as an $E_{*}$-module.
\end{warning}

\subsection{Adams spectral sequence and adapted homology theories}

One way in which Adams spectral sequences associated to an adapted homology theory $H$ have an advantage over a general one associated to an epimorphism class is that their $E_{2}$-term is readily computable as the Ext-groups in the target abelian category. In this section, we will show that this is in fact always true, as every epimorphism class with enough injectives arises from a unique adapted homology theory. 

In fact, we will show that for an idempotent-complete, stable $\infty$-category $\ccat$, the following three sets of data are equivalent: 

\begin{enumerate}
\item epimorphism classes with enough injectives
\item localizing subcategories of $A(\ccat)$ such that the Gabriel quotient has enough injectives
\item adapted homology theories 
\end{enumerate}
As each of these gives rise to an Adams spectral sequence, this result can be interpreted as a classification of Adams spectral sequences. 

We will first begin by relating (1) and (2), where the relationship holds in a little bit more generality.

\begin{proposition}[Beligiannis]
\label{proposition:epimorphism_classes_same_as_localizing_subcategories}
Let $\ecat$ be an epimorphism class on a stable $\infty$-category $\ccat$. Let $K$ be the collection of those objects $k \in A(\ccat)$ such that there exists a cokernel sequence
\[
y(c) \rightarrow y(d) \rightarrow k \rightarrow 0
\]
with $c \rightarrow d \in \ecat$. Then, 

\begin{enumerate}
    \item $K$ is a Serre subcategory of $A(\ccat)$ closed under the local grading and 
    \item any such Serre subcategory of $A(\ccat)$ arises in this way from a unique epimorphism class in $\ccat$. 
\end{enumerate}
\end{proposition}

\begin{proof}
This is essentially a restatement of a result of Beligiannis, who proves it in the triangulated context \cite{beligiannis2000relative}. For completeness, we give the argument here. 

We first check that $K$ is closed under quotients. Suppose we have a cokernel sequence $y(c) \rightarrow y(d) \rightarrow k \rightarrow 0$ with $c \rightarrow d \in \ecat$ and let $k \rightarrow k^{\prime}$ be a quotient map.  It follows that the composite $y(d) \rightarrow k^{\prime}$ is also a surjection, and so its kernel is an image of some map $y(c^{\prime}) \rightarrow y(d)$. Using projectivity of $y(c)$, we see we can lift its map to $y(c^{\prime})$ obtaining a commutative diagram \[
\begin{tikzcd}
	{y(c)} & {y(d)} & k & 0 \\
	{y(c^{\prime})} & {y(d)} & {k^{\prime}} & 0
	\arrow[from=1-1, to=2-1]
	\arrow[from=1-1, to=1-2]
	\arrow[from=2-1, to=2-2]
	\arrow[from=1-2, to=1-3]
	\arrow[from=2-2, to=2-3]
	\arrow["id"{description}, from=1-2, to=2-2]
	\arrow[from=1-3, to=2-3]
	\arrow[from=1-4, to=2-4]
	\arrow[from=1-3, to=1-4]
	\arrow[from=2-3, to=2-4]
\end{tikzcd}
\]
the rows of which are cokernel sequences. The left square arises from a diagram
\[
c \rightarrow c^{\prime} \rightarrow d,
\]
and since the whole composite is in $\ecat$, so must be the latter map. Thus, $k^{\prime} \in K$, as needed.

We now check the case of subobjects, so let $k^{\prime} \subseteq k$. Choose a surjection $y(d^{\prime}) \rightarrow k^{\prime} \times_{k} y(d)$ and consider the diagram
\[
\begin{tikzcd}
	{y(c \times_{d} d^{\prime}) } & {y(d^{\prime})} & {k^{\prime}} & 0 \\
	{y(c)} & {y(d)} & k & 0
	\arrow[from=1-3, to=1-4]
	\arrow[from=2-3, to=2-4]
	\arrow[from=1-2, to=1-3]
	\arrow[from=1-1, to=1-2]
	\arrow[from=2-1, to=2-2]
	\arrow[from=2-2, to=2-3]
	\arrow[from=1-3, to=2-3]
	\arrow[from=1-2, to=2-2]
	\arrow[from=1-1, to=2-1]
\end{tikzcd}
\]
The arrow $y(c \times_{d} d^{\prime}) \rightarrow y(c) \times_{y(d)} y(d^{\prime})$ is rarely an isomorphism, but it is always a surjection, and it follows that the upper row is also a cokernel sequence. Since $c \times_{d} d^{\prime} \rightarrow d^{\prime}$ is a pullback of a map in $\ecat$, it is also contained in it, which is what we wanted to show.

We now deal with extensions, so suppose that
\[
0 \rightarrow k^{\prime} \rightarrow k \rightarrow k^{\prime \prime} \rightarrow 0
\]
is an exact sequence such that $k^{\prime}, k^{\prime \prime} \in K$ so can be written as cokernels of appropriate maps in $\ecat$. Since $y(d^{\prime \prime})$ is projective, the given map into $k^{\prime \prime}$ lifts to $k$. We claim that 
\[
y(c^{\prime} \oplus c^{\prime \prime}) \rightarrow y(d^{\prime} \oplus d^{\prime \prime}) \rightarrow k
\]
is a cokernel sequence. This can be checked using an easy diagram chase like in the classical Horseshoe lemma of homological algebra. 


Because $\ecat$ is closed under suspension, and the local grading on the Freyd envelope is induced by the suspension, it is clear that $K$ is closed under it. This ends the first part of the argument. 

Conversely, if $K$ is such a Serre subcategory, we claim that the class $\ecat$ of maps in $\ccat$ whose cokernels in $A(\ccat)$ belong to $K$, forms an epimorphism class. Since $K$ is assumed to be closed under the local grading, the quotient acquires a unique local grading from that of the Freyd envelope. It follows that the composite
\[
\ccat \rightarrow A(\ccat) \rightarrow A(\ccat) / K.
\]
is a homological functor, from which it readily follows that $\ecat$ satisfies the axioms of an epimorphism class by \cref{example:homology_surjections}. 

To show uniqueness, we just have to check that if $\ecat$ is an epimorphism class and $K$ is as above, then $\ecat$ is \emph{exactly} the class of maps whose formal cokernels are in $K$. Thus, let $c \rightarrow d$ be arbitrary and suppose that $k = \mathrm{coker}(y(c) \rightarrow (d))$ is in $K$. It follows that there is a different map $c^{\prime} \rightarrow d^{\prime} \in \ecat$ such that we also have $k = \mathrm{coker}(y(c^{\prime}) \rightarrow y(d^{\prime}))$. Using projectivity of representables, we can construct a commutative diagram
\[
\begin{tikzcd}
	{y(c)} & {y(d)} & {k} & 0 \\
	{y(c^{\prime})} & {y(d^{\prime})} & k & 0.
	\arrow[from=1-3, to=1-4]
	\arrow[from=2-3, to=2-4]
	\arrow[from=1-2, to=1-3]
	\arrow[from=1-1, to=1-2]
	\arrow[from=2-1, to=2-2]
	\arrow[from=2-2, to=2-3]
	\arrow["{id}", from=1-3, to=2-3]
	\arrow[from=1-2, to=2-2]
	\arrow[from=1-1, to=2-1]
\end{tikzcd}
\]
To say that the left square above is injective on cokernels means that a map into $x \to d$ in $h\ccat$ factors through $c$ if and only if the composite map into $d^{\prime}$ factors through $c'$. It follows that the map $c^{\prime} \times _{d^{\prime}} d \rightarrow d$, which is in $\ecat$, being a base-change of one which is, factors through $c$:
\[c^{\prime} \times _{d^{\prime}} d \to c \to d.\]
This by (2) of \cref{definition:epimorphism_class} implies that $c \to d$ is in $\ecat$. 
 
\end{proof}

\begin{remark} If $\ccat$ admits $\kappa$-indexed coproducts, then an analogue of  \cref{proposition:epimorphism_classes_same_as_localizing_subcategories} holds with $\kappa$-indexed sums. More precisely, epimorphism classes $\ecat$ closed under $\kappa$-small coproducts in the $\infty$-category $\Fun(\Delta^{1}, \ccat)$ are in one to one correspondence with $\kappa$-small coproduct-closed Serre subcategories of $A(\ccat)$ closed under the local grading and closed under arbitrary coproducts.

To see this, suppose that $k_{\alpha} \in K$ be a $\kappa$-small family of objects, each of which is a cokernel of a map $y(c_{\alpha}) \rightarrow y(d_{\alpha})$ in $\ecat$. We want to show that $\oplus_{\alpha} k_{\alpha} \in K$, and we have a cokernel sequence
\[
y(\oplus c_{\alpha}) \rightarrow y(\oplus d_{\alpha}) \rightarrow \oplus k_{\alpha} \rightarrow 0,
\]
since $y\colon \ccat \rightarrow A(\ccat)$ preserves coproducts by \cref{lemma:coproducts_in_the_freyd_envelope} and cokernels commute with direct sums. This is the needed presentation, since $\oplus c_{\alpha} \rightarrow \oplus d_{\alpha} \in \ecat$. 

\end{remark}

One hope might be that if $\ecat$ is an epimorphism class and $K$ the corresponding Serre subcategory of the Freyd envelope, then the composite 
\[
\ccat \rightarrow A(\ccat) \rightarrow A(\ccat) / K
\]
is an adapted homology theory. However, for the latter to hold, we need to know that $A(\ccat) / K$ has enough injectives, which is not automatic. 

Worse yet, for a general choice of $K$, there is no good reason for the quotient to be well-behaved. In fact, it is not clear that it is locally small, because categories of the form $A(\ccat)$ are usually not well-copowered - that is, it is possible to have an object with a proper class of subobjects, even for presentable $\ccat$. 

Luckily, it turns out that resolving the first issue also resolves the latter - having enough injectives guarantees that $K$ is localizing, so that $A(\ccat)/K$ can be identified with a full subcategory of $A(\ccat)$ and so is locally small. 

\begin{theorem}[Classification of Adams spectral sequences]
\label{theorem:adapted_homology_theories_correspond_to_injective_epimorphism_classes}
The following three sets of data one can associate to an idempotent-complete stable $\infty$-category $\ccat$ are equivalent:
\begin{enumerate}
    \item an epimorphism class $\ecat$ of morphisms such that $\ccat$ has enough $\ecat$-injectives,
    \item localizing subcategories $K$ of $A(\ccat)$ of such that the Gabriel quotient $A(\ccat)/K$ has enough injectives and 
    \item adapted homology theories $H\colon \ccat \rightarrow \acat$.
\end{enumerate}
\end{theorem}

\begin{proof}
The bijection between $(2)$ and $(3)$ is provided by the Gabriel quotient construction and by \cref{theorem:characterization_of_adapted_homology_theories}. Since we also know that epimorphism classes on $\ccat$ correspond to Serre classes on $A(\ccat)$, by \cref{proposition:epimorphism_classes_same_as_localizing_subcategories}, to show the bijection between $(1)$ and $(2)$ we just have to check that the ``enough injectives'' conditions match up.

First assume that $\ccat$ has enough $\ecat$-injectives and let $i \in \ccat$ be one. For ease of reading, we will not distinguish between an object of $\ccat$ and its image in $A(\ccat)$. We then see that for any $k \in K$, say $k = \mathrm{coker}(c \rightarrow d)$ for a map $c \rightarrow d$ in $\ecat$, we have
\[
\Hom_{A(\ccat)}(k, i) \simeq \mathrm{ker}([d, i] \rightarrow [c, i]) = 0,
\]
using $\ecat$-injectivity condition.

We claim that this implies that not only is $i$ injective in $A(\ccat)$, which is true for any representable, but the same is true for its image in $A(\ccat) / K$. We first check that for any $a \in A(\ccat)$, 
\begin{equation}
\label{equation:maps_into_injective_same_in_quotient}
\Hom_{A(\ccat)/K}(a, i) \simeq \Hom_{A(\ccat)}(a, i).
\end{equation}
By the direct construction of the Gabriel quotient, the left hand side is isomorphic to
\[
\varinjlim \Hom_{A(\ccat)}(a^{\prime}, i^{\prime}),
\]
where the colimit is taken over those subobjects $a^{\prime} \hookrightarrow a$ such that the quotient is in $K$, and those quotients $i \twoheadrightarrow i^{\prime}$ such that the kernel is in $K$. Since $i$ does not admit non-zero maps out of objects in $K$, there are no quotients of the latter type except isomorphisms. Similarly, since $i$ is injective not admitting maps out of objects of $K$, we see that $\Hom_{A(\ccat)}(a, i) \rightarrow \Hom_{A(\ccat)}(a^{\prime}, i)$ is an isomorphism for any such $a^{\prime}$. This shows (\ref{equation:maps_into_injective_same_in_quotient}).

We now check that $\Hom_{A(\ccat)/K}(-, i)$ takes monomorphisms to epimorphisms. Any map in the Gabriel quotient is isomorphic to an image of a map $a \rightarrow b$, and it is monic in the quotient if and only if its kernel $k$ in $A(\ccat)$ is in $K$. Then, we have an exact sequence
\[
\Hom_{A(\ccat)}(b, i) \rightarrow \Hom_{A(\ccat)}(a, i) \rightarrow \Hom_{A(\ccat)}(k, i) \rightarrow 0,
\]
where the third term is zero. Since we have checked these maps coincide with those in the quotient, we see that the image of $i$ is injective in $A(\ccat)/K$, as claimed.

We now claim that $A(\ccat)/K$ has enough injectives. To see this, observe that any object of $A(\ccat)$ is a subobject of a representable $x$. The same is then true in the quotient, so it is enough to show that any image of a representable embeds into an injective of $A(\ccat)/K$. If $\ccat$ has enough $\ecat$-injectives, we have an $\ecat$-monic map $x \rightarrow i$ into an $\ecat$-injective. Then, the image of this map is monic in the quotient, which is what we wanted. 

Consider the composite homology theory 
\[
\ccat \rightarrow A(\ccat) \rightarrow A(\ccat)/K,
\]
if $i \in \ccat$ is an $\ecat$-injective, then (\ref{equation:maps_into_injective_same_in_quotient}) implies that its image in the quotient can be lifted to an associated injective of $\ccat$, which will be $i$ itself. Since any object of the quotient embeds into one of the form $y(i)$, by passing to direct summands, we deduce that any injective of the quotient can be lifted to $\ccat$. Then, \cref{theorem:characterization_of_adapted_homology_theories} implies that the quotient functor has a fully-faithful right adjoint; in other words, $K$ is localizing.

Conversely, suppose we have a localizing subcategory $K$ such that $A(\ccat)/K$ has enough injectives. Let $c \in \ccat$ be an object and choose a monic $c \rightarrow i$ into an injective in the quotient. Since the projection is exact, the right adjoint $R\colon A(\ccat)/K \hookrightarrow A(\ccat)$ preserves injectives, and the chosen map is adjoint to one of the form $c \rightarrow Ri$, where the target is an injective of $A(\ccat)$. Thus, the latter must be representable by \cref{remark:injectives_are_representable_when_c_idemcomplete}, and one can verify immediately that the corresponding homotopy class is the needed $\ecat$-monic map into an $\ecat$-injective.
\end{proof}

\begin{corollary}
\label{corollary:millers_adams_sseq_has_e2_term_given_by_exts}
Let $\ccat$ be an idempotent complete stable $\infty$-category. Then, any epimorphism class $\ecat$ such that $\ccat$ has enough $\ecat$-injectives is the class of $H$-epimorphisms for a unique adapted homology theory $H\colon \ccat \rightarrow \acat$. 

In particular, if $c, d \in \ccat$ are any two objects, then in the $\ecat$-Adams spectral sequence obtained by applying $[c, -]$ to an $\ecat$-Adams resolution of $d$ we have
\[
E^{s, t}_{2} \simeq \Ext_{\acat}^{s, t}(H(c), H(d)).
\]
That is, all Miller's Adams spectral sequences have $E_{2}$-pages given by $\Ext$-groups in appropriate abelian category with enough injectives. 
\end{corollary}

\begin{remark}
\label{remark:characterization_of_the_image_of_right_adjoint}
In the context of \cref{theorem:adapted_homology_theories_correspond_to_injective_epimorphism_classes}, it is standard that the image of the right adjoint $A(\ccat)/K \hookrightarrow A(\ccat)$ can be characterized as those objects $x$ such that $\Ext^{i}_{A(\ccat)}(k, x) = 0$ for any $k \in K$ and $i = 0, 1$.
\end{remark}

\subsection{Digression: Adapted factorization}
\label{subsection:digression_adapted_factorization}

As we observed in \cref{corollary:millers_adams_sseq_has_e2_term_given_by_exts}, the fact that any epimorphism class with enough injectives is the class of homology epimorphisms for some adapted homology theory $H\colon \ccat \rightarrow \acat$ is useful in that it allows one to use homological algebra to compute the $E_{2}$-page of the relevant spectral sequence. To make the best use of this, we would like to have as explicit description of $\acat$ as possible. 

In this short section, we will give one such description in the special case where $\ecat$ is given by homology epimorphisms, but the homology theory in question is not necessarily adapted.

\begin{theorem}[Adapted factorization]
\label{theorem:adapted_factorization_of_homology_theories}
Let $H\colon \ccat \rightarrow \acat$ be a homology theory such that $\acat$ has enough injectives which lift to injectives of $\ccat$. Then, there is a unique factorization
\[
\begin{tikzcd}
	{\ccat} && {\acat} \\
	& {\acat^{\star}}
	\arrow["{H}", from=1-1, to=1-3]
	\arrow["{H^{\star}}"', from=1-1, to=2-2]
	\arrow["{U}"', from=2-2, to=1-3]
\end{tikzcd},
\]
where $H^{\star}$ is an adapted homology theory and $U$ is a comonadic exact left adjoint. In particular, $\acat^{\star}$ is the category of comodules over a certain left exact comonad $C$ over $\acat$. 
\end{theorem}

\begin{proof}
By \cref{theorem:characterization_of_adapted_homology_theories}, to give a factorization as above is equivalent to giving a factorization 
\[
\begin{tikzcd}
	{A(\ccat)} && {\acat} \\
	& {\acat^{\star}}
	\arrow["{L}", from=1-1, to=1-3]
	\arrow["{L^{\star}}"', from=1-1, to=2-2]
	\arrow["{U}"', from=2-2, to=1-3]
\end{tikzcd},
\]
where $L^{\star}$ is a quotient by a localizing subcategory, $U$ is a comonadic exact left adjoint, and $\acat^{\star}$ has enough injectives. By \cref{theorem:properties_of_comonadic_factorization_of_an_exact_left_adjoint}, the category of comodules over the comonad $C \colonequals LR$ has the first two properties. It has enough injectives 
by \cref{proposition:category_of_comodules_enough_injectives}. 
\end{proof}

\begin{example}
\label{example:adams_type_factorization}
Let $R$ be a ring spectrum, so that we have a homology theory $R_{*}\colon \spectra \rightarrow \Mod_{R_{*}}$ where the target is Grothendieck, in particular has enough injectives. 

If $R$ is Adams-type, then this factors through the category of $R_{*}R$-comodules, and the uniqueness of the adapted factorization is exactly
\begin{center}
\begin{tikzcd}
	{\spectra} && {\Mod_{R_{*}}} \\
	& {\Comod_{R_*R}}
	\arrow["{L}", from=1-1, to=1-3]
	\arrow["{L^{\star}}"', from=1-1, to=2-2]
	\arrow["{U}"', from=2-2, to=1-3]
\end{tikzcd}.
\end{center}
Indeed, $U$ is exact and conservative, so the two homology theories determine the same class of epimorphisms, and $L^{\star}$ is adapted by a theorem of Devinatz \cite{devinatz1997morava}[1.5]. 
\end{example}

\begin{warning}
In \cref{example:adams_type_factorization}, the relevant comonad is $R_{*}R \otimes_{R_{*}} -$, which is also right exact. This is not always the case for factorizations obtained by \cref{theorem:adapted_factorization_of_homology_theories}, and is one advantage of having the Adams-type condition which guarantees that the universal abelian category we are looking for is exactly the category $\Comod_{R_*R}$. 
\end{warning}

\begin{remark}
As one possible application, \cref{theorem:adapted_factorization_of_homology_theories} is used by Burklund and the second author to describe the $E_2$-term of various Adams spectral sequences based on a non-flat homology theory in terms of quiver representations \cite{burklundpstragowski2023}. 
\end{remark}

\begin{remark}
Ideas similar to our proof of \cref{theorem:characterization_of_adapted_homology_theories} appear in the recent work of Balmer and Cameron, who use comonads to describe certain homological residue fields \cite{balmer2021computing}. In particular, applying \cref{example:adams_type_factorization} to the case of Morava $K$-theories, we recover their calculation of the homological residue fields of $\spectra^{\omega}$ as being equivalent to $K(n)_{*}K(n)$-comodules. 
\end{remark}

\section{Prestable Freyd envelope}
\label{section:prestable_freyd_envelope}

In this section, we will study the $\infty$-categorical generalization of the classical Freyd envelope discussed in \S\ref{subsection:recollections_on_freyd_envelope}. This will be a prestable $\infty$-category obtained from a given additive $\infty$-category by freely attaching certain kinds of colimits. 

\subsection{Almost perfect presheaves}
\label{subsection:almost_perfect_presheaves}

Recall from \cref{remark:freyd_envelope_obtained_by_adjoining_cokernels} that the classical Freyd envelope $A(\ccat)$ is obtained from the homotopy category $h\ccat$ by formally adjoining cokernels. In an additive category, any cokernel sequence
\[
a \rightarrow b \rightarrow c
\]
can be rewritten as a reflexive coequalizer, namely
\[
a \oplus b \rightrightarrows b \rightarrow c.
\]
Using this, one can show that the classical Freyd envelope can be described as the category obtained from $h\ccat$ by formally adjoining reflexive coequalizers. 

The $\infty$-categorical analogue of a reflexive coequalizers is given by a geometric realization of a simplicial object, which suggests that to obtain an $\infty$-categorical analogue of the Freyd envelope, we should enlarge a given $\ccat$ by freely adjoining geometric realizations. 

Alternatively, by \cref{proposition:if_c_has_finite_limits_freyd_envelope_is_an_abelian_subcat}, if $\ccat$ admits finite limits, then $A(\ccat)$ can be described as the smallest subcategory of $\Fun_{\Sigma}(\ccat^{op}, \sets)$ containing all the representables and closed under finite colimits. This suggests a different generalization of the Freyd envelope, where we take the same definition, but replace sets by the $\infty$-category of spaces. 

As it turns out, these two approaches do not in general coincide. However, they \emph{both} lead to important $\infty$-categories, each interesting in their own right. The two possible results are closely related, and we will discuss both, starting with the former idea of freely attaching geometric realizations. 

\begin{definition}
Let $\ccat$ be an additive $\infty$-category. We say a product-preserving presheaf $X\colon \ccat^{op} \rightarrow \spaces$ of spaces is \emph{almost perfect} if it belongs to the smallest subcategory of $\Fun_{\Sigma}(\ccat^{op}, \spaces)$ containing the representables and closed under geometric realizations.
\end{definition}

\begin{definition}
\label{definition:prestable_freyd_envelope}
The \emph{prestable Freyd envelope} $A_{\infty}(\ccat)$ is the $\infty$-category of almost perfect presheaves on $\ccat$. It is an additive subcategory of $\Fun_{\Sigma}(\ccat^{op}, \spaces)$. 
\end{definition}

\begin{remark}
Since geometric realizations commute with products, we can alternatively describe $A_{\infty}(\ccat)$ as the smallest subcategory of the full presheaf $\infty$-category $\Fun(\ccat^{op}, \spaces)$ which contains representables and is closed under geometric realizations. 
\end{remark}

\begin{remark}
\label{remark:prestable_envelope_obtained_by_freely_adding_geometric_realizations}
By construction, the $\infty$-category $A_{\infty}(\ccat)$, has the following universal property: for any $\infty$-category $\dcat$ admitting geometric realizations, restriction along the Yoneda embedding gives an equivalence
\[
\Fun_{\sigma}(A_{\infty}(\ccat), \dcat) \simeq \Fun(\ccat, \dcat)
\]
between geometric-realization preserving functors out of $A_{\infty}(\ccat)$, and all functors out of $\ccat$, see \cite{higher_topos_theory}[5.3.5.9]. The inverse to the above equivalence is given by left Kan extension.
\end{remark}

\begin{remark}
It is implicit in our account that $\Fun_{\Sigma}(\ccat^{op}, \spaces)$, and its direct-product closed subcategory $A_{\infty}(\ccat)$, is additive. To see this, observe that since $\ccat$ is additive, any product-preserving presheaf of spaces has a canonical lift to a presheaf of $\mathbf{E}_{\infty}$-spaces and hence can be identified with a presheaf of connective spectra. 
\end{remark}

It is only understandable that a reader would feel unease at considering functor $\infty$-categories out of a non-small source, which does often lead to set-theoretic difficulties. We first show that this is not the case here, and that $A_{\infty}(\ccat)$ inherits good finiteness properties from $\ccat$. 

\begin{proposition}
The $\infty$-category $A_{\infty}(\ccat)$ is locally small.
\end{proposition}

\begin{proof}
The $\infty$-category of those presheaves $X \in \Fun_{\Sigma}(\ccat^{op}, \spaces)$, such that for an arbitrary other $Y$, the mapping space $\map(X, Y)$ is small, is closed under geometric realization, since small spaces are closed under totalizations. Since it also contains all the representables, by the Yoneda lemma, it necessarily contains all of $A_{\infty}(\ccat)$. 
\end{proof}

To give a better characterization of being almost perfect, we need to introduce some notation.
\begin{notation}
The Yoneda embedding of $\ccat$ factors through almost perfect presheaves, and we will denote this factorization by $\nu\colon \ccat \rightarrow A_{\infty}(\ccat)$. Explicitly, we have 
\[
\nu(c)(d) \colonequals \Map_{\ccat}(d, c).
\]
We will keep using $y$ to denote the Yoneda embedding in the homotopy category, so that 
\[
y(c)(d) \colonequals \Map_{h\ccat}(d, c) \simeq [d, c].
\]
Notice that $y(c) \colonequals \tau_{\leq 0} \nu(c)$, the Postnikov $0$-truncation in the $\infty$-category $\Fun_{\Sigma}(\ccat^{op}, \spaces)$.
\end{notation}

\begin{remark}
The $\nu$ notation for the Yoneda embedding was first used in the context of the $\infty$-category of synthetic spectra of \cite{pstrkagowski2018synthetic}.
\end{remark}

\begin{proposition}
\label{proposition:characterization_of_almost_perfect_presheaves}
Let $\ccat$ be an additive $\infty$-category which admits finite limits. Then, the following conditions on a a product-preserving presheaf $X\colon \ccat^{op} \rightarrow \spaces$ are equivalent: 

\begin{enumerate}
    \item $X$ can be written as a geometric realization $X \simeq | \nu(c_{\bullet})|$ of a simplicial diagram of representables,
    \item $X$ is almost perfect, 
    \item $\pi_{k}X$ is finitely presented presheaf of abelian groups; that is, $\pi_{k}X \in A(\ccat)$, for every $k \geq 0$. 
\end{enumerate}
\end{proposition}

\begin{proof}
$(1) \Rightarrow (2)$ is clear. To see $(2) \Rightarrow (3)$, observe that for every representable we have
\[
(\pi_{k} \nu(c))(d) \simeq \pi_{k} \Map(d, c) \simeq [d, \Omega^{k} c]
\]
which is finitely presented. Thus, it is enough to check that the last condition is closed under geometric realizations, so suppose that $X \simeq | X_{\bullet} |$ where all of $X_{\bullet}$ have finitely-presented homotopy groups. We have a convergent first quadrant spectral sequence
\[
H_{s} \pi_{t} X_{\bullet} \Rightarrow \pi_{s+t} X,
\]
and the left hand side is finitely presented in each degree, since finitely presented presheaves form an abelian subcategory of $\Fun_{\Sigma}(\ccat^{op}, \abeliangroups)$ by \cref{proposition:if_c_has_finite_limits_freyd_envelope_is_an_abelian_subcat}. Since they are also closed under extensions, we deduce that $\pi_{s+t}X$ is also finitely presented, as needed. 

We now move on to $(3) \Rightarrow (1)$. Observe that since the classical Freyd envelope is an abelian subcategory, presheaves with finitely presented homotopy groups are closed under fibres, cofibres and products in $\Fun_{\Sigma}(\ccat^{op}, \spaces)$, and so under all finite limits and colimits. 

Now suppose that $X$ has finitely presented homotopy, so that in particular there is a surjection $y(c) \rightarrow \pi_{0}X$ for some $c \in \ccat$. This determines a point of $\pi_{0}X(c)$, and so a homotopy class of maps $\nu(c) \rightarrow X$ which is $\pi_{0}$-surjective. Proceeding inductively, and using that all latching objects will again have finitely presented homotopy groups, we see that $X$ admits a levelwise-hypercover by representables in $\Fun_{\Sigma}(\ccat^{op}, \spaces)$, see \cite{higher_algebra}[7.2.1.4]. All hypercovers of connective spectra are colimit diagrams, ending the argument.
\end{proof}

\begin{corollary}
\label{corollary:almost_perfect_presheaves_closed_under_finite_limits_colimits_truncations}
Suppose that $\ccat$ is an additive $\infty$-category which admits finite limits. Then, the prestable Freyd envelope $A_{\infty}(\ccat)$ is closed under finite limits, finite colimits, extensions and Postnikov truncations in $\Fun_{\Sigma}(\ccat^{op}, \spaces)$. In particular, it is prestable. 
\end{corollary}

\begin{proof}
This is immediate from the third characterization of \cref{proposition:characterization_of_almost_perfect_presheaves}, as the classical Freyd envelope enjoys the same closure properties in additive presheaves of abelian groups. 
\end{proof}

\begin{remark}[Weakening existence of finite limits]
\label{remark:for_characterization_of_ap_presheaves_no_finite_limits_needed}
Observe that in the proof of  \cref{proposition:characterization_of_almost_perfect_presheaves} we did not use the full power of the assumption that $\ccat$ has finite limits, and in fact we need slightly less than that. The two conditions required for \cref{proposition:characterization_of_almost_perfect_presheaves} and consequently \cref{corollary:almost_perfect_presheaves_closed_under_finite_limits_colimits_truncations} to hold are that
\begin{enumerate}
    \item the homotopy category $h \ccat$ has weak kernels, which guarantees that the classical Freyd envelope is closed under kernels inside additive presheaves and 
    \item for any $c \in \ccat$ and any $k \geq 0$, the additive presheaf $\pi_{k}\Map(-, c)$ is finitely presented. 
\end{enumerate}
These conditions are strictly weaker; for example, if $\ccat$ satisfies them, then so do all of the higher homotopy categories $h_{n} \ccat$ for $n \geq 1$. On the other hand, existence of limits is rarely if ever preserved by passing to homotopy categories. 
\end{remark}

\begin{corollary}
\label{corollary:almost_perfect_presheaves_preserve_all_direct_sums}
Let $X\colon \ccat^{op} \rightarrow \spaces$ be an almost perfect presheaf. Then $X$ takes all small direct sums which exist in $\ccat$ to products of spaces.
\end{corollary}

\begin{proof}
This can be tested on homotopy groups, where it is true by \cref{proposition:characterization_of_almost_perfect_presheaves} and \cref{lemma:coproducts_in_the_freyd_envelope}. \end{proof}

\begin{proposition}
\label{proposition:nu_preserves_direct_sums}
Let $\ccat$ have $\kappa$-indexed direct sums. Then so does $A_{\infty}(\ccat)$, and $\nu\colon \ccat \rightarrow A_{\infty}(\ccat)$ preserves them. 
\end{proposition}

\begin{proof}
Suppose that $X_{\alpha}$ is a family of almost perfect presheaves, so that we can write each one 
\[
X_{\alpha} \simeq | \nu (c_{\alpha, \bullet}) |
\]
as a geometric realization of representables by \cref{proposition:characterization_of_almost_perfect_presheaves}.  We claim that 
\[
| \nu (\oplus_{\alpha} c_{\alpha, \bullet}) |
\]
is the needed direct sum $\bigoplus _{\alpha} X_{\alpha}$ in the prestable Freyd envelope. To see this, suppose that $Y \in A_{\infty}(\ccat)$ is another almost perfect presheaf. Then, 
\[
\Map(| \nu (\bigoplus_{\alpha} c_{\alpha, \bullet}) |, Y) \simeq \textnormal{Tot}(Y(\bigoplus_{\alpha} c_{\alpha, \bullet})) \simeq \textnormal{Tot}(\prod_{\alpha} Y(c_{\alpha, \bullet})) 
\]
where in the middle equivalence we have used \cref{corollary:almost_perfect_presheaves_preserve_all_direct_sums}, and further
\[
\textnormal{Tot}(\prod_{\alpha} Y(c_{\alpha, \bullet})) \simeq \prod_{\alpha} \textnormal{Tot}  Y(c_{\alpha, \bullet})) \simeq \prod_{\alpha} \Map(|\nu (c_{\alpha, \bullet})|, Y) \simeq \prod_{\alpha} \Map(X_{\alpha}, Y),
\]
as needed.
\end{proof}

\begin{proposition}
\label{proposition:products_in_prestable_freyd_envelope}
Let $\kappa$ be infinite and suppose that $\ccat$ has $\kappa$-indexed products. Then, $A_{\infty}(\ccat)$ is closed under such products in $\Fun_{\Sigma}(\ccat^{op}, \spaces)$, and hence under all $\kappa$-small limits. Moreover, if that is the case, then $\nu\colon \ccat \rightarrow A_{\infty}(\ccat)$ preserves all $\kappa$-small limits. 
\end{proposition}

\begin{proof}
Since products in the presheaf $\infty$-category are detected at the level of levelwise homotopy groups, by \cref{proposition:characterization_of_almost_perfect_presheaves} this is an immediate consequence of the same closure for the classical Freyd envelope, which we have shown in \cref{lemma:freyd_envelope_has_products_when_c_does}. The case of $\kappa$-small limits follows from the fact that an additive $\infty$-category with $\kappa$-small products and fibres has all $\kappa$-small limits. 

Finally, the claim about $\nu$ is immediate, as the Yoneda embedding preserves all limits which exist in the source. 
\end{proof}

\begin{remark}
As in the classical case, there is a distinction between \cref{proposition:nu_preserves_direct_sums} and \cref{proposition:products_in_prestable_freyd_envelope}. Namely, the inclusion $A_{\infty}(\ccat) \hookrightarrow \Fun_{\Sigma}(\ccat^{op}, \spaces)$ preserves even infinite products, but not necessarily infinite direct sums. 
\end{remark}

As we have observed before, $A_{\infty}(\ccat)$ is a prestable $\infty$-category. In the latter context, it is possible to define an internal variant of homotopy groups, as we will now recall.

\begin{notation}
\label{notation:t_structure_homotopy}
If $\ccat$ is prestable, the \emph{heart} is $\ccat^{\heartsuit} \colonequals \tau_{\leq 0} \ccat$, the subcategory of discrete objects. If $\ccat$ has finite limits, then the inclusion of the heart admits a left adjoint denoted by 
\[
\pi_{0}\colon \ccat \rightarrow \ccat^{\heartsuit};
\]
to check it exists one can verify that the cofibre of the canonical map $\Sigma \Omega c \rightarrow c$ satisfies the needed universal property of $\pi_{0} c$. 

It then follows from \cite{BBD} (see also \cite{higher_algebra}[1.2.1.12]) that the heart has all finite limits and colimits, and in fact that it is abelian, since it can be identified with the heart on the corresponding $t$-structure on the $\infty$-category of Spanier-Whitehead objects in $\ccat$, see \cite{lurie_spectral_algebraic_geometry}[C.1.2.9 and C.1.2.11]. 

Using the heart and $\pi_{0}$ we define the higher \emph{homotopy groups} by letting
\[
\pi_{0}\colon \ccat \rightarrow \ccat^{\heartsuit}
\]
be the truncation, and setting 
\[
\pi_{k}(c) \colonequals \pi_{0}(\Omega^{k} c).
\]
These correspond to the usual $t$-structure homotopy groups on the stabilization, although in the prestable setting they are necessarily non-negatively graded. 
\end{notation}

\begin{example}
\label{example:heart_of_prestable_envelope}
Suppose that $\ccat$ is an additive $\infty$-category which admits finite limits, so that $A_{\infty}(\ccat)$ is prestable. A presheaf $X \in A_{\infty}(\ccat)$ is discrete if and only if it is discrete as an object of $\Fun(\ccat^{op}, \spaces)$; that is, when it is valued in discrete spaces. It follows that we have an equivalence 
\[
A_{\infty}(\ccat)^{\heartsuit} \simeq A(\ccat)
\]
between the heart of the prestable Freyd envelope and the classical abelian Freyd envelope. In terms of this equivalence, the t-structure homotopy groups of \cref{notation:t_structure_homotopy} coincide with the levelwise ones in the sense that 
\[
(\pi_{k}X)(c) \simeq \pi_{k} X(c).
\]
\end{example}

\begin{remark}
\label{remark:local_grading_on_prestable_freyd_envelope}
Using the universal property of $A_{\infty}(\ccat)$ outlined in \cref{remark:prestable_envelope_obtained_by_freely_adding_geometric_realizations}, we see that if $\ccat$ is locally graded, then $A_{\infty}(\ccat)$ inherits an essentially unique local grading such that $\nu$ admits a structure of a locally graded functor. 

Using the latter constraint, one verifies directly that the induced local grading is given explicitly by
\[
(X[1])(c) \colonequals X(c[-1]).
\]
In the important case when $\ccat$ is stable and locally graded through suspension, we thus have 
\[
(X[1])(c) \colonequals X(\Sigma^{-1}c);
\]
we will study this particular grading in more detail in \S\ref{subsection:thread_structure_on_prestable_freyd_envelope}.
\end{remark}

\subsection{Perfect presheaves} 

As we discussed at the beginning of \S\ref{subsection:almost_perfect_presheaves}, another way to characterize the classical Freyd envelope is as a finite colimit closure of representables in the category of discrete additive presheaves. This naturally generalizes to the $\infty$-categorical context in the following way. 

\begin{definition}
\label{definition:perfect_presheaf}
Let $\ccat$ be an additive $\infty$-category. Then, an additive presheaf $X \in \Fun_{\Sigma}(\ccat^{op}, \spaces)$ is \emph{perfect} if it belongs to the smallest subcategory containing all representables and closed under finite colimits. 
\end{definition}

\begin{definition}
The \emph{perfect prestable Freyd envelope} is the $\infty$-category $A_{\infty}^{\omega}(\ccat)$ of perfect presheaves. 
\end{definition}

\begin{remark}
\label{remark:universal_property_of_perfect_freyd_envelope}
By construction, $A_{\infty}^{\omega}(\ccat)$ has the following universal property: any additive functor $\ccat \rightarrow \dcat$ into an additive $\infty$-category with finite colimits extends uniquely to a right exact functor $A_{\infty}^{\omega}(\ccat) \rightarrow \dcat$. 
\end{remark}

\begin{lemma}
\label{lemma:perfect_presheaves_are_almost_perfect}
Let $\ccat$ be an additive $\infty$-category. Then, any perfect presheaf $X$ is almost perfect; that is, $A_{\infty}^{\omega}(\ccat)$ is naturally a subcategory of $A_{\infty}(\ccat)$. 
\end{lemma}

\begin{proof}
Since $A_{\infty}(\ccat)$ contains all representables and is closed under direct sums, it is enough to check that it is closed under cofibres. Suppose that $X \rightarrow Y$ is a map of almost perfect presheaves. As both are additive, this map presents $X$ as acting on $Y$ with respect to the cartesian symmetric monoidal structure, so that we can write a bar construction of the form 
\[
\ldots \ X \oplus X \oplus Y \triplerightarrow X \oplus Y \rightrightarrows Y
\]
The geometric realization of this bar construction is equivalent to the cofibre of $X \rightarrow Y$, so that it is also almost perfect. 
\end{proof}
In \cref{corollary:almost_perfect_presheaves_closed_under_finite_limits_colimits_truncations} we observed that if $\ccat$ admits finite limits, then almost perfect presheaves are closed under limits in the $\infty$-category of all additive presheaves. Our main result of this section will be an analogous result for perfect presheaves. 

\begin{remark}
While the conclusions are the same, to prove that perfect presheaves are closed under finite limits is both more involved and more surprising. For example, if $R$ is a classical commutative ring, then connective almost perfect complexes are closed under finite limits in the derived $\infty$-category $\dcat_{\geq 0}(R)$ as soon as $R$ is coherent \cite{higher_algebra}[7.2.4.18]; for example, noetherian. 

On the other hand, for connective perfect complexes to be closed under finite limits in $\dcat_{\geq 0}(R)$ requires that every finitely presented module has a finite projective resolution, a very strong assumption on the ring $R$. 

Beware that the crux of the difficulty is that we are asking for finite limit closure inside the connective derived $\infty$-category rather than the stable one; in the latter finite colimit closure and finite limit closure coincide, and asking for finite limit closure has no teeth. 
\end{remark}

\begin{lemma}
\label{lemma:discrete_almost_perfect_presheaves_are_perfect}
Let $\ccat$ be an additive $\infty$-category with finite limits and let $x \in A(\ccat)$ be a discrete finitely presented presheaf. Then, any $\pi_{0}$-surjection $\nu(c) \rightarrow x$ can be completed to a cofibre sequence
\[
\mathrm{cofib}(\nu(e) \rightarrow \nu(d)) \rightarrow \nu(c) \rightarrow x;
\]
for a suitable map $e \rightarrow d$; in particular, $x$ is perfect. 
\end{lemma}

\begin{proof}
Since $x$ is finitely presented, the kernel of $\pi_{0} \nu(c) \simeq y(c) \rightarrow x$ is finitely presented and hence we can write it as a quotient of $y(d)$ for some other $d \in \dcat$. This determines a homotopy class of maps $d \rightarrow c$ such that $x \simeq \pi_{0}(\mathrm{cofib}(\nu(d) \rightarrow \nu(c))$.

Let $e$ be a fibre of $d \rightarrow c$ in $\ccat$, so that we have a fibre sequence
\[
\nu(e) \rightarrow \nu(d) \rightarrow \nu(c)
\]
of perfect presheaves. The long exact sequence of homotopy groups tells us that if $z$ denotes the cofibre of $\nu(e) \rightarrow \nu(d)$, then the induced map $z \rightarrow \nu(c)$ is an isomorphism on $\pi_{k}$ for $k > 0$ and is injective on $\pi_{0}$, with the image equal to $\mathrm{im}(\pi_{0} \nu(d)).$ Thus, we have a cofibre sequence
\[
z \rightarrow \nu(c) \rightarrow x
\]
as needed. 
\end{proof}

\begin{lemma}
\label{lemma:perfection_detected_by_looping_into_a_representable}
Let $\ccat$ be an additive $\infty$-category with finite limits and let $x \in A_{\infty}^{\omega}(\ccat)$ be a perfect presheaf. Then, $\Omega^{k} x$ is representable for $k$ large enough. 
\end{lemma}

\begin{proof}
Since $\Omega \nu(c) \simeq \nu(\Omega c)$, the collection of presheaves with this property contains all the representables. Thus, it is enough to verify that it is closed under finite colimits. 

Suppose that $x \rightarrow y \rightarrow z$ is a cofibre sequence such that $x, y$ have the needed property. For each $k \geq 0$ we then have a fibre sequence
\[
\Omega^{k+1} z \rightarrow \Omega^{k} y \rightarrow \Omega^{k} x.
\]
Choose $k$ large enough so that $\Omega^{k} y \simeq \nu(d), \Omega^{k}x \simeq \nu(c)$ for some $d, c \in \ccat$. It follows that $\Omega^{k+1} z \simeq \nu(e)$, where $e$ is the fibre of $d \rightarrow c$, and so is representable. 
\end{proof}

\begin{theorem}
\label{theorem:if_ccat_has_finite_limits_perfect_presheaves_closed_under_limits}
Let $\ccat$ be an additive $\infty$-category with finite limits. Then, the perfect prestable Freyd envelope $A_{\infty}^{\omega}(\ccat)$ is closed under finite limits in additive presheaves of spaces. In particular, it is a prestable $\infty$-category with finite limits. 
\end{theorem}

\begin{proof}
The collection of those almost perfect $x$ such that $\Omega^{k} x$ is representable for $k$ large enough is closed under finite limits when $\ccat$ has them by the argument of \cref{lemma:perfection_detected_by_looping_into_a_representable}. Thus, it is enough to check that any presheaf with this property is necessarily perfect. As any representable is perfect, it is enough to show that if $\Omega x$ is perfect, then so is $x$. 
Choose a $\pi_{0}$-surjection $\nu(c) \rightarrow x$, which by taking fibres we can complete to a diagram of almost perfect presheaves
\[\begin{tikzcd}
	{\Omega x} & {f} & {f_{0}} \\
	{0} & {\nu(c)} & {\nu(c)} \\
	{x_{\geq 1}} & {x} & {\pi_{0}x}
	\arrow[from=3-1, to=3-2]
	\arrow[from=3-2, to=3-3]
	\arrow[from=2-3, to=3-3]
	\arrow[from=2-2, to=3-2]
	\arrow[from=2-1, to=3-1]
	\arrow[from=2-1, to=2-2]
	\arrow[from=2-2, to=2-3]
	\arrow[from=1-3, to=2-3]
	\arrow[from=1-2, to=2-2]
	\arrow[from=1-1, to=2-1]
	\arrow[from=1-1, to=1-2]
	\arrow[from=1-2, to=1-3]
\end{tikzcd}\]
where every column and row is a fibre sequence. In fact, the two right most columns are even cofibre sequences by construction, as the latter maps are $\pi_{0}$-surjective. We claim so is the top row. 

To see this, it is enough to check that $\pi_{0}f \rightarrow \pi_{0} f_{0}$ is surjective. Since $\pi_{0}x$ is a discrete presheaf, taking $\pi_{0}$ of the two right-most columns gives a diagram 
\[
\begin{tikzcd}
	& {\pi_{0}f} & {y(c)} & {\pi_{0}x} & 0 \\
	0 & {\pi_{0} f_{0}} & {y(c)} & {\pi_{0}x}
	\arrow[from=1-2, to=1-3]
	\arrow[from=1-3, to=1-4]
	\arrow[from=2-3, to=2-4]
	\arrow["{id}", from=1-4, to=2-4]
	\arrow[from=1-3, to=2-3]
	\arrow[from=1-2, to=2-2]
	\arrow[from=2-2, to=2-3]
	\arrow[from=1-4, to=1-5]
	\arrow[from=2-1, to=2-2]
\end{tikzcd}.
\]
Applying the snake lemma we see that $\pi_{0} f \rightarrow \pi_{0} f_{0}$ is surjective, as needed. 

By \cref{lemma:discrete_almost_perfect_presheaves_are_perfect}, we can write $f_{0} \simeq \mathrm{cofib}(\nu(e) \rightarrow \nu(d))$ for a suitable map $e \rightarrow d$. This determines a map $\nu(d) \rightarrow f_{0}$ which by $\pi_{0}$-surjectivity lifts to an arrow $\nu(d) \rightarrow f$ which by taking fibres we can complete to a diagram 
\[\begin{tikzcd}
	{\nu(e)} & {\nu(d)} & {f_{0}} \\
	{\Omega x} & f & {f_{0}}
	\arrow[from=2-2, to=2-3]
	\arrow[from=1-2, to=1-3]
	\arrow[from=1-2, to=2-2]
	\arrow["{id}", from=1-3, to=2-3]
	\arrow[from=1-1, to=2-1]
	\arrow[from=1-1, to=1-2]
	\arrow[from=2-1, to=2-2]
\end{tikzcd}\]
such that both rows are cofibre sequences. It follows that the left-most square is a pushout, and hence $f$ is perfect as a pushout of perfects. As $x \simeq \mathrm{cofib}(f \rightarrow \nu(c))$, we deduce that $x$ is also perfect, ending the argument. 
\end{proof}

\begin{remark}
By \cref{lemma:discrete_almost_perfect_presheaves_are_perfect}, the inclusion  $A_{\infty}^{\omega}(\ccat) \rightarrow A_{\infty}(\ccat)$ induces an equivalence on the hearts. 
\end{remark}

It will be useful to have an explicit description of perfect presheaves in terms of certain geometric realizations, analogous to \cref{proposition:characterization_of_almost_perfect_presheaves} in the almost perfect case. To show that any perfect presheaf has a suitable finite resolution, we will need the concept of projective dimension. 

\begin{notation}
\label{definition:ext_groups_in_prestable_freyd_envelope}
Let $x, y \in A_{\infty}(\ccat)$ be almost perfect presheaves. Then, for any $k \geq 0$ we put 
\[
\Ext^{k}(x, y) \colonequals \pi_{0} \Map_{A_{\infty}(\ccat)}(x, \Sigma^{k} y)
\]
and similarly for any $k < 0$ 
\[
\Ext^{k}(x, y) \colonequals \pi_{0} \Map_{A_{\infty}(\ccat)}(x, \Omega^{-k} y).
\]
\end{notation}

\begin{remark}
\label{remark:ext_groups_homotopy_groups_of_mapping_spectrum_hence_les}
The $\infty$-category $A_{\infty}(\ccat)$ is prestable and so can be identified with the positive part of a $t$-structure on its Spanier-Whitehead $\infty$-category. Then, the $\Ext$-groups of \cref{definition:ext_groups_in_prestable_freyd_envelope} can be identified with the homotopy groups of internal mapping spectra in the latter $\infty$-category, which is already stable. 

It follows that if $x \rightarrow y \rightarrow z$ is a cofibre sequence of almost perfect presheaves, then for any other $w$ we have long exact sequences
\[
\ldots \rightarrow \Ext^{-1}(x, w) \rightarrow \Ext^{0}(z, w) \rightarrow \Ext^{0}(y, w) \rightarrow \Ext^{0}(x, w) \rightarrow \Ext^{1}(z, w) \rightarrow \ldots
\]
and 
\[
\ldots \rightarrow \Ext^{-1}(w, z) \rightarrow \Ext^{0}(w, x) \rightarrow \Ext^{0}(w, y) \rightarrow \Ext^{0}(w, z) \rightarrow \Ext^{1}(w, x) \rightarrow \ldots.
\]
\end{remark}

\begin{definition}
Let $x \in A_{\infty}(\ccat)$ be an almost perfect presheaf. If $k \geq 0$, we say $x$ is of \emph{projective dimension at most $k$} if $\Ext^{s}(x, w) = 0$ for any other almost perfect $w$ and any $s > k$. 
\end{definition}

\begin{example}
\label{example:representables_are_of_proj_dimension_zero}
Let $x \colonequals \nu(c)$ be representable. Then, Yoneda lemma implies that $x$ is of projective dimension zero. 
\end{example}

\begin{lemma}
\label{lemma:fibres_along_surjection_from_rep_lowers_projective_dimension}
Let $x$ be of projective dimension at most $k > 0$, and choose a $\pi_{0}$-surjection $\nu(c) \rightarrow x$ from a representable. Then, the fibre $f$ of this map is at most of projective dimension $k-1$. 
\end{lemma}

\begin{proof}
Since $f \rightarrow \nu(c) \rightarrow x$ is a cofibre sequence by construction, this is immediate from \cref{example:representables_are_of_proj_dimension_zero} and the long exact sequence of \cref{remark:ext_groups_homotopy_groups_of_mapping_spectrum_hence_les}.
\end{proof}

\begin{lemma}
\label{lemma:perfects_of_projective_dimension_zero_are_rep}
Let $x$ be of projective dimension zero. Then, $x$ is a direct summand of a representable. If $x$ is perfect, then it is representable. 
\end{lemma}

\begin{proof}
Choose a $\pi_{0}$-surjection $\nu(c) \rightarrow x$ from a representable. If $x$ is of projective dimension zero, then the long exact sequence of \cref{remark:ext_groups_homotopy_groups_of_mapping_spectrum_hence_les} implies that it has a left lifting property with respect to $\pi_{0}$-surjections and hence the map from $\nu(c)$ is split, as needed. 

The other part is more interesting. By what we have shown above, if $x$ is of projective dimension zero, then it is a direct summand of a representable. It follows that $\Omega^{k} x$ is also of projective dimension zero for all $k \geq 0$, and it is again perfect by \cref{theorem:if_ccat_has_finite_limits_perfect_presheaves_closed_under_limits}. As it is always representable for $k$ large enough by \cref{lemma:perfection_detected_by_looping_into_a_representable}, it will be enough to verify that if $x$ is of projective dimension zero and $\Omega x$ is representable, then $x$ is representable. 

In the proof of \cref{theorem:if_ccat_has_finite_limits_perfect_presheaves_closed_under_limits}, we have constructed a cofibre sequence
\[
f \rightarrow \nu(c) \rightarrow x
\]
such that $f$ itself fits into a cofibre sequence
\[
\nu(e) \rightarrow \nu(d) \oplus \Omega x \rightarrow f.
\]
When $x$ is of projective dimension zero, then the first sequence is split. Thus, $f$ is of projective dimension zero and hence the second one is also split. 

By inductive assumption, we can we write $\Omega x \simeq \nu(s)$. It follows that $f \simeq \nu(\mathrm{cofib}(e \rightarrow d \oplus s))$, as $\nu$ preserves cofibres along split cofibre sequences, as any additive functor does. Thus, $f$ is representable, say $f \simeq \nu(t)$ and $x \simeq \nu(\mathrm{cofib}(t \rightarrow c))$ by the same argument. 
\end{proof}

\begin{lemma}
\label{lemma:any_perfect_sheaf_of_finite_projective_dimension}
Let $x$ be a perfect presheaf. Then, $x$ is of finite projective dimension at most $k$ for some $k \geq 0$. 
\end{lemma}

\begin{proof}
By the long exact sequence of \cref{remark:ext_groups_homotopy_groups_of_mapping_spectrum_hence_les}, the subcategory of additive presheaves with this property is closed under cofibres. As it also contains all representables, we deduce that it contains all perfect presheaves. 
\end{proof}

\begin{corollary}
\label{corollary:perfect_presheaves_realizations_of_skeletal_simplicial_objects}
Let $\ccat$ be an additive $\infty$-category with finite limits. Then an additive presheaf $x \in \Fun_{\Sigma}(\ccat^{op}, \spaces)$ is perfect if and only if it is a geometric realization of a simplicial object $\nu(c_{\bullet})$ which is $k$-skeletal for some $k \geq 0$. 
\end{corollary}

\begin{proof}
Geometric realization of a $k$-skeletal simplicial object coincides with its $k$-skeleton, so it is a finite colimit. Thus, any presheaf with this property is perfect. 

Conversely, if $x$ is perfect, then it is also almost perfect by \cref{lemma:perfect_presheaves_are_almost_perfect}. Thus, the inductive procedure of choosing $\pi_{0}$-surjections from a representable in the proof of \cref{proposition:characterization_of_almost_perfect_presheaves} shows that it is a geometric realization of representables. 

Since $x$ is perfect, by \cref{lemma:any_perfect_sheaf_of_finite_projective_dimension} it is of projective dimension at most $k$ for some $k \geq 0$. It follows from \cref{lemma:fibres_along_surjection_from_rep_lowers_projective_dimension} that this inductive procedure results in a perfect of projective dimension zero after $k$-steps, and hence a representable by  \cref{lemma:perfects_of_projective_dimension_zero_are_rep}. Thus, the process can be stopped there, ending with a $k$-skeletal simplicial object. 
\end{proof}

\begin{remark}[Comparison to the almost perfect case]
As we observed in \cref{remark:for_characterization_of_ap_presheaves_no_finite_limits_needed}, the corresponding statement for almost perfect presheaves \cref{corollary:almost_perfect_presheaves_closed_under_finite_limits_colimits_truncations} requires slightly less assumptions; in particular it applies to the higher homotopy categories $h_{n} \ccat$ of an additive $\infty$-category with finite limits. 

This is not true in the perfect case. For example, consider the stable $\infty$-category of perfect complexes over a field $F$ and let $\ccat_{2} \mathrm{Perf}(F)$ be its homotopy $2$-category. Then, one can show that if we consider $F$ as a perfect complex concentrated in degree zero, then $y(F) = \pi_{0} \nu_{\ccat}(F)$ is not perfect by a suitable count of dimensions. Thus, perfect presheaves over $\ccat$ are not stable under finite limits. 

This is analogous to classical results about complexes over a ring. More precisely, connective almost perfect complexes are stable under limits in $\dcat(R)_{\geq 0}$ as soon as $R$ is coherent, but to deduce the same stability for perfect complexes requires the ring to be regular, a much stronger condition. 
\end{remark}

It will be useful to have a concrete criterion to decide whether a given presheaf is perfect in the sense of  \cref{definition:perfect_presheaf}, analogous to our characterization of almost perfect presheaves given in \cref{proposition:characterization_of_almost_perfect_presheaves}. The two cases of most interest to us are when $\ccat$ is either an ordinary category or stable; we will deal with the former here and characterize the latter in later in \cref{lemma:characterization_of_finite_presheaves_on_stable_inftycat}.

\begin{lemma}
\label{lemma:recognition_of_finite_presheaves_on_an_ordinary_category}
Let $\ccat$ be an additive $\infty$-category with finite limits which is an $n$-category for some finite $n$. Then, the following are equivalent for a presheaf $X \in \Fun_{\Sigma}(\ccat^{op}, \spaces)$:
\begin{enumerate}
    \item $X$ is perfect or 
    \item $\pi_{k}X$ is finitely presented for all $k \geq 0$ and vanishes for almost all of them. 
\end{enumerate}
\end{lemma}

\begin{proof}
Observe that additive presheaves satisfying $(2)$ are closed under finite colimits. When $\ccat$ is an $n$-category, we have $\pi_{k} \nu(c) = 0$ for $k \geq n$, and we deduce that every representable satisfies $(2)$, so that $(1) \Rightarrow (2)$. 

Conversely, every additive presheaf satisfying $(2)$ can be written using iterated extensions and colimits starting from discrete finitely presented presheaves. Since every such presheaf is perfect by \cref{lemma:discrete_almost_perfect_presheaves_are_perfect}, using \cref{theorem:if_ccat_has_finite_limits_perfect_presheaves_closed_under_limits} we deduce that $(2) \Rightarrow (1)$. 
\end{proof}

\subsection{Prestable enhancements to a homological functor}
\label{subsection:prestable_enhancements_to_homological_functor}

As we observed in \cref{remark:prestable_envelope_obtained_by_freely_adding_geometric_realizations}, by construction the prestable Freyd envelope $A_{\infty}(\ccat)$ has a certain universal property regarding geometric realizations, mirroring the relationship between the abelian Freyd envelope and cokernels. However, we have seen in \cref{theorem:homological_universal_property_of_freyd_envelope} that when $\ccat$ is stable, the latter has \emph{another} universal property, namely it is the target of a universal homology theory. In this section, we will describe a similar universal property of $A_{\infty}(\ccat)$. 

The idea of characterization of the prestable Freyd envelope of a stable $\infty$-category rests on the observation that the functor $\nu\colon \ccat \rightarrow A_{\infty}(\ccat)$ behaves like a homology theory in the sense of \cref{definition:homology_theory}, with the obvious difference being that the target is prestable rather than abelian. More precisely, it satisfies the following axioms: 

\begin{definition}
\label{definition:prestable_enhancement}
Let $\ccat$ be a stable $\infty$-category. We say a functor $\euscr{H}\colon \ccat \rightarrow \dcat$ into a prestable $\infty$-category $\dcat$ with finite limits is a \emph{prestable enhancement} to a homological functor when
\begin{enumerate}
\item $\euscr{H}$ is additive and 
\item $\euscr{H}$ is left exact.
\end{enumerate}
\end{definition}

\begin{remark}
A left exact functor between additive $\infty$-categories is automatically additive; thus, the first axiom is superfluous. We have written the two axioms in this way to underline the analogy with \cref{definition:homological_functor} of a homological functor valued in an abelian category.
\end{remark}

\begin{remark}[Why ``enhancement''?]
\label{remark:why_enhancement}
The terminology of \cref{definition:prestable_enhancement} is motivated as follows. If $c \rightarrow d \rightarrow e$ is a cofibre sequence in $\ccat$, we get a fibre sequence
\[
\euscr{H}(c) \rightarrow \euscr{H}(d) \rightarrow \euscr{H}(e)
\]
in $\dcat$. Taking homotopy groups, we obtain a long exact sequence
\[
\ldots \rightarrow \pi_{1} \euscr{H}(d) \rightarrow \pi_{1} \euscr{H}(e) \rightarrow \pi_{0} \euscr{H}(c) \rightarrow \pi_{0} \euscr{H}(d) \rightarrow \pi_{0} \hcat(e),
\]
in particular $\pi_{0} \circ \hcat\colon \ccat \rightarrow \dcat^{\heartsuit}$ is a homological functor in the classical sense of \cref{definition:homological_functor}.
\end{remark}

\begin{notation}
If $H\colon \ccat \rightarrow \dcat^{\heartsuit}$ is a homological functor, then by a \emph{prestable enhancement of $H$} we will mean a prestable enhancement $\euscr{H}\colon \ccat \rightarrow \dcat$ together with a chosen equivalence $\pi_{0} \circ \euscr{H} \simeq H$. 
\end{notation}

Note that many homological functors $H\colon \ccat \rightarrow \acat$ that arise in practice are in fact homology theories; that is, there is a chosen local grading on $\acat$ and $H$ has a structure of a locally graded functor, where $\ccat$ is locally graded using the suspension. 

Similarly, most examples of prestable enhancements which we construct also carry this additional structure: there is a chosen local grading on $\dcat$ and $\euscr{H}$ is a locally graded functor, so that there is a canonical equivalence 
\[
\euscr{H}(\Sigma c) \simeq (\euscr{H}(c))[1]_{\dcat},
\]
We decided not to make this data a part of of a definition of prestable enhancement, as universal properties stated without mentioning the local grading are strictly more general, and they imply they locally graded counterparts as in \cref{remark:universal_properties_of_local_gradings_on_prestable_freyd_envelopes} below.

\begin{remark}
Suppose that we are in the common situation where $\pi_{0} \circ \euscr{H}\colon \ccat \rightarrow \dcat^{\heartsuit}$ is a homology theory; that is, is locally graded. Then, using left exactness we see that 
\[
\pi_{k} \hcat(c) \simeq \pi_{0} \Omega^{k} \hcat(c) \simeq \pi_{0} \hcat(\Sigma^{-k} c) \simeq (\pi_{0} \euscr{H}(c))[-k]_{\dcat}.
\]
Thus, the homotopy groups of $\hcat$ exhibit a form of periodicity, and do not contain any more information than what is already visible in the  underlying homology theory $\pi_{0} \circ \hcat$. The additional information captured is more subtle than this, pertaining to how these homotopy groups are threaded together. 
\end{remark}

\begin{example}
For any spectrum $R$, we have a homology theory
\[
R_{*}\colon \spectra \rightarrow grAb,
\]
which for simplicity here we will treat as valued in graded abelian groups. We can write
\[
R_{n}(X) \colonequals \pi_{0}(R \otimes \Sigma^{-n} X),
\]
which suggests a prestable enhancement obtained by remembering the whole connective part of the relevant spectrum, rather than just the $\pi_{0}$. Indeed, it is not too difficult to check that we have a prestable-valued homology theory
\[
\euscr{R}\colon \spectra \rightarrow gr\spectra_{\geq 0}
\]
valued in the $\infty$-category of graded connective spectra given by 
\[
\euscr{R}_{n}(X)\colonequals  (R \otimes \Sigma^{-n}X)_{\geq 0}.
\]
Note that the above formula encodes the homotopy type of connective covers of $R \otimes X$, and so contains much more information than the $R$-homology groups of $X$ alone. 
\end{example}

\begin{example}
The functor $\nu\colon \ccat \rightarrow A_{\infty}(\ccat)$ is a prestable enhancement compatible with the local grading, by \cref{corollary:almost_perfect_presheaves_closed_under_finite_limits_colimits_truncations}, as the Yoneda embedding into presheaves preserves all limits.
\end{example}
It is the second example which is more relevant to us. As we observed in \cref{example:heart_of_prestable_envelope}, we have an equivalence $A_{\infty}^{\heartsuit}(\ccat) \simeq A(\ccat)$ between the heart of the prestable Freyd envelope and the abelian Freyd envelope. The underlying homology theory
\[
\pi_{0} \circ \nu\colon \ccat \rightarrow A(\ccat)
\]
is the Yoneda embedding for the homotopy category, which has the property of being the universal homology theory by \cref{theorem:homological_universal_property_of_freyd_envelope}. This suggests that perhaps $\nu$ itself is universal in a certain sense. This is indeed the case, as we will show below. 

Recall from \cref{remark:prestable_envelope_obtained_by_freely_adding_geometric_realizations} that $A_{\infty}(\ccat)$ already has one universal property; namely, it is obtained from $\ccat$ be freely attaching geometric realizations. As in the case of the classical Freyd envelope, the goal is to rephrase this property in the language of homology theories when $\ccat$ is stable. 

\begin{lemma}
\label{lemma:unique_extension_from_prestable_envelope_preserves_colimits_iff_original_functor_preserves_finite_sums}
Let $f\colon \ccat \rightarrow \dcat$ be an additive functor of additive $\infty$-categories. Suppose moreover that $\dcat$ admits geometric realizations and that $\ccat$ admits finite limits and $\kappa$-small sums. Then, the following are equivalent:

\begin{enumerate}
    \item $f$ preserves $\kappa$-small sums
    \item the unique geometric realization-preserving extension $\euscr{F}\colon A_{\infty}(\ccat) \rightarrow \dcat$ preserves all $\kappa$-small colimits.
\end{enumerate}
In particular, the left Kan extension of an additive functor is right exact. 
\end{lemma}

\begin{proof}
Observe that $(2 \Rightarrow 1)$ is clear, since $\nu\colon \ccat \rightarrow A_{\infty}(\ccat)$ preserves all direct sums which exist in $\ccat$ by \cref{proposition:nu_preserves_direct_sums}. An argument proving $(1 \Rightarrow 2)$ appears in the work of Lurie as \cite{lurie2011derived}[4.2.14], but we will recall it here for the ease of reference. 

We first check that $\euscr{F}$ preserves $\kappa$-small direct sums. Suppose that $X_{\alpha} \in A_{\infty}(\ccat)$ is a $\kappa$-small family. Then, we can write each $X_{\alpha} \simeq | \nu(c_{\alpha, \bullet})|$ as a geometric realizations of objects in the image of $\nu$ and 
\[
\bigoplus_{\alpha} X_{\alpha} \simeq | \nu(\bigoplus c_{\alpha, \bullet}) |
\]
since $\nu$ commutes with directs sums by another application of \cref{proposition:nu_preserves_direct_sums}. Applying $\euscr{F}$ to the right hand side we get 
\[
| f(\bigoplus_{\alpha} c_{\alpha, \bullet}) | \simeq | \bigoplus _{\alpha} f(c_{\alpha, \bullet})| \simeq \bigoplus _{\alpha} | f(c_{\alpha, \bullet}) | \simeq \bigoplus_{\alpha} \euscr{F}(X_{\alpha}),
\]
which is what we wanted. 

Having verified direct sums, it is enough to check that $\euscr{F}$ preserves coequalizers, so let 
\[
X^{\prime} \rightrightarrows X
\]
consist of two paraller arrows in $A_{\infty}(\ccat)$. This admits a left Kan extension to a simplicial object which will have the form 
\[
\ldots X^{\prime} \oplus X \oplus X \triplerightarrow X^{\prime} \oplus X \rightrightarrows X,
\]
and whose colimit is the same as that of the original diagram by virtue of being a left Kan extension. Since $\euscr{F}$ preserves geometric realizations and direct sums, the above diagram will be taken to the left Kan extension of the image of $\euscr{F}(X^{\prime}) \rightrightarrows \euscr{F}(X)$. As $\euscr{F}$ preserves geometric realizations and finite sums, we deduce that it preserves coequalizer diagrams, as needed.
\end{proof}

\begin{theorem}
\label{theorem:universal_property_of_prestable_freyd_envelope}
Let $\hcat\colon \ccat \rightarrow \dcat$ be a prestable enhancement to a homological functor and suppose that $\dcat$ admits geometric realizations. Then, there exists an essentially unique exact, geometric realization-preserving functor $\euscr{L}\colon A_{\infty}(\ccat) \rightarrow \dcat$ such that the following diagram commutes 

\begin{center}
\begin{tikzcd}
	{\ccat} && {\dcat} \\
	& {\acat_{\infty}(\ccat)}
	\arrow["{\hcat}", from=1-1, to=1-3]
	\arrow["{\nu}"', from=1-1, to=2-2]
	\arrow["{\exists_{!} \euscr{L}}"', from=2-2, to=1-3, dashed]
\end{tikzcd}.
\end{center}
In other words, to give an exact, geometric realization-preserving functor out of the prestable Freyd envelope $A_{\infty}(\ccat)$ into a prestable $\infty$-category $\dcat$ with finite limits and geometric realizations is essentially the same datum as to give a homological functor $\ccat \rightarrow \dcat^{\heartsuit}$ together with a prestable enhancement.
\end{theorem}

\begin{proof}
Let $\euscr{L}\colon A_{\infty}(\ccat) \rightarrow \dcat$ be the left Kan extension of $\hcat$ along $\nu$; by \cref{remark:prestable_envelope_obtained_by_freely_adding_geometric_realizations} we know that this is the unique extension of $\hcat$ along $\nu$ which preserves geometric realizations. We have to verify that this left Kan extension is exact if and only if $\euscr{H}$ was a prestable enhancement. 

By \cref{lemma:unique_extension_from_prestable_envelope_preserves_colimits_iff_original_functor_preserves_finite_sums}, we know that $\euscr{L}$ is right exact if and only if $\euscr{H}$ is additive. Thus, to finish proving the theorem we are left with verifying that one is left exact if and only if the other is. 

As $\nu$ is left exact, it is clear that if $\euscr{L}$ is left exact then so will be $\euscr{H}$. Conversely, assume that the latter holds. As $\hcat \simeq \euscr{L} \circ \nu$, $\euscr{L}$ commutes with $\Omega$ on objects of the form $\nu(c)$; that is, the canonical map
\[
\euscr{L} \Omega \nu(c) \rightarrow \Omega \euscr{L} \nu(c)
\]
is an equivalence. We claim that this implies that $\euscr{L}$ takes projectives in the heart $A_{\infty}(\ccat)^{\heartsuit}~\simeq~A(\ccat)$ into the heart $\dcat^{\heartsuit}$. 

We know that projectives in the Freyd envelope are retracts of those of the form $y(c) \simeq \pi_{0} \nu(c)$ for $c \in \ccat$, that is, of representables in the homotopy category of $\ccat$. We have a cofibre sequence
\[
\Sigma \Omega \nu (c)\rightarrow \nu(c) \rightarrow y(c)
\]
which by above is taken by $\euscr{L}$ to a cofibre sequence
\[
\Sigma \Omega \hcat(c) \rightarrow \hcat(c) \rightarrow \euscr{L}(y(c)),
\]
because $\euscr{L}$ commutes both with all suspensions by virtue of being cocontinuous, as well as with loops on objects of the form $\nu(c)$. This implies that $\euscr{L}(y(c)) \simeq \pi_{0} \hcat(c)$, which is discrete, so that the same is true for retracts.

Now suppose that $x \in A(\ccat)$ is arbitrary, then we can find a projective resolution of $x$
\[
\ldots \rightarrow y(c_{2}) \rightarrow y(c_{1}) \rightarrow y(c_{0})
\]
where every three terms are an image of a cofibre sequence in $\ccat$. The projective terms determine, by Dold-Kan correspondence, a simplicial object in $A(\ccat)$. As the above is a resolution, the geometric realization spectral sequence collapses on the second page and shows that we have an equivalence $| y(c_{\bullet}) | \simeq x$ in $A_{\infty}(\ccat)$.  

Since $\euscr{L}$ preserves geometric realizations and $\euscr{L}(y(c)) \simeq \pi_{0} \hcat(c)$, we deduce that $\euscr{L}(x)$ is the geometric realization of the simplicial object corresponding to the chain complex
\[
\ldots \rightarrow \pi_{0} \hcat(c_{2}) \rightarrow \pi_{0} \hcat(c_{1}) \rightarrow \pi_{0} \hcat(c_{0})
\]
But this is exact outside of the last term, since every three terms are an image of cofibre sequence in $\ccat$, and $\pi_{0} \circ \hcat$ is homological, as we observed in \cref{remark:why_enhancement}. Thus, the geometric realization spectral sequence proves that $\euscr{L}(x)$ is discrete, as claimed. 

Since $\euscr{L}$ is a right exact functor of prestable $\infty$-categories which takes discrete objects to discrete objects, we claim it follows formally that it is left exact. To check this, it is enough to verify that the induced exact extension $SW(\euscr{L})\colon SW(A_{\infty}(\ccat)) \rightarrow SW(\dcat)$ between stable $\infty$-categories of Spanier-Whitehead objects, is $t$-exact \cite{lurie_spectral_algebraic_geometry}[C.3.2.1]. 

The extension clearly preserves connective objects, and since  coconnective object of the Spanier-Whitehead $\infty$-category are precisely those which can be written as the $n$-th delooping of the image of some $n$-truncated object of the prestable $\infty$-category \cite{lurie_spectral_algebraic_geometry}[C.1.2.9], it is enough to check that $\euscr{L}$ itself preserves $n$-truncated objects. But every such is an extension of suspensions of discrete objects, so we can reduce to the case of $n=0$, which we already covered. This ends the argument that $\euscr{L}$ is left exact. 
\end{proof}

\begin{remark}
\label{remark:preserving_geometric_realizations_automatic_for_exact_functor_into_complete}
A slightly restrictive part of \cref{theorem:universal_property_of_prestable_freyd_envelope} is the requirement for the classifying functor $\euscr{L}$ to preserve geometric realizations. In practice, this is often satisfied automatically for formal reasons. 

For example, suppose that $\dcat$ is complete as a prestable $\infty$-category; that is, that we have $\dcat \simeq \varprojlim \tau_{\leq n} \dcat$. Then, to give an exact functor $A_{\infty}(\ccat) \rightarrow \dcat$ is the same as to give a compatible family of exact functors $\tau_{\leq n} A_{\infty}(\ccat) \rightarrow \tau_{\leq n} \dcat$. Each of the latter is a functor of $(n+1)$-categories, so that it preserves geometric realization as in this context it is equivalent to a finite colimit. We deduce that the original functor also preserves geometric realizations, as needed. 
\end{remark}

\begin{remark} In view of \cref{lemma:unique_extension_from_prestable_envelope_preserves_colimits_iff_original_functor_preserves_finite_sums} we see that under the conditions of \cref{theorem:universal_property_of_prestable_freyd_envelope}, the prestable enhancement $\hcat$ preserves $\kappa$-small sums if and only if the unique extension $\euscr{L}$ does. In particular, if $\hcat$ preserves all small sums, then $\euscr{L}$ is cocontinuous. 
\end{remark}

As we observed in \cref{remark:preserving_geometric_realizations_automatic_for_exact_functor_into_complete}, a slightly unsatisfying part of this result is the requirement that the classifying functor $\euscr{L}\colon A_{\infty}(\ccat) \rightarrow \dcat$ should preserve geometric realizations. No such additional restriction is needed in the classical setting of  \cref{theorem:homological_universal_property_of_freyd_envelope}, this comes down to the fact that in ordinary categories geometric realizations are equivalent to reflexive coequalizers, hence are in fact finite colimits. This issue can be avoided by working with the perfect prestable Freyd envelope, as the following shows. 

\begin{theorem}
\label{theorem:universal_property_of_finite_presheaves}
Suppose that $\ccat$ is stable and let $\hcat\colon \ccat \rightarrow \dcat$ be a prestable enhancement to a homological functor. Then, there exists an essentially unique exact functor $\euscr{L}\colon A^{\omega}_{\infty}(\ccat) \rightarrow \dcat$  such that the following diagram commutes 

\begin{center}
\begin{tikzcd}
	{\ccat} && {\dcat} \\
	& {\acat^{\omega}_{\infty}(\ccat)}
	\arrow["{\hcat}", from=1-1, to=1-3]
	\arrow["{\nu}"', from=1-1, to=2-2]
	\arrow["{\exists_{!} \euscr{L}}"', from=2-2, to=1-3, dashed]
\end{tikzcd}.
\end{center}
In other words, to give an exact functor out of the perfect prestable Freyd envelope $A^{\omega}_{\infty}(\ccat)$ into a prestable $\infty$-category $\dcat$ with finite limits is essentially the same datum as to give a homological functor $\ccat \rightarrow \dcat^{\heartsuit}$ together with a prestable enhancement.
\end{theorem}

\begin{proof}
Consider the universal right exact embedding $i\colon \dcat \rightarrow \textnormal{Ind}(\dcat)$ into a prestable $\infty$-category which admits all small colimits. Then, by \cref{theorem:universal_property_of_prestable_freyd_envelope}, $\euscr{H}$ induces a unique exact, geometric realization-preserving extension $\euscr{L}\colon A_{\infty}(\ccat) \rightarrow \textnormal{Ind}(\dcat)$. As the essential image of $\dcat$ in the $\textnormal{Ind}$-completion is closed under finite colimits, we deduce that the restriction of $\euscr{L}$ to $A_{\infty}^{\omega}(\ccat)$ factors uniquely through $\dcat$. This is the needed unique extension to an exact functor. 
\end{proof}

\begin{remark}[Universal property of the local grading]
\label{remark:universal_properties_of_local_gradings_on_prestable_freyd_envelopes}
As we observed in \cref{remark:local_grading_on_prestable_freyd_envelope}, the prestable Freyd envelope admits a unique local grading such that $\nu\colon \ccat \rightarrow A_{\infty}(\ccat)$ becomes a locally graded functor, where we grade $\ccat$ using the suspension functor. 

By the same argument as in the discrete case given in \cref{remark:universal_property_of_classical_freyd_envelope_and_local_gradings}, the universal property of \cref{theorem:universal_property_of_prestable_freyd_envelope} implies a universal property of $A_{\infty}(\ccat)$ as a locally graded $\infty$-category. To be more precise, for any locally graded prestable $\infty$-category $\dcat$ which admits geometric realizations, geometric-realization preserving, locally graded exact functors $\euscr{L}\colon A_{\infty}(\ccat) \rightarrow \dcat$ correspond uniquely to prestable enhancements $\ccat \rightarrow \dcat$ compatible with local gradings. 

Analogous discussion applies to the perfect prestable Freyd envelope $A_{\infty}^{\omega}(\ccat)$ with the difference being that $\dcat$ need not admit geometric realizations and $\euscr{L}$ need not preserve them. 
\end{remark}

\subsection{Thread structure on the prestable Freyd envelope}
\label{subsection:thread_structure_on_prestable_freyd_envelope}

Suppose that $\ccat$ is a stable $\infty$-category. As we observed in \cref{remark:local_grading_on_prestable_freyd_envelope}, the Freyd envelope $A_{\infty}(\ccat)$ has an induced local grading given by
\[
(X[1])(c) \colonequals X(\Sigma^{-1}c).
\]
This is the unique local grading for which $\nu\colon \ccat \rightarrow A_{\infty}(\ccat)$ is a locally graded functor; that is, such that $\nu(\Sigma c) \simeq \nu(c)[1]$. 

This local grading is different from the suspension of the prestable Freyd envelope; however, they are related by a natural transformation. The precise difference is the subject of this section.

\begin{construction}
\label{construction:canonical_thread_structure_on_prestable_freyd}
Since $\nu$ is a functor, for any $c \in \ccat$ we have a canonical map 
\[
\Sigma \nu c \rightarrow \nu(\Sigma c) \simeq \nu(c)[1].
\]
This is in fact a natural transformation of functors $\ccat \rightarrow A_{\infty}(\ccat)$, and by left Kan extension we obtain a natural transformation 
\[
\tau\colon \Sigma X \rightarrow X[1]
\]
of geometric realization-preserving endofunctors of the prestable Freyd envelope. 
\end{construction}

\begin{definition}
\label{definition:canonical_thread_structure_on_prestable_freyd_envelope}
The \emph{canonical thread structure} on $A_{\infty}(\ccat)$ is the natural transformation 
\[
\tau\colon \Sigma X \rightarrow X[1]
\]
of \cref{construction:canonical_thread_structure_on_prestable_freyd}.
\end{definition}

\begin{remark}
The thread structure on $X \in A_{\infty}(\ccat)$ can be described concretely. It is adjoint to the map $X \rightarrow \Omega X[1]$ which is levelwise over $c \in \ccat$ given by the canonical map of spaces 
\[
X(c) \simeq X(\Sigma \Omega c) \rightarrow \Omega X( \Omega c)
\]
measuring the extent to which $X$ takes suspensions to loops. To see this, observe that these two descriptions coincide on representables, and since both functors in question preserve geometric realizations, they must coincide for a general almost perfect presheaf. 
\end{remark}

An immediate calculation shows that if $c \in \ccat$, then the thread structure fits into a cofibre sequence
\[
\Sigma \nu(c)[-1] \rightarrow \nu(c) \rightarrow y(c);
\]
in other words, the map $\tau$ into a representable is a $1$-connective cover. More generally:

\begin{lemma}
\label{lemma:cofibre_of_powers_of_tau_into_representable_a_truncation}
For any $n \geq 1$, the map $\tau^{n}$ into a representable fits into a cofibre sequence
\[
\Sigma^{n} \nu(c)[-n] \rightarrow \nu(c) \rightarrow \nu(c)_{\leq (n-1)}.
\]
\end{lemma}

\begin{proof}
After unwrapping definitions, levelwise over $d \in \ccat$, the above is the cofibre sequence of grouplike $E_{\infty}$-spaces $\mathrm{B}^{n} \Omega^{n} \Map_{\ccat}(d, c) \rightarrow \Map_{\ccat}(d, c) \rightarrow \Map_{\ccat}(d, c)_{\leq (n-1)}$. 
\end{proof}
Note that \cref{lemma:cofibre_of_powers_of_tau_into_representable_a_truncation} is not true for a general almost perfect presheaf. To obtain the correct generalization, observe that the cofibre $\nu(c)_{\leq (n-1)}$ can be interpreted as the presheaf represented by $c$ in the $n$-th homotopy category $h_{n} \ccat$. This suggests that the cofibre of powers of $\tau$ is related to the (higher) homotopy categories; this is indeed the case, as we will now show. 

\begin{notation}
\label{notation:n_th_homotopy_category_projection_adjunction_on_ap_presheaves} 
Let $\pi(n)\colon \ccat \rightarrow h_{n} \ccat$ be the projection onto the $n$-th homotopy category. We then have an adjunction
\[
\pi(n)^{*} \dashv \pi(n)_{*}\colon A_{\infty}(\ccat) \rightleftarrows A_{\infty}(h_{n} \ccat)
\]
between the $\infty$-categories of almost perfect presheaves given by left Kan extension and restriction. 
\end{notation}

\begin{remark}
Note that the restriction $\pi(n)_{*}$ takes representables to almost perfect presheaves by \cref{lemma:cofibre_of_powers_of_tau_into_representable_a_truncation}. As it also preserves geometric realizations, it follows that it preserves the property of being almost perfect, which is implicit in \cref{notation:n_th_homotopy_category_projection_adjunction_on_ap_presheaves}. 
\end{remark}

\begin{remark}
Note that the characterization of almost perfect presheaves in terms of having finitely presented homotopy groups also applies to presheaves over $h_{n} \ccat$, see \cref{remark:for_characterization_of_ap_presheaves_no_finite_limits_needed}. This will not be needed for what follows. 
\end{remark}

\begin{proposition}
\label{proposition:spiral_cofibre_sequence}
For any $n \geq 1$, the map $\tau^{n}$ into an almost perfect $X$ fits into a canonical cofibre sequence
\[
\Sigma^{n} X[-n] \rightarrow X \rightarrow \pi(n)_{*} \pi(n)^{*} X
\]
\end{proposition}

\begin{proof}
Observe that each of these three functors of $X$ preserves geometric realization. Since cofibre sequences are stable under geometric realizations, in fact all colimits, the case of an arbitrary almost perfect presheaf follows from \cref{lemma:cofibre_of_powers_of_tau_into_representable_a_truncation}, as after unwrapping the definitions we see that in the representable case the unit $\nu(c) \rightarrow \pi(n)_{*} \pi(n)^{*} \nu(c)$ identifies the target with the $(n-1)$-th Postnikov truncation of the source. 
\end{proof}

\begin{remark}
The fibre sequence established by \cref{proposition:spiral_cofibre_sequence} is also known in the literature as the \emph{spiral fibre sequence}, and arises in much more general contexts than the additive one presented in the current work, see \cite{balderrama2021deformations}, \cite{pstrkagowski2017moduli}. 
\end{remark}
The above fundamental observation places a lot of importance on this monad, so that it deserves its own notation. 

\begin{notation}
\label{notation:monad_associated_to_hnc_adjunction_is_tensor_with_ctau}
We will write 
\[
C\tau^{n} \otimes - \colonequals \pi(n)_{*} \pi(n)^{*}
\]
for the monad associated to the adjunction of \cref{notation:n_th_homotopy_category_projection_adjunction_on_ap_presheaves}. We will refer to modules over the monad $C\tau^{n} $ as \emph{$C\tau^{n}$-modules}. 
\end{notation}

\begin{warning}
Beware that \cref{notation:monad_associated_to_hnc_adjunction_is_tensor_with_ctau} is potentially abusive, as unless $\ccat$ is monoidal, there is no natural notion of a tensor product of almost perfect presheaves. 

Our notation is motivated by the fact that if $\ccat$ \emph{was} stably monoidal, then $A_{\infty}(\ccat)$ would inherit a monoidal structure right exact in each variable. In this case, we could identify the functor $C\tau^{n}$ with tensoring with $\nu(\monunit_{\ccat})_{\leq n-1} \simeq C\tau^{n} \otimes \nu(\monunit_{\ccat})$, and the monad structure would correspond to the canonical associative algebra structure on the latter. This is what happens in the case of the $\infty$-category of synthetic spectra of \cite{pstrkagowski2018synthetic}.
\end{warning}

\begin{remark}[Passing to perfect presheaves]
\label{remark:perfect_presheaves_and_ctau_monad}
Note that since perfect presheaves are stable under cofibres, \cref{proposition:spiral_cofibre_sequence} implies that if $X \in A_{\infty}^{\omega}(\ccat)$ is perfect, so is $C\tau^{n} \otimes X$ for any $n \geq 1$. In particular, the monad $C\tau^{n} $ restricts to a monad on the perfect prestable Freyd envelope. 
\end{remark}

One part of the importance of the monad $C\tau^{n} $ is that it allows one to recover the Freyd envelopes of the homotopy categories, as the following shows. 

\begin{proposition}
\label{proposition:adjunction_induced_by_projection_onto_homotopy_category_is_monadic}
The adjunction $\pi(n)^{*} \dashv \pi(n)_{*}\colon A_{\infty}(\ccat) \rightleftarrows A_{\infty}(h_{n} \ccat)$ is monadic; that is, it induces an equivalence \[
\Mod_{C\tau^{n} \otimes -}(A_{\infty}(\ccat)) \simeq A_{\infty}(h_{n} \ccat)
\]
with modules over the monad $C\tau^{n} $. 
\end{proposition}

\begin{proof}
By \cite{higher_algebra}[4.7.3.5] it is enough to show that $\pi(n)_{*}$ preserves geometric realizations and is conservative. The first part is clear, as it is given by restriction, and the second follows from the fact that $\ccat \rightarrow h_{n} \ccat$ is surjective on objects. 
\end{proof}

\begin{remark}[Linearity of $\tau$]
Note that if $X$ is a module over the monad $C\tau^{n} \otimes -$, then the map
\[
\tau\colon \Sigma X[-1] \rightarrow X
\]
of \cref{definition:canonical_thread_structure_on_prestable_freyd_envelope} has a canonical lift to a morphism of modules. In this sense, $C\tau^{n} \otimes -$-modules behave as if they were modules over a commutative algebra.

To see this, observe that to any $c \in \ccat$ we can functorially associate a pullback square 
\[\begin{tikzcd}
	{\Omega c} & 0 \\
	0 & c
	\arrow[from=1-1, to=1-2]
	\arrow[from=1-1, to=2-1]
	\arrow[from=2-1, to=2-2]
	\arrow[from=1-2, to=2-2]
\end{tikzcd}\]
This yields a functor $\ccat \rightarrow \Fun(\Delta^{1} \times \Delta^{1}, \ccat)$ which descends to a functor $h_{n} \ccat \rightarrow \Fun(\Delta^{1} \times \Delta^{1}, h_{n}\ccat)$, since the target is an $n$-category. Composing with the Yoneda embedding
\[
h_{n} \ccat \hookrightarrow A_{\infty}(h_{n} \ccat) \simeq \Mod_{C\tau^{n} \otimes -}(A_{\infty}(\ccat)),
\]
for every $c \in h_{n} \ccat$ we get a canonical square 
\[\begin{tikzcd}
	{C\tau^{n} \otimes (\Omega c)} & 0 \\
	0 & C\tau^{n} \otimes \nu(c)
	\arrow[from=1-1, to=1-2]
	\arrow[from=1-1, to=2-1]
	\arrow[from=2-1, to=2-2]
	\arrow[from=1-2, to=2-2]
\end{tikzcd}\]
which is usually not a pullback square, but induces a morphism of $C\tau^{n}$-modules 
\[
(C\tau^{n} \otimes \nu(c))[-1] \simeq C\tau^{n} \otimes \nu(\Omega c) \rightarrow \Omega (C\tau^{n} \otimes \nu(c))
\]
the adjoint of which we can identify with the map $\tau$ of \cref{definition:canonical_thread_structure_on_prestable_freyd_envelope}. 

Since the $\infty$-category $A_{\infty}(\ccat) \simeq \Mod_{C\tau^{n} \otimes -}(A_{\infty}(\ccat))$ is freely generated under geometric realizations by $h_{n} \ccat$ and both $\Sigma$ and $[-1]$ preserve these, this natural transformation extends uniquely to all of $C\tau^{n} \otimes -$-modules.
\end{remark}

\begin{proposition}
\label{proposition:forgetful_functor_from_modules_is_exact}
The forgetful functor 
\[
\Mod_{C\tau^{n} \otimes -}(A_{\infty}(\ccat)) \rightarrow A_{\infty}(\ccat)
\]
is exact and induces an equivalence 
\[
\tau_{\leq n-1} (\Mod_{C\tau^{n} \otimes -}(A_{\infty}(\ccat)))  \simeq \tau_{\leq n-1}A_{\infty}(\ccat)
\]
between the subcategories of $(n-1)$-truncated objects. 
\end{proposition}

\begin{proof}
By \cref{proposition:adjunction_induced_by_projection_onto_homotopy_category_is_monadic}, we can identify the above forgetful functor with the restriction functor $\pi(n)_{*}\colon A_{\infty}(h_{n} \ccat) \rightarrow A_{\infty}(\ccat)$ along the projection $\ccat \rightarrow h_{n} \ccat$. It follows that it must be exact, as the latter is. 

Similarly, since any presheaf of $(n-1)$-truncated spaces $X\colon \ccat^{op} \rightarrow \tau_{\leq n-1} \spaces$ uniquely factors through $h_{n}\ccat$ by the universal property of the latter, we deduce we have an equivalence between $(n-1)$-truncated objects. 
\end{proof}

\begin{remark}
In the setting of \cref{proposition:forgetful_functor_from_modules_is_exact}, we can deduce further that for any $k \geq 0$, the forgetful functor $\Mod_{C\tau^{n} \otimes -}(A_{\infty}(\ccat)) \rightarrow A_{\infty}(\ccat)$ induces an equivalence between subcategories of objects with homotopy concentrated in degrees $k \leq d < k+n$. Indeed, these subcategories are equivalent to those of $n$-truncated objects by a $k$-fold application of $\Sigma$. 
\end{remark}
As $n$ varies, the relevant $\infty$-categories and adjunctions can be assembled into a tower, as in the following. 

\begin{construction}[Relative case]
\label{construction:goerss_hopkins_tower_for_prestable_freyd_envelope}
We have a tower of additive $\infty$-categories 
\[
\ccat \rightarrow \ldots h_{2} \ccat \rightarrow h_{1} \ccat=h \ccat 
\]
which induces a tower of prestable Freyd envelopes and adjunctions 
\[
A_{\infty}(\ccat) \rightleftarrows \ldots \rightleftarrows A_{\infty}(h_{2} \ccat) \rightleftarrows A_{\infty}(h_1 \ccat).
\]
Each of these adjunctions is also monadic by the same argument as in \cref{proposition:adjunction_induced_by_projection_onto_homotopy_category_is_monadic}. Using notation analogous to the slightly abusive \cref{notation:monad_associated_to_hnc_adjunction_is_tensor_with_ctau}, for any $n \geq k \geq 1$, the adjunction induced by $h_{n+1} \ccat \rightarrow h_{k+1} \ccat$ induces a monad $C\tau^{k} \otimes_{C\tau^{n}} -$ on $C\tau^{n} $-modules with the property that 
\[
C\tau^{k} \otimes_{C\tau^{n}} (C\tau^{n} \otimes X) \simeq C\tau^{k} \otimes X.
\]
Note that any $C\tau^{k}$-module has a canonical structure of a $C\tau^{n}$-module, with the above monad corresponding to the left adjoint to that forgetful functor. 
\end{construction}

\begin{remark}
\label{remark:relative_monad_preserves_pi_0}
Note that the forgetful functors $\Mod_{C\tau^{k}}(A_{\infty}(\ccat)) \rightarrow \Mod_{C\tau^{n}}(A_{\infty}(\ccat))$ for $k < n$ also induce equivalences on the hearts, as they commute with forgetful functors into $A_{\infty}(\ccat)$ which both have this property by \cref{proposition:forgetful_functor_from_modules_is_exact}.

It follows that their left adjoints preserve $\pi_{0}$-s, and we deduce that after identifying both sides with objects of $A(\ccat)$, there is a canonical isomorphism 
\[
\pi_{0}(C\tau^{k} \otimes_{C\tau^{n}} X) \simeq \pi_{0}X.
\]
\end{remark}

\begin{warning}[Relative situation and perfection]
\label{warning:relative_monads_dont_preserve_perfection}
In \cref{construction:goerss_hopkins_tower_for_prestable_freyd_envelope}, we introduced the ``relative'' monads $C\tau^{k} \otimes_{C\tau^{n}} -$ on the $\infty$-category $\Mod_{C\tau^{n} \otimes -}(A_{\infty}(\ccat))$. These do not in general preserve modules with the property that the underlying presheaf is perfect, unlike the ``absolute'' monads $C\tau^{k} \otimes -$. 

As an explicit example, a reader can compute that $C\tau \otimes_{C\tau^{2}} y(c)$ for a non-zero $c \in \ccat$ has infinitely many homotopy groups on which $\tau$ necessarily acts by zero as the result is a $C\tau$-module. Thus, $C\tau \otimes _{C\tau^{2}} y(c)$ is not perfect. This is analogous to the observation that $k$ is not a perfect $k[x]/x^{2}$-module, even though both are perfect as modules over $k[x]$. 
\end{warning}

The process of forming $C\tau^{n}$-modules corresponds to modding out by some power of $\tau$. Going in the opposite direction, one can ask if there is a way to \emph{invert} $\tau$. We have the following construction, where for complete generality we will need to restrict to perfect presheaves. 

\begin{construction}[$\tau$-inversion]
\label{construction:tau_inversion}
Let $\ccat$ be a stable $\infty$-category. Then, $\ccat$ is also prestable, and the identity functor $\mathrm{id} \colon \ccat \rightarrow \ccat$ is a left exact and hence a prestable enhancement. Thus, by \cref{theorem:universal_property_of_finite_presheaves}, the left Kan extension provides a unique exact extension 
\[
\tau^{-1}\colon A_{\infty}^{\omega}(\ccat) \rightarrow \ccat
\]
which we will call the \emph{$\tau$-inversion functor}. 
\end{construction}

\begin{remark}[Why ``$\tau$-inversion''?]
Observe that for any $c \in \ccat$, the map 
\[
\Sigma \nu(c) \rightarrow \nu(\Sigma c) 
\]
becomes the identity of $\Sigma c$ after applying $\tau^{-1}\colon A_{\infty}^{\omega}(\ccat) \rightarrow \ccat$. Thus, we deduce that $\tau^{-1}$ inverts $\tau$ for all representables, and since it is exact, we deduce that it inverts the map $\tau$ on any perfect presheaf. One can show that it is a universal exact functor with this property. 

Note that this in particular means that $\tau^{-1}$ annihilates all presheaves of the form $C\tau \otimes X$, and hence all perfect presheaves with only finitely many non-zero homotopy groups.
\end{remark}

\begin{remark}
It is not too difficult to verify that there's a canonical isomorphism
\[
y(\tau^{-1} X) \simeq \varinjlim \pi_{k}X[k]
\]
in $A(\ccat)$, where the latter colimit exists as this diagram is eventually constant, as we show in \cref{lemma:characterization_of_finite_presheaves_on_stable_inftycat} below. Thus, the equivalence class of $\tau^{-1}X$ is completely determined by the homotopy groups of $X$.
\end{remark}

\begin{remark}[$\tau$-inversion of almost perfects]
If $\ccat$ admits geometric realizations, then the restriction to perfect presheaves can be avoided. In this case, \cref{theorem:universal_property_of_prestable_freyd_envelope} implies that the identity of $\ccat$ extends uniquely to an exact, geometric realization-preserving functor $\tau^{-1}\colon A_{\infty}(\ccat) \rightarrow \ccat$. On perfect presheaves, this necessarily coincides with the functor of \cref{construction:tau_inversion}.
\end{remark}
We now give the promised characterization of perfect presheaves using the thread structure. 

\begin{lemma}
\label{lemma:characterization_of_finite_presheaves_on_stable_inftycat}
Let $\ccat$ be a stable $\infty$-category. Then, the following are equivalent for an additive presheaf $X \in \Fun_{\Sigma}(\ccat^{op}, \spaces)$:
\begin{enumerate}
    \item $X$ is perfect, 
    \item $\pi_{k}X$ is finitely presented for each $k \geq 0$, and there exists an $N$ such that for any $k \geq N$, $\pi_{k}X$ is a representable discrete presheaf and $\tau\colon \pi_{k}X[-1] \rightarrow \pi_{k+1}(X)$ is an isomorphism.
\end{enumerate}
\end{lemma}

\begin{proof}
If $X$ is perfect, then it is almost perfect and thus it has finitely presented homotopy groups. Moreover, we have $\Omega^{k} X \simeq \nu(c)$ for some $c \in \ccat$ and $k \geq 0$ by \cref{lemma:perfection_detected_by_looping_into_a_representable}. Thus, $\tau$ acts isomorphically on $\pi_{i}X$ for $i \geq k$. 

Conversely, suppose that $X$ satisfies $(2)$. For any $N \geq 0$ we have a fibre sequence 
\[
X \rightarrow X_{\leq N-1} \rightarrow \Sigma^{N+1} \Omega^{N} X,
\]
where by assumption the middle presheaf is bounded with finitely presented homotopy groups and hence perfect. Thus, it is enough to verify that $\Omega^{N}X$ is representable, and we can replace $X$ by the latter. 

By choosing $N$ large enough and we can assume that $\pi_{0}X$ is a representable discrete presheaf, and that $\tau\colon \pi_{k}X[-1] \rightarrow \pi_{k+1}X$ is an isomorphism for all $k \geq 0$. 
Choose an element of $\pi_{0}X(c)$ which corresponds to an isomorphism $\pi_{0}X \simeq y(c)$; this determines a homotopy class of maps $\nu(c) \rightarrow X$. The latter is a $\pi_{0}$-isomorphism by construction and since $\tau$ acts isomorphically on homotopy of both presheaves, we deduce that it is an isomorphism on all homotopy groups and hence an equivalence.
\end{proof}

\begin{corollary}
\label{corollary:finite_ctaun_modules_are_bounded}
Let $n \geq 1$ and $X \in \Mod_{C\tau^{n} \otimes -}(A_{\infty}^{\omega}(\ccat))$. Then, $X$ has bounded homotopy groups. 
\end{corollary}

\begin{proof}
Observe that if an almost perfect presheaf $X$ admits a structure of $C\tau^{n} $-module, then the map 
\[
\tau^{n}\colon \Sigma^{n} X [-n] \rightarrow X 
\]
is necessarily null. It then follows from the criterion $(2)$ of \cref{lemma:characterization_of_finite_presheaves_on_stable_inftycat} that $\pi_{k} X$ vanishes for $k$ large enough as needed when $X$ is perfect. 
\end{proof}

\begin{corollary}
\label{corollary:perfect_prestable_freyd_envelope_generated_under_limits_by_representables_and_em_presheaves}
The perfect prestable Freyd envelope $A_{\infty}^{\omega}(\ccat)$ is generated under finite limits by presheaves with homotopy group in a single degree and representables. 
\end{corollary}

\begin{proof}
Using Postnikov decompositions, we see that presheaves with homotopy group in a single degree generate under finite limits all bounded presheaves. Thus, it is enough to show that bounded presheaves and representables suffice. 

By \cref{lemma:characterization_of_finite_presheaves_on_stable_inftycat}, if $X$ is perfect, then there exists a $c \in \ccat$ such that for $n$ large enough we have a fibre sequence
\[
\nu(c)_{\geq n} \rightarrow X \rightarrow X_{\leq n-1}.
\]
This then rotates to a fibre sequence
\[
X \rightarrow X_{\leq n-1} \rightarrow \Sigma \nu(c)_{\geq n} 
\]
and we are done since $\Sigma \nu(c)_{\geq n} \simeq \nu(\Sigma c)_{\geq n+1}$ is the fibre of $\nu(\Sigma c) \rightarrow \nu(\Sigma c)_{\leq n}$. 
\end{proof}

\section{Derived $\infty$-categories}
\label{section:derived_infty_cats}

A reasonable abelian category $\acat$ has as associated connective derived $\infty$-category $\dcat(\acat)$, which is a prestable $\infty$-category with heart canonically equivalent to $\acat$, and which is in a precise sense initial with respect to this property, allowing one to form derived functors.

The construction is extremely useful, and one would often like to be able to construct such ``derived $\infty$-categories'' in a more general context, including when the input is not abelian. In this section, we will focus on a particular variant of this construction, where the input will be a stable $\infty$-category $\ccat$ together with an epimorphism class. 

\begin{remark}
The need to start with an epimorphism class comes from the fact that stable $\infty$-categories do not come with an interesting non-trivial notion of epimorphism, as we observed in \cref{warning:all_maps_are_effective_epi_in_a_stable_infty_category}. 
\end{remark}

\begin{remark}
Informally, the objects of the derived $\infty$-category of an abelian category can be described as appropriate ``resolutions'' of objects of $a$; for example, when $\acat$ has enough projectives, then $\dcat(\acat)$ can be described as the localization of the category of levelwise-projective chain complexes at the class of quasi-isomorphisms. This is also true in our context, but for technical reasons it will be convenient to proceed slightly differently, by defining the derived $\infty$-category as certain kinds of sheaves. 
\end{remark}

\begin{warning}
In the current work, by the derived $\infty$-category we always mean the \emph{connective} derived $\infty$-category, which is prestable rather than stable, and we denote it by $\dcat(\acat)$ rather than the more common $\dcat_{\geq 0}(\acat)$. We will have no need to consider any of the stable derived $\infty$-categories and thus we felt that to keep writing $_{\geq 0}$ would clutter the notation with little obvious gain. 
\end{warning}

Note that by \cref{theorem:adapted_homology_theories_correspond_to_injective_epimorphism_classes}, epimorphism classes with enough injectives are in one to one correspondence with adapted homology theories $H\colon \ccat \rightarrow \acat$. Thus, alternatively we can think that the starting input for our construction is an adapted homology theory. It will turn out that what we obtain can be considered as a categorification of the Adams spectral sequence. 

\subsection{Bounded derived $\infty$-category of an abelian category as sheaves}
\label{subsection:bounded_derived_cat_of_abelian_cat}

If $\acat$ is an abelian category with enough injectives, its coconnective derived $\infty$-category can be obtained by starting with the subcategory of injectives and freely attaching totalizations, see \cite{higher_algebra}[1.3.3.8] for the dual statement. Passing to right bounded objects and stabilizing, we obtain a construction of the bounded derived $\infty$-category which is valid for any $\acat$ and which does not directly use chain complexes or any form of non-reflective localization. 

Unfortunately, this description is somewhat inconvenient for many purposes, as the process of freely attaching totalizations is most easily described using \emph{copresheaves}, which is often unnatural. In this section, we will give an alternative description of the bounded derived $\infty$-category using \emph{sheaves} and the prestable Freyd envelope. 

In what follows, let $\acat$ be an abelian category with enough injectives. In \cref{definition:classical_freyd_envelope}, we have introduced the classical Freyd envelope, which by  \cref{proposition:if_c_has_finite_limits_freyd_envelope_is_an_abelian_subcat} is itself an abelian category. The two are related in the following way.

\begin{lemma} \label{lemma:yoneda has left adjoint}
The Yoneda embedding $y\colon \acat \rightarrow A(\acat)$ has a left adjoint which presents $\acat$ as a quotient of its Freyd envelope by a localizing subcategory. 
\end{lemma}

\begin{proof}
The left adjoint is given by the left Kan extension along the identity of $\acat$, which exists since the latter already has cokernels. To say that it presents $\acat$ as a quotient by a localizing subcategory is the same as to say that the right adjoint $y$ is fully faithful, which is the classical Yoneda lemma as $\acat$ is an ordinary category. 
\end{proof}

We first describe the image of the Yoneda embedding in sheaf-theoretic terms. 

\begin{definition}[{\cite[\href{https://stacks.math.columbia.edu/tag/05PM}{Tag 05PM}]{stacks-project}}]
The \emph{epimorphism topology} on an abelian category $\acat$ is a Grothendieck pretopology where covering families $\{ a_{i} \rightarrow b \}$ consist of a single epimorphism. 
\end{definition}

\begin{lemma}
\label{lemma:yoneda_embedding_of_abelian_category_identifies_with_sheaves}
The Yoneda embedding $y\colon \acat \hookrightarrow A(\acat)$ identifies the source with the subcategory of those finitely presented presheaves which are also sheaves for the epimorphism topology. 
\end{lemma}

\begin{proof}
By \cite[\href{https://stacks.math.columbia.edu/tag/05PN}{Tag 05PN}]{stacks-project}, the Yoneda embedding factors through the subcategory of sheaves and the induced functor into sheaves is exact. Since the image of $\acat$ generates the Freyd envelope under finite colimits, the same will be true for the subcategory of sheaves. It follows that the functor into finitely presented sheaves has to be an equivalence. 
\end{proof}

In \cref{definition:prestable_freyd_envelope}, we have introduced the $\infty$-categorical analogue of the classical Freyd envelope, namely the prestable Freyd envelope. Since we have just seen that $\acat$ can be recovered from the classical Freyd envelope as a category of sheaves, it is natural to expect that we can similarly recover the derived $\infty$-category by starting with $A_{\infty}(\acat)$ and again enforcing some kind of descent. 

Unfortunately, there are set-theoretic issues involved in this plan, since the sheaves we consider are not indexed by a small $\infty$-category, and so the existence of a suitable sheafification functor is not immediate. The situation is not all that dire, however - as a consequence of \cref{lemma:yoneda_embedding_of_abelian_category_identifies_with_sheaves} and \cref{lemma:yoneda has left adjoint}, we see that the needed sheafification does exist on the heart. This will be enough to construct at least the bounded versions of the needed sheaf $\infty$-categories. 

\begin{notation}
\label{notation:sheafication_in_large_spaces_on_abelian_category}
It will be convenient to argue in a context where we know that sheafification exists. Recall that $A_{\infty}(\acat)$ is by definition a subcategory of presheaves $\Fun(\acat^{op}, \spaces)$, choosing a larger universe, we can embed the latter into the $\infty$-category
\[
\widehat{P}_{\Sigma}(\acat) \colonequals \Fun_{\Sigma}(\acat^{op}, \widehat{\spaces})
\]
of product-preserving presheaves valued in large spaces. Relative to this larger universe, $\acat$ is small, and so by standard results there exists a sheafification functor which we will denote by $L\colon \widehat{P}_{\Sigma}(\acat) \rightarrow \widehat{P}_{\Sigma}(\acat)$ \cite{higher_topos_theory}[6.2.2.7].
\end{notation}

\begin{remark}
This is implicit in what we wrote above, but since $\acat$ is additive and the pretopology on $\acat$ consists only of single maps, the sheafification (and hypercomplete sheafification) functors do preserve product-preserving presheaves by \cite{pstrkagowski2018synthetic}[2.5]
\end{remark}

The following criterion for being a sheaf will be useful. 

\begin{lemma}
\label{lemma:detecting_sheaves_in_additive_setting}
Let $\dcat$ be an additive $\infty$-category equipped with a Grothendieck pretopology where covering families consists of single morphisms and let $X \in \widehat{P}_{\Sigma}(\dcat)$. Then, $X$ is a sheaf if and only if for every fibre sequence $a \rightarrow b \rightarrow c$, where the second map is a covering, the induced sequence $X(c) \rightarrow X(b) \rightarrow X(a)$ of spaces is a fibre sequence. 
\end{lemma}

\begin{proof}
This is \cite{pstrkagowski2018synthetic}[2.8].
\end{proof}

\begin{remark}
In the particular case of an abelian category with the epimorphism topology, the characterization of \cref{lemma:detecting_sheaves_in_additive_setting} reduces to asking that $X$ takes short exact sequences to fibre sequences of spaces. 
\end{remark}

\begin{proposition}
\label{proposition:sheafication_of_suspended_representable_in_freyd_envelope_of_abelian_cat}
Suppose that $a \in \acat$ and consider the suspension $\Sigma^{n} y(a) \in A_{\infty}(\acat)$ in almost perfect presheaves, which is explicitly defined by the $n$-fold bar construction
\[
(\Sigma^{n} y(a))(b) = \mathrm{B}^{n}(\Hom_{\acat}(b, a)).
\]
Then, we have that 
\begin{enumerate}
    \item $L \Sigma^{n} y(a)$ is almost perfect and 
    \item $\pi_{k}(L \Sigma^{n} y(a))(b) \simeq \Ext_{\acat}^{n-k}(b, a)$.
\end{enumerate}
\end{proposition}

\begin{proof}
Observe that both statements are true when $a = i$ is injective, as then $\Sigma^{n} y(a)$ is already a sheaf as a consequence of \cref{lemma:detecting_sheaves_in_additive_setting}. 

Now let $a$ be arbitrary and choose an injective resolution 
\[
0 \rightarrow a \rightarrow i_{0} \rightarrow i_{1} \rightarrow \ldots 
\]
in $\acat$ which through the Dold-Kan correspondence we can identify with an coaugmented cosimplicial object. We claim that the induced map of almost perfect presheaves
\[
\theta\colon \Sigma^{n} y(a) \rightarrow \Tot(\Sigma^{n} y(i_{\bullet}))
\]
is a sheafification, this shows (2) after unwrapping the definitions. Moreover, since all presheaves here are $n$-bounded, we can replace the totalization by a suitably large partial totalization, which is a finite limit. As almost perfect presheaves are closed under finite limits by \cref{corollary:almost_perfect_presheaves_closed_under_finite_limits_colimits_truncations}, this also shows (1). 

As $\theta$ is a map of bounded presheaves, to see that it is an equivalence after sheafification it is enough to check that it is an isomorphism on sheaf homotopy groups. We can identify discrete sheaves with objects of $\acat$ by \cref{lemma:yoneda_embedding_of_abelian_category_identifies_with_sheaves}. Under this identification, the totalization spectral sequence shows that the $k$-th sheaf homotopy groups of $\Tot(\Sigma^{n} y(i_{\bullet}))$ can be identified with $(n-k)$-th cohomology group of the complex
\[
0 \rightarrow i_{0} \rightarrow i_{1} \rightarrow \ldots
\]
These vanish unless $k = n$, in which case the cohomology group is $a$, as needed. 
\end{proof}

\begin{remark}
One consequence of  \cref{proposition:sheafication_of_suspended_representable_in_freyd_envelope_of_abelian_cat} is that the $\Ext$-groups of an abelian category with enough injectives can be identified with cohomology groups of $\acat$ with respect to the epimorphism topology. This can be suitably generalized to a statement about more general derived functors. 
\end{remark}

\begin{corollary}
\label{corollary:sheafication_exists_for_bounded_sheaves_in_abelian_case}
Let $X \in A^{\omega}_{\infty}(\acat)$ be a perfect presheaf. Then $LX$ is perfect. 
\end{corollary}

\begin{proof}
Since $L$ is left exact and almost perfect sheaves are closed under finite limits, we deduce that so is the class of those almost perfect $X$ such that $LX$ is again almost perfect. As the latter contains all presheaves with homotopy concentrated in a single degree by \cref{proposition:sheafication_of_suspended_representable_in_freyd_envelope_of_abelian_cat}, we deduce that it contains all perfect presheaves. By \cref{lemma:recognition_of_finite_presheaves_on_an_ordinary_category}, we conclude that $LX$ is perfect.  
\end{proof}
Knowing that the needed localization exists, we can make the main definition of this section. 

\begin{definition}
\label{definition:bounded_derived_infty_cat_of_abelian_cat_as_sheaves}
The \emph{bounded (connective) derived $\infty$-category} of $\acat$, denoted by $\dcat^{b}(\acat)$, is the $\infty$-category of bounded almost perfect sheaves on $\acat$ with respect to the epimorphism topology. 
\end{definition}

\begin{warning}
As we first warned the reader in \S\ref{subsection:conventions}, unlike most sources, we use $\dcat(-)$ to denote the \emph{connective} derived $\infty$-categories. Note that these are prestable rather than stable. In terms of \cref{definition:bounded_derived_infty_cat_of_abelian_cat_as_sheaves}, our convention is that 
\[
\dcat^{b}(\acat) \colonequals \dcat_{\geq 0}(\acat).
\]
When we will need to consider a stable variant of the derived $\infty$-category, which happens very little, we will be very explicit about this. 
\end{warning}

The following summarizes the main properties of the derived $\infty$-category. 
\begin{proposition}
\label{proposition:properties_of_finite_derived_infty_cat_of_abelian_cat}
We have that 
\begin{enumerate}
\item the functor $L\colon A_{\infty}^{\omega}(\acat) \rightarrow \dcat^{b}(\acat)$ is an exact localization, in particular 
\item $\dcat^{b}(\acat)$ is a prestable $\infty$-category with finite limits. Moreover, 
\item we have a canonical equivalence $\dcat^{b}(\acat)^{\heartsuit} \simeq \acat$ and 
\item for any $a, b \in \acat$, we have $\pi_{0} \Map_{\dcat^{b}(\acat)}(a, \Sigma^{n} b) \simeq \Ext^{n}_{\acat}(a, b)$.
\end{enumerate}
\end{proposition}

\begin{proof}
Property $(1)$ follows from \cref{corollary:sheafication_exists_for_bounded_sheaves_in_abelian_case}, which tells us that the sheafification and inclusion adjunction restricts to bounded almost perfect presheaves, and $(2)$ is an immediate consequence. Property $(3)$ is \cref{lemma:yoneda_embedding_of_abelian_category_identifies_with_sheaves} and the last one is \cref{proposition:sheafication_of_suspended_representable_in_freyd_envelope_of_abelian_cat}.
\end{proof}

We will now prove that the bounded derived $\infty$-category, as we defined it, satisfies the expected universal property. 

\begin{lemma}
\label{lemma:any_perfect_presheaf_over_abelian_category_rep_by_a_bounded_chain_complex}
Every perfect presheaf $X \in A_{\infty}^{\omega}(\acat)$ can be written as a geometric realization of a simplicial object corresponding under the Dold-Kan equivalence to a bounded chain complex $C_{\bullet}$ in $\acat$. Moreover, the t-structure homotopy groups of its sheafification satisfy 
\[
\pi_{k}^{\acat}(LX) \simeq H_{k}(C_{\bullet}).
\]
\end{lemma}

\begin{proof}
The first part is immediate from \cref{corollary:perfect_presheaves_realizations_of_skeletal_simplicial_objects}, as under the Dold-Kan equivalence skeletal simplicial objects correspond to bounded complexes. For the second part, observe that we have a realization spectral sequence of presheaf homotopy groups 
\[
H_{k}(\pi_n y(C_{\bullet})) \Rightarrow \pi_{k+n}X.
\]
which collapses on the second page giving an isomorphism 
\[H_k(y(C_{\bullet})) \simeq  \pi_k X\]
of objects of $A(\acat)$. After sheafifying, this gives the desired isomorphism. 
\end{proof}

\begin{theorem}
\label{theorem:universal_property_of_bounded_derived_cat_of_an_abelian_cat}
Let $\acat$ be an abelian category with enough injectives and $\dcat$ a prestable $\infty$-category with finite limits. Then, left Kan extension gives an equivalence between the following two kinds of data: 

\begin{enumerate}
    \item exact functors $\acat \rightarrow \dcat^{\heartsuit}$ of abelian categories and 
    \item exact functors $\dcat^{b}(\acat) \rightarrow \dcat$ of prestable $\infty$-categories with finite limits. 
\end{enumerate}
\end{theorem}

\begin{proof}
The correspondence from $(2)$ to $(1)$ is given by restricting a given functor to the hearts. We claim left Kan extension provides an inverse. 

By \cref{remark:universal_property_of_perfect_freyd_envelope}, any additive functor $f\colon \acat \rightarrow \dcat$ extends uniquely to a right exact functor $F\colon A_{\infty}^{\omega}(\acat) \rightarrow \dcat$. We have to show that if $f$ is an exact functor, then the extension $F$ factors uniquely through an exact functor out of $\dcat^{b}(\acat)$. 

We first claim that $F$ factors through the bounded derived $\infty$-category in the first place. We have to show that it inverts $L$-equivalences, which are exactly those maps whose cofibre is $L$-acyclic. As $F$ is right exact, it is enough to check that $F$ annihilates $L$-acyclic objects. 

By \cref{lemma:any_perfect_presheaf_over_abelian_category_rep_by_a_bounded_chain_complex}, any $L$-acyclic object can be represented as a geometric realization $X \simeq y(a_{\bullet})$ of a simplicial object of $\acat$ corresponding to an acyclic bounded chain complex. Since $F$ is right exact, it follows that $F(X)$ is a geometric realization of $f(a_{\bullet})$, which is an acyclic complex in $\dcat^{\heartsuit}$ as $f$ is assumed to be exact. It follows from the geometric realization spectral sequence that $F(X) = 0$, as needed. 

We are left with showing that the unique right exact factorization through the derived $\infty$-category, which we will also denote by $F\colon \dcat^{b}(\acat) \rightarrow \dcat$, is left exact. The proof of \cite{lurie_spectral_algebraic_geometry}[C.3.2.2] shows that it is enough to verify that $F$ takes $n$-truncated objects to $n$-truncated objects. 

We use the same argument as above. If $X \in \dcat^{b}(\acat)$ is $n$-truncated, then by \cref{lemma:any_perfect_presheaf_over_abelian_category_rep_by_a_bounded_chain_complex} we can write $X \simeq y(a_{\bullet})$, where $a_{\bullet}$ corresponds to a chain complex in $\acat$ with no homology above degree $n$. Thus, $F(X) \simeq |f(a_{\bullet})|$ also has no $\dcat^{\heartsuit}$-valued $t$-structure homotopy groups above this degree, proving that it is $n$-truncated. 
\end{proof}

\subsection{Perfect derived $\infty$-category of a stable $\infty$-category} 
\label{subsection:perfect_derived_infty_cat_of_a_stable_infty_cat}

In this section, we associate to a stable $\infty$-category $\ccat$ equipped with a choice of an adapted homology theory an appropriate \emph{perfect derived $\infty$-category} $\dcat^{\omega}(\ccat)$. This will be a prestable $\infty$-category which is universal with respect to the property of admitting a prestable enhancement to the chosen homology theory in the sense of \cref{definition:prestable_enhancement}. 

\begin{remark}[Why relative to a homology theory?]
To construct a derived $\infty$-category, one has to fix a good notion of an epimorphism. In the abelian context, there is a natural such choice. However, as discussed at more length in \S\ref{subsection:epimorphism_classes}, a stable $\infty$-category does not have a canonical non-trivial notion of an epimorphism.

Thus, our derived $\infty$-category will depend not only on a stable $\ccat$, but also on a choice of an epimorphism class. This is the same as choosing an adapted homology theory by \cref{theorem:adapted_homology_theories_correspond_to_injective_epimorphism_classes}. 
\end{remark}

\begin{remark}[Why ``perfect''?]
We chose the term ``perfect'' for $\dcat^{\omega}(\ccat)$ as it will have the property of being generated under finite colimits by the image of a suitable embedding of $\ccat$. Thus, it will be analogous to the bounded derived $\infty$-category $\dcat^{b}(\acat)$ of an abelian category introduced in \cref{definition:bounded_derived_infty_cat_of_abelian_cat_as_sheaves}. 

The reason to avoid the word ``bounded'' in the stable case is that $\dcat^{\omega}(\ccat)$ will not be bounded in the sense that its every object is $n$-truncated for some $n$. This is a consequence of the fact that a non-zero stable  $\infty$-category $\ccat$ contains no $n$-truncated objects. 
\end{remark}
Informally, one would like to define the objects of the derived $\infty$-category associated to an adapted homology theory $H$ to consist of ``formal $H$-Adams resolutions''. Like in the abelian case, one can obtain the coconnective variant by starting with the subcategory of $H$-injectives and freely attaching totalizations, then stabilizing and passing to a suitable subcategory. For technical reasons, this is not the route we will take, instead giving a sheaf-theoretic description, mirroring our construction of the derived $\infty$-category of an abelian category. 

\begin{remark}[Expected properties]
Before proceeding with the construction, let us remark that the properties we expect from the derived $\infty$-category associated to $H\colon \ccat \rightarrow \acat$ mirror those of the $\infty$-category of $\synspectra_{E}$ of synthetic spectra. In particular, we will see that
\begin{enumerate}
    \item the derived $\infty$-category $\dcat^{\omega}(\ccat)$ is related by a well-behaved  adjunction to $\dcat^{b}(\acat)$,
    \item that its Postnikov tower describes the $H$-Adams spectral sequence and that 
    \item a suitable tower of its subcategories is a Goerss-Hopkins tower.
\end{enumerate}
These will be explored in, respectively,  \S\ref{subsection:homology_adjunction_and_thread_structure}, \S\ref{subsection:postnikov_towers_and_adams_filtration} and \S\ref{subsection:goerss_hopkins_theory}.
\end{remark}
We now move on to the construction. 

\begin{notation}
We fix a stable $\infty$-category equipped with a choice of an adapted homology theory $H\colon \ccat \rightarrow \acat$. We recall that by \cref{theorem:characterization_of_adapted_homology_theories} the induced functor $A(\ccat) \rightarrow \acat$ presents the latter as a Gabriel quotient of the classical Freyd envelope by a localizing subcategory which we will denote by $K$.
\end{notation}
Recall that our construction of $\dcat^{b}(\acat)$ given in \S\ref{subsection:bounded_derived_cat_of_abelian_cat} proceeded by starting with the prestable Freyd envelope $A_{\infty}(\acat)$ and enforcing descent with respect to epimorphisms. This is the route we will also take in the stable case. 
\begin{definition}
The \emph{$H$-epimorphism topology} is a Grothendieck pretopology on $\ccat$ in which a family of maps $\{ c_{i} \rightarrow d \}$ is a covering if it consists of a single map which is an $H$-epimorphism. 
\end{definition}

We first verify that this gives the right answer in case of discrete almost perfect presheaves.

\begin{proposition}
\label{proposition:k_local_objects_are_sheaves}
A finitely presented discrete presheaf $x \in A(\ccat)$ is a sheaf with respect to the $H$-epimorphism topology if and only if it is in the image of the right adjoint $\acat \hookrightarrow A(\ccat)$.
\end{proposition}

\begin{proof}
Since $\acat \simeq A(\ccat)/K$ is a Gabriel quotient, it is known that the image of the right adjoint can be characterized as those $x \in A(\ccat)$ such that $\Ext^{s}_{A(\ccat)}(k, x)$ for all $k \in K$ and $s = 0, 1$ (Remark \ref{remark:characterization_of_the_image_of_right_adjoint}). In our context, the objects belonging to $K$ are precisely those that can be written as a cokernel of a map $y(d) \rightarrow y(e)$ induced by an $H$-epimorphism $d \rightarrow e$. 

Completing the latter to a fibre sequence, we obtain a projective resolution
\[
y(c) \rightarrow y(d) \rightarrow y(e) \rightarrow k \rightarrow 0
\]
in $A(\ccat)$ and the relevant $\Ext$-groups can be computed as zeroth and first cohomology of the chain complex
\[
x(e) \rightarrow x(d) \rightarrow x(c).
\]
These vanish precisely when the above is a kernel sequence, hence the claimed result follows from the characterization of additive presheaves given in \cref{lemma:detecting_sheaves_in_additive_setting}.
\end{proof}
We will be working as in \cref{notation:sheafication_in_large_spaces_on_abelian_category}, with the $\infty$-category 
\[
\widehat{P}_{\Sigma}(\ccat) \colonequals \Fun_{\Sigma}(\ccat^{op}, \widehat{\spaces}) 
\]
of product-preserving presheaves of large spaces. In this context, we have a sheafification functor $L_{\ccat}\colon \widehat{P}_{\Sigma}(\ccat) \rightarrow \widehat{P}_{\Sigma}(\ccat)$ for formal reasons. 

Our goal is to show that $L_{\ccat}$ preserves perfect presheaves, so that sheaves form a localization of the Freyd envelope. We could argue directly, as we did in the abelian case in \cref{corollary:sheafication_exists_for_bounded_sheaves_in_abelian_case}, but it will be more convenient to use the work we did in the previous section. 

\begin{definition}
\label{definition:homology_adjunction_for_freyd_envelopes}
Let $H^{*}\colon \widehat{P}_{\Sigma}(\ccat) \rightarrow \widehat{P}_{\Sigma}(\acat)$ be the left Kan extension of the homology functor $H\colon \ccat \rightarrow \acat$. This preserves almost perfect presheaves and hence induces an adjunction
\[
H^{*} \dashv H_{*}\colon A_{\infty}(\ccat) \rightleftarrows A_{\infty}(\acat),
\]
where $H_{*}$ is the restriction along $H$ functor. We will call $H^{*} \dashv H_{*}$ the \emph{homology adjunction}. 
\end{definition}
Observe that, by definition, $H\colon \ccat \rightarrow \acat$ preserves (and reflects) covering morphisms if we equip $\acat$ with the epimorphism topology. Thus, it is a morphism of sites in the sense of \cite{pstrkagowski2018synthetic}[A.10], and $H_{*}$ preserves sheaves. In fact, more is true, as we have the following:

\begin{proposition}
\label{proposition:homology_has_covering_lifting_property}
The functor $H\colon \ccat \rightarrow \acat$ has the covering lifting property. That is, for every $c \in \ccat$ and every epimorphism $a \rightarrow H(c)$, there exists an $H$-epimorphism $d \rightarrow c$ and a commutative diagram 
\[
\begin{tikzcd}
	& {H(d)} \\
	a & {H(c).}
	\arrow[two heads, from=1-2, to=2-2]
	\arrow[two heads, from=2-1, to=2-2]
	\arrow[from=1-2, to=2-1]
\end{tikzcd}
\]
\end{proposition}

\begin{proof}
Since $H$ is adapted, we have $\acat \simeq A(\ccat)/K$. Since in the Freyd envelope, any object is a quotient of a representable, any object $a \in \acat$ is a quotient of an object of the form $H(e)$, so that we can assume that $a = H(e)$ in the first place. 

We claim that there exists an $H$-epimorphism $d \rightarrow e$, such that the composite 
\[
H(d) \rightarrow H(e) \rightarrow H(c)
\]
lifts to an actual map $d \rightarrow c$, this proves the proposition. To see this, observe that by \cref{proposition:k_local_objects_are_sheaves}, $H(c)$ can be identified with the sheafification of $y(c)$ in the Freyd envelope. Thus, the plus-construction of the sheafification shows that the map $y(e) \rightarrow H(c)$ has an honest representative $y(d) \rightarrow y(c)$, where $d \rightarrow e$ is an $H$-epimorphism covering. 
\end{proof}

\begin{remark}
Having in mind a reader feeling a little uneasy with the sheaf-theoretic argument given above, we add that it is possible to construct the $H$-epimorphism $d \rightarrow e$ appearing in the proof of \cref{proposition:homology_has_covering_lifting_property} directly, by embedding $c$ into an injective and using adaptedness. 
\end{remark}

\begin{corollary}
\label{corollary:restriction_along_homology_commutes_with_sheafication_and_hp_sheafication}
The restriction $H_{*}\colon \widehat{P}_{\Sigma}(\acat) \rightarrow \widehat{P}_{\Sigma}(\ccat)$ commutes with sheafification and hypercomplete sheafification.
\end{corollary}

\begin{proof}
This follows from the covering lifting property, see \cite{pstrkagowski2018synthetic}[A.13] for the proof. 
\end{proof}
The above gives us a good method of showing that sheafification preserves desirable properties, as long as we can present a given presheaf as being of the form $H_{*}X$ for some $X \in \widehat{P}_{\Sigma}(\acat)$, as then sheafification can be computed in the latter. Luckily, we have a good supply of such presheaves. 

\begin{lemma}
\label{lemma:homology_adjunction_is_an equivalence_on_hearts}
The restriction of the right adjoint to the hearts 
\[
H_{*}\colon A(\acat) \rightarrow A(\ccat)
\]
induces an adjoint equivalence between the categories of finitely presented sheaves. In particular, any discrete sheaf $X \in A(\ccat)$ can be written uniquely in the form $H_{*}Y$ for some sheaf $Y \in A(\acat)$. 
\end{lemma}

\begin{proof}
The categories of finitely presented discrete sheaves on both sides can be identified with the abelian category $\acat$, as a consequence of \cref{proposition:k_local_objects_are_sheaves} and \cref{lemma:yoneda_embedding_of_abelian_category_identifies_with_sheaves}. Chasing through the definitions shows that $H_{*}$ is compatible with these equivalences. 
\end{proof}

\begin{corollary}
\label{corollary:presheaves_with_single_homotopy_group_can_be_sheafified_in_abelian_cat}
Let $X \in A_{\infty}^{\omega}(\ccat)$ be a presheaf with homotopy groups concentrated in a single degree $n$. Then, $L_{\ccat} X \simeq H_{*}(L_{\acat} Y)$ for some $Y \in A_{\infty}(\acat)$ with homotopy concentrated in single degree $n$.
\end{corollary}

\begin{proof}
Write $X \simeq \Sigma^{n} x$ for some $x \in A(\ccat)$ and $n \geq 0$. Since $L$ is a localization, the morphism $\Sigma^{n} x \rightarrow \Sigma^{n} Lx$ becomes an equivalence after applying $L$. Thus, we can assume that $X \simeq \Sigma^{n} x$, where $x \in A(\ccat)$ is already a sheaf. 

By \cref{lemma:homology_adjunction_is_an equivalence_on_hearts}, we can then write $\Sigma^{n} x \simeq H_{*} \Sigma^{n} y$ for some sheaf $y \in A(\acat)$, and the result follows from \cref{corollary:restriction_along_homology_commutes_with_sheafication_and_hp_sheafication} by taking $Y \colonequals \Sigma^{n} y$.
\end{proof}

\begin{proposition}
\label{proposition:bounded_almost_perfect_presheaves_have_sheafication}
Let $X \in A^{\omega}_{\infty}(\ccat)$ be a perfect presheaf. Then, the sheafification $LX$ is perfect.
\end{proposition}

\begin{proof}
We will first assume that $X$ has bounded presheaf homotopy groups. In this case, it belongs to the smallest subcategory of $A^{\omega}_{\infty}(\ccat)$ closed under fibres and containing all perfect presheaves with a single non-zero homotopy group. As the subcategory of those perfect presheaves such that $LX$ is perfect is closed under fibres since $L$ is left exact, this case is an immediate consequence of \cref{corollary:presheaves_with_single_homotopy_group_can_be_sheafified_in_abelian_cat} and the abelian case covered in \cref{corollary:sheafication_exists_for_bounded_sheaves_in_abelian_case}.

Now, consider the case when $X \simeq \nu(c)$ for $c \in \ccat$ is a representable. In this case, $X$ satisfies the criterion of \cref{lemma:detecting_sheaves_in_additive_setting}, as it is already a sheaf. 

Finally, the general case follows from the perfectness criterion of \cref{lemma:characterization_of_finite_presheaves_on_stable_inftycat} which implies that any perfect presheaf is a fibre of a map from a bounded presheaf to one of the form $\nu(c)_{\geq n}$ for some $n$. Since the latter itself is a fibre of a map $\nu(c) \rightarrow \nu(c)_{\leq n-1}$, we are done. 
\end{proof}

\begin{definition} \label{perfderived}
\label{definition:perfect_derived_cat_of_stable_cat}
Let $H\colon \ccat \rightarrow \acat$ be an adapted homology theory. Then, the perfect derived $\infty$-category of $\ccat$ relative to $H$ is given by 
\[
\dcat^{\omega}(\ccat) \colonequals A_{\infty}^{\omega, sh}(\ccat), 
\]
the $\infty$-category of perfect sheaves on $\ccat$ with respect to the $H$-epimorphism topology. 
\end{definition}

\begin{warning}
Note that we assume that the adapted homology theory $H$ is understood from the context. It would be more correct to write
\[
\dcat^{\omega}(\ccat, H)
\]
rather than $\dcat^{\omega}(\ccat)$, as the derived $\infty$-category depends on the choice of homology theory. However, this would clutter our notation, so we will not do so.
\end{warning}

The following lists the main properties of the perfect derived $\infty$-category. 

\begin{proposition}
\label{proposition:properties_of_finite_derived_infty_cat_of_stable_cat}
We have that 
\begin{enumerate}
\item $L\colon A_{\infty}^{\omega}(\ccat) \rightarrow \dcat^{\omega}(\ccat)$ is an exact localization compatible with local grading, in particular 
\item $\dcat^{\omega}(\ccat)$ is a locally graded, prestable $\infty$-category with finite limits. Moreover, 
\item we have a canonical equivalence $\dcat^{\omega}(\ccat)^{\heartsuit} \simeq \acat$ and
\item the kernel $\mathrm{ker}(L)$ of the localization functor contains only bounded objects and as a localization of $A^{\omega}_{\infty}(\ccat)$, the perfect derived $\infty$-category is determined by its heart. Finally, 
\item the synthetic analogue functor $\nu\colon \ccat \rightarrow A_{\infty}^{\omega}(\ccat)$ factors through $\dcat^{\omega}(\ccat)$ and 
\item a cofibre sequence $c \rightarrow d \rightarrow e$ in $\ccat$ induces a cofibre sequence $\nu(c) \rightarrow \nu(d) \rightarrow \nu(e)$ if and only if $H(d) \rightarrow H(e)$ is surjective. 
\end{enumerate}
\end{proposition}

\begin{proof}
The first two properties are immediate from \cref{proposition:bounded_almost_perfect_presheaves_have_sheafication}, where we verified that sheafification takes perfect presheaves to perfect presheaves, and the fact that the Grothendieck pretopology on $\ccat$ is compatible with suspension which includes the local grading of the Freyd envelope.

The properties $(3)-(5)$ follow from the proof of \cref{proposition:bounded_almost_perfect_presheaves_have_sheafication}, as we observed that $\nu(c)$ for $c \in \ccat$ is already a sheaf, and that any other perfect presheaf is an extension of a sheaf of this form and a bounded presheaf. 

For the last property, observe that since $\nu$ is left exact, if $\nu(c) \rightarrow \nu(d) \rightarrow \nu(e)$ is always fibre. It is cofibre if and only if the last map is an epimorphism on $\dcat^{\omega}(\ccat)^{\heartsuit} \simeq \acat$-valued zeroth homotopy group, but we have 
\begin{equation}
\label{equation:nu_is_a_prestable_enhancement_of_h}
\pi_{0}^{\heartsuit} \nu(c) \simeq H(c)
\end{equation}
as objects of $\acat$ by construction.
\end{proof}

Let us observe that the perfect derived $\infty$-category has the following elegant universal property similar to the one of \cref{theorem:universal_property_of_bounded_derived_cat_of_an_abelian_cat} appearing in the abelian case.

Note that as a consequence of the natural isomorphism  (\ref{equation:nu_is_a_prestable_enhancement_of_h}) above, $\nu\colon \ccat \rightarrow \dcat^{\omega}(\ccat)$ is a prestable enhancement to $H\colon \ccat \rightarrow \acat$ in the sense of \cref{definition:prestable_enhancement}; our main result is that it is in fact universal with respect to this property. 

\begin{theorem}[Universal property of the perfect derived $\infty$-category]
\label{theorem:universal_property_of_finite_derived_category}
Let $H\colon \ccat \rightarrow \acat$ be an adapted homology theory,  $\dcat^{\omega}(\ccat)$ the corresponding perfect derived $\infty$-category and $\dcat$ a prestable $\infty$-category with finite limits. Then, left Kan extension along $\nu\colon \ccat \rightarrow \dcat^{\omega}(\ccat)$ induces an equivalence between the two following collections of data:
\begin{enumerate} 
\item exact functors $G\colon \dcat^{\omega}(\ccat) \rightarrow \dcat$ of prestable $\infty$-categories and 
\item exact functors $G_{0}\colon \acat \rightarrow \dcat^{\heartsuit}$ of abelian categories \emph{together} with a prestable enhancement of $G_{0} \circ H$.
\end{enumerate}
\end{theorem}

\begin{proof}
Since $\acat$ is a localization of $A(\ccat)$, to give a collection of data of the second type is equivalent to a prestable enhancement $\euscr{H}$ \emph{with the property} that the composite $\pi_{0} \circ \euscr{H}$ is a homology theory such that the induced exact functor $A(\ccat) \rightarrow \dcat^{\heartsuit}$ factors through $\acat$. 

By \cref{theorem:universal_property_of_finite_presheaves}, prestable enhancements are classified by exact functors out of $A_{\infty}^{\omega}(\ccat)$, and the ones corresponding to those that with the above factorization property are those such that the induced map on the hearts factors through $\acat$. These are exactly those exact functors out of $A_{\infty}^{\omega}(\ccat)$ which factor uniquely through the perfect derived $\infty$-category by \cref{proposition:properties_of_finite_derived_infty_cat_of_stable_cat}.
\end{proof}

\begin{remark}
\label{remark:local_grading_on_perfect_derived_infty_cat}
As in the case of the prestable Freyd envelopes discussed in \cref{remark:universal_properties_of_local_gradings_on_prestable_freyd_envelopes}, $\dcat^{\omega}(\ccat)$ also acquires a unique local grading such that $\nu\colon \ccat \rightarrow \dcat^{\omega}(\ccat)$ acquires a structure of a locally graded functor. Moreover, the universal property of \cref{theorem:universal_property_of_finite_derived_category} implies a universal property of $\dcat^{\omega}(\ccat)$ as a locally graded prestable $\infty$-category. 
\end{remark}

It will be useful to have a concrete way to recognize the perfect derived $\infty$-category, and the following criterion will be sufficient for our purposes. For some intuition about the second condition below, see \cref{remark:intuition_about_second_condition_in_detection_of_perfect_derived_cat} which follows the proof. 

\begin{proposition}
\label{proposition:criterion_for_functor_out_of_perfect_derived_cat_to_be_ff}
Let $\dcat$ be a prestable $\infty$-category with finite limits with $\dcat^{\heartsuit} \simeq \acat$ and let $\nu_{\dcat}\colon \ccat \rightarrow \dcat$ be a prestable enhancement of $H$. Let $G\colon \dcat^{\omega}(\ccat) \rightarrow \dcat$ be the functor induced by \cref{theorem:universal_property_of_finite_derived_category}; that is, the unique exact functor such that $G \circ \nu \simeq \nu_{\dcat}$ and $G_{0} \colonequals G |_{\acat}$ is the chosen equivalence with $\dcat^{\heartsuit}$. Then, $G$ is fully faithful if and only if 
\begin{enumerate}
    \item $\nu_{\dcat}$ is fully faithful and,
    \item $\Ext^{s}_{\dcat}(\nu_{\dcat}(c), i) \colonequals \pi_{0} \Map_{\dcat}(\nu_{\dcat}(c), \Sigma^{s} i) = 0$ for any $c \in \ccat$, $i \in \acat^{inj}$ and $s > 0$. 
\end{enumerate}
Moreover, if the above holds then $G$ is an equivalence if and only if the image of $\nu_{\dcat}$ generates $\dcat$ under finite colimits. 
\end{proposition}

\begin{proof}
Observe that the above two properties hold for the identity functor of $\dcat^{\omega}(\ccat)$ itself, the first one by the Yoneda lemma and the second one since 
\[
\Map_{\dcat^{\omega}(\ccat)}(\nu(c), \Sigma^{s} i) \simeq \mathrm{B}^{s}\Hom_{\acat}(H(c), i),
\]
as the latter formula defines a sheaf with respect to the $H$-epimorphism topology when $i$ is injective with the correct sheaf homotopy groups.

In the general case, since the functor $G$ is an equivalence on the hearts by construction, any injective $i \in \acat$ is in its image. As so are all objects of the form $\nu_{\dcat}$, we deduce that the above two conditions are necessary, as they both hold for $\dcat^{\omega}(\ccat)$ itself.

Now suppose that the above two properties hold, we have to show that $G$ is fully faithful; that is, that for any $X, Y \in \dcat^{\omega}(\ccat)$, the induced morphism of mapping spaces
\[
\Map_{\dcat^{\omega}(\ccat)}(X, Y) \rightarrow \Map_{\dcat}(GX, GY)
\]
is an equivalence. Since $G$ is right exact and the perfect derived $\infty$-category is generated under finite colimits by representables $\nu(c)$ for $c \in \ccat$, we can restrict to the case when $X \simeq \nu(c)$ is of this form. 

Similarly, since $G$ is left exact, by \cref{corollary:perfect_prestable_freyd_envelope_generated_under_limits_by_representables_and_em_presheaves} we can assume that $Y$ is either also a representable or has sheaf homotopy concentrated in a single degree, as these kinds of sheaves generate the perfect derived $\infty$-category under finite limits. 

In the first case, we have $Y \simeq \nu(d)$ and the claim follows from 
\[
\Map_{\dcat^{\omega}(\ccat)}(\nu(c), \nu(d)) \simeq \Map_{\ccat}(c, d) \simeq \Map_{\dcat}(\nu_{\dcat}(c), \nu_{\dcat}(d)) \simeq \Map_{\dcat}(G(\nu(c)), G(\nu(d))
\]
since both $\nu$ and $\nu_{\dcat}$ are fully faithful by assumption.

Now assume that $Y \simeq \Sigma^{s} a$ for $a \in \acat$ is a sheaf with homotopy concentrated in a single degree. If 
\[
a \rightarrow i_{0} \rightrightarrows i_{1} \triplerightarrow\ldots
\]
is a cosimplicial injective resolution, then the totalization spectral sequence implies that 
\[
\Sigma^{s} a \rightarrow \Sigma^{s} i_{0} \rightrightarrows \Sigma^{s} i_{1} \triplerightarrow \ldots
\]
is a limit diagram. As this is a limit of $s$-truncated sheaves, it can be replaced by a finite totalization, and we deduce that we can further assume that $a = i$ is injective. 

After reducing to $X= \nu(c)$ and $Y \colonequals \Sigma^{s} i$, then for $s = 0$ we have 
\begin{equation}
\label{equation:equivalence_of_mapping_spaces_in_criterion_for_being_domega}
\Map_{\dcat^{\omega}(\ccat)}(\nu(c), i) \simeq \Hom_{\acat}(H(c), i) \simeq \Map_{\dcat}(\nu_{\dcat}(c), i) \simeq \Map_{\dcat}(G(\nu(c)), i) 
\end{equation}
as $G$ is an equivalence on the hearts and both of $\nu, \nu_{\dcat}$ are prestable enhancements of $H$. 

If $s > 0$, then both of the relevant mapping spaces have no homotopy groups below degree $s$ by assumption $(2)$ above, and so it is enough to check it is an equivalence after applying $\Omega^{s}$, reducing to the already covered case $s = 0$. 

For the last part, observe that if $G$ is fully faithful, then it is an equivalence if and only if it is essentially surjective. Since $\dcat^{\omega}(\ccat)$ is generated under finite colimits by representables, the claim follows. 
\end{proof}

\begin{remark}[Relation to the homology adjunction]
\label{remark:intuition_about_second_condition_in_detection_of_perfect_derived_cat}
Condition $(2)$ appearing in \cref{proposition:criterion_for_functor_out_of_perfect_derived_cat_to_be_ff} should be interpreted as saying that $\nu(c)$ are relatively projective in the perfect derived $\infty$-category, with relative projectivity measured with respect to the injectives of the heart. 

To see that this holds for $\dcat^{\omega}(\ccat)$ in a more intuitive manner than the direct argument given in the proof above, it is convenient to use the homology adjunction $H^{*} \dashv H_{*}\colon \dcat^{\omega}(\ccat) \leftrightarrows \dcat^{b}(\acat)$ which we will later introduce in  \S\ref{subsection:homology_adjunction_and_thread_structure}. 

Since the latter induces an adjunction between the hearts and $H^{*}(\nu(c)) \simeq H(c)$, we can write 
\[
\Ext^{s}_{\dcat^{\omega}(\ccat)}(\nu(c), i) \simeq \Ext^{s}_{\dcat^{\omega}(\ccat)}(\nu(c), H_{*}i) \simeq \Ext^{s}_{\dcat^{b}(\acat)}(H(c), i)
\]
which now vanishes if $s > 0$ since $i$ is assumed injective and the right hand side is just the classical $\Ext$-group of $\acat$. Thus, condition $(2)$ is automatic whenever we have a sufficiently well-behaved adjunction with the derived $\infty$-category of $\acat$ which is an equivalence on the hearts and such that $H^{*}(\nu(c))$ is discrete. 

Note that we can identify $H^{*} \dashv H_{*}$ with the free $C\tau$-module adjunction, as we will show later in  \cref{theorem:finite_ctau_modules_same_as_derived_category}. Thus, condition $(2)$ also holds whenever the relevant $\infty$-category admits some form of the $C\tau$-formalism. 
\end{remark}

\subsection{Homology adjunction and the thread structure}
\label{subsection:homology_adjunction_and_thread_structure}

In \S\ref{subsection:thread_structure_on_prestable_freyd_envelope}, we have introduced the thread structure on the prestable Freyd envelope of a stable $\infty$-category $\ccat$, which is a certain transformation 
\[
\tau\colon \Sigma X \rightarrow X[1]
\]
natural in $X \in A_{\infty}(\ccat)$. This had the property that the functor taking $X$ to the cofibre of 
\[
C\tau^{n} \otimes X \colonequals \mathrm{cofib}(\Sigma^{n} X[-n] \rightarrow X) 
\]
had a natural structure of a monad. In this section, we will show that these functors are compatible with the localization giving rise the the derived $\infty$-category, and that they can be used to construct a Goerss-Hopkins tower associated to $H\colon \ccat \rightarrow \acat$. 

One can work with either finite or almost perfect presheaves, but since we have introduced the derived $\infty$-category in full generality only in the finite case, we will focus on the former. 

In what follows, let $L$ denote the localization functor $A_{\infty}^{\omega}(\ccat) \rightarrow \dcat^{\omega}(\ccat)$ which corresponds to sheafification with respect to the $H$-epimorphism topology. We say a morphism is an $L$-equivalence if it is taken to an equivalence of sheaves.

\begin{lemma}
The functor $C\tau^{n} \otimes -\colon A_{\infty}^{\omega}(\ccat) \rightarrow A_{\infty}^{\omega}(\ccat)$ takes $L$-equivalences to $L$-equivalences, so that it descends to a monad on the derived $\infty$-category $\dcat^{\omega}(\ccat)$. 
\end{lemma}

\begin{proof}
Since $L$ is an exact localization compatible with the local grading by \cref{proposition:properties_of_finite_derived_infty_cat_of_stable_cat}, the subcategory of endofunctors of $A_{\infty}^{\omega}(\ccat)$ with this property is stable under finite limits, colimits and the local grading. Since it also contains the identity and 
\[
C\tau^{n} \otimes X \colonequals \mathrm{cofib}(\Sigma^{n} X[-n] \rightarrow X) 
\]
by \cref{proposition:spiral_cofibre_sequence}, we deduce that $C\tau^{n} $ is also in this subcategory. 
\end{proof}

\begin{notation}
We will denote the induced monad on $\dcat^{\omega}(\ccat)$ also by $C\tau^{n} {\otimes} -$, hoping that this does not lead to confusion. Note that it differs from the one on $A_{\infty}^{\omega}(\ccat)$ by a sheafification; in other words, it corresponds to taking the cofibre in sheaves rather than in presheaves. 
\end{notation}

In \cref{definition:homology_adjunction_for_freyd_envelopes} we have observed that the functor $H\colon \ccat \rightarrow \acat$ induces a left Kan extension and restriction adjunction 
\[
H^{*} \dashv H_{*}\colon A_{\infty}(\ccat) \leftrightarrows A_{\infty}(\acat).
\]
This restricts to perfect sheaves, and is compatible with the topologies on both sides and so induces an adjunction on $\infty$-categories of sheaves by standard arguments. 

\begin{definition}
\label{definition:homology_adjunction_for_derived_cats}
The \emph{homology adjunction} is the adjunction
\[
H^{*} \dashv H_{*}\colon \dcat^{\omega}(\ccat) \leftrightarrows \dcat^{b}(\acat)
\]
induced by restriction along $H\colon \ccat \rightarrow \acat$. 
\end{definition}

\begin{warning}
The notation used in \cref{definition:homology_adjunction_for_derived_cats}, the right adjoint $H_{*}$ given by restriction coincides with the one on whole Freyd envelope, but the left adjoint differs from the one on presheaves by the process of sheafification. This should not lead to too much confusion.
\end{warning}

\begin{lemma}
\label{lemma:right_adjoint_in_homology_adjunction}
The right adjoint $H_{*}\colon \dcat^{b}(\acat) \rightarrow \dcat^{\omega}(\ccat)$ is exact and induces an equivalence between the hearts; thus, it commutes with $\acat$-valued homotopy groups. 
\end{lemma}

\begin{proof}
The right adjoint is given by restriction, and it is exact when restricted to sheaves as the restriction functor on presheaves commutes with sheafification by \cref{corollary:restriction_along_homology_commutes_with_sheafication_and_hp_sheafication}. The second part is \cref{lemma:homology_adjunction_is_an equivalence_on_hearts}.
\end{proof}

\begin{lemma}
\label{lemma:bounded_derived_cat_monadic_over_derived_of_ccat}
For any $n \geq 0$, the induced adjunction
\[
(-)_{\leq n} \circ H^{*} \dashv H_{*}\colon \tau_{\leq n} \dcat^{\omega}(\ccat) \leftrightarrows \tau_{\leq n} \dcat^{b}(\acat)
\]
between subcategories of $n$-truncated objects is monadic. 
\end{lemma}

\begin{proof}
Since these are $(n+1)$-categories which admit finite colimits, they also admit geometric realizations, which gets preserved by both adjoints, in the latter case by \cref{lemma:right_adjoint_in_homology_adjunction}. Thus, by the Barr-Beck-Lurie criterion \cite{higher_algebra}[4.7.3.5], it is enough to verify that $H_{*}$ is conservative, which follows from the same result as on both sides equivalences are detected by $\acat$-valued homotopy groups. 
\end{proof}

\begin{theorem}
\label{theorem:finite_ctau_modules_same_as_derived_category}
The right adjoint $H_{*}\colon \dcat^{b}(\acat) \rightarrow \dcat^{\omega}(\ccat)$ has a canonical lift to $C\tau$-modules and induces an equivalence
\[
 \Mod_{C\tau \otimes -}(\dcat^{\omega}(\ccat)) \simeq \dcat^{b}(\acat)
\]
between $C\tau$-modules whose underlying sheaf is perfect and the bounded derived $\infty$-category of $\acat$. 
\end{theorem}

\begin{proof}
Since the right adjoint $H_{*}$ is exact, the universal property of the bounded derived $\infty$-category of \cref{theorem:universal_property_of_bounded_derived_cat_of_an_abelian_cat} implies that it is uniquely determined by its restriction to the heart. As the forgetful functor is an equivalence on the hearts by \cref{proposition:forgetful_functor_from_modules_is_exact}, it follows that there exists a unique exact lift
\[
U\colon \dcat^{b}(\acat) \rightarrow \Mod_{C\tau \otimes -}(\dcat^{\omega}(\ccat))).
\]
We first claim that for any $n \geq 0$, the restriction
\[
U\colon \tau_{\leq n} \dcat^{b}(\acat) \rightarrow \tau_{\leq n}( \Mod_{C\tau \otimes -}(\dcat^{\omega}(\ccat)))
\]
to subcategories of $n$-truncated objects is an equivalence. 

Observe that both of these $(n+1)$-categories are monadic over $\tau_{\leq n} \dcat^{\omega}(\ccat)$, the former by \cref{lemma:bounded_derived_cat_monadic_over_derived_of_ccat}. Thus, it follows from Lurie's criterion \cite{higher_algebra}[4.7.3.16] that to verify that the restricted lift is an equivalence we only have to check that for any $X \in \tau_{\leq n} \dcat^{\omega}(\ccat)$, the induced map 
\[
C\tau \otimes X \rightarrow U(H^{*}X)
\]
is an equivalence in $\tau_{\leq n}( \Mod_{C\tau \otimes -}(\dcat^{\omega}(\ccat)))$. Since $U$ is exact, both sides are right exact in $X$. 

Any object of $\tau_{\leq n} \dcat^{\omega}(\ccat)$ can be written as a geometric realization of representables, and since the latter is an $(n+1)$-category, this is equivalent to a realization of a finite skeleton. Thus, an arbitrary $X$ can be written as a finite colimit of $\nu(c)_{\leq n}$ for $c \in \ccat$, and it is enough to verify the latter case. 

As $H^{*}(\nu(c)) = H(c)$ by definition and $C\tau \otimes \nu(c) \simeq \nu(c)_{\leq 0}$ by \cref{lemma:cofibre_of_powers_of_tau_into_representable_a_truncation}, which are already discrete, after forgetting down to $\tau_{\leq n} \dcat^{\omega}(\ccat)$, the relevant arrow can be identified with a map
\[
 C\tau \otimes \nu(c) \rightarrow H_{*} H^{*} \nu(c)
\]
with the property that its composite with the unit $\nu(c) \rightarrow  C\tau \otimes \nu(c)$ is the unit of $H^{*} \dashv H_{*}$. Thus, to check that the relevant map is an isomorphism, it is enough to verify that both units are $\pi_{0}$-isomorphisms, which is a consequence of the fact that both adjunctions induce adjoint equivalences between the hearts. 

We have shown that for any $n \geq 0$, the restriction of the lift 
\[
U\colon \dcat^{b}(\acat) \rightarrow \Mod_{C\tau \otimes -}(\dcat^{\omega}(\ccat)))
\]
to subcategories of $n$-truncated objects is an equivalence. However, both prestable $\infty$-categories have the property that each object is $n$-truncated for some $n \geq 0$, the first one by construction and $C\tau^{n} $-modules by \cref{corollary:finite_ctaun_modules_are_bounded}. This ends the argument. 
\end{proof}

\begin{remark}[Special fibre]
One way to interpret \cref{theorem:finite_ctau_modules_same_as_derived_category} is as identifying the special fibre of the ``deformation'' $\infty$-category $\dcat^{\omega}(\ccat)$, for more on this perspective see \cite{gheorghe2018special}, \cite{burklund2020galois} and \cite{balderrama2021deformations}.
\end{remark}

\begin{remark}[Generic fibre]
\label{remark:generic_fibre_of_derived_category_of_homology_theory}
Similarly, one can ask what is the generic fibre of the derived $\infty$-category; that is, the $\infty$-category obtained by inverting all of the maps $\tau\colon \Sigma X[-1] \rightarrow X$ for $X \in \dcat^{\omega}(\ccat)$. This is not difficult. 

We argue as in \cref{construction:tau_inversion} that $\ccat$ is itself a prestable $\infty$-category, and hence the identity determines a prestable enhancement. As the heart of $\ccat$ (in the prestable sense) is zero, the induced classical homology theory vanishes and so canonically factors through $\acat$. It follows from the universal property of the derived $\infty$-category of \cref{theorem:universal_property_of_finite_derived_category} that the left Kan extension exists and is a unique exact functor 
\[
\tau^{-1}\colon \dcat^{\omega}(\ccat) \rightarrow \ccat.
\]
This can be shown to be the universal functor inverting $\tau$, so that the generic fibre of the derived $\infty$-category can be identified with $\ccat$. In particular, it is independent of the chosen homology theory. 
\end{remark}

\subsection{Postnikov towers and Adams filtrations}
\label{subsection:postnikov_towers_and_adams_filtration}

We have previously claimed that the tower of Postnikov covers in the derived $\infty$-category associated to adapted $H\colon \ccat \rightarrow \acat$ encodes the $H$-Adams spectral sequence. In this section, we will make this precise. 

In more detail, we will use connective coverings in the derived $\infty$-category to construct for any $c, d \in \ccat$ an $H$-Adams filtration on the mapping spectrum $F_{\ccat}(c, d)$. This filtration will be functorial in both $c$ and $d$ at the level of $\infty$-categories and have the property that the induced spectral sequence computing $\pi_{*} F(c, d) \simeq [c, d]_{*}$ is the $H$-Adams spectral sequence. 

To have a spectral sequence at the level of mapping spectra rather than mapping spaces, it will be convenient to stabilize our derived $\infty$-categories using the following standard construction.

\begin{construction}
\label{construction:spanier_whitehead_category}
If $\dcat$ is a prestable $\infty$-category, then it can be stabilized by taking the colimit of $\infty$-categories along the diagram 
\[
\begin{tikzcd}
	\dcat & \dcat & \dcat & \ldots
	\arrow["\Sigma", from=1-3, to=1-4]
	\arrow["\Sigma", from=1-2, to=1-3]
	\arrow["\Sigma", from=1-1, to=1-2],
\end{tikzcd}
\]
resulting in the $\infty$-category $SW(\dcat)$ of Spanier-Whitehead objects \cite{lurie_spectral_algebraic_geometry}[C.1.1]. Informally, the objects of the latter are pairs $(X, n)$, $(Y, m)$, where $X, Y \in \dcat$ and $n, m \in \mathbb{Z}$ and with mapping spaces computed by 
\[
\Map_{SW(\dcat)}((X, n), (Y, m)) \simeq \varinjlim_{k} \Map_{\dcat}(\Sigma^{k+n}X, \Sigma^{k+m}Y).
\]
That is, $(X, n)$ represents the formal $n$-th suspension of $X$. 
\end{construction}

\begin{remark}
The $\infty$-category $SW(\dcat)$ is stable and has $\dcat$ as the full subcategory spanned by objects of the form $(X, 0)$, which is a connective part of a unique $t$-structure. 
\end{remark}

\begin{definition}
\label{definition:spanier_whitehead_cat_of_d_is_d_minus}
If $H\colon \ccat \rightarrow \acat$ is a fixed adapted homology theory, we write 
\[
\dcat^{\omega}_{-}(\ccat) \colonequals SW(\dcat^{\omega}(\ccat))
\]
for the Spanier-Whitehead $\infty$-category of \cref{construction:spanier_whitehead_category} and refer to it as the \emph{bounded below perfect derived $\infty$-category}. Similarly, we write 
\[
\dcat^{b}_{-}(\acat) \colonequals SW(\dcat^{b}(\acat))
\]
\end{definition}

\begin{warning}
Note that \cref{definition:spanier_whitehead_cat_of_d_is_d_minus} is slightly awkward, as $\dcat^{\omega}(\ccat)$ is a \emph{subcategory} of $\dcat^{\omega}_{-}(\ccat)$. This is a price for our choice of denoting connective derived $\infty$-categories with $\dcat(-)$.
\end{warning}

\begin{remark}
\label{remark:stabilized_equivalence_of_ctau_modules_and_derived_cat}
Note that \cref{construction:spanier_whitehead_category} is clearly functorial with respect to right exact functors, as it only involves the suspension. Thus, \cref{theorem:finite_ctau_modules_same_as_derived_category} passes to Spanier-Whitehead $\infty$-categories to induce an equivalence
\[
\Mod_{C\tau}(\dcat^{\omega}_{-}(\ccat)) \simeq \dcat^{b}_{-}(\acat).
\]
\end{remark}

\begin{remark}
\label{remark:stabilized_dcat_omega_as_sheaves_of_spectra}
In the same way as $\dcat^{\omega}(\spectra)$ can be viewed as a full subcategory of product-preserving sheaves $Sh_{\Sigma}(\ccat, \widehat{\spaces})$ of large spaces (spanned by those sheaves that happen to be perfect), we can think of $\dcat^{\omega}_{-}(\ccat)$ as a subcategory of $Sh_{\Sigma}(\ccat, \widehat{\spectra})$, product-preserving sheaves of spectrum objects in large spaces. 

To see this, observe that the latter is stable, and admits a $t$-structure whose connective cover we can identify with sheaves of spaces \cite{pstrkagowski2018synthetic}[2.13], so that the universal property of the Spanier-Whitehead construction provides a fully faithful embedding 
\[
\dcat^{\omega}_{-}(\ccat) \hookrightarrow Sh_{\Sigma}(\ccat, \widehat{\spectra})
\]
Using the description given in \cref{construction:spanier_whitehead_category}, we see that the image consists of desuspensions of perfects. 
\end{remark}

\begin{remark}
In the context of \cref{remark:stabilized_dcat_omega_as_sheaves_of_spectra}, the image of $\dcat^{\omega}_{-}(\ccat)$ is is contained already in sheaves of (small) spectra. However, the passage to large spaces has the effect of making the sheaf $\infty$-category presentable with respect to a larger universe, which is often useful for formal reasons: for example, it admits a sheafification functor. 
\end{remark}
As we have promised before, the Postnikov towers in the derived $\infty$-category of $\ccat$ are supposed to encode the $H$-Adams spectral sequence. Now that we stabilize, we would like these towers to extend infinitely in both directions.

\begin{notation}
If $\ccat$ is an $\infty$-category, let us treat $\mathbb{Z}$ as a poset and write 
\[
\Fil(\ccat) \colonequals \Fun(\mathbb{Z}^{op}, \ccat)
\]
for the $\infty$-category of decreasingly \emph{filtered objects}, 
\end{notation}

\begin{proposition}
\label{proposition:existence_of_double_sided_postnikov_towers}
There is an exact functor $\mathrm{Post}: \ccat \rightarrow \Fil(\dcat^{\omega}_{-}(\ccat))$ valued in filtered objects which takes any $c \in \ccat$ to the diagram 
\[
\ldots \rightarrow \mathrm{Post}_{1}(c) \rightarrow \mathrm{Post}_{0}(c) \rightarrow \mathrm{Post}_{-1}(c) \rightarrow \ldots \colonequals \ldots \rightarrow \Sigma \nu(c)[-1] \rightarrow \nu(c) \rightarrow \Sigma^{-1} \nu(c)[1] \rightarrow \ldots
\]
such that each arrow $\mathrm{Post}_{n+1}(c) \rightarrow \mathrm{Post}_{n}(c)$ is an $n+1$-connective cover which can be identified with $\tau$ of \cref{definition:canonical_thread_structure_on_prestable_freyd_envelope}.
\end{proposition}

\begin{proof}
Note that the maps $\tau$ on representables do have the property that their source is a connective cover of the target, by \cref{lemma:cofibre_of_powers_of_tau_into_representable_a_truncation}. So it is enough to construct the Postnikov tower with the first property. 

Let us consider the bounded below derived $\infty$-category as a subcategory of sheaves of spectra, as in  \cref{remark:stabilized_dcat_omega_as_sheaves_of_spectra}. After unwrapping the definitions, we see that 
\[
\mathrm{Post}_{n}(c) \colonequals \Sigma^{n} \nu(c)[-n]
\]
can be identified with the sheafification of the presheaf of spectra defined by 
\[
d \mapsto \Sigma^{n} (F_{\ccat}(d, \Omega^{n} c)_{\geq 0}) \simeq F_{\ccat}(d, c)_{\geq n},
\]
where $F_{\ccat}$ is the internal mapping spectrum. Thus, if $Y(c)$ denotes the spectral sheaf 
\[
d \rightarrow F_{\ccat}(d, c),
\]
then the needed Postnikov tower is given by the tower of connective covers of $Y(c)$. 
\end{proof}

\begin{remark}
Note that the filtered object $\mathrm{Post}_{\star}$ constructed in \cref{proposition:existence_of_double_sided_postnikov_towers} consists of connective covers rather than truncations and so it is really a Whitehead tower rather than a Postnikov tower \cite{hatcher2005algebraic}[4.20]. Our terminology in the current work is thus slightly abusive; we refer to both as Postnikov towers. 
\end{remark}

We are now ready to define the $H$-Adams filtration by mapping into the above Postnikov tower.  

\begin{definition}
\label{definition:adams_filtered_spectrum}
Let $c, d \in \ccat$. Then, the \emph{$H$-Adams filtration} on $F_{\ccat}(d, c)$ is given by the decreasingly filtered spectrum
\[
F^{H}_{\star}(d, c) \colonequals F(\nu(d), \mathrm{Post}_{\star}(c)),
\]
where the internal mapping spectrum on the right is computed in $\dcat^{\omega}_{-}(\ccat)$. 
\end{definition}

To convince the reader that this is a reasonable definition to make, let us compute the associated graded pieces to the above filtration. 

\begin{remark}
\label{remark:associated_graded_of_adams_filtration}
Note that since by \cref{proposition:existence_of_double_sided_postnikov_towers} the map $\mathrm{Post}_{t+1}(c) \rightarrow \mathrm{Post}_{t}(c)$ can be identified with multiplication by $\tau$, there is a functorial equivalence
\[
\mathrm{gr}_{t} F^{H}_{\star}(d, c) \simeq F(\nu(d), \mathrm{cofib}(\mathrm{Post}_{t+1}(c) \rightarrow \mathrm{Post}_{t}(c)))
\]
which we can further rewrite as 
\[
F(\nu(d), C\tau \otimes \Sigma^{t} \nu(c)[-t]) \simeq F_{C\tau}(C\tau \otimes \nu(d), C\tau \otimes \Sigma^{t} \nu(c)[-t]) 
\]
and finally as 
\[
F(\nu(d), C\tau \otimes \Sigma^{t} \nu(c)[-t]) \simeq F_{C\tau}(C\tau \otimes \nu(d), C\tau \otimes \Sigma^{t} \nu(c)[-t]) \simeq F_{\dcat^{b}_{-}(\acat)}(H(d), \Sigma^{t} H(c)[-t])
\]
using the equivalence of $C\tau$-modules and the derived $\infty$-category of $\acat$ of  \cref{remark:stabilized_equivalence_of_ctau_modules_and_derived_cat}. In particular, 
\[
\pi_{t} \mathrm{gr}_{s} F^{H}_{\star}(d, c) \simeq \Ext^{s-t, s}_{\acat}(H(c), H(d)).
\]
In other words, the homotopy groups of the associated graded of the Adams filtration is a reindexed form of the $H$-Adams $E_{2}$-term. 
\end{remark}
Note that \cref{remark:associated_graded_of_adams_filtration} is one strong piece of evidence that the associated graded of the Postnikov tower has something to do with the $H$-Adams filtration. In fact, it is not far from implying it directly, as the following result shows. 

\begin{theorem}
\label{theorem:properties_of_the_h_adams_filtration}
Let $H: \ccat \rightarrow \acat$ be an adapted homology theory and $c, d \in \ccat$. Then, the $H$-Adams filtration $F^{H}_{\star}(d, c)$ has the following properties: 
\begin{enumerate}
    \item there is a canonical equivalence $\varinjlim F^{H}_{\star}(d, c) \simeq F(d, c)$
    \item For every $s, t \in \mathbb{Z}$, 
    \[
    \mathrm{im}(\pi_{s} F^{H}_{t} F(d, c)) \subseteq \pi_{s} F(d, c) \simeq [d, c]_{s}
    \]
    coincides with the subgroup of elements of $H$-Adams filtration at least $t-s$,
    \item the spectral sequence 
    \[
    \pi_{s} \mathrm{gr}_{t}(F^{H}_{\star}(d, c)) \Rightarrow [d, c]_{t-s}
    \]
    induced by this filtration of $F(d, c)$ coincides up to regrading with the $H$-based Adams spectral sequence.
\end{enumerate}. 
\end{theorem}

\begin{proof}
After unwrapping the definitions, we see that 
\[
\pi_{s} F_{t}^{H}(d, c) \simeq [\nu(d), \Sigma^{t-s} \nu(\Omega^{t} c)],
\]
where on the right hand side we have homotopy classes of maps in $\dcat^{\omega}_{-}(\ccat)$. If $t-s<0$, we can replace the target by its connective cover $\nu(\Omega^{s} c)$ and we deduce that in this case 
\[
\pi_{s} F_{t}^{H}(d, c) \simeq [\nu(d), \nu(\Omega^{s} c)] \simeq [d, \Omega^{s} c]_{\ccat},
\]
where we have used that $\nu$ is fully faithful. This shows that 
\[
\pi_{s} F_{t}^{H}(d, c) \simeq \pi_{s} F(d, c)
\]
for $t-s < 0$, the isomorphism induced by the connective covers of $Y(c)$ appearing in the proof of  \cref{proposition:existence_of_double_sided_postnikov_towers}. This shows part $(1)$. 

Now let us first show parts $(2)$ and $(3)$ in the special case when $c = i$ is $H$-injective. In this case, for any $n \in \mathbb{Z}$, the presheaf of spectra 
\[
d \rightarrow F_{\ccat}(d, i)_{\geq n} 
\]
considered in \cref{proposition:existence_of_double_sided_postnikov_towers} is already a sheaf, and we deduce using the Yoneda lemma that the $H$-Adams filtration of $F(d, i)$ is given by 
\[
\ldots \rightarrow F(d, i)_{\geq n+1} \rightarrow F(d, i)_{\geq n} \rightarrow F(d, i)_{\geq n-1} \rightarrow \ldots
\]
This has both the needed properties, as an $H$-Adams filtration of a non-zero map $d \rightarrow i$ is always zero when $i$ is $H$-injective. 

Now suppose that $c$ is arbitrary. As we observed in \cref{remark:cosimplicial_adams_resolutions}, by repeatedly embedding $c$ into an injective, we can construct a cosimplicial $H$-Adams resolution
\[
c \rightarrow i_{0} \rightrightarrows i_{1} \triplerightarrow \ldots
\]
such that mapping $d$ into it induces the $H$-Adams spectral sequence. We claim that the induced augmented cosimplicial diagram of filtered spectra
\[
F^{H}_{\star}(d, c) \rightarrow F^{H}_{\star}(d, i_{0}) \rightrightarrows F^{H}_{\star}(d, i_{1}) \triplerightarrow \ldots  
\]
is a limit diagram on the associated graded spectra. 

Indeed, by \cref{remark:associated_graded_of_adams_filtration}, the degree $t$ graded pieces of the above diagram are given by the mapping spectra
\[
F_{\dcat^{b}_{-}(\acat)}(H(d), \Sigma^{t} H(c)[-t]) \rightarrow F_{\dcat^{b}_{-}(\acat)}(H(d), \Sigma^{t} H(i_{\bullet})[-t])
\]
in the derived $\infty$-category of $\acat$. This is a limit diagram of spectra, since
\[
H(c) \rightarrow H(i_{0}) \rightrightarrows H(i_{1}) \triplerightarrow \ldots
\]
is an injective resolution by construction and so a limit diagram in $\dcat^{b}_{-}(\acat)$. It follows that the map of filtered spectra
\[
F^{H}_{\star}(d, c) \rightarrow \Tot(F^{H}_{\star}(d, i_{\bullet}))
\]
into the levelwise totalization is an equivalence on associated graded pieces and hence both induce isomorphic spectral sequences. 

Using the above description of the filtration for an injective object, we deduce that the spectral sequence induced by $F^{H}_{\star}(d, c)$ is the same as the one induced by the filtered spectrum 
\[
\ldots \rightarrow \Tot(F(d, i_{\bullet})_{\geq n+1}) \rightarrow \Tot(F(d, i_{\bullet})_{\geq n}) \rightarrow \Tot(F(d, i_{\bullet})_{\geq n-1}) \rightarrow \ldots,
\]
where notice that we first take connective covers and only then totalize. This is known as the \emph{d\'{e}calage} of the spectral sequence induced by the cosimplicial spectrum
\[
F(d, i_{0}) \rightrightarrows F(d, i_{1}) \triplerightarrow \ldots
\]
and it is a result of Levine that the two resulting spectral sequence agree up to reindexing \cite{levine2015adams}[6.3]. This proves $(3)$ for an arbitrary $c \in \ccat$, as the latter is exactly the $H$-Adams spectral sequence.

Finally, part $(2)$ is a consequence of $(3)$ and that $\pi_{s} F_{s+k}^{H}(d, c) \simeq [d, c]_{s}$ for $k < 0$. Indeed, in a spectral sequence of a filtered spectrum, a class $x \in \pi_{s} F_{s} $ will lift to $\pi_{s} F_{s+k}$ for $k > 0$ if and only if it is not detected in $\pi_{s} \mathrm{gr}_{s+i} \simeq \Ext^{i}_{\acat}$ for $i < k$, which is the same as being of Adams filtration at least $k$. 
\end{proof}

\begin{remark}
Note that $F^{H}(d, c)$ is not an exact functor in either of its variables. However, it does preserve cofibre sequences which are $H$-exact, as these are exactly the ones which are preserved by the synthetic analogue functor by the last part of \cref{proposition:properties_of_finite_derived_infty_cat_of_stable_cat}.

Note that the induced cofibre sequence translates into relations between the relevant $H$-Adams spectral sequences as a consequence of \cref{theorem:properties_of_the_h_adams_filtration}. We will not pursue this here, but we believe that this makes $F^{H}$ into a natural home for ``connecting homomorphism'' type results for the Adams spectral sequence, as in \cite{ravenel_complex_cobordism}[\S2.3].
\end{remark}

\begin{remark}
\label{remark:up_to_graded_the_adams_filtration_is_d\'{e}calage}
Observe that if one is only interested in constructing the filtered spectrum $F^{H}(d, c)$ for a fixed pair of $c, d \in \ccat$, then the proof of \cref{theorem:properties_of_the_h_adams_filtration} provides an elementary recipe which gives the same result up to completion. Namely, one can choose a cosimplicial $H$-Adams resolution
\[
c \rightarrow i_{0} \rightrightarrows i_{1} \triplerightarrow \ldots
\]
and consider the d\'{e}calage
\[
n \mapsto \Tot(F(d, i_{\bullet})_{\geq n}).
\]

However, from this approach it is virtually impossible to lift $F^{H}$ to a functor of $\infty$-categories, as there are too many arbitrary choices. This highlights the intuition that the perfect derived $\infty$-category $\dcat^{\omega}(\ccat)$ is as a place where one can discuss $H$-Adams resolutions without encountering these types of issues. 
\end{remark}

\subsection{Digression: Monoidality of the Adams filtration}
\label{subsection:monoidality_of_the_adams_filtration}

In \cref{definition:adams_filtered_spectrum}, we used derived $\infty$-categories to associate to any adapted $H: \ccat \rightarrow \acat$ a functorial filtration on the internal mapping spectrum $F(d, c)$ which describes the $H$-based Adams spectral sequence. 

In this section, we will extend this result by showing that in the classical setting of spectra and a homology theory based on $R$-split maps, where $R$ is a ring spectrum, formation of this filtration defines a symmetric monoidal functor. 

In fact, in what is perhaps the most surprising, we will not need $R$ to be a ring spectrum of any kind, but only satisfy the following weak condition. 

\begin{definition}
We say $R \in \spectra$ \emph{admits a right unital multiplication} if there exist maps $u: S^{0} \rightarrow R$ and $m\colon R \otimes R \rightarrow R$ such that the composite
\[
\begin{tikzcd}
	R & {R \otimes R} & R
	\arrow["{R \otimes u}", from=1-1, to=1-2]
	\arrow["m", from=1-2, to=1-3]
\end{tikzcd}
\]
is homotopic to the identity. 
\end{definition}
We will show below, in \cref{lemma:spectra_have_enough_r_split_injectives_if_r_has_right_unital_multiplication}, that if $R$ admit a right unital multiplication, then if $\ecat$ denotes the class of $R$-split maps; that is, those $X \rightarrow Y$ such that $R \otimes X \rightarrow R \otimes Y$ admits a section, then $\spectra$ has enough $\ecat$-injectives and thus we obtain an Adams spectral sequence. 

In fact, in this case this spectral sequence is completely classical: for any $X \in \spectra$, the coaugmented semicosimplicial object 
\begin{equation}
\label{equation:semicosimplicial_resolution_based_on_spectrum_r}
X \rightarrow R \otimes X \rightrightarrows R \otimes R \otimes X \triplerightarrow \ldots
\end{equation}
is an $R$-Adams resolution of $X$. Thus, the $R$-based Adams spectral sequence is obtained by mapping any other $Y \in \spectra$ into it and studying the totalization spectral sequence of the corresponding semicosimplicial object. 

If $X$ is a spectrum, then \cref{definition:adams_filtered_spectrum} provides a filtration on $F(S^{0}, X) \simeq X$ which up to completion we can identify with the d\'{e}calage of (\ref{equation:semicosimplicial_resolution_based_on_spectrum_r}). The following surprising result shows that this filtration is compatible with the symmetric monoidal structure on spectra. 

\begin{theorem}[Monoidality of the classical Adams filtration]
\label{theorem:monoidality_of_the_classical_adams_filtration}
Let $R$ be a spectrum which admits a right unital multiplication. Then, the $R$-Adams filtration functor 
\[
F^{R} \colonequals F^{R}(S^{0}, -)\colon \spectra \rightarrow \Fil(\spectra)
\]
is canonically lax symmetric monoidal, where we consider the target with the Day convolution symmetric monoidal structure. In particular, if $X$ is an $\mathbf{E}_{n}$-algebra in spectra, then $F^{R}(X)$ is a filtered $\mathbf{E}_{n}$-algebra. 
\end{theorem}

\begin{remark}
In the particular case where $R$ is an Adams-type ring spectrum, \cref{theorem:monoidality_of_the_classical_adams_filtration} is a consequence of the construction of synthetic spectra \cite{pstrkagowski2018synthetic}. The reason it is important was explained to us in many conversations with Robert Burklund, Jeremy Hahn and Andy Senger, who encouraged us to write this result down in the general case.
\end{remark}

Note that \cref{theorem:monoidality_of_the_classical_adams_filtration} provides numerous examples of filtrations of ring spectra, in fact one for each $R$ with a right unital multiplication. To understand why producing such filtrations is difficult, let us discuss the classical approach to doing so. 

For example, the complex bordism spectrum admits an $\mathbf{E}_{\infty}$-ring structure, so that its Adams resolution
\begin{equation}
\label{equation:mu_adams_resolution_of_s0}
S^{0} \rightarrow MU \rightarrow MU \otimes MU \triplerightarrow \ldots
\end{equation}
is a cosimplicial object in $\mathbf{E}_{\infty}$-rings. By passing to Deligne's d\'{e}calage, we obtain an Adams filtration of the sphere which is an algebra in filtered spectra, as in the work of Gheorghe, Ricka, Krause and Isaksen \cite{gheorghe2018c}.

The problem with constructing a commutative monoidal structure on the Adams filtration of the sphere using (\ref{equation:mu_adams_resolution_of_s0}) is that it requires complex bordism to be $\mathbf{E}_{\infty}$ to begin with. This stands in large contrast with the Adams filtration itself, which does not depend on this monoidal structure or even $MU$ itself, but only on the class of $MU$-split maps. In many interesting cases, such as those of Morava $K$-theories, even $\mathbf{E}_{2}$-structures are already known not to exist, and so this method cannot produce $\mathbf{E}_{2}$-monoidal $K$-Adams filtrations. 

On the other hand, \cref{theorem:monoidality_of_the_classical_adams_filtration} requires virtually no assumptions on $R$. For example, it applies to $S^{0}/p$ at odd primes, which is known not to carry a structure of an $\mathbf{E}_{1}$-ring, but which is known to admit a right unital multiplication. 

\begin{remark}
\label{remark:monoidality_of_adams_filtration_analogous_to_bousfield_localization}
The surprising monoidality of the Adams filtration of \cref{theorem:monoidality_of_the_classical_adams_filtration} can be compared to the analogous property of the Bousfield localization. 

To see this, suppose for simplicity that $R, S$ are $\mathbf{E}_{1}$-ring spectra and that the $S$-based Adams spectral sequence converges to the Bousfield localization of $R$. In this case, we have a cosimplicial diagram 
\[
R \rightarrow S \otimes R \rightrightarrows S \otimes S \otimes R \triplerightarrow \ldots
\]
which induces an equivalence
\[
R_{S} \simeq \Tot(S^{\otimes \bullet} \otimes R),
\]
furnishing the Bousfield localization with its own $\mathbf{E}_{1}$-ring structure. However, it is well-known that $R_{S}$ admits an $\mathbf{E}_{1}$-structure as soon as $R$ does, irrespective of any structure on $S$, as the Bousfield localization $L\colon \spectra \rightarrow \spectra_{S}$ is symmetric monoidal. 

The situation of \cref{theorem:monoidality_of_the_classical_adams_filtration} is similar: to define the $R$-Adams spectral sequence, we do not need any information except which maps become split after applying $R \otimes -$, and so any further structure on $R$ should not be needed.
\end{remark}

The rest of this section will be devoted to the proof of \cref{theorem:monoidality_of_the_classical_adams_filtration}. We begin with the promised assertion that if $R$ admits a right unital multiplication, then we do indeed have an Adams spectral sequence. 

\begin{lemma}
\label{lemma:spectra_have_enough_r_split_injectives_if_r_has_right_unital_multiplication}
Let $\ecat$ be the class of $R$-split maps; that is, those $X \rightarrow Y$ such that $R \otimes X \rightarrow R \otimes Y$ admits a section. If $R$ admits a right unital multiplication, then $\spectra$ has enough $\ecat$-injectives. 
\end{lemma}

\begin{proof}
We claim that for any $Z \in \spectra$, the unit map $u \otimes Z\colon Z \rightarrow R \otimes Z$ is an $\ecat$-monic map into an $\ecat$-injective. To show that it is $\ecat$-monic, we have to verify that 
\[
R \otimes u \otimes Z\colon R \otimes Z \rightarrow R \otimes R \otimes Z 
\]
is split monic. This is clear, as it has a one-sided inverse given by $m \otimes Z\colon R \otimes R \otimes Z \rightarrow R \otimes Z$.

To verify that $R \otimes Z$ is an $\ecat$-injective, we have to check that if $f\colon X \rightarrow Y$ is a map of spectra such that $R \otimes f\colon R \otimes X \rightarrow R \otimes Y$ admits a one-sided inverse $\pi\colon R \otimes Y \rightarrow R \otimes X$, then any map $s\colon X \rightarrow R \otimes Z$ factors through $Y$. Consider the commutative diagram 
\[
\begin{tikzcd}
	X & Y \\
	{R \otimes X} & {R \otimes Y} & {R \otimes X} & {R  \otimes R \otimes Z} & {R \otimes R \otimes Z} & {R \otimes Z}
	\arrow["{u \otimes X}"', from=1-1, to=2-1]
	\arrow["f", from=1-1, to=1-2]
	\arrow["{R \otimes f}", from=2-1, to=2-2]
	\arrow["{u \otimes Y}", from=1-2, to=2-2]
	\arrow["\pi", from=2-2, to=2-3]
	\arrow["{t \otimes Z}", from=2-4, to=2-5]
	\arrow["{m \otimes Z}", from=2-5, to=2-6]
	\arrow["{R \otimes s}", from=2-3, to=2-4]
\end{tikzcd},
\]
where $t\colon R \otimes R \rightarrow R \otimes R$ is the twist map. Since $R \otimes f$ is a section of $\pi$ and $m$ is right unital, we see that the long composite $X \rightarrow R \otimes Z$ can be identified with $s$. It follows that the composite $Y \rightarrow R \otimes Z$ is the needed factorization. 
\end{proof}
For the rest of the section, we will assume that $R$ a spectrum which admits a right unital multiplication has been fixed. Our arguments are inspired by the case of Bousfield localization mentioned in  \cref{remark:monoidality_of_adams_filtration_analogous_to_bousfield_localization}, we want to express the Adams filtration as arising from a certain exact localization compatible with the symmetric monoidal structure. In this case, the localization will not be of the $\infty$-category of spectra itself, but of its perfect prestable Freyd envelope.

Recall that $\dcat^{\omega}(\spectra)$ is, by definition, given by the $\infty$-category of perfect sheaves on $\spectra$ with respect to the Grothendieck pretopology in which coverings are given by a single $R$-split epimorphic map. Our first goal is to equip this $\infty$-category with a symmetric monoidal structure induced from spectra. This requires the following property special to the class of $R$-split maps. 

\begin{lemma}
\label{lemma:r_split_pretopology_compatible_with_monoidal_structure}
The Grothendieck pretopology of $R$-split maps is compatible with the tensor product of spectra; that is, if $X \rightarrow Y$ is $R$-split, then so is $X \otimes A \rightarrow Y \otimes A$ for any $A \in \spectra$. 
\end{lemma}

\begin{proof}
This is clear, as if $s\colon R \otimes Y \rightarrow R \otimes X$ is a section of $R \otimes X \rightarrow R \otimes Y$, then $s \otimes A$ is a section of $R \otimes X \otimes A \rightarrow R \otimes Y \otimes A$. 
\end{proof}

\begin{proposition}
\label{proposition:perfect_derived_infty_cat_of_r_split_maps_acquires_a_monoidal_structure}
The perfect derived $\infty$-category $\dcat^{\omega}(\spectra)$ admits a unique symmetric monoidal structure which preserves finite colimits in each variable and such that $\nu\colon \spectra \rightarrow \dcat^{\omega}(\spectra)$ is symmetric monoidal. 
\end{proposition}

\begin{proof}
We can identify $\dcat^{\omega}(\spectra)$ with the smallest  subcategory $Sh_{\Sigma}(\spectra, \widehat{\spaces})$ of sheaves of possibly large spaces which contains representables and is closed under finite colimits. We claim that the larger $\infty$-category admits a unique symmetric monoidal structure compatible with $\nu$ and which preserves large colimits in each variable, this will imply what we need as the symmetric monoidal structure will restrict to $\dcat^{\omega}(\spectra)$. 

The $\infty$-category $Sh_{\Sigma}(\spectra, \widehat{\spaces})$ is a localization of the $\infty$-category $P(\spectra, \widehat{\spaces})$ of presheaves of large spaces. The latter is presentable with respect to a larger universe and thus has a unique Day convolution symmetric monoidal structure extending that of $\spectra$ which preserves large colimits in each variable \cite{higher_algebra}[4.8.1]. 

Since the tensor product of spectra is additive, to check that Day convolution descends to product-preserving sheaves, it is enough to verify that that for any $A \in \spectra$, $\nu(A) \otimes_{Day} -$ takes the \v{C}ech nerve 
\[
\ldots \triplerightarrow \nu(Y \times_{X} Y) \rightrightarrows \nu(Y) \rightarrow \nu(X) 
\]
of any $R$-split $Y \rightarrow X$ to a diagram which becomes a colimit in sheaves. Additivity is clear, and since Day convolution is compatible with $\nu$ the image of the above nerve is 
\[
\ldots \triplerightarrow \nu(A \otimes (Y \times_{X} Y)) \rightrightarrows \nu(A \otimes Y) \rightarrow \nu(A \otimes X). 
\]
As $A \otimes -$ is exact, this is the \v{C}ech nerve of $A \otimes X \rightarrow A \otimes Y$ which is itself $R$-split by \cref{lemma:r_split_pretopology_compatible_with_monoidal_structure} and hence becomes a colimit in sheaves. 
\end{proof}

As the induced symmetric monoidal structure of \cref{proposition:perfect_derived_infty_cat_of_r_split_maps_acquires_a_monoidal_structure} preserves finite colimits in each variable, it extends formally to a symmetric monoidal structure on the bounded below derived $\infty$-category $\dcat^{\omega}_{-}(\spectra)$ of \cref{definition:spanier_whitehead_cat_of_d_is_d_minus}, exact in each variable and compatible with the $t$-structure. The monoidal unit is given by $\nu(S^{0})$. 

We are now ready to prove the main result of this section.

\begin{proof}[Proof of \cref{theorem:monoidality_of_the_classical_adams_filtration}:]
Recall from \cref{definition:adams_filtered_spectrum} that the $R$-Adams filtration of a spectrum $X$ is given by 
\[
\ldots \rightarrow F(\nu(S^{0}), \mathrm{Post}_{1}(X)) \rightarrow F(\nu(S^{0}), \mathrm{Post}_{0} (X)) \rightarrow F(\nu(S^{0}), \mathrm{Post}_{-1}(X)) \rightarrow \ldots
\]
where $\mathrm{Post}_{n}(X) \simeq \Sigma^{n} \nu(X)[-n]$, with connecting maps given by $\tau$. 

As in \cref{remark:stabilized_dcat_omega_as_sheaves_of_spectra}, we can identify $\dcat^{\omega}_{-}(\spectra)$ with a subcategory of product-preserving sheaves of spectra. We first claim that the functor
\[
\mathrm{Post}_{\star}\colon \spectra \rightarrow \Fil(\dcat^{\omega}_{-}(\spectra)) \rightarrow \Fil (Sh_{\Sigma}(\spectra, \widehat{\spectra}))
\]
is lax symmetric monoidal with respect to the symmetric monoidal structure of \cref{proposition:perfect_derived_infty_cat_of_r_split_maps_acquires_a_monoidal_structure}. As constructed in the proof of \cref{proposition:existence_of_double_sided_postnikov_towers}, in terms of sheaves of spectra $\mathrm{Post}_{\star}(X)$ is given by the Postnikov tower of the sheaf $Y(X)$ defined by 
\[
X^{\prime} \mapsto F(X^{\prime}, X).
\]
More precisely, $\mathrm{Post}_{n}(X)$ can be identified with the sheafification of the presheaf 
\[
X^{\prime} \mapsto F(X^{\prime}, X)_{\geq n},
\]
which is the $n$-connective cover of $Y(X)$ in sheaves. 

Thus, $\mathrm{Post}_{\star}$ can be described as a composite of three operations:
\begin{enumerate}
    \item form the spectral sheaf $Y(X)(-) \colonequals F(-, X)$,
    \item take the constant filtered object 
    \[
    \ldots \rightarrow Y(X) \rightarrow Y(X) \rightarrow Y(X) \rightarrow \ldots
    \]
    \item take the connective cover in the $t$-structure on $\Fil(Sh_{\Sigma}(\spectra, \widehat{\spectra}))$ in which $(F_{n})_{n \in \mathbb{Z}}$ is connective if each $F_{n}$ is $n$-connective as a sheaf of (large) spectra.
\end{enumerate}
We claim each of these operations is lax symmetric monoidal. 

Note that the last $t$-structure exists by \cite{higher_algebra}[1.4.4.11], which applies to $Sh_{\Sigma}(\spectra, \widehat{\spectra})$ as it is presentable with respect to the larger universe. The Day convolution symmetric monoidal structure is clearly compatible with this $t$-structure, so that $(3)$ is lax symmetric monoidal, as needed.

Let us move to $(2)$. Taking the constant diagram is right adjoint to the colimit functor 
\[
\varinjlim\colon \Fil(Sh_{\Sigma}(\spectra, \widehat{\spectra})) \rightarrow Sh_{\Sigma}(\spectra, \widehat{\spectra}).
\]
The latter is even strongly symmetric monoidal, and we deduce that its right adjoint is lax symmetric monoidal, as needed. 

For $(1)$, observe that $Y(X)$ is defined even if $X$ is a large spectrum, so that we have a functor 
\[
Y(X)(-) \colonequals F(-, X)\colon \widehat{\spectra} \rightarrow Sh_{\Sigma}(\spectra, \widehat{\spectra})
\]
This is a right adjoint, with left adjoint given by the left Kan extension 
\[
Sh_{\Sigma}(\spectra, \widehat{\spectra}) \rightarrow \widehat{\spectra}
\]
along the inclusion $\spectra \hookrightarrow \widehat{\spectra}$. As the latter is strongly symmetric monoidal, so is the left Kan extension, and we deduce formally that $Y$ is lax symmetric monoidal. 

Finally, by the Yoneda lemma applying $F(\nu(S^{0}
), -)$ to a spectral sheaf is the same as evaluating at $S^{0}$. We deduce that the $R$-Adams filtration functor is the composite 
\[
\spectra \rightarrow \Fil(Sh_{\Sigma}(\spectra, \widehat{\spectra})) \rightarrow \Fil(\widehat{\spectra})
\]
of the $\mathrm{Post}_{\star}$ functor followed by levelwise evaluation of a sheaf at $S^{0} \in \spectra$ (note that the composite actually lands in filtered spectra rather than filtered large spectra, but the second functor does not). The latter functor is obtained by applying the Day convolution to the evaluation at $S^{0}$ functor
\[
Sh_{\Sigma}(\spectra, \widehat{\spectra}) \rightarrow \widehat{\spectra}.
\]
The latter is right adjoint to the unique strongly symmetric monoidal functor $\widehat{\spectra} \rightarrow Sh_{\Sigma}(\spectra, \widehat{\spectra})$ which preserves even large colimits and exists by the universal property of the source. We deduce the evaluation is lax symmetric monoidal, ending the argument. 
\end{proof}

\section{Further topics in derived $\infty$-categories}
\label{section:further_topics_in_derived_infty_cats}

In this chapter, we study deeper topics in the subject of derived $\infty$-categories associated to homology theories, such as their completions and the Goerss-Hopkins tower. We also describe variants of derived $\infty$-categories which are Grothendieck prestable and compare our constructions to those already occurring in the literature. 

\subsection{Completing derived $\infty$-categories}
\label{subsection:completing_derived_infty_cats}

The perfect derived $\infty$-category of $H\colon \ccat \rightarrow \acat$ of \cref{definition:perfect_derived_cat_of_stable_cat} can be sometimes a little inconvenient; for example, even if $\ccat$ is cocomplete, $\dcat^{\omega}(\ccat)$ will in general only admit finite direct sums. The issue does not lie in the sheaf condition, as  \cref{lemma:detecting_sheaves_in_additive_setting} implies that product-preserving sheaves on $\ccat$ are stable under arbitrary direct sums, but rather in the perfectness assumption. 

The issue is easier to observe in the abelian context, where perfectness of an almost perfect presheaf is equivalent to being bounded: an infinite direct sum of bounded presheaves will in general not be bounded anymore.  This is in some sense the only obstruction, one can show that for any $n \geq 0$, the subcategory $\tau_{\leq n}\dcat^{\omega}(\ccat)$ of $n$-truncated object does admit all direct sums which exist in $\ccat$, as a consequence of \cref{proposition:nu_preserves_direct_sums}. 

Thus, one way to endow the derived $\infty$-category with better categorical properties is to consider the prestable completion, which is defined as the limit 
\[
\widehat{\dcat}(\ccat) \colonequals \varprojlim \tau_{\leq n} \dcat^{\omega}(\ccat). 
\]
of $\infty$-categories; this is again prestable by an application of \cite{higher_algebra}[1.2.1.17]. 

A serious drawback to using a limit definition of the complete derived $\infty$-category is that its objects can then be identified with certain diagrams of sheaves, rather than sheaves themselves. This is often cumbersome in practice. In this section, we will show that under certain homological finiteness conditions on one can give an explicit sheaf-theoretic description of the complete derived $\infty$-category. 

The discussion above applies also to the abelian case, with which we will begin as it is the more familiar one. The needed finiteness assumptions are as follows. 

\begin{definition}
\label{definition:homologically_finitary_abelian_category}
We will say an abelian category $\acat$ with enough injectives is \emph{homologically finitary} if either
\begin{enumerate}
    \item $\acat$ is of finite cohomological dimension or 
    \item $\acat$ admits countable products and there exists a $d < \infty$ such that each object of $\acat$ is a quotient of an object $a$ of projective dimension at most $d$; that is, such that $\Ext^{s}_{\acat}(a, b) = 0$ for $s > d$ and any $b \in \acat$. 
\end{enumerate}
\end{definition}

\begin{remark}
In the resolution of Franke's conjecture, we will only be using the case where $\acat$ is of finite cohomological dimension. However, this is a very strong assumption, and in the interest of future applications we will show that the second condition also turns out to be sufficient. A natural example of $\acat$ satisfying the second condition is an abelian category with countable products and enough projectives. 
\end{remark}

\begin{remark}
The notion of homologically finitary will not be used outside of this subsection, it is only introduced to avoid repeating the same assumptions.
\end{remark}
Since we want the resulting derived $\infty$-category to be complete, in particular it should be separated;  that is, have the property that its internal homotopy groups detect equivalences. This is usually not the case for unbounded sheaves, and so we will instead work with the \emph{hypercomplete sheafification} functor $\widehat{L}\colon \Fun_{\Sigma}(\acat^{op}, \widehat{\spaces}) \rightarrow \Fun_{\Sigma}(\acat^{op}, \widehat{\spaces})$ on product-preserving presheaves of large spaces. 

\begin{remark}
For the discussion on the difference between hypercomplete sheaves and sheaves, see \cite{higher_topos_theory}[6.5.4]. Informally, hypercomplete sheaves are those for which sheaf homotopy groups detect equivalences. This is analogous to the property of the classical derived $\infty$-category that all quasi-isomorphisms are equivalences.
\end{remark}

\begin{proposition}
\label{proposition:sheafication_exists_in_abelian_context_when_bounded_or_finite_dimension}
Let $\acat$ be a homologically finitary abelian and suppose that $X \in A_{\infty}(\acat)$ is almost perfect. Then, the hypercomplete sheafification $\widehat{L}X$ is again almost perfect. 
\end{proposition}

\begin{proof}
First suppose that $\acat$ is of finite cohomological dimension $d$. Consider the fibre sequences
\[
X_{\leq n} \rightarrow X_{\leq n-1} \rightarrow \Sigma^{n+1} \pi_{n}X,
\]
as these presheaves are bounded, their usual sheafification coincides with the hypercomplete one. Thus, if we apply $\widehat{L}$, the resulting long exact sequence of homotopy groups together with the calculation of homotopy of $\widehat{L}(\Sigma^{n+1} \pi_{n}X)$ of  \cref{proposition:sheafication_of_suspended_representable_in_freyd_envelope_of_abelian_cat} shows that $\pi_{k} \varprojlim \widehat{L} X_{\leq n} \simeq \pi_{k} \widehat{L}\tau_{\leq m} X$ for $m > k+d$. It follows that the natural map 
\[
X \rightarrow \varprojlim \widehat{L}X_{\leq n}
\]
is an isomorphism on sheafified homotopy groups, as $X \rightarrow \widehat{L}X_{\leq m}$ has this property in degrees $k \leq m$.  As the target is a limit of hypercomplete sheaves, it is hypercomplete and we deduce that $\widehat{L}X \simeq \varprojlim \widehat{L} X_{\leq n}$. As each of $\widehat{L}X_{\leq n}$ has finitely presented homotopy groups by \cref{corollary:sheafication_exists_for_bounded_sheaves_in_abelian_case} by the case done above, we deduce that so does $\widehat{L}X$, so that it is almost perfect by  \cref{proposition:characterization_of_almost_perfect_presheaves}. 

Alternatively, suppose that $\acat$ admits countable products and that its every object is a quotient of one lying in the subcategory $\acat^{pd \leq d}$ of objects of projective dimension at most $d$. Observe that the latter is closed under pullbacks along epimorphisms and so inherits a Grothendieck pretopology from $\acat$; moreover, the quotient conditions guarantees that the restriction $i_{*}$ along the inclusion $i\colon \acat^{pd \leq d} \hookrightarrow \acat$ induces an equivalence on the categories of discrete sheaves \cite{lurie2018ultracategories}[B.6.4]. 

We claim that again we have $\widehat{L}X \simeq \varprojlim \widehat{L} X_{\leq n}$. As sheaf homotopy groups detect equivalences of hypercomplete sheaves, it is enough to check that $i_{*} \widehat{L}X \simeq \varprojlim i_{*} \widehat{L} X_{\leq n}$ is an equivalence. This is the same as saying that \[
\widehat{L}X(p) \simeq \varprojlim (\widehat{L} X_{\leq n})(p)
\]
for any $p \in \acat^{pd \leq d}$, which follows from the argument given above in the finite cohomological dimension case. 

Since $\acat$ has countable products, almost perfect presheaves are closed under countable limits by \cref{proposition:products_in_prestable_freyd_envelope}, and we deduce that $\widehat{L}X \simeq \varprojlim \widehat{L} X_{\leq n}$ is almost perfect as needed. 
\end{proof}

\begin{remark}
One might wonder if working with hypercomplete sheafification is necessary in \cref{proposition:sheafication_exists_in_abelian_context_when_bounded_or_finite_dimension}; that is, whether the ordinary sheafification $L$ also preserves almost perfect presheaves. This is the case when $\acat$ has finite cohomological dimension, but not otherwise. 
\end{remark}
We are now ready to define the completed derived $\infty$-category. 

\begin{definition}
\label{definition:complete_derived_infty_cat}
Let $\acat$ be homologically finitary abelian category. The \emph{complete derived $\infty$-category} is given by 
\[
\widehat{\dcat}(\acat) \colonequals A_{\infty}^{hyper}(\acat),
\]
the $\infty$-category of almost perfect hypercomplete sheaves with respect to the epimorphism topology. 
\end{definition}
Note that by \cref{proposition:sheafication_exists_in_abelian_context_when_bounded_or_finite_dimension}, $\widehat{\dcat}(\acat)$ is an exact localization of $A_{\infty}(\acat)$; in particular, it is itself prestable with finite limits. The terminology is motivated by the following. 

\begin{proposition}
\label{proposition:under_finite_dimension_derived_category_already_complete}
Let $\acat$ be a homologically finitary abelian category. Then, the the canonical inclusion $\dcat^{b}(\acat) \rightarrow \widehat{\dcat}(\acat)$ presents the complete derived $\infty$-category as a prestable completion of the bounded one. 
\end{proposition}

\begin{proof}
Observe that since any bounded almost perfect sheaf is automatically hypercomplete and finite, the inclusion induces an equivalence $\tau_{\leq n} \dcat^{b}(\acat) \rightarrow \tau_{\leq n} \widehat{\dcat}(\acat)$ for every $n \geq 0$. Thus, we only have to verify that the target is indeed complete. 

We first cover the case when $\acat$ is of finite cohomological dimension using the criterion of \cite{higher_topos_theory}[5.5.6.26]. Let $(X_{n})_{n \geq 0} \in \varprojlim \tau_{\leq n} \widehat{\dcat}(\acat)$, be a Postnikov pretower, which we can informally identify with a family of $X_{n}$ of $n$-truncated sheaves together with maps $X_{n} \rightarrow X_{n-1}$ which identify the target as the $(n-1)$-th truncation of the source. We will show that if $X \colonequals \varprojlim X_{n}$ is the limit in the $\infty$-category of all hypercomplete sheaves, then
\begin{enumerate}
    \item $X$ is almost perfect and hence a limit in $\widehat{\dcat}(\acat)$ and 
    \item the canonical maps $X \rightarrow X_{n}$ present the target as the $n$-truncation.
\end{enumerate}
It follows that $X$ extends it to a limit Postnikov tower, as needed. 

It is enough to show that for any fixed $k$, the map $\pi_{k} X \rightarrow \pi_{k} X_n$ on sheaf homotopy groups is an isomorphism for $n$ large enough. Indeed, if this is the case, then the same will be true for the composite 
\[
\pi_{k} X \rightarrow \pi_{k} X_n \rightarrow \pi_{k} X_m
\]
for every $k \leq m \leq n$, as the latter map also has this property.
 
Suppose that $\acat$ is of cohomological dimension $d$. Considering Postnikov fibre sequences as in the proof of \cref{proposition:sheafication_exists_in_abelian_context_when_bounded_or_finite_dimension}, the calculations of  \cref{proposition:sheafication_of_suspended_representable_in_freyd_envelope_of_abelian_cat} show that $\pi_{k} X_{n+1} \simeq \pi_{k} X_n$ for $n > k+d$, as the relevant $\Ext$-groups vanish. This readily implies that $\pi_{k} X \simeq \pi_{k} X_{n}$ for $n > k+d$, even before sheafification, as claimed. It follows that $X$ has finitely presented homotopy groups and hence is almost perfect, ending the argument when $\acat$ is of finite cohomological dimension.

The case where $\acat$ has countable products and every object is a quotient of one in $\acat^{pd \leq d}$ follows from the argument given above by restricting to this subcategory as in the last part of the proof of \cref{proposition:sheafication_exists_in_abelian_context_when_bounded_or_finite_dimension}.
\end{proof}

We now move on to the case of a stable $\infty$-category together with a choice of a homology theory, this will be largely similar to the one done above. Again, we will need some finiteness assumptions. 

\begin{definition}
\label{definition:homologically_finitary_adapted_homology_theory}
We say an adapted homology theory $H\colon \ccat \rightarrow \acat$ is \emph{homologically finitary} if either: 

\begin{enumerate}
    \item $\acat$ is of finite cohomological dimension or 
    \item $\ccat$ admits countable products and there exists a $d$ such that for every $c \in \ccat$ there exists an $H$-epimorphism $c^{\prime} \rightarrow c$ with $H(c^{\prime})$ of projective dimension at most $d$. 
\end{enumerate}
\end{definition}
Note that this is compatible with the previously introduced notion in the following sense. 
\begin{lemma}
If $H\colon \ccat \rightarrow \acat$ is homologically finitary in the sense of \cref{definition:homologically_finitary_adapted_homology_theory}, then $\acat$ is homologically finitary in the sense of \cref{definition:homologically_finitary_abelian_category}.
\end{lemma}

\begin{proof}
This is clear if $\acat$ is of finite cohomological dimension, so assume that instead $\ccat$ satisfies the second condition. In this case, $A(\ccat)$ also has countable products by \cref{lemma:freyd_envelope_has_products_when_c_does}, and so does $\acat$ as it is a localization of the Freyd envelope by adaptedness. Moreover, since any object of $A(\ccat)$ is a quotient of a representable, we deduce the same is true in $\acat$, so that any object therein is a quotient of one of the form $H(c^{\prime})$ of projective dimension at most $d$. 
\end{proof}

We will denote the relevant hypercomplete sheafification on the $\infty$-category $\Fun_{\Sigma}(\ccat^{op}, \widehat{\spaces})$ by $\widehat{L}_{\ccat}$, to distinguish it from the one we used in the abelian case. The following is the stable analogue of  \cref{proposition:sheafication_exists_in_abelian_context_when_bounded_or_finite_dimension}. 

\begin{proposition}
\label{proposition:in_homologically_finitary_stable_case_hp_sheafication_preserves_ap_presheaves}
Let $H\colon \ccat \rightarrow \acat$ be a homologically finitary adapted homology theory. Then, the hypercomplete sheafification $\widehat{L}_{\ccat}X$ of an almost perfect presheaf $X \in A_{\infty}(\ccat)$ is again almost perfect.
\end{proposition}

\begin{proof}
First, assume that $\acat$ is of finite cohomological dimension. For any $n \geq 1$, we have a fibre sequence 
\[
\widehat{L}_{\ccat}(X_{\leq n}) \rightarrow \widehat{L}_{\ccat}(X_{\leq n-1}) \rightarrow \widehat{L}_{\ccat}(\Sigma^{n+1} \pi_{n} X).
\]
The presheaf $\Sigma^{n+1} \pi_{n} X$ has homotopy concentrated in a single degree, so that by \cref{corollary:presheaves_with_single_homotopy_group_can_be_sheafified_in_abelian_cat} we can write
\[
\widehat{L}_{\ccat}(\Sigma^{n+1} \pi_{n} X) \simeq H_{*}(\widehat{L}_{\acat}(\Sigma^{n+1} y)),
\]
where $y \in A(\acat)$ is a sheaf with respect to the epimorphism topology, which we can identify with an object of $\acat$ by \cref{lemma:yoneda_embedding_of_abelian_category_identifies_with_sheaves}. It then follows from the calculation given in \cref{proposition:sheafication_of_suspended_representable_in_freyd_envelope_of_abelian_cat} that 
\[
\pi_{k} (\widehat{L}_{\ccat}(\Sigma^{n+1} \pi_{n} X))(c) \simeq \Ext^{n+1-k}_{\acat}(H(c), y).
\]
for any $c \in \ccat$. In particular, if $\acat$ is of cohomological dimension $d$, then these groups vanish below degree $n+1-d$ and we deduce that 
\[
\pi_{k} \widehat{L}_{\ccat}(X_{\leq n}) \rightarrow \pi_{k} \widehat{L}_{\ccat}(X_{\leq n-1})
\]
is an isomorphism in degrees $k < n-d$. It then follows as in the proof of \cref{proposition:sheafication_exists_in_abelian_context_when_bounded_or_finite_dimension} that $\widehat{L}_{\ccat}X \simeq \varprojlim \widehat{L}_{\ccat}X_{\leq n}$ and that 
\[
\pi_{k} \widehat{L}_{\ccat}X \simeq \pi_{k} \widehat{L}_{\ccat} X_{\leq n}
\]
is an isomorphism for $n > k+d$. As each $\widehat{L}_{\ccat} X_{\leq n}$ is almost perfect by \cref{proposition:bounded_almost_perfect_presheaves_have_sheafication}, we deduce that $\widehat{L}_{\ccat}X$ is almost perfect by \cref{proposition:characterization_of_almost_perfect_presheaves}, as it has finitely presented homotopy groups. 

Let us now consider the case where $\ccat$ admits countable products and every object admits an $H$-epimorphism from one in the subcategory $\ccat^{pd \leq d}$ of those $c \in \ccat$ such that $H(c)$ is of projective dimension at most $c$. This subcategory is closed under pullbacks along $H$-epimorphisms, and so it inherits a Grothendieck topology from $\ccat$. Restricting our sheaves along the inclusion $i\colon \ccat^{pd \leq d} \hookrightarrow \ccat$ as in the last part of \cref{proposition:sheafication_exists_in_abelian_context_when_bounded_or_finite_dimension} we reduce to the case done above, ending the argument.  
\end{proof}

\begin{definition}
\label{definition:complete_derived_infty_cat_of_an_homologically_finitary_adapted_homology_theory}
Let $H\colon \ccat \rightarrow \acat$ be a homologically finitary adapted homology theory. Then, its \emph{complete derived $\infty$-category} is the $\infty$-category 
\[
\widehat{\dcat}(\ccat) \colonequals A_{\infty}^{hsh}(\ccat)
\]
of almost perfect, hypercomplete sheaves on $\ccat$ with respect to the $H$-epimorphism topology. 
\end{definition}

\begin{remark}
As we have observed in an analogous situation above, by  \cref{proposition:in_homologically_finitary_stable_case_hp_sheafication_preserves_ap_presheaves} the $\infty$-category $\widehat{\dcat}(\ccat)$ is an exact localization of $A_{\infty}(\ccat)$, in particular it is prestable with finite limits. Moreover, by construction we have $\widehat{\dcat}(\ccat)^{\heartsuit} \simeq \acat$, as in the finite case, since every discrete sheaf is automatically hypercomplete. 
\end{remark} 
One major difference between the abelian and stable cases is that in the latter a perfect sheaf is not necessarily bounded, and so need not be hypercomplete. In fact, even the synthetic analogue $\nu(c)$ can fail to be hypercomplete, which is related to Bousfield localization, as we discuss at more length in \S\ref{subsection:hypercompletion_and_locality} below.

However, if there is still a natural functor $\dcat^{\omega}(\ccat) \rightarrow \widehat{\dcat}(\ccat)$ which sends any perfect sheaf to its hypercomplete sheafification. The following is the stable analogue of \cref{proposition:under_finite_dimension_derived_category_already_complete}, and is proven in the same way. 

\begin{proposition}
\label{proposition:complete_derived_cat_of_stable_cat_a_completion_of_finite_one}
Let $H\colon \ccat \rightarrow \acat$ be a homologically finitary adapted homology theory. Then, the hypercompletion functor $\dcat^{\omega}(\ccat) \rightarrow \widehat{\dcat}(\ccat)$ presents the target as the prestable completion of the source.
\end{proposition}

\begin{remark}
As a consequence of \cref{proposition:complete_derived_cat_of_stable_cat_a_completion_of_finite_one} and the universal property of $\dcat^{\omega}(\ccat)$ proven in \cref{theorem:universal_property_of_finite_derived_category}, the complete derived $\infty$-category also has a universal property. Namely, the composite 
\[
\ccat \rightarrow \dcat^{\omega}(\ccat) \rightarrow \widehat{\dcat}(\ccat)
\]
is a prestable enhancement to $H\colon \ccat \rightarrow \acat$, and it is universal among all such \emph{valued in a complete prestable $\infty$-category}. 
\end{remark}

\subsection{Digression: Hypercompletion and locality} 
\label{subsection:hypercompletion_and_locality}

The distinction between hypercomplete and non-hypercomplete sheaves is more subtle in the stable context than the abelian one. In the latter, any perfect sheaf is automatically bounded, and so is hypercomplete. In the stable case, this is not what happens; the synthetic analogues $\nu(c)$ are perfect and are sheaves, but are not always hypercomplete, as the following example shows. 

\begin{warning}
\label{warning:hypercompletion_locality}
Let $H\colon \ccat \rightarrow \acat$ be an adapted homology theory, and let $c \in \ccat$ be a non-$H$-local object in the sense that there exists a non-zero map $d \rightarrow c$ from an $H$-acyclic $d$; that is, an object such that $H(d) = 0$. Then, we claim that $\nu(c)$ cannot be hypercomplete. 

To see this, observe that since $\nu\colon \ccat \rightarrow \dcat^{\omega}(\ccat)$ is fully faithful, we have a non-zero homotopy class of maps $\nu(d) \rightarrow \nu(c)$. If the latter was hypercomplete, the map would factor through the hypercompletion $\widehat{L}\nu(d)$; we claim that the latter is zero. To see this, observe that we have $\pi_{k} \nu(d) \simeq H(d)[-k] = 0 $, where on the left we have the sheaf homotopy groups. It follows that $\widehat{L} \nu(d) = 0$, since we assume that $d$ is $H$-acyclic.
\end{warning} 
Thus, if we want $\nu(c)$ to be hypercomplete, we should at least assume that $c$ is $H$-local. Under fairly weak hypotheses, this turns out to be sufficient. 

\begin{lemma}
\label{lemma:nu_hypercomplete_iff_object_is_local}
Let $H\colon \ccat \rightarrow \acat$ be an adapted homology theory and assume that either:
\begin{enumerate}
    \item $\ccat$ admits geometric realizations or 
    \item $\acat$ is of finite cohomological dimension.
\end{enumerate}
Then, $c \in \ccat$ is $H$-local if and only if $\nu(c)$ is hypercomplete. 
\end{lemma}

\begin{proof}
The necessity of the locality condition was shown in \cref{warning:hypercompletion_locality}, so we only verify that it is sufficient.

First assume that $\ccat$ has geometric realizations. By \cite{pstrkagowski2018synthetic}[A.23], $\nu(c)$ is hypercomplete if and only if it takes any $H$-epimorphism hypercover 
\[
\ldots \triplerightarrow d_{1} \rightrightarrows d_{0} \rightarrow d
\]
to a limit diagram. By definition of $\nu$, this happens precisely when 
\[
\Map(d, c) \rightarrow \Map(d_{0}, c) \rightrightarrows \Map(d_{1}, c) \triplerightarrow \ldots 
\]
is a totalization. Since $\ccat$ admits geometric realizations, this is the same as asking for the induced morphism $\Map(d, c) \rightarrow \Map(|d_{\bullet}|, c)$ to be an equivalence. As $c$ is $H$-local by assumption, it is enough to check that $|d_{\bullet}| \rightarrow d$ is an $H$-isomorphism. 

By adaptedness, to be an $H$-isomorphism is equivalent to having the property that for every injective $i \in \acat$, and any associated injective lift $i_{\ccat} \in \ccat$, $\Map(d, i_{\ccat}) \rightarrow \Map(|d_{\bullet}|, i_{\ccat})$ is an equivalence of spaces. 

We have a totalization spectral sequence of the form 
\[
E_{2}^{s, t} \colonequals H^{s} [d_{\bullet}, i_{\ccat}]_{t} \simeq H^{s} \Hom^{t}_{\acat}(H(d_{\bullet}), i) \Rightarrow [|d_{\bullet}|, i_{\ccat}]_{t-s}.
\]
Since $H(d_{\bullet})$ is an epimorphism hypercover of $H(d)$, the $E_{2}$-term vanishes for $s \neq 0$ and gives $\Hom_{\acat}^{t}(H(d), i)$ for $s = 0$. We deduce that the spectral sequence collapses and $H(|d_{\bullet}|) \simeq H(d)$, as needed. 

Alternatively, assume that $\acat$ is of finite cohomological dimension, so that we can choose a finite injective resolution of $H(c)$. This can be lifted using \cref{construction:adams_spectral_sequence_associated_to_a_homology_theory} to a finite Adams resolution of $c$, which since $c$ is assumed to be $H$-local will terminate in finitely many steps in an object which is both $H$-acyclic and local, hence zero. It follows that $c$ belongs to the smallest subcategory of $\ccat$ containing all $H$-injectives and closed under taking fibres. 

As $\nu\colon \ccat \rightarrow \dcat^{\omega}(\ccat)$ preserves fibres and sheaves which are limits of their Postnikov tower are closed under taking these, we deduce that it is enough to show this when $c$ is itself $H$-injective. In this case, the presheaf truncations of $\nu(c)$ are already sheaves, and so $\nu(c)_{\leq n}$ is computed levelwise. It follows that $\nu(c) \simeq \varprojlim \nu(c)_{\leq n}$ as needed, so it is hypercomplete as a limit of bounded sheaves.  
\end{proof}

\begin{proposition}
\label{proposition:adapted_homology_theory_conservative_iff_finite_sheaves_hypercomplete} Let $H\colon \ccat \rightarrow \acat$ be an adapted homology theory. 
Suppose that either $\ccat$ admits geometric realizations or that $\acat$ is of finite cohomological dimension. Then, the following two conditions are equivalent:

\begin{enumerate}
    \item $H$ is conservative or 
    \item every perfect sheaf $X \in \dcat^{\omega}(\ccat)$ is hypercomplete.
\end{enumerate}
\end{proposition}

\begin{proof}
As all of the objects of $\ccat$ are $H$-local if and only if $H$ is conservative, \cref{lemma:nu_hypercomplete_iff_object_is_local} implies that conservativity is equivalent to all of the representables $\nu(c)$ for $c \in \ccat$ being hypercomplete. 

As we have seen in the proof of \cref{proposition:bounded_almost_perfect_presheaves_have_sheafication}, any perfect sheaf can be obtained through iterated extensions from the representables $\nu(c)$ and bounded sheaves. Since bounded sheaves are hypercomplete, and the statement follows from the fact that hypercomplete sheaves are stable under extensions.
\end{proof}

\begin{remark}
\label{remark:embedding_of_derived_categories_when_a_of_finite_dimension}
It follows from \cref{proposition:adapted_homology_theory_conservative_iff_finite_sheaves_hypercomplete} that if $\acat$ is of finite cohomological dimension, then $H$ is conservative if and only if the natural functor $\dcat^{\omega}(\ccat) \rightarrow \widehat{\dcat}(\ccat)$ is a fully faithful embedding of prestable $\infty$-categories. 
\end{remark}

\begin{remark}[Bousfield localization]
\label{remark:bousfield_localization}
One can show that for many homology theories any object $c \in \ccat$ admits an $H$-isomorphism into an $H$-local object, necessarily essentially unique. This presents the subcategory $\ccat_{H}$ of $H$-local objects as a localization of $\ccat$ by the class of $H$-isomorphisms, known as the Bousfield localization \cite{bousfield1979localization}, \cite{higher_topos_theory}[5.5.4].

If that is the case, then $H$ canonically factors as a composition
\[
\ccat \rightarrow \ccat_{H} \rightarrow \acat
\]
and one can show that the latter arrow is an adapted homology theory if and only if $H$ was. Moreover, the latter arrow is conservative; in this way, the study of arbitrary adapted homology theories can be reduced to the case of conservative ones. 
\end{remark}

\subsection{Goerss-Hopkins theory}
\label{subsection:goerss_hopkins_theory}

In this section we show how the derived $\infty$-category $\dcat^{\omega}(\ccat)$ associated to an adapted homology theory is an appropriate context for Goerss-Hopkins theory. Since we already established the needed formal properties of the derived $\infty$-category, the derivation of the needed obstructions is completely classical; our account will closely follow the arguments of \cite{moduli_problems_for_structured_ring_spectra} and \cite{pstrkagowski2017moduli, abstract_gh_theory}.

\begin{definition} 
\label{def:general_potential_l_stage}
Let $\infty \geq n \geq 0$. An object $X$ of $\dcat^{\omega}(\ccat)$ is called a \emph{potential $n$-stage} if $X$ is $n$-truncated and the morphism
\[X[-1] \to \Omega X\]
adjoint to $\tau$ identifies the target with the $(n-1)$-truncation of the source.
\end{definition}

\begin{notation}
\label{notation:goerss_hopkins_tower}
We denote the full subcategory of $\dcat^{\omega}(\ccat)$ spanned by potential $n$-stages by $\mathcal{M}_{n}(\ccat)$. Observe that the Postnikov $(n-1)$-truncation in the derived $\infty$-category restricts to functors 
\[
(-)_{\leq n-1}\colon \mathcal{M}_{n}(\ccat) \rightarrow \mathcal{M}_{n-1}(\ccat).
\]
We call the resulting tower of $\infty$-categories of the form 
\[
\mathcal{M}_{\infty}(\ccat) \rightarrow \ldots \rightarrow \mathcal{M}_{2}(\ccat)  \rightarrow \mathcal{M}_{1}(\ccat) \rightarrow \mathcal{M}_{0}(\ccat)
\]
the \emph{Goerss-Hopkins tower} of $\ccat$ relative to the adapted homology theory $H$. 
\end{notation}

\begin{remark}
A tower with remarkably similar properties to the one of \cref{notation:goerss_hopkins_tower} was constructed in the work of Biedermann using Bousfield's resolution model structures \cite{Biedermann}. One would expect that this construction is related to ours, but we did not try to construct a direct comparison.
\end{remark}

The beginning and the end of the tower can be identified very explicitly, as the following two examples show. 

\begin{example}
\label{example:potential_o_stages_are_discrete_objects}
A potential $0$-stage is precisely a discrete object of $\dcat^{\omega}(\ccat)$, with no further conditions, and so can be identified with an object of $\acat$ by \cref{proposition:properties_of_finite_derived_infty_cat_of_stable_cat}. In other words, we have a canonical equivalence $\mathcal{M}_{0}(\ccat) \simeq \acat$.
\end{example}

\begin{example}
\label{example:potential_infty_stages_are_objects_of_c}
For any $c \in \ccat$, the object $\nu(c)$ is a potential $\infty$-stage. In fact, we claim that all potential $\infty$-stages are of this form; that is, $\nu$ induces a canonical equivalence $\ccat \simeq \mathcal{M}_{\infty}(\ccat)$. 

Suppose that $X$ is a potential $\infty$-stage, and observe that since the map $X[-1] \rightarrow \Omega X$ is assumed to be an equivalence, it induces isomorphisms $\pi_{k}X[-1] \simeq \pi_{k+1}X$ also on the unsheafified homotopy groups; that is, the homotopy groups in $A_{\infty}^{\omega}(\ccat)$. It follows from \cref{lemma:characterization_of_finite_presheaves_on_stable_inftycat} that all of these unsheafified homotopy groups are representable; in particular, there is an isomorphism $\pi_{0} X \simeq y(c)$ in $A(\ccat)$  for some $c \in \ccat$. 

A choice of such an isomorphism determines a homotopy class of maps $\nu(c) \rightarrow X$ which is by construction an isomorphism on all unsheafified homotopy groups. Thus, it is a levelwise an equivalence, and hence an isomorphism. 
\end{example}

\begin{remark}
The above two examples show that the beginning of the Goerss-Hopkins tower can be identified with $\ccat$, and so is of homotopy-theoretic nature, while the bottom is the abelian category $\acat$ and so is in a sense purely algebraic. The intermediate $\infty$-categories $\mathcal{M}_{n}(\ccat)$ interpolate between these two extremes in a very controlled manner. 
\end{remark}

The notion of a potential $n$-stage is more interesting in the intermediate case of $\infty > n > 0$. In this case, an $n$-truncated $X$ is a potential $n$-stage if $\tau$ induces isomorphisms 
\[
\pi_{i}X[-1] \rightarrow \pi_{i+1}X 
\]
on $\acat$-valued homotopy groups for $0 \leq i < n$. This condition can be equivalently characterized using homology adjunction
\[
H^{*} \dashv H_{*}\colon \dcat^{\omega}(\ccat) \leftrightarrows \dcat^{b}(\acat), 
\]
as we will now do. Recall that by \cref{theorem:finite_ctau_modules_same_as_derived_category} this adjunction uniquely factors through an equivalence
\[
\Mod_{C\tau}(\dcat^{\omega}(\ccat)) \simeq \dcat^{b}(\acat).
\]
In particular, for any perfect sheaf $X \in \dcat^{\omega}(\ccat))$ we have a canonical cofibre sequence
\[
\Sigma X[-1] \xrightarrow{\tau} X \rightarrow H_*H^* X.
\]
induced by multiplication by $\tau$.

\begin{lemma} \label{characterizing l-stages} Let $X \in \dcat^{\omega}(\ccat)$ and $\infty > n \geq 0$. Then, the following are equivalent: 
\begin{enumerate}
    \item $X$ is a potential $n$-stage or 
    \item $X$ is $n$-truncated and $\pi_i(H_*H^*X) = \pi_i (C \tau \otimes X)=0$ for all $1 \leq i \leq n+1$.
\end{enumerate}
\end{lemma}

\begin{proof}
By definition, an $n$-truncated object $X$ is a potential $n$-stage if and only if the map
\[\tau \colon \Sigma X[-1] \rightarrow X\]
induces an isomorphism on sheaf homotopy groups $\pi_i$ for all $1 \leq i \leq n$. Using the cofibre sequence
\[
\Sigma X[-1] \rightarrow X \rightarrow H_*H^*X,
\]
we see that the latter is the case if and only if $\pi_i (H_*H^*X)=0$ for all $1 \leq i \leq n+1$. 
\end{proof}
There is a different characterization of potential $n$-stages, more in line with \cite{abstract_gh_theory}, \cite{lurie_hopkins_brauer_group}, which is also sometimes useful. Observe that since any potential $n$-stage is $n$-truncated, it admits a canonical structure of a $C\tau^{n+1}$-module by \cref{proposition:forgetful_functor_from_modules_is_exact}. In fact, a prototypical example of a potential $n$-stage is given by $\nu(c) _{\leq n} \simeq C\tau^{n+1} \otimes \nu(c)$.

In \cref{construction:goerss_hopkins_tower_for_prestable_freyd_envelope}, we introduced relative monads $C\tau^{k} \otimes_{C\tau^{n+1}} -$ on $\Mod_{C\tau^{n+1}}(A_{\infty}(\ccat))$ coming from the adjunction between $C\tau^{n+1}$ and $C\tau^{k}$-modules, where $k \leq n+1$. Morally, one would expect that we should have 
\begin{equation}
\label{equation:moral_equation_using_relative_tensor_products}
C\tau \otimes_{C\tau^{n+1}} (C\tau^{n+1} \otimes \nu(c)) \simeq C\tau \otimes \nu(c),
\end{equation}
which is discrete, in fact the zero-truncation of $C\tau^{n+1} \otimes \nu(c)$.  Moreover, this property should characterize potential $n$-stages, see \cite{abstract_gh_theory}[4.5].

The issue with making this argument precise is that, as we observed in  \cref{warning:relative_monads_dont_preserve_perfection}, the monads $C\tau \otimes _{C\tau^{n+1}} - $ do not in general preserve modules whose underlying presheaf is perfect. This matter cannot be avoided by working in the almost perfect case, as we are interested in sheaves and we have only proven the existence of sheafification in the perfect context. These issues are somewhat technical, but the statement we are after is morally true, although some care must be taken in formalizing it.

Let $X \in \dcat^{\omega}(\ccat)$ be an $n$-truncated sheaf, so that in particular it admits a structure of a $C\tau^{n+1}$-module. Consider the map $\rho\colon X \rightarrow X_{\leq 0}$ into its $0$-truncation, and observe that the target admits a structure of a $C\tau$-module and that $\rho$ uniquely lifts to a map of $C\tau^{n+1}$-modules.

\begin{lemma}
\label{lemma:ctau_tensor_over_ctaun_is_discrete_fora_potential_n_stage}
Let $X$ be a $n$-truncated sheaf and $\rho\colon X \rightarrow X_{\leq 0}$ be the zero-truncation. Then, $X$ is a potential $n$-stage if and only if for any $Y \in \Mod_{C\tau}(\dcat^{\omega}(\ccat))$, $\rho$ induces a homotopy equivalence
\begin{equation}
\label{equation:zero_truncation_of_potential_n_stage_a_relative_tensor_product}
\Map_{C\tau}(X_{\leq 0}, Y) \simeq \Map_{C\tau^{n+1}}(X, Y).
\end{equation}
\end{lemma}

\begin{proof}
First assume that $X$ is a potential $n$-stage. Choose an epimorphism $y(c) \rightarrow \pi_{0}X$ for some $c \in \ccat$, this determines a homotopy class of maps $\nu(c) \rightarrow X$. Using the canonical $C\tau^{n+1}$-module structure on the latter, this determines a homotopy class of maps $C\tau^{n+1} \otimes \nu(c) \rightarrow X$ of modules. Note that since this map is surjective on $\pi_{0}$ by construction, it is surjective on all $\acat$-valued homotopy groups as it is natural with respect to $\tau$. Thus, the long exact sequence of homotopy tells us that the fibre $F$ is again a potential $n$-stage with $\pi_{0} F \simeq \mathrm{ker}(H(c) \rightarrow \pi_{0}X)$. 

Proceeding inductively, we construct a hypercover of $X$ of the form 
\[
\ldots \triplerightarrow C\tau^{n+1} \otimes \nu(c_{1}) \rightrightarrows C\tau^{n+1} \otimes \nu(c_{0}) \rightarrow X 
\]
Since hypercovers are colimit diagrams in hypercomplete sheaves, and since every perfect $C\tau^{n+1}$-module is bounded and hence hypercomplete, we deduce that this is a colimit diagram in $\Mod_{C\tau^{n+1}}(\dcat^{\omega}(\ccat))$. 

Since taking fibres of a $\pi_{0}$-surjective map of potential $n$-stages commutes with taking homotopy groups, in particular we deduce that $\pi_{0}(C\tau^{n+1} \otimes \nu(c_{\bullet}))$ is a resolution of $\pi_{0}X$ in $\acat$. Thus, we deduce that 
\[
\ldots \triplerightarrow (C\tau^{n+1} \otimes \nu(c_{1}))_{\leq 0} \rightrightarrows (C\tau^{n+1} \otimes \nu(c_{0}))_{\leq 0} \rightarrow X_{\leq 0}
\]
is also a colimit diagram. 

Since both sides of (\ref{equation:zero_truncation_of_potential_n_stage_a_relative_tensor_product}) take colimits in the source to limits in the target, the above shows us that it is enough to verify the statement for potential $n$-stages of the form $C\tau^{n+1} \otimes \nu(c)$. In this case, since $Y$ is a sheaf, we can replace both $C\tau^{n+1} \otimes \nu(c)$ and its zero-truncation by their unsheafified versions computed in $A_{\infty}(\ccat)$, in which case this is a consequence of the equation (\ref{equation:moral_equation_using_relative_tensor_products}) which holds in almost perfect presheaves. 

Conversely, assume that $X$ is any $n$-truncated sheaf, proceeding as above we can again build a hypercover using objects of the form $C\tau^{n+1} \otimes \nu(c)$, which will be a colimit diagram of modules. If the map (\ref{equation:zero_truncation_of_potential_n_stage_a_relative_tensor_product}) is a homotopy equivalence, then we deduce that 
\[
\Map_{C\tau}(X_{\leq 0}, Y) \simeq \Map_{C\tau^{n+1}}(X, Y) \simeq \varprojlim \Map_{C\tau^{n+1}}(C\tau^{n+1} \otimes \nu(c_{\bullet}), Y) 
\]
and further 
\[
\varprojlim \Map_{C\tau^{n+1}}(C\tau^{n+1} \otimes \nu(c_{\bullet}), Y) \simeq \varprojlim \Map_{C\tau}(C\tau \otimes \nu(c_{\bullet}), Y) \simeq \Map_{C\tau}(\varinjlim C\tau \otimes \nu(c_{\bullet}), Y).
\]
Thus, we have $\varinjlim C\tau \otimes \nu(c_{\bullet}) \simeq X_{\leq 0}$ in $C\tau$-modules, so that
\[
\pi_{0} (C\tau \otimes \nu(c_{\bullet})) \simeq \pi_{0}(C\tau^{n+1} \otimes \nu(c_{\bullet}))
\]
is a resolution of $\pi_{0}X$. Since 
\[
\pi_{k} (C\tau^{n+1} \otimes \nu(c_{\bullet})) \simeq (\pi_{0}(C\tau^{n+1} \otimes \nu(c_{\bullet}))[-k]
\]
for $0 \leq k \leq n$ and vanish otherwise, we deduce that the geometric realization spectral sequence of $\varinjlim C\tau^{n+1} \otimes \nu(c_{\bullet}) \simeq X$ collapses on the second page and that $X$ is a potential $n$-stage. 
\end{proof}

Having characterized potential $n$-stages, let us prove the main property of the Goerss-Hopkins tower, namely that it admits obstructions to lifting objects and morphisms. We begin with the former. 

\begin{construction}[Obstructions to lifting objects] 
\label{construction:obstruction_to_lifting_objects}
Suppose that $X$ is a potential $n$-stage and consider the cofibre sequence
\[
\Sigma X[-1] \xrightarrow{\tau} X \rightarrow H_*H^* X.
\]
Using the long exact sequence of homotopy groups, we see that we have canonical isomorphisms $\pi_{0} H_{*} H^{*} X \simeq \pi_{0}X$, $\pi_{n+2} H_{*} H^{*}X \simeq \pi_{n}X[-1]$ and that the other homotopy groups vanish. 

As the right adjoint $H_{*}$ is exact and induces an equivalence on the hearts, it commutes with homotopy groups and we deduce that the same is true for $H^{*}X \in \dcat^{b}(\acat)$. Thus, Postnikov truncation provides a canonical fibre sequence 
$$\xymatrix{ H^* X \ar[r] & \pi_0 X \ar[r]^-{u_X} & \Sigma^{n+3} \pi_{0}X[-n-1]}. $$
where we have used that $\pi_{n} (X[-1]) \simeq \pi_{0} X[-n-1]$ since $X$ is a potential $n$-stage. 

It follows that $H^{*}X$ is classified up to equivalence as an object of the derived $\infty$-category of $\acat$ by a class 
\[
u_X \in \Ext^{n+3, n+1}_{\acat}(\pi_0X, \pi_0X),
\]
well-defined up to the automorphisms of $\pi_{0}X$.
\end{construction}

\begin{proposition}
\label{prop: general obstruction for objects}
Let $X \in \mathcal{M}_{n}(\ccat)$ be a potential $n$-stage. Then, $X$ can be lifted to a potential $(n+1)$-stage $Y \in \mathcal{M}_{n+1}(\ccat)$ if and only if $H^{*}X$ splits; that is, if and only if the obstruction $u_X\in \Ext^{n+3, n+1}_{\acat}(\pi_0X, \pi_0X)$ vanishes.  
\end{proposition}

\begin{proof}
First assume that $u_X=0$, so that $H^* X$ splits 
\[
H^*X \simeq \pi_0X \oplus \Sigma^{n+2}\pi_{0}X[-n-1]
\]
and hence there exists a morphism $H^*X \to \Sigma^{n+2}\pi_0X[-n-1]$ in $\dcat^{b}(\acat)$ inducing an isomorphism on $\pi_{n+2}$. By adjunction this gives a morphism
\[X \to \Sigma^{n+2} H_{*} \pi_0X[-n-1]\]
in the derived $\infty$-category of $\ccat$. By taking the fibre of the latter, we get a fiber sequence
\[
Y \to X \to \Sigma^{n+2} H_* \pi_0X[-n-1]
\]
Since X is $n$-truncated, it follows that $Y$ is $(n+1)$-truncated. As the latter map is $\pi_{0}$-surjective, this is also a cofibre sequence and hence
\[
H^*Y \to H^*X \to \Sigma^{n+2}H^* H_* \pi_0X[-n-1]
\]
is a cofibre sequence. Using prestability we see that this is also a fibre sequence. 

The composite 
\[
H^*X \to \Sigma^{n+2}H^* H_{*} \pi_0X[-n-1] \rightarrow \Sigma^{n+2} \pi_{0}X[-n-1]
\]
is the splitting map we chose, and hence the first map is also an isomorphism on $\pi_{n+2}$, as the only other non-trivial homotopy of $\Sigma^{n+2}H^* H_{*} \pi_0X[-n-1]$ lives in degree $n+4$. By the long exact sequence, we deduce that $\pi_{n+2} H^*Y=0$ and 
$\pi_{i} H^*Y \simeq \pi_{i} H^*X$, for $0 \leq i \leq n+1$. By \cref{characterizing l-stages}, we conclude that $Y$ is a potential $(n+1)$-stage, as needed.  

Conversely, suppose $X$ extends to a potential $(n+1)$-stage $Y$, so that we have a cofibre sequence
\[Y \to X \to \Sigma^{n+2}\pi_0X[-n-1].\]
As the latter object is concentrated in a single degree, it is canonically in the image of $H_{*}$ and hence by adjunction we get a morphism 
\[H^*X \to \Sigma^{n+2}\pi_0X[-n-1] \]
in the derived $\infty$-category of $\acat$. We claim that this induces an isomorphism on $\pi_{n+2}$, so that $H^*X$ splits. Using the long exact sequence of the cofibre sequence
\[H^*Y \to H^* X \to \Sigma^{n+2} H^{*} \pi_0X[-n-1]\]
as above it is enough to see that $\pi_{n+2}H^*Y=0$ and $\pi_{n+1}H^*Y=0$. Both hold by \cref{characterizing l-stages}, since $Y$ is assumed to be a potential $(n+1)$-stage. 
\end{proof}

\begin{proposition} \label{proposition:gh_fibre_sequence_for_mapping_spaces}
Suppose that $X$ and $Y$ are potential $n$-stages where $n \geq 1$. Then, we have a canonical fibre sequence of mapping spaces of the form
\[
\xymatrix{ \Map_{\mathcal{M}_n(\ccat)}(X, Y) \ar[r]^-{(-)_{\leq n-1}} & \Map_{\mathcal{M}_{n-1}(\ccat)}(X_{\leq n-1}, Y_{\leq n-1}) \ar[r] & \Map_{\dcat^{b}(\acat)}(\pi_0X, \Sigma^{n+1} \pi_0 Y[-n]). } 
\]
\end{proposition}

\begin{proof}
We have a cofibre sequence in $\dcat^{\omega}(\ccat)$ of the form
\[
\xymatrix{Y \ar[r] & Y_{\leq n-1} \ar[r] & \Sigma^{n+1} \pi_n Y.}
\]
Since the first map has an $n$-truncated source and target, it can be lifted to a map of $C\tau^{n+1}$-modules, and hence the same is true for the whole cofibre sequence. 

Applying the functor $\Map_{C\tau^{n+1}}(X,-)$ we get a fibre sequence
\begin{equation}
\label{equation:fibre_sequence_in_obstructions_to_lifting_morphisms}
\xymatrix{\Map_{C\tau^{n+1}}(X,Y) \ar[r] & \Map_{C\tau^{n+1}}(X, Y_{\leq n-1}) \ar[r] & \Map_{C\tau^{n+1}}(X,\Sigma^{n+1} \pi_n Y).}
\end{equation}
We claim this is of the needed form. 

To begin with, since truncations are a left adjoint, and all object in question are $n$-truncated, we have 
\[
\Map_{C\tau^{n+1}}(X, Y_{\leq n-1}) \simeq \Map_{\dcat^{\omega}(\ccat)}(X, Y_{\leq n-1}) \simeq \Map_{\dcat^{\omega}(\ccat)}(X_{\leq n-1}, Y_{\leq n-1}).
\]
The left-most mapping space in (\ref{equation:fibre_sequence_in_obstructions_to_lifting_morphisms}) is immediately identified using \cref{proposition:forgetful_functor_from_modules_is_exact}, so we only have the right-most one left. However, by \cref{lemma:ctau_tensor_over_ctaun_is_discrete_fora_potential_n_stage} we have 
\[
\Map_{C\tau^{n+1}}(X,\Sigma^{n+1} \pi_n Y) \simeq \Map_{C\tau}(X_{\leq 0}, \Sigma^{n+1} \pi_{n}Y),
\]
and the statement follows from the equivalence $\Mod_{C\tau}(\dcat^{\omega}(\ccat)) \simeq \dcat^{b}(\acat)$. 
\end{proof}

\begin{remark}
\label{remark:obstructions_to_lifting_morphisms}
Since we have 
\[
\pi_{k} \Map_{\dcat^{b}(\acat)}(\pi_0X, \Sigma^{n+1} \pi_0 Y[-n]) \simeq \Ext^{n+1-k, n}_{\acat}(\pi_{0} X, \pi_{0} Y), 
\]
it follows from \cref{proposition:gh_fibre_sequence_for_mapping_spaces} that if $X, Y$ are potential $n$-stages, then there are obstruction to lifting homotopy classes of morphisms $Y_{\leq n-1} \rightarrow X_{\leq n-1}$ along the functor $\mathcal{M}_{n}(\ccat) \rightarrow \mathcal{M}_{n-1}(\ccat)$ living in $\Ext^{n+1, n}$. In this way, this statement is very reminiscent of \cref{prop: general obstruction for objects}.

Similarly, lower $\Ext$-groups provide obstructions to lifting homotopies between morphisms and so on. 
\end{remark}
The last ingredient needed in applying Goerss-Hopkins theory is the question of convergence; that is, whether we have
\[
\ccat \simeq \mathcal{M}_{\infty}(\ccat) \simeq \varprojlim \mathcal{M}_{n}(\ccat).
\]
This is related to the convergence of the $H$-based Adams spectral sequence, and hence quite subtle, see \cite{abstract_gh_theory} for more discussion.

In our application to Franke's conjecture, we will only need the simple case when $\acat$ is of finite cohomological dimension. 

\begin{proposition} 
\label{proposition:convergence_of_gh_tower}
Let $H\colon \ccat \rightarrow \acat$ be a conservative, adapted homology theory such that $\acat$ is of finite cohomological dimension. Then, the Goerss-Hopkins tower 
\[
\mathcal{M}_{\infty}(\ccat) \rightarrow \ldots \rightarrow \mathcal{M}_{2}(\ccat)  \rightarrow \mathcal{M}_{1}(\ccat) \rightarrow \mathcal{M}_{0}(\ccat)
\]
is a limit diagram of $\infty$-categories. 
\end{proposition}

\begin{proof}
We will work with the complete derived $\infty$-category of \cref{definition:complete_derived_infty_cat}, the $\infty$-category of almost perfect hypercomplete sheaves on $\ccat$. Since $H$ is conservative, every perfect sheaf is already hypercomplete by \cref{proposition:adapted_homology_theory_conservative_iff_finite_sheaves_hypercomplete}, so that we have a natural fully faithful embedding $\dcat^{\omega}(\ccat) \hookrightarrow \widehat{\dcat}(\ccat)$. 

By \cref{proposition:complete_derived_cat_of_stable_cat_a_completion_of_finite_one}, this embedding presents the target as a completion of the source; in particular,
\[
\widehat{\dcat}(\ccat) \simeq \varprojlim \tau_{\leq n} \widehat{\dcat}(\ccat) \simeq \varprojlim \tau_{\leq n} \dcat^{\omega}(\ccat).
\]

Let us say that $X \in \widehat{\dcat}(\ccat)$ is a \emph{homotopy $\infty$-stage} if $X_{\leq n}$ is a potential $n$-stage for every $n$. If we denote by $\widehat{\mathcal{M}}_{\infty}(\ccat)$ the $\infty$-category of homotopy $\infty$-stages, then the above equivalence restricts to one of the form 
\[
\widehat{\mathcal{M}}_{\infty}(\ccat) \simeq \varprojlim \mathcal{M}_{n}(\ccat).
\]
We have to show that every homotopy $\infty$-stage $X$ is an $\infty$-stage in the sense of \cref{def:general_potential_l_stage}, as perfect sheaves are hypercomplete and so their equivalences are detected by homotopy groups, it is enough to show that $X$ is perfect. 

We prove this by induction on injective dimension of $\pi_{0}X \in \acat$. First assume that $\pi_{0}X = i$ is injective and let $i_{\ccat}$ be an associated injective lift in $\ccat$, so that $\pi_{0} \nu(i_{\ccat}) \simeq i$. 

We claim that any isomorphism $\pi_{0} X \simeq \pi_{0} \nu(i_{\ccat})$ can be lifted to an equivalence $X \simeq \nu(i_{\ccat})$. To see this, observe that by the limit formula above we have 
\[
\Map_{\widehat{\mathcal{M}}_{\infty}(\ccat)}(X, \nu(i_{\ccat})) \simeq \varprojlim \Map_{\mathcal{M}_{n}(\ccat)}(X_{\leq n}, \nu(i_{\ccat})_{\leq n}).
\]
Thus, it is enough to verify that any equivalence of $n$-truncations can be lifted to one of $(n+1)$-truncations. By \cref{remark:obstructions_to_lifting_morphisms},  obstructions to doing the latter live in $\Ext_{\acat}^{n+2,n+1}(\pi_{0}X, \pi_{0} \nu(i_{\ccat}))$ and so vanish for all $n \geq 0$ by injectivity of the target. Thus, $X \simeq \nu(i_{\ccat})$, which is perfect as needed. 

Now assume that $\pi_{0}X$ is of injective dimensions $k > 0$. We can choose an embedding $\pi_{0} X \hookrightarrow i$ into an injective of $\acat$. As above, this determines a map $\pi_{0}X \rightarrow \pi_{0} \nu(i_{\ccat})$ which can be lifted to a morphism $X \rightarrow \nu(i_{\ccat})$ as the relevant obstructions again vanish. 

Completing this map to a cofibre sequence
\[
X \rightarrow \nu(i_{\ccat}) \rightarrow C;
\]
studying the resulting long exact sequence of homotopy groups, which in this case reduces to short exact sequences, we see that $C \in \widehat{\mathcal{M}}_{\infty}(\ccat)$ and that $\pi_{0}C$ is of injective dimension $k-1$. Thus, $C$ is perfect by inductive assumption and so is $X$ as it is a fibre of a map between perfect sheaves and hence perfect itself. 

By induction, we deduce that every homotopy potential $\infty$-stage such that $\pi_{0}X$ is of finite injective dimension is a potential $\infty$-stage. As $\acat$ is of finite cohomological dimension, this ends the argument. 
\end{proof}

\begin{corollary}
\label{corollary:c_which_admits_conservative_adapted_homology_theory_in_fin_dim_a_is_idempotent_complete}
Let $H\colon \ccat \rightarrow \acat$ be a conservative, adapted homology theory such that $\acat$ is of finite cohomological dimension. Then, $\ccat$ is idempotent-complete.
\end{corollary}

\begin{proof}
As idempotent-complete $\infty$-categories are stable under limits, by a combination of \cref{proposition:convergence_of_gh_tower} and \cref{example:potential_infty_stages_are_objects_of_c} it is enough to show that each of $\mathcal{M}_{n}(\ccat)$ is idempotent-complete. By definition, the latter is a subcategory of $\tau_{\leq n} \dcat^{\omega}(\ccat)$, closed under direct summands. As $\tau_{\leq n} \dcat^{\omega}(\ccat)$ is an $(n+1)$-category with finite colimits, it follows that it is idempotent complete \cite{lurie2009derived}[1.5.12], and hence so is $\mathcal{M}_{n}(\ccat)$. 
\end{proof}

\subsection{Digression: The Grothendieck abelian case} 
\label{subsection:digression_derived_cat_in_grothendieck_abelian_case}

Our construction of the perfect derived $\infty$-category of \cref{definition:perfect_derived_cat_of_stable_cat} is very general, as it takes an arbitrary homology theory as an input. This comes at a price; for example, as we mentioned at the beginning of \S\ref{subsection:completing_derived_infty_cats}, $\dcat^{\omega}(\ccat)$ virtually never has infinite coproducts.

One way to ensure better $\infty$-categorical properties of $\dcat^{\omega}(\ccat)$ is to complete it, but depending on the context, this can come at a price as well. For instance, as we discussed in \S\ref{subsection:hypercompletion_and_locality}, if $c \in \ccat$ is $H$-acyclic, then $\nu(c)$ already vanishes in $\widehat{\dcat}(\ccat)$; that is, the completion cannot tell the difference between $\ccat$ and its Bousfield $H$-localisation. 

In this section, we will describe a variant of the derived $\infty$-category which takes as an input only a restricted class of homology theories, but which has excellent $\infty$-categorical properties: it will be a Grothendieck prestable $\infty$-category in the sense of Lurie \cite{lurie_spectral_algebraic_geometry}[Appendix C].

The following is the class of homology theories we will be interested in.

\begin{definition}
\label{definition::grothendieck_homology_theory}
We say a homological functor $H\colon \ccat \rightarrow \acat$ is \emph{Grothendieck} if 
\begin{enumerate}
    \item $\ccat$ is presentable, 
    \item $\acat$ is Grothendieck abelian and 
    \item $H$ preserves all small coproducts. 
\end{enumerate}
We say a homology theory $H$ is Grothendieck if the underlying homological functor is. 
\end{definition}

\begin{example}
If $\ccat$ is a presentably monoidal stable $\infty$-category with a compact unit, then for any $c \in \ccat$, the \emph{$c$-homology} functor 
\[
c_{*}(-)\colon \ccat \rightarrow \mathrm{gr}\abeliangroups
\]
given by 
\[
c_{*}d \colonequals [\monunit_{\ccat}, c \otimes d]_{*}
\]
is a Grothendieck homology theory. In particular, homology with respect to a fixed spectrum in the most classical sense of the word is an example. 
\end{example}

\begin{warning}
An important example of a \emph{non}-Grothendieck homology theory is the universal homology theory 
\[
y\colon \ccat \rightarrow A(\ccat)
\]
valued in the classical Freyd envelope. Indeed, the latter is usually not presentable, much less Grothendieck, as we observed in \cref{warning:freyd_envelope_usually_not_presentable}. Many further examples arise through the equivalence of  \cref{theorem:adapted_homology_theories_correspond_to_injective_epimorphism_classes} by picking an appropriate class of epimorphisms; for example, $E \otimes -$-split maps for $E$ a ring spectrum almost always lead to a non-Grothendieck homology theory. 
\end{warning}

The goal of this section is to associate to any  $H\colon \ccat \rightarrow \acat$ as above a Grothendieck prestable $\infty$-category satisfying a variant on the universal property of \cref{theorem:universal_property_of_finite_derived_category}. Adding to the complication (as well as the richness of the subject) is the fact that even in the abelian case, there are several $\infty$-categories satisfying closely related universal properties, all of which are good candidates for \emph{the} derived $\infty$-category, depending on the context. 

\begin{remark}[Variants on the derived $\infty$-category of an abelian category]
\label{remark:variants_on_da_in_grothendieck_case}
In \cite{lurie_spectral_algebraic_geometry}[C.5], Lurie associates to a Grothendieck abelian category
\begin{enumerate}
    \item an \emph{unseparated derived $\infty$-category} $\widecheck{\dcat}(\acat)$, which is universal among \emph{all} Grothendieck prestable $\infty$-categories $\dcat$ admitting an exact, cocontinuous functor $\acat \rightarrow \dcat^{\heartsuit}$, in the sense that any such functor has a unique exact, cocontinuous extension $\widecheck{\dcat}(\acat) \rightarrow \dcat$,
    \item a \emph{(separated) derived $\infty$-category} $\dcat(\acat)$, which is universal in the same sense among all separated Grothendieck prestable $\infty$-categories, that is, those for which homotopy groups detect equivalences and 
    \item a \emph{complete derived $\infty$-category} $\widehat{\dcat}(\acat)$ which is universal among all Postnikov complete prestable $\infty$-categories.
\end{enumerate}
These sometimes coincide, for example whenever $\acat$ is of finite cohomological dimension, but not in general. 
\end{remark}

\begin{remark}[Why focus on the unseparated case?]
Note that among the three variants discussed in \cref{remark:variants_on_da_in_grothendieck_case}, it is the middle which can be obtained from the abelian category of chain complexes in $\acat$ by inverting quasi-isomorphisms, and so has been classically referred to as the derived $\infty$-category of $\acat$. 

On the other hand, from the perspective of Grothendieck prestable $\infty$-categories, it is the unseparated variant which is the most fundamental, as the other two can be easily obtained from it either through the process of taking the separated quotient or completion \cite{lurie_spectral_algebraic_geometry}[C.3.6.1, C.3.6.3], while going the other way around is slightly more involved.

It is for this reason that in our construction of the derived $\infty$-category associated to a Grothendieck homology theory we will also focus on the unseparated case. 
\end{remark}

Our goal in this section will be to prove the following result. 

\begin{theorem}
\label{theorem:grothendieck_derived_infty_cat_of_h_universal_property}
Let $H\colon \ccat \rightarrow \acat$ be a Grothendieck homology theory. Then, there exists a Grothendieck prestable $\infty$-category $\widecheck{\dcat}(\ccat)$ together with an equivalence $\widecheck{\dcat}(\ccat)^{\heartsuit} \simeq \acat$ and a fully faithful, small coproduct-preserving prestable enhancement $\widecheck{\nu}\colon \ccat \rightarrow \widecheck{\dcat}(\ccat)$ of $H$, which we call the \emph{unseparated derived $\infty$-category of $H$}, such that for any other Grothendieck prestable $\dcat$, the following two collections of data are equivalent: 

\begin{enumerate}
    \item exact, cocontinuous functors $G\colon \widecheck{\dcat}(\ccat) \rightarrow \dcat$,
    \item exact, cocontinuous functors $G_{0}\colon \acat \rightarrow \dcat^{\heartsuit}$ together with a small coproduct-preserving prestable enhancement of $G_{0} \circ H$.
\end{enumerate}
\end{theorem}
Note that the universal property is analogous to that of the perfect derived $\infty$-category, which we have proven in \cref{theorem:universal_property_of_finite_derived_category}, with the two differences being that the prestable $\infty$-categories in question are Grothendieck and that all functors are required to preserve small coproducts.

\begin{remark}
As in the case of the perfect derived $\infty$-category of \cref{remark:local_grading_on_perfect_derived_infty_cat} and of prestable Freyd envelopes discussed in \cref{remark:universal_properties_of_local_gradings_on_prestable_freyd_envelopes}, the universal property of $\widecheck{\dcat}(\ccat)$ implies that it admits a canonical local grading compatible with that of $\ccat$ and $\acat$.
\end{remark}
Our proof of \cref{theorem:grothendieck_derived_infty_cat_of_h_universal_property} will proceed in stages, first starting in the compactly-generated case and a very special homology theory. We will need the following result of Henning Krause. 

\begin{lemma}[Krause]
\label{lemma:embedding_into_product_pres_sheaves_a_universal_grothendieck_homological_functor}
Let $\ccat$ be a compactly-generated presentable, stable $\infty$-category. Then, the restricted Yoneda embedding $y\colon \ccat \rightarrow P_{\Sigma}(\ccat^{\omega}, \abeliangroups)$ into product-preserving presheaves of abelian groups on compact objects, is a universal Grothendieck homological functor. That is, for any other Grothendieck $H\colon \ccat \rightarrow \acat$, there exists an essentially unique exact, cocontinuous $L\colon P_{\Sigma}(\ccat^{\omega}, \abeliangroups) \rightarrow \acat$ such that $H \simeq L \circ y$. \end{lemma}

\begin{proof}
This is \cite{krause_2010}[6.10.1].
\end{proof}

\begin{remark}
\label{remark:p_sigma_of_compacts_the_grothendieck_freyd_envelope}
If $\ccat$ is compactly-generated, then $P_{\Sigma}(\ccat^{\omega}, \abeliangroups)$ plays the same role for Grothendieck homology theories on $\ccat$ as the Freyd envelope $A(\ccat)$ plays for all homology theories. 

To be more precise, by a result of Neeman $y\colon \ccat \rightarrow P_{\Sigma}(\ccat^{\omega}, \abeliangroups)$ is adapted \cite{neeman2001triangulated}[6.5]. Thus, a Grothendieck homology theory $H\colon \ccat \rightarrow \acat$ induces functors 
\[
 A(\ccat) \rightarrow P_{\Sigma}(\ccat^{\omega}, \abeliangroups) \rightarrow \acat
\]
where the first arrow is quotient by a localizing subcategory. It follows that the same is true for the second one if and only if this holds for the composite, and thus that adapted Grothendieck homology theories are in one-to-one correspondence with localizing subcategories of $P_{\Sigma}(\ccat^{\omega}, \abeliangroups)$.

In the special case where $\ccat$ is monoidal, certain such localizing subcategories have been studied in the work of Balmer, Krause and Stevenson, who called the corresponding quotients the \emph{homological residue fields of $\ccat$} \cite{balmer2019tensor}, \cite{balmer2020nilpotence}. These can be often described as comodules over comonads or coalgebras as in the work of Balmer, Cameron and Stevenson \cite{balmer2021computing}, \cite{cameron2021homological}, see also \cref{remark:adapted_factorization_for_grothendieck_categories} below.
\end{remark}

\begin{remark}[A general Grothendieck-Freyd envelope]
\label{remark:general_grothendieck_freyd_envelope}
We will not need this, but a universal Grothendieck homological functor $H\colon \ccat \rightarrow \acat$ analogous to the restricted Yoneda embedding of \cref{lemma:embedding_into_product_pres_sheaves_a_universal_grothendieck_homological_functor} exists for an arbitrary presentable, stable $\infty$-category $\ccat$, not necessarily compactly-generated. 

In the general case, one can construct a localization $L\colon \ccat_{0} \rightarrow \ccat$ from a compactly-generated presentable, stable $\infty$-category. It is not difficult to verify that the needed universal homological functor is obtained from $P_{\Sigma}(\ccat_{0}^{\omega}, \abeliangroups)$ by taking a quotient by the localizing subcategory $K$ generated by $y(k)$ for $k \in \ker(L)$. 

It follows that adapted Grothendieck homology theories on an arbitrary $\ccat$ are parameterized by the localizing subcategories of the Gabriel quotient $P_{\Sigma}(\ccat_{0}^{\omega}, \abeliangroups)/K$; equivalently, localizing subcategories of $P_{\Sigma}(\ccat_{0}^{\omega}, \abeliangroups)$ which contain $K$. This is the Grothendieck analogue of \cref{theorem:adapted_homology_theories_correspond_to_injective_epimorphism_classes}.
\end{remark}

\begin{remark}
\label{remark:adapted_factorization_for_grothendieck_categories}
Any Grothendieck abelian category has enough injectives, and an application of  \cref{corollary:coproduct_preserving_homology_theory_on_presentable_cat_has_lifts_of_injectives} tells us that any Grothendieck homology theory has lifts for injectives. By \cref{theorem:adapted_factorization_of_homology_theories}, we thus have a canonical factorization
\[
\begin{tikzcd}
	{\ccat} && {\acat} \\
	& {\acat^{\star}}
	\arrow["{H}", from=1-1, to=1-3]
	\arrow["{H^{\star}}"', from=1-1, to=2-2]
	\arrow["{U}"', from=2-2, to=1-3]
\end{tikzcd}
\]
where $H^{\star}$ is adapted and $U$ is an exact comonadic left adjoint. Observe that $\acat^{\star}$ can be identified with a quotient of the Grothendieck-Freyd envelope of \cref{remark:general_grothendieck_freyd_envelope}, and so is itself Grothendieck abelian. It follows that any any Grothendieck homology theory canonically factors through an adapted one, which is also Grothendieck, followed by a comonadic exact functor. 
\end{remark}

Note that the abelian category of $P_{\Sigma}(\ccat^{\omega}, \abeliangroups)$ of \cref{lemma:embedding_into_product_pres_sheaves_a_universal_grothendieck_homological_functor} can be canonically identified with the heart of $P_{\Sigma}(\ccat^{\omega})$, the $\infty$-category of product-preserving presheaves of spaces. This leads to a natural guess that perhaps this prestable $\infty$-category is the unseparated Grothendieck derived $\infty$-category $\widecheck{\dcat}(\ccat)$ associated to the universal Grothendieck homology theory $y\colon \ccat \rightarrow P_{\Sigma}(\ccat^{\omega}, \abeliangroups)$; that is, that it satisfies the conditions of \cref{theorem:grothendieck_derived_infty_cat_of_h_universal_property}. This is indeed the case, and the proof will require the following. 

\begin{lemma}
\label{lemma:coproduct_preserving_enhancement_into_grothendieck_pcat_preserves_filtered_colimits}
Let $\ccat$ be a presentable, stable $\infty$-category, $\dcat$ a Grothendieck prestable $\infty$-category and $\euscr{H}\colon \ccat \rightarrow \dcat$ a prestable enhancement which preserves all small coproducts. Then $\euscr{H}$ preserves filtered colimits. 
\end{lemma}

\begin{proof}
Since $\euscr{H}$ is left exact, it induces a functor between $\infty$-categories of spectrum objects, giving a diagram 
\[
\begin{tikzcd}
	\ccat & \dcat \\
	{\spectra(\ccat)} & {\spectra(\dcat)}
	\arrow[from=1-1, to=1-2]
	\arrow["{\Omega^{\infty}_{C}}", from=2-1, to=1-1]
	\arrow["{\Omega^{\infty}_{D}}"', from=2-2, to=1-2]
	\arrow[from=2-1, to=2-2]
\end{tikzcd}.
\]
Since $\ccat$ is stable, the left vertical arrow is an equivalence, and as $\dcat$ is Grothendieck prestable, the right vertical arrow can be identified with taking connective covers with respect to a $t$-structure compatible with filtered colimits. In particular, the right vertical arrow preserves filtered colimits. 

The lower horizontal arrow is a coproduct-preserving exact functor of stable $\infty$-categories, and so it is cocontinuous. We deduce that the upper horizontal arrow also preserves filtered colimits, as needed. 
\end{proof}

\begin{proposition}
\label{proposition:p_sigma_of_c_omega_an_unsep_derived_cat_for_universal_homology_theory}
Let $\ccat$ be a compactly-generated presentable, stable $\infty$-category. Then, the $\infty$-category of product-preserving presheaves of spaces on compact objects together with the restricted Yoneda embedding $\widecheck{\nu}\colon \ccat \rightarrow P_{\Sigma}(\ccat^{\omega})$ satisfies the conditions of \cref{theorem:grothendieck_derived_infty_cat_of_h_universal_property}.
\end{proposition}

\begin{proof}
First observe that $P_{\Sigma}(\ccat^{\omega})$ is indeed Grothendieck prestable, by \cite{lurie_spectral_algebraic_geometry}[C.1.5.10], and has the correct heart. Moreover, since $\ccat \simeq \Ind(\ccat^{\omega})$, the restricted Yoneda embedding $\ccat \rightarrow P_{\Sigma}(\ccat^{\omega})$ can be identified with the inclusion of presheaves which are filtered colimits of representables. In particular, it is fully faithful and preserves all small coproducts. 

Suppose that $\dcat$ is a Grothendieck prestable $\infty$-category, $G_{0}\colon P_{\Sigma}(\ccat^{\omega}, \abeliangroups) \rightarrow \dcat^{\heartsuit}$ is exact, cocontinuous and $\euscr{G}\colon \ccat \rightarrow \dcat$ is a small coproduct-preserving prestable enhancement of $G_{0} \circ y$, so that we have a chosen isomorphism $\pi_{0} \circ \euscr{G} \simeq G_{0} \circ y$. 

Observe that since $y$ is universal by \cref{lemma:embedding_into_product_pres_sheaves_a_universal_grothendieck_homological_functor}, $G_{0}$ and the natural isomorphism are uniquely determined by $\euscr{G}$. Thus, we have to show that there is a one-to-one correspondence between exact, cocontinuous functors out of $P_{\Sigma}(\ccat)$ and left exact, small coproduct-preserving $\euscr{G}\colon \ccat \rightarrow \dcat$. Since $\widecheck{\nu}$ is left exact and preserves all coproducts, precomposition with $\widecheck{\nu}$ provides one direction of this correspondence. 

To go the other way, we have to show that the left Kan extension $\euscr{L}\colon P_{\Sigma}(\ccat^{\omega}) \rightarrow \dcat$ of the composite 
\[
\ccat^{\omega} \rightarrow \ccat \rightarrow \dcat,
\]
which is cocontinuous as the composite was additive, is left exact and satisfies $\euscr{L} \circ \widecheck{\nu} \simeq \euscr{G}$. The second part is clear, as they agree on compacts and both preserve filtered colimits, the right hand side by \cref{lemma:coproduct_preserving_enhancement_into_grothendieck_pcat_preserves_filtered_colimits}.

For left exactness, observe that using the criterion of \cite{higher_topos_theory}[5.3.5.11, 5.5.8.13] we can identify $P_{\Sigma}(\ccat^{\omega})$ with the $\Ind$-completion of its prestable subcategory $A_{\infty}^{\omega}(\ccat^{\omega})$. Since $\euscr{L}$ is cocontinuous, it is a left Kan extension of its restriction to the perfect prestable Freyd envelope, and we deduce that it is enough to verify that its restriction is left exact. This is exactly \cref{theorem:universal_property_of_finite_presheaves}.
\end{proof}
Now let us move to the more general case where $\ccat$ is compactly-generated, but we have an arbitrary adapted Grothendieck $H\colon \ccat \rightarrow \acat$. By \cref{remark:p_sigma_of_compacts_the_grothendieck_freyd_envelope}, the induced functor out of $P_{\Sigma}(\ccat^{\omega}, \abeliangroups)$ identifies $\acat$ with a quotient by some localizing subcategory $K$. 

There is canonical way to lift a localizing subcategory $K$ of the Grothendieck abelian \[
P_{\Sigma}(\ccat^{\omega}, \abeliangroups) \simeq P_{\Sigma}(\ccat^{\omega})^{\heartsuit}
\]
to a localizing subcategory $\euscr{K}$ of the Grothendieck prestable $P_{\Sigma}(\ccat^{\omega})$. Namely, one can take the smallest localizing subcategory $\euscr{K}$ which contains $K$, considered as a subcategory of discrete objects, see \cite{lurie_spectral_algebraic_geometry}[C.5.2.8]. Note that we will then have 
\[
(P_{\Sigma}(\ccat^{\omega}) / \euscr{K})^{\heartsuit} \simeq P_{\Sigma}(\ccat^{\omega})^{\heartsuit} / K,
\]
so that the composite 
\[
\ccat \rightarrow P_{\Sigma}(\ccat^{\omega}) \rightarrow P_{\Sigma}(\ccat^{\omega}) / \euscr{K}
\]
can be canonically considered as a prestable enhancement of $H$. Using this notation, we have the following. 

\begin{proposition}
\label{proposition:unsep_grothendieck_exists_when_ccat_is_compactly_generated}
The composite $\ccat \rightarrow P_{\Sigma}(\ccat^{\omega}) / \euscr{K}$ identifies the target with the unseparated derived $\infty$-category of $H$; that is, it satisfies the conditions of \cref{theorem:grothendieck_derived_infty_cat_of_h_universal_property}.
\end{proposition}

\begin{proof}
By construction, $P_{\Sigma}(\ccat^{\omega}) / \euscr{K}$ has heart equivalent to $\acat$, as needed. Moreover, since we chose the smallest localizing subcategory containing $K$, cocontinuous, left exact functors $P_{\Sigma}(\ccat^{\omega}) / \euscr{K} \rightarrow \dcat$ into another Grothendieck prestable $\infty$-category correspond to those cocontinuous, exact functors $\euscr{L}\colon P_{\Sigma}(\ccat^{\omega}) \rightarrow \dcat$ whose restriction to the heart annihilates $K$. By \cref{proposition:p_sigma_of_c_omega_an_unsep_derived_cat_for_universal_homology_theory}, to give the latter is the same as to give a functor 
\[
\acat \simeq P_{\Sigma}(\ccat^{\omega})^{\heartsuit} / K \rightarrow \dcat^{\heartsuit}
\]
together with a coproduct-preserving prestable enhancement of its composition with $H$, which is the needed universal property. 

Finally, we claim that $\ccat \rightarrow P_{\Sigma}(\ccat^{\omega})/K$ is fully faithful. Since $\widecheck{\nu}$ is fully faithful before passing to the quotient, it is enough to verify that for any $c \in \ccat$, the object $\widecheck{\nu}(c) \in P_{\Sigma}(\ccat^{\omega})$ is in the image of the right adjoint from the quotient, which consists of the subcategory of those objects $X$ such that $\Map(k, X) \simeq 0$ for any $k \in \euscr{K}$ \cite{lurie_spectral_algebraic_geometry}[C.2.3.9].

Observe that the identity of $\ccat$ is a coproduct-preserving prestable enhancement and so induces an exact, cocontinuous functor $\tau^{-1}\colon P_{\Sigma}(\ccat^{\omega}) \rightarrow \ccat$. This is a left adjoint, with right adjoint given by $\widecheck{\nu}$, so it is enough to verify that $\euscr{K} \subseteq \ker(\tau^{-1})$. 

The latter is also localizing, so it is sufficient to check that $k \in \ker(\tau^{-1})$ for any $k \in K$. As the latter is contained in the heart, we have
\[
\Omega (\tau^{-1}k) \simeq \tau^{-1}(\Omega k) = 0
\]
Since $\ccat$ is stable, this means $\tau^{-1}k = 0$ as needed. 
\end{proof}

Finally, we are ready to prove the general case. 

\begin{proof}[Proof of \cref{theorem:grothendieck_derived_infty_cat_of_h_universal_property}:]
Let $H\colon \ccat \rightarrow \acat$ be an arbitrary adapted Grothendieck homology theory. As in the proof of \cref{proposition:brown_representability}, we can write $\ccat$ as a localization
\[
L\colon \ccat_{0} \rightarrow \ccat
\]
of a compactly generated, presentable stable $\infty$-category $\ccat_{0}$. The composite 
\[
\ccat_{0} \rightarrow \ccat \rightarrow \acat
\]
is also adapted Grothendieck, and by \cref{proposition:unsep_grothendieck_exists_when_ccat_is_compactly_generated} there exists an unseparated derived $\infty$-category $\widecheck{\dcat}(\ccat_{0})$ satisfying the conditions of \cref{theorem:grothendieck_derived_infty_cat_of_h_universal_property}. 

Intuitively, the main issue is that $\widecheck{\nu}\colon \ccat_{0} \rightarrow \widecheck{\dcat}(\ccat_{0})$ is fully faithful, and so cannot factor through $\ccat$ other than in the trivial case of $\ccat_{0} = \ccat$. To fix this, we will take the largest possible quotient such that the composite functor from $\ccat_{0}$ factors through $\ccat$. 

Let $\euscr{K} \colonequals \ker(L)$, this is an accessible, stable subcategory of $\ccat_{0}$. The composite 
\[
\euscr{K} \rightarrow \ccat_{0} \rightarrow \widecheck{\dcat}(\ccat_{0})
\]
is left exact; we claim that it is also right exact. To see this, we have to check that for any $k \rightarrow k^{\prime}$ in $\euscr{K}$, 
\[
\pi_{0} \widecheck{\nu}(k) \rightarrow \pi_{0} \widecheck{\nu}(k)
\]
is an epimorphism in $\acat$. Since $\widecheck{\nu}$ is a prestable enhancement of $H \circ L$, we can identify the above map with $H(L(k)) \rightarrow H(L(k^{\prime}))$ and as $k, k^{\prime} \in \ker(L)$, this is a map between zero objects and so an epimorphism, as needed. 

It follows that the essential image $\widecheck{\nu}(\euscr{K})$ is a localizing subcategory of $\widecheck{\dcat}(\ccat_{0})$; so that we can consider the quotient 
\[
\widecheck{\dcat}(\ccat) \colonequals \widecheck{\dcat}(\ccat_{0}) / \widecheck{\nu}(\euscr{K}).
\]
By construction, the composite 
\begin{equation}
\label{equation:nucheck_for_a_quotient}
\ccat_{0} \rightarrow \widecheck{\dcat}(\ccat_{0}) \rightarrow \widecheck{\dcat}(\ccat)
\end{equation}
annihilates $\ker(L)$ and so uniquely factors through $\ccat$. We claim that this unique factorization $\ccat \rightarrow \widecheck{\dcat}(\ccat)$ satisfies the needed properties of the unseparated derived $\infty$-category of $H\colon \ccat \rightarrow \acat$. 

The fact that the composite of (\ref{equation:nucheck_for_a_quotient}) is fully faithful follows from fact that the first arrow is and that 
\[
\Map_{\widecheck{\dcat}(\ccat_{0})}(\widecheck{\nu}(k), \widecheck{\nu}(Rc)) \simeq \Map_{\ccat_{0}}(k, Rc) = 0
\]
for any $k \in \euscr{K}$ and any $c \in \ccat$, where $R\colon \ccat \hookrightarrow \ccat_{0}$ is the fully faithful right adjoint to the localization. Thus, $\widecheck{\nu}(Rc)$ already lands in $\widecheck{\dcat}(\ccat)$ considered as a subcategory of $\widecheck{\dcat}(\ccat_{0})$ through its own right adjoint. 

Observe that since $\pi_{i} \widecheck{\nu}(k) = H(L(k))[-i] = 0 $ for any $k \in \euscr{K}$, the latter contains only $\infty$-connective objects, so that the quotient functor 
\[
\widecheck{\dcat}(\ccat_{0}) \rightarrow \widecheck{\dcat}(\ccat)
\]
is an equivalence on the hearts. We then have $\widecheck{\dcat}(\ccat)^{\heartsuit} \simeq \acat$, as needed. 

Finally, by construction and the universal property of $\widecheck{\dcat}(\ccat_{0})$ which we have already verified in \cref{proposition:unsep_grothendieck_exists_when_ccat_is_compactly_generated}, $\widecheck{\dcat}(\ccat)$ classifies a pair of an exact, cocontinuous functor $G_{0}\colon \acat \rightarrow \dcat^{\heartsuit}$ together with a small coproduct-preserving prestable enhancement $\euscr{G}\colon \ccat_{0} \rightarrow \dcat$ of $G_{0} \circ H \circ L$ with the property that $\euscr{G}$ annihilates $\ker(L)$. The latter is equivalent to asking that $\euscr{G}$ uniquely factor through $L\colon \ccat_{0} \rightarrow \ccat$, giving the needed universal property. 
\end{proof}

\begin{remark}[Separated and complete derived $\infty$-categories]
\label{remark:separated_and_complete_derived_infty_cats_of_a_homology_theory}
As in the abelian case discussed in \cref{remark:variants_on_da_in_grothendieck_case}, there are two further natural variants on the unseparated derived $\infty$-category $\widecheck{\dcat}(\ccat)$ associated a Grothendieck homology theory $H\colon \ccat \rightarrow \acat$, namely 
\begin{enumerate}
    \item the \emph{(separated) derived $\infty$-category} $\dcat(\ccat)$, which is given by the separated quotient of $\widecheck{\dcat}(\ccat)$ as in \cite{lurie_spectral_algebraic_geometry}[C.3.6.1]
    \item the \emph{complete derived $\infty$-category} $\widehat{\dcat}(\ccat)$, given by completion as in \cite{lurie_spectral_algebraic_geometry}[C.3.6.3].
\end{enumerate}

Since the separated quotient and completion constructions are left adjoint to the inclusions of the respective subcategories into the $\infty$-category of Grothendieck prestable $\infty$-categories and exact, cocontinuous functors, we deduce from \cref{theorem:grothendieck_derived_infty_cat_of_h_universal_property} that $\dcat(\ccat)$ and $\widehat{\dcat}(\ccat)$ have analogous universal properties among all, respectively, separated and complete Grothendieck prestable $\infty$-categories. 
\end{remark}

\begin{remark}
Note that in \cref{definition:complete_derived_infty_cat} we already constructed a complete derived $\infty$-category under the assumption that $H\colon \ccat \rightarrow \acat$ is homologically finitary. If $H$ is also Grothendieck, then this constructions agrees with the one of \cref{remark:separated_and_complete_derived_infty_cats_of_a_homology_theory} right above, as a consequence of \cref{proposition:complete_derived_cat_of_stable_cat_a_completion_of_finite_one} and \cref{proposition:comparison_between_perfect_and_unseparated_derived_infty_cats} below.
\end{remark}

It is natural to ask how our construction of $\widecheck{\dcat}(\ccat)$ associated to a Grothendieck homology theory is related to the perfect derived $\infty$-category of \cref{definition:perfect_derived_cat_of_stable_cat}. We claim there is a canonical comparison functor. 

To see this, observe that $\widecheck{\nu}\colon \ccat \rightarrow \widecheck{\dcat}(\ccat)$ is a prestable enhancement to $H$ and so by the universal property of the perfect derived $\infty$-category proven in \cref{theorem:universal_property_of_finite_derived_category} is classified by a unique exact functor $\euscr{L}$ inducing equivalence on the hearts and such that $\euscr{L} \circ \nu \simeq \widecheck{\nu}$.

\begin{proposition}[Comparison between perfect and unseparated derived $\infty$-categories]
\label{proposition:comparison_between_perfect_and_unseparated_derived_infty_cats}
The exact functor $\euscr{L}\colon \dcat^{\omega}(\ccat) \rightarrow \widecheck{\dcat}(\ccat)$ is a fully faithful embedding whose essential image is finite colimit closure of $\widecheck{\nu}(c)$ for $c \in \ccat$.
\end{proposition}

\begin{proof}
Since $\nu(c)$ generate the perfect derived $\infty$-category under colimits, we only have to verify that $\euscr{L}$ is fully faithful. As $\widecheck{\nu}$ is fully faithful, by the criterion of  \cref{proposition:criterion_for_functor_out_of_perfect_derived_cat_to_be_ff} it is enough to verify that 
\[
\Ext^{s}_{\widecheck{\dcat}(\ccat)}(\widecheck{\nu}(c), i) \colonequals \pi_{0} \Map_{\widecheck{\dcat}(\ccat)}(\widecheck{\nu}(c), \Sigma^{s} i) = 0
\]
for any $c \in \ccat$, $i \in \acat^{inj}$ and $s > 0$.

Let $\ccat^{\prime}$ be a small subcategory of $\ccat$ which generates it under filtered colimits and which contains the chosen $c \in \ccat$. We then have a localization $L\colon \ccat_{0} \rightarrow \ccat$, where $\ccat_{0} \colonequals \Ind(\ccat^{\prime})$ is compactly-generated and moreover has the property that if $R$ denotes the right adjoint, then $Rc$ is compact. 

By retracing the construction of the unseparated derived $\infty$-category, we see that $\widecheck{\dcat}(\ccat)$ can be identified with a certain subcategory of local objects in $P_{\Sigma}(\ccat_{0}^{\omega}) \simeq P_{\Sigma}(\ccat^{\prime})$. One can check that under this description, $\widecheck{\nu}(c)$ corresponds to the presheaf on $d \in \ccat^{\prime}$ given by 
\[
\widecheck{\nu}(c)(d) \colonequals \Map_{\ccat}(d, c) 
\]
and likewise, $\Sigma^{s} i$ corresponds to the presheaf 
\[
(\Sigma^{s} i)(d) \colonequals \mathrm{B}^{s}(\Hom_{\acat}(H(d), i)),
\]
both of which are already local. We thus have that 
\[
\Map_{\widecheck{\dcat}(\ccat)}(\widecheck{\nu}(c), \Sigma^s i) \simeq \mathrm{B}^{s}(\Hom_{\acat}(H(c), i))
\]
by the Yoneda lemma, which is connected for $s > 0$ as needed.
\end{proof}

\subsection{The case of modules and synthetic spectra}
\label{subsection:case_of_modules_and_synthetic_spectra}

Due to its generality, the definition of the perfect derived $\infty$-category $\dcat^{\omega}(\ccat)$ associated to an adapted $H\colon \ccat \rightarrow \acat$ or its  unseparated Grothendieck variant studied in \S\ref{subsection:digression_derived_cat_in_grothendieck_abelian_case}, is somewhat opaque. In this section, we show that for specific choices of homology theories, these recover many $\infty$-categories studied previously in the literature, as well as encapsulating their important properties into a coherent framework. 

On a different note, these specific examples are somewhat easier to understand than the general case, and we hope that understanding our constructions in this more classical context can illuminate our arguments. The two cases we will discuss is that of modules over a ring and of an Adams-type homology theory, starting with the first one. 

\begin{example}
Let $R$ be an $\mathbf{E}_1$-algebra in spectra, and for brevity let us write
\[
\Mod_{R} \colonequals \Mod_{R}(\spectra)
\]
and
\[
\Mod_{R_{*}} \colonequals \Mod_{R_{*}}(\mathrm{gr}\abeliangroups)
\]
for the corresponding module categories; this should not lead to much confusion since one is stable and the other abelian. Note that taking homotopy groups defines a functor
\[
\pi_{*}\colon \Mod_{R} \rightarrow \Mod_{R_{*}}
\]
which is canonically a homology theory. 
\end{example}

\begin{lemma}
\label{lemma:homotopy_of_r_module_an_adapted_homology_theory}
The homology theory $\pi_{*}\colon \Mod_{R} \rightarrow \Mod_{R_{*}}$ is adapted. 
\end{lemma}

\begin{proof}
By Brown representability, for any injective $I_{*} \in \Mod_{R_{*}}$ there exists an $R$-module $R(I)$ such that 
\[
[M, R(I)]_{R} \simeq \Hom_{R_{*}}(\pi_{*}M, I_{*}).
\]
Taking $M = R^{s}$ for $s \in \mathbb{Z}$, we see that $\pi_{*} R(I) \simeq I_{*}$ as needed. 
\end{proof}
Slightly more surprisingly, the homology theory $\pi_{*}\colon \Mod_{R} \rightarrow \Mod_{R_{*}}$ is also co-adapted; that is, $\Mod_{R}^{op} \rightarrow \Mod_{R_{*}}^{op}$ is also adapted. Explicitly, this is saying that $\Mod_{R_{*}}$ has enough projectives and that any projective $R_{*}$-module can be lifted to an $R$-module, namely a direct summand of a suitable free module.

This leads for any $M, N \in \Mod_{R}$ to an appropriate ``co-Adams'' spectral sequence of signature 
\[
\Ext^{s, t}_{R_{*}}(M_{*}, N_{*}) \Rightarrow [M, N]_{R}^{s-t}
\]
obtained by resolving $M$ using projectives, which is just the classical universal coefficient spectral sequence as in \cite{higher_algebra}[7.2.1.24] or \cite{elmendorf1997rings}[IV.5]. Note that by a variation on the standard balancing argument from homological algebra, this will coincide with the Adams spectral sequence obtained by injectively resolving $N$. 

As a motivation for our description, recall that if $\acat$ is a Grothendieck abelian category generated under colimits by its subcategory $\acat^{cp}$ of compact projectives, then its derived $\infty$-category is already complete, and can be described alternatively as $\dcat(\acat) \simeq P_{\Sigma}(\acat^{cp})$, the $\infty$-category of product preserving presheaves \cite{higher_algebra}[1.3.3.14, 1.3.5.24]. In particular, if we consider modules over the graded ring $R_{*}$, then $\dcat(\Mod_{R_{*}}) \simeq P_{\Sigma}(\Mod_{R_{*}}^{ff})$, where $\Mod_{R_{*}}^{ff}$ is the subcategory of free modules of finite rank. 

The above suggests that there should be an alternative projective-based construction of the derived $\infty$-category associated to $\pi_{*}\colon \Mod_{R} \rightarrow \Mod_{R_{*}}$, which is indeed the case. 

\begin{construction}
Let $\Mod_{R}^{ff}$ denote the subcategory of finite free $R$-modules; that is, those which are finite sums of $\Sigma^{n} R$ for varying $n \in \mathbb{Z}$. Then, the $\infty$-category $P_{\Sigma}(\Mod_{R}^{ff})$ is Grothendieck prestable \cite{lurie_spectral_algebraic_geometry}[C.1.5.10].

Let us denote the restricted Yoneda embedding by $\nu_{\Sigma}\colon \Mod_{R} \rightarrow P_{\Sigma}(\Mod_{R}^{ff})$ to distinguish it from the synthetic analogue construction associated to the perfect derived $\infty$-category of \cref{definition:perfect_derived_cat_of_stable_cat}. By construction, it is left exact and preserves all small coproducts, since all finite free $R$-modules are compact. Moreover, one can show that it is fully faithful as in \cite{lurie_hopkins_brauer_group}[4.2.5].

Using a Lawvere theory argument we see that 
\[
P_{\Sigma}(\Mod_{R}^{ff})^{\heartsuit} \simeq P_{\Sigma}(\Mod_{R}^{ff}, \abeliangroups) \simeq \Mod_{R_{*}}
\]
can be identified with the category of $R_{*}$-modules and moreover that $\pi^{\heartsuit}_{0} \nu_{\Sigma}(M) \simeq \pi_{*}M$, where on the right hand side we have the heart-valued homotopy group. In other words, $\nu_{\Sigma}$ is a prestable enhancement of $\pi_{*}\colon \Mod_{R} \rightarrow \Mod_{R_{*}}$.
\end{construction}
Since $\nu_{\Sigma}\colon \Mod_{R} \rightarrow P_{\Sigma}(\Mod_{R}^{ff})$ is a prestable enhancement of the homotopy group functor, by  \cref{theorem:universal_property_of_finite_derived_category} there exists a unique exact functor $\euscr{L}$ out of the perfect derived $\infty$-category inducing an equivalence on the hearts and such that $\euscr{L} \circ \nu \simeq \nu_{\Sigma}$. 

\begin{proposition}
\label{proposition:product_pres_presheaves_on_finite_projectives_the_complete_derived_cat}
The induced exact functor $\euscr{L}\colon \dcat^{\omega}(\Mod_{R}) \rightarrow P_{\Sigma}(\Mod_{R}^{ff})$ is fully faithful and induces an equivalence $P_{\Sigma}(\Mod_{R}^{ff}) \simeq \widehat{\dcat}(\Mod_{R})$ with the complete derived $\infty$-category. 
\end{proposition}

\begin{proof}
Since the given exact functor is an equivalence on the hearts by construction, if it is fully faithful it will be an equivalence on subcategories of bounded objects. Since $P_{\Sigma}(\Mod_{R}^{ff})$ is complete, as it has Postnikov towers computed levelwise, it will then be identified with with the completion of $\dcat^{\omega}(\Mod_{R})$. 

To show that $\euscr{L}$ is fully faithful, we can use the criterion of \cref{proposition:criterion_for_functor_out_of_perfect_derived_cat_to_be_ff}. As we already observed $\nu_{\Sigma}$ is fully faithful, by \cref{remark:intuition_about_second_condition_in_detection_of_perfect_derived_cat}, it is enough to show that $P_{\Sigma}(\Mod_{R}^{ff})$ is related by a suitable adjunction with the derived $\infty$-category of $R_{*}$-modules. The necessary left adjoint is obtained by the left Kan extension
\[
P_{\Sigma}(\Mod_{R}^{ff}) \rightarrow P_{\Sigma}(\Mod_{R_{*}}^{ff}) \simeq \dcat(\Mod_{R_{*}})
\]
of the restriction $\pi_{*} |_{\Mod_{R}^{ff}}$, and it takes $\nu_{\Sigma}(M)$ to a discrete object by the same argument as in \cref{proposition:spiral_cofibre_sequence}. 
\end{proof}

\begin{remark}
Note that the homology theory $\pi_{*}\colon \Mod_{R} \rightarrow \Mod_{R_{*}}$ is Grothendieck in the sense of \cref{definition::grothendieck_homology_theory}, so that in \cref{remark:separated_and_complete_derived_infty_cats_of_a_homology_theory} we have associated to it the \emph{separated derived $\infty$-category} $\dcat(\Mod_{R})$.

Since $\nu_{\Sigma}$ preserves arbitrary coproducts and $P_{\Sigma}(\Mod_{R}^{ff})$ is complete, in particular separated, by \cref{theorem:grothendieck_derived_infty_cat_of_h_universal_property} we have a canonical comparison functor $\dcat(\Mod_{R}) \rightarrow P_{\Sigma}(\Mod_{R}^{ff})$. One can check that it is an equivalence by verifying that the source is complete using argument analogous to those of \S\ref{subsection:completing_derived_infty_cats}, and observing that both the source and target can be identified with Postnikov completions of $\dcat^{\omega}(\Mod_{R})$, as they admit functors out of it which are equivalences on subcategories of bounded objects by, respectively, \cref{proposition:product_pres_presheaves_on_finite_projectives_the_complete_derived_cat} and \cref{proposition:comparison_between_perfect_and_unseparated_derived_infty_cats}.
\end{remark}

It follows from \cref{proposition:product_pres_presheaves_on_finite_projectives_the_complete_derived_cat} that in the important case of $\pi_{*}\colon \Mod_{R} \rightarrow \Mod_{R_{*}}$, our abstract formalism yields the $\infty$-category of product-preserving presheaves on finite free modules, the natural guess for the $\mathbf{E}_{1}$-ring spectrum analogue of the derived $\infty$-category the discrete graded ring $R_{*}$, for which we have the same constructions recover the classical derived $\infty$-category $\dcat(\Mod_{R_{*}})$. 

Note that the idea to use product-preserving presheaves on a suitable choice of compact generators as a suitable place for $\infty$-categorical resolutions goes back to the classical work of Blanc, Dwyer, Goerss, Kan and Stover \cite{blanc2004realization}, \cite{dwyer1993e2} using $E_{2}$-model structures. These ideas, and the resulting obstruction theory, were translated into the language of product-preserving presheaves in the work of Hopkins-Lurie \cite{lurie_hopkins_brauer_group} and the second author \cite{pstrkagowski2017moduli}. 

Let us discuss our second  example, the case of an Adams-type homology theory $E$. In this case, Goerss and Hopkins constructed an appropriate model category of resolutions using simplicial spectra in \cite{goerss2014hopkins}. Translating these ideas into the $\infty$-categorical framework, one obtains the the $\infty$-category of $E$-based (connective) synthetic spectra, which we will denote by $\synspectra_{E}$ \cite{pstrkagowski2018synthetic}.

The $\infty$-category $\synspectra_{E}$ is Grothendieck prestable, with a canonical equivalence
\[
\synspectra_{E}^{\heartsuit} \simeq \ComodE
\]
between the heart and $E_{*}E$-comodules and a small coproduct-preserving prestable enhancement $\nu_{E}\colon \spectra \rightarrow \synspectra_{E}$ of the $E_{*}$-homology functor \cite{pstrkagowski2018synthetic}[2.9, 4.4, 4.16, 4.21]. It follows from \cref{theorem:grothendieck_derived_infty_cat_of_h_universal_property} that we obtain a canonical cocontinuous, exact functor 
\[
\widecheck{\dcat}(\spectra) \rightarrow \synspectra_{E}
\]
from the unseparated derived $\infty$-category associated to the homology theory $E_{*}\colon \spectra \rightarrow \ComodE$. This is not too far from being an equivalence of $\infty$-categories, as the following shows. 

\begin{proposition}[Comparison with synthetic spectra]
\label{proposition:comparison_between_dcat_and_synthetic_spectra}
The canonical functor $\widecheck{\dcat}(\spectra) \rightarrow \synspectra_{E}$ induces an equivalence between separated quotients; that is, an equivalence $\dcat(\spectra) \simeq \synspectra_{E}^{h}$ between the separated derived $\infty$-category of $E_{*}$ and the $\infty$-category of hypercomplete synthetic spectra. 
\end{proposition}

\begin{proof}
Let us recall that $\synspectra_{E}^{h}$ is by definition the $\infty$-category of product-preserving, hypercomplete sheaves on $\spectra^{fp}$, the $\infty$-category of finite spectra with projective $E_{*}$-homology, where we consider the latter as equipped with the $E_{*}$-surjection Grothendieck topology. In particular, we have a canonical equivalence
\[
Sh_{\Sigma}(\spectra^{fp}, \sets) \simeq \ComodE,
\]
between product-preserving sheaves of sets and comodules, see \cite{pstrkagowski2018synthetic}[3.26].

On the other side, since separated quotients are determined by localizing subcategories of the heart and $\spectra$ is compactly-generated, the description in the proof of \cref{proposition:unsep_grothendieck_exists_when_ccat_is_compactly_generated} shows that $\dcat(\spectra)$ is the unique separated quotient of $P_{\Sigma}(\spectra^{\omega})$ such that $\dcat(\spectra)^{\heartsuit} \simeq \ComodE$. 

There is a natural inclusion of additive $\infty$-sites $\spectra^{fp} \hookrightarrow \spectra^{\omega}$ of finite spectra with projective $E_{*}$-homology into all finite spectra, both considered with the $E_{*}$-surjection topology. By the Adams condition, any finite spectrum admits an $E_{*}$-surjection from a finite one with $E_{*}$-projective, and we deduce that this is an inclusion of a basis of the pretopology and so induces a further equivalence 
\[
Sh_{\Sigma}(\spectra^{\omega}, \sets) \simeq Sh_{\Sigma}(\spectra^{fp}, \sets) \simeq \ComodE,
\]
see \cite{lurie2018ultracategories}[B.6.4]. It follows that we have a canonical identification $\dcat(\spectra) \simeq Sh_{\Sigma}^{h}(\spectra^{\omega})$ of the separated derived $\infty$-category with hypercomplete sheaves on finite spectra, as both are separated localizations of $P_{\Sigma}(\spectra^{\omega})$ with the same heart. 

Since the canonical functor $\dcat(\spectra) \rightarrow \synspectra_{E}^{h}$ is obtained by left Kan extension along the synthetic analogue construction, by tracing through definitions we see that it can be identified with the functor between $\infty$-categories of hypercomplete sheaves of spaces 
\[
Sh^{h}_{\Sigma}(\spectra^{\omega}) \rightarrow Sh^{h}_{\Sigma}(\spectra^{fp})
\]
left adjoint to restriction along $\spectra^{fp} \hookrightarrow \spectra^{\omega}$. As the latter is an inclusion of a basis and hence induces an equivalence between sheaves of sets, it also yields an equivalence between $\infty$-categories of hypercomplete sheaves of spaces \cite{barwick2018exodromy}[3.12.11], ending the argument. 
\end{proof}

\begin{remark}
It is natural to ask if the canonical functor $\widecheck{\dcat}(\spectra) \rightarrow \synspectra_{E}$ is an equivalence even before passing to separated quotients. We expect that this is not the case in general: the definition of $\synspectra_{E}$ privileges finite spectra with projective $E_{*}$-homology among all finite spectra in a way which $\widecheck{\dcat}(\spectra)$ does not. This distinction happens by design, as taking $E_{*}$-homology on spectra with projective homology satisfies the K\"{u}nneth isomorphism, leading to a symmetric monoidal structure on $\synspectra_{E}$. 

On the other hand, if $E$ is a field, such as $H\mathbb{F}_{p}$ or Morava $K$-theory, so that $\spectra^{fp} = \spectra^{\omega}$, then we have $\widecheck{\dcat}(\spectra) \simeq \synspectra_{E}$ even before passing to separated quotients, by a variation on the argument given in the proof of \cref{proposition:comparison_between_dcat_and_synthetic_spectra}
\end{remark}

\section{Franke's algebraicity conjecture}
\label{section:proving_algebraicity}

The goal of this chapter is to prove the main theorem, Franke's algebraicity conjecture. 

\subsection{The algebraic model} 
\label{subsection:algebraic_model} 

In this section we describe, for each well-behaved abelian category, a canonical example of a stable $\infty$-category equipped with an adapted homology theory valued in $\acat$, the periodic derived $\infty$-category $\dcat^{per}(\acat)$. The subject of Franke's conjecture will be a comparison between the homotopy category of this derived category and of any other stable $\infty$-category which also admits an $\acat$-valued adapted homology theory. 

\begin{definition}
Let $\acat$ be a locally graded abelian category. A \emph{differential object} is a pair $(a, \partial)$, where $a \in \acat$ and $\partial\colon a \rightarrow a[1]$ is a differential in the sense that $\partial^{2} \colonequals \partial[1] \circ \partial = 0$. 
\end{definition}

\begin{notation}
We denote the category of differential objects by $d\acat$. Note that it is an abelian category, with a forgetful exact functor to $\acat$. 
\end{notation}

\begin{remark}
\label{remark:differential_objects_same_as_periodic_chain_complexes}
In his original manuscript, instead of differential objects Franke considers what he calls \emph{$([1]_{\acat}, 1)$-periodic chain complexes.}, which are chain complexes in $\acat$ equipped with a suitable periodicity isomorphism \cite{franke1996uniqueness}. Any differential object $(a, \partial)$ determines such a chain complex of the form 
\[
\ldots \rightarrow a[-1] \rightarrow a \rightarrow a[1] \rightarrow \ldots
\]
and one can show this yields an equivalence of categories \cite{patchkoria2017exotic}, \cite{pstragowski_chromatic_homotopy_algebraic}[3.3].
\end{remark}

\begin{notation}
If $(a, \partial)$ is a differential object, then its \emph{homology} is given by 
\[
H(a,\partial) \colonequals \mathrm{ker}(\partial)/\mathrm{im}(\partial[-1])
\]
A map $(a, \partial_{a}) \rightarrow (b, \partial_{b})$ of differential objects is a \emph{quasi-isomorphism} if it is an isomorphism on homology.
\end{notation}

\begin{remark}
\label{remark:short_exact_sequence_of_dobjects_gives_les_of_homology}
A short exact sequence $0 \rightarrow a \rightarrow b \rightarrow c \rightarrow 0$ of differential objects determines a short exact sequence of chain complexes as in \cref{remark:differential_objects_same_as_periodic_chain_complexes}, and hence a long exact sequence
\[
\ldots \rightarrow H(c)[-1] \rightarrow H(a) \rightarrow H(b) \rightarrow H(c) \rightarrow H(a)[1] \rightarrow \ldots
\]
on homology. 
\end{remark}
We will define the needed periodic derived $\infty$-category $\dcat^{per}(\acat)$ as a $\infty$-categorical localization of the category of differential objects at quasi-isomorphisms. To describe the localization explicitly, it is convenient to use the language of model categories. 

\begin{proposition}
\label{model structure periodic} 
Let $\acat$ be an abelian category with enough injectives and of finite cohomological dimension. Then, the category $d \acat$ of differential objects admits a model structure where
\begin{enumerate}
\item weak equivalences are quasi-isomorphisms and 
\item cofibrations are monomorphisms.
\end{enumerate}
\end{proposition} 

\begin{proof}
This is \cite{franke1996uniqueness}[p.~15, Proposition 3].
\end{proof}

\begin{definition}
The \emph{periodic derived $\infty$-category} of an abelian category $\acat$
\[
\dcat^{per}(\acat) \colonequals (d \acat)[W^{-1}]
\]
is the $\infty$-categorical localization of the category of differential objects, where $W$ is the class of quasi-isomorphisms. 
\end{definition}

\begin{remark}[Stability]
\label{remark:suspension_in_periodic_derived_infty_cat}
If $(a, \partial)$ is a differential object, its \emph{cone} is given by $C(a) \colonequals a \oplus a[1]$ together with the differential informally given by $\partial_{C(a)}(x, y) = (\partial(x)+y, -\partial(y))$. This is quasi-isomorphic to zero, and there is a canonical short exact sequence
\[
a \rightarrow C(a) \rightarrow a[1].
\]
Since the first map is a cofibration, this becomes a cofibre sequence in $\dcat^{per}(\acat)$ and we deduce that 
\[
\Sigma (a, \partial) = (a[1], -\partial)
\]
in the periodic derived $\infty$-category. Thus, the suspension functor is induced by the local grading of $\acat$; in particular, it is an equivalence and $\dcat^{per}(\acat)$ is stable.
\end{remark}

\begin{remark}
Identifying differential objects with periodic chain complexes using \cref{remark:differential_objects_same_as_periodic_chain_complexes}, one can show that $\dcat^{per}(\acat)$ coincides up to equivalence with the differential graded nerve of a suitable dg-category, providing an explicit way to compute mapping spaces, see \cite{higher_algebra}[1.3.4]. 
\end{remark}

Observe that by definition, the homology functor 
\[
H\colon d \acat \rightarrow \acat
\]
inverts quasi-isomorphisms and hence descends to a functor on the periodic derived $\infty$-category. 

\begin{proposition}
\label{proposition:homology_functor_on_periodic_derived_cat_is_a_homology_theory}
The functor $H \colon \dcat^{per}(\acat) \to \acat$ is a conservative, adapted homology theory.
\end{proposition}

\begin{proof} 
Compatibility with local grading follows from \cref{remark:suspension_in_periodic_derived_infty_cat}, and long exact sequence of homology from  \cref{remark:short_exact_sequence_of_dobjects_gives_les_of_homology} as any cofibre sequence in $\dcat^{per}(\acat)$ can be represented by a short exact sequence in $d \acat$. Similarly, conservativity is clear, as $H(a, \partial) = 0$ if and only if the canonical map $(a, \partial) \rightarrow (0, 0)$ of differential objects is a quasi-isomorphism and hence becomes an equivalence in $\dcat^{per}(\acat)$.

To prove adaptedness, we have to verify that for any injective $i \in \acat$, 
\[
\Hom_{\acat}(H(-), i)
\]
is representable in the homotopy category $h \dcat^{per}(\acat)$ and that the representing object has $i$ as homology. We claim that this functor is represented by $(i, 0)$. 

By \cite{franke1996uniqueness}[p.~15, Proposition 3], $(i, 0)$ is fibrant as a differential object and hence standard arguments about model categories show that 
\[
[(a,d), (i,0)]
\]
can be computed as left homotopy classes of morphisms of differential objects in the sense of model categories. These classes can be identified with $\Hom_{\acat}(H(a,d), i)$ as needed.
\end{proof}

\subsection{Splittings of abelian categories and the Bousfield functor} 
\label{subsection:splitting_of_abelian_categories_and_the_bousfield_functor}

One special property of the adapted homology theory $H\colon \dcat^{per}(\acat) \rightarrow \acat$ of \cref{proposition:homology_functor_on_periodic_derived_cat_is_a_homology_theory} is that it has a one-sided inverse; more precisely, the association 
\begin{equation}
\label{equation:formula_for_bousfield_functor_in_periodic_derived_cat}
a \mapsto (a, 0)
\end{equation}
that sends an $a \in \acat$ to itself equipped with a zero differential, descends to a functor 
\[
\acat \rightarrow \dcat^{per}(\acat)
\]
of $\infty$-categories which is a section of the homology functor. 
Since Franke's conjecture, which we will prove below in \S\ref{subsection:complete_proof_of_frankes_conjecture}, asserts that any adapted homology theory $H\colon \ccat \rightarrow \acat$ with sufficiently simple target induces an equivalence $h \ccat \simeq h \dcat^{per}(\acat)$, it is natural for the first step towards the proof to be establish that under certain assumptions on $\acat$, we have a partial inverse to $H$ for arbitrary $\ccat$. This is the goal of this short section.

The simplicity assumption we will need is as follows. 

\begin{definition}[{\cite{franke1996uniqueness}[p.~56]}]
\label{definition:splitting_of_a_locally_graded_abelian_category}
Let $\acat$ be an abelian category with a local grading $[1]\colon \acat \rightarrow \acat$. A \emph{splitting of order $q+1$} of $\acat$ is a collection of Serre subcategories 
\[
\acat_{\phi} \subseteq \acat
\]
indexed by $\phi \in \mathbb{Z}/(q+1)$ such that
\begin{enumerate}
    \item $[k](\acat_{\phi}) \subseteq \acat_{\phi+k}$ for any $k \in \mathbb{Z}$ and 
    \item the functor $\prod_{\phi \in \mathbb{Z}/(q+1)} \acat_{\phi} \rightarrow \acat$ given by $(a_{\phi}) \mapsto \bigoplus_{\phi} a_{\phi}$ is an equivalence of categories.
\end{enumerate}
\end{definition}

\begin{notation}
We will refer to objects in the essential image of $\acat_{\phi}$ as \emph{pure of weight $\phi$}. The second condition above says that any $a \in \acat$ has a canonical weight decomposition
\[
a \simeq a_{0} \oplus a_{1} \ldots \oplus a_{q}
\]
There are no non-zero maps between objects in different weights. 
\end{notation}

\begin{remark}
Note that the splitting decomposition restricts to an equivalence of categories
\[
\acat^{inj} \simeq \acat^{inj}_{0} \times \ldots \times \acat^{inj}_{q}
\]
In other words, every injective of $\acat$ decomposes essentially uniquely as a finite sum of injectives of pure weight. 
\end{remark}

For the remainder of the section, let us assume that $H\colon \ccat \rightarrow \acat$ is an adapted homology theory and that $\acat$ has a chosen splitting of order $q+1$. Recall from \cref{definition:injectives_in_a_stable_category} that we say that $c \in \ccat$ is $H$-injective if 
\begin{enumerate}
    \item $H(c) \in \acat$ is injective and 
    \item $[d, c] \rightarrow \Hom_{\acat}(H(d), H(c))$ is a bijection for any $d \in \ccat$. 
\end{enumerate}
The adaptedness of $H$ tells us that every injective of $\acat$ lifts uniquely up to homotopy equivalence to an $H$-injective of $\ccat$. 

\begin{definition}
We say an $H$-injective $c \in \ccat$ is \emph{pure of weight $\phi$} if $H(c) \in \acat_{\phi}$. We denote the subcategory of $H$-injectives of fixed weight by  $\ccat^{inj}_{\phi}$; note that by definition homology restricts to a functor $\ccat^{inj}_{\phi} \rightarrow \acat^{inj}_{\phi}$.
\end{definition}

\begin{lemma}
\label{lemma:homology_on_pure_injectives_equivalence_between_higher_homotopy_categories}
For every $\phi \in \mathbb{Z}/(q+1)$, the functor $H\colon \ccat^{inj}_{\phi} \rightarrow \acat^{inj}_{\phi}$ induces an equivalence
\[
h_{q+1} \ccat^{inj}_{\phi} \simeq \acat^{inj}_{\phi}
\]
between the homotopy $(q+1)$-category of injectives of weight $\phi$ in $\ccat$ and the category of injectives of the same weight in $\acat$. In particular, $h_{q+1} \ccat^{inj}_{\phi}$ is a $1$-category. 
\end{lemma}

\begin{proof}
Suppose that $c, d \in \ccat^{inj}_{\phi}$. Then, since both are injective we have
\[
\pi_{k} \Map_{\ccat}(c, d) \simeq \Hom_{\acat}(H(c)[k], H(d))
\]
for any $k \geq 0$. Taking $k = 0$ we see that $h \ccat^{inj}_{\phi} \simeq \acat^{inj}_{\phi}$.

Considering the above, we only have to show that if $k > 0$ and $k \leq q$, the right hand side vanishes, so that $h_{q+1} \ccat^{inj}_{\phi} \simeq h_{1} \ccat^{inj}_{\phi}$. This is clear, as then $H(c)[k]$ and $H(d)$ are of different weight, as $\phi + k \neq \phi \in \mathbb{Z}/(q+1)$. 
\end{proof}

\begin{construction}[Bousfield functor]
\label{construction:bousfield_functor_associated_to_splitting}
By \cref{lemma:homology_on_pure_injectives_equivalence_between_higher_homotopy_categories}, for each each $\phi \in \mathbb{Z}/q+1$ we have an essentially unique equivalence
\[
R_{\phi}\colon \acat^{inj}_{\phi} \rightarrow h_{q+1} \ccat^{inj}_{\phi}
\]
which is inverse to the restricted homology functor. The \emph{Bousfield functor} 
\[
\beta^{inj}\colon \acat^{inj} \rightarrow h_{q+1} \ccat^{inj}
\]
is defined by 
\[
\beta^{inj}(i) \colonequals R_{0}(i_{0}) \oplus \ldots \oplus R_{q}(i_{q}),
\]
where $i \simeq i_{0} \oplus \ldots \oplus i_{q}$ is the weight decomposition induced by the chosen splitting of $\acat$. 
\end{construction}

\begin{remark}
\label{remark:bousfield_functor_partial_inverse_to_homology}
Observe that by construction the Bousfield functor is a partial inverse to the homology functor; that is, we have 
\[
H(\beta^{inj}(i)) \simeq i
\]
for any $i \in \acat^{inj}$. Moreover, it is the unique additive functor $\acat^{inj} \rightarrow h_{q+1} \ccat^{inj}$ with this property. 
\end{remark}

\begin{warning}
Note that there are important differences between the general Bousfield functor of \cref{construction:bousfield_functor_associated_to_splitting} associated to an adapted homology theory $H\colon \ccat \rightarrow \acat$ where $\acat$ admits a splitting, and the functor $\acat \rightarrow \dcat^{per}(\acat)$ associated to the periodic derived $\infty$-category induced by \cref{equation:formula_for_bousfield_functor_in_periodic_derived_cat}. Namely, 
\begin{enumerate}
    \item in the general case, the Bousfield functor is only defined on injectives, rather than all of $\acat$ and 
    \item in the general case, the functor is valued in a (higher) homotopy category, rather than the stable $\infty$-category itself. 
\end{enumerate}
Franke's conjecture asserts that this is still enough to induce an equivalence $h \ccat \simeq h \dcat^{per}(\acat)$ of homotopy categories if we know that $\acat$ is of sufficiently small cohomological dimension, so that a general object of $\acat$ is not too far from being an injective. 

Note that it is in general not reasonable to expect a stronger Bousfield functor. One can show using arguments similar to the ones employed in our proof of Franke's conjecture that if $H$ has a section \emph{as a functor of $\infty$-categories}, then we necessarily must have $\ccat \simeq \dcat^{per}(\acat)$. This is not possible for many different stable $\infty$-categories $\ccat$; for example, whenever $\ccat$ is not $H\mathbb{Z}$-linear. 
\end{warning}

\begin{remark}
Note that our further arguments toward the Franke's conjecture will only require the existence of the Bousfield functor; that is, an additive functor $\beta^{inj}\colon h_{q+1} \ccat^{inj} \rightarrow \acat^{inj}$ which is a section of $H$. It does not matter whether the functor itself comes from a splitting of $\acat$ or not. 

We think it is an interesting question to find other assumptions on either $\acat$ or $\ccat$ that would guarantee the existence of the Bousfield functor, as that would extend the reach of algebraicity results. 
\end{remark}

\subsection{Bousfield adjunction}
\label{subsection:bousfield_adjunction}

Suppose that $H\colon \ccat \rightarrow \acat$ is an adapted homology theory such that have an additive Bousfield functor 
\[
\beta^{inj}\colon \acat^{inj} \rightarrow h_{q+1} \ccat^{inj}
\]
such that $H(\beta^{inj}(i)) \simeq i$ for any $i \in \acat^{inj}$. For example, in \cref{construction:bousfield_functor_associated_to_splitting}, we have shown that this is the case whenever $\acat$ admits a splitting of order $q+1$. Our goal in this section is to show that this induces a well-behaved adjunction between $C\tau^{q+1}$-modules and the derived $\infty$-category of $\acat$. 

Observe that for any $c \in \ccat^{inj}$
\[
C\tau^{q+1} \otimes \nu(c) \simeq \nu(c)_{\leq q}
\]
is $q$-truncated. It follows that the functor 
\[
C\tau^{q+1} \otimes \nu(-)\colon \ccat \rightarrow \Mod_{C\tau^{q+1}}(\dcat^{\omega}(\ccat))
\]
uniquely factors through the homotopy $(q+1)$-category so that we can consider the composite 
\[
\acat^{inj} \rightarrow h_{q+1} \ccat \rightarrow \Mod_{C\tau^{q+1}}(\dcat^{\omega}(\ccat)).
\]
We can naturally extend this functor to all of $\acat$ using injective resolutions, in the following way. 

\begin{lemma}
\label{lemma:bousefield_functor_exists_and_is_unique}
There exists an essentially unique functor $\beta\colon \acat \rightarrow \Mod_{C\tau^{q+1}}(\dcat^{\omega}(\ccat)) $ such that
\begin{enumerate}
    \item $\beta(i) \simeq C\tau^{q+1} \otimes \nu(\beta^{inj}(i))$ for any injective $i \in \acat$ and 
    \item any cosimplicial injective resolution
\[
a \rightarrow i_{0} \rightrightarrows i_{1} \triplerightarrow \ldots
\]
of $a \in \acat$ is taken by $\beta$ to a limit diagram; that is, it induces an equivalence
\[
\beta(a) \colonequals \textnormal{Tot}(C\tau^{q+1} \otimes \nu(\beta^{inj}(i_{\bullet})).
\]
\end{enumerate}
\end{lemma}

\begin{proof}
First observe that the totalization exists since $C\tau^{q+1} \otimes \nu(\beta^{inj}(-))$ takes values in the subcategory of $q$-truncated objects, which is a $(q+1)$-category so that the totalization can be identified with a finite limit. Thus, the needed unique extension can be identified with the composite 
\[
\acat \hookrightarrow \dcat^{-}_{\leq 0}(\acat) \rightarrow \Mod_{C\tau^{q+1}}(\dcat^{\omega}(\ccat)),
\]
of the inclusion of the heart into the \emph{coconnective derived $\infty$-category} and the second arrow is the unique totalization-preserving extension of $C\tau^{q+1} \otimes \nu(\beta^{inj}(-))$, see \cite{higher_algebra}[1.3.3.8] for the dual statement about the connective derived $\infty$-category and projectives. 
\end{proof}

\begin{definition}
We call the unique functor $\beta\colon \acat \rightarrow \Mod_{C\tau^{q+1}}(\dcat^{\omega}(\ccat))$ of \cref{lemma:bousefield_functor_exists_and_is_unique} the \emph{extended Bousfield functor}. 
\end{definition}

\begin{lemma}
\label{lemma:beta_on_homotopy_groups}
Let $a \in \acat$. Then, $\beta(a)$ is $q$-truncated, and with respect to the  $t$-structure on $\Mod_{C\tau^{q+1}}(\dcat^{\omega}(\ccat))$ induced by the forgetful functor to $\dcat^{\omega}(\ccat)$ we have 
\[
\pi_{k} \beta(a) \simeq a[-k]
\]
as objects of $\acat$ for any $0 \leq k \leq q$. 
\end{lemma}

\begin{proof}
Choosing an injective resolution $a \rightarrow i_{\bullet}$ get a totalization spectral sequence
\[
H^{s} (\pi_{t} \ C\tau^{q+1} \otimes \nu(\beta^{inj}(i_{\bullet}))) \Rightarrow \pi_{t-s} \beta(a).
\]
The first page is given by 
\[
\pi_{t} \ \nu(\beta^{inj}(i_{\bullet}))_{\leq q} = \begin{cases} i_{\bullet}[-t] &\mbox{if } 0 \leq t \leq q \\
0& \mbox{otherwise } \end{cases}
\]
which immediately gives the needed result, since $i_{\bullet}$ is an injective resolution and so the relevant spectral sequence collapses on the second page. 
\end{proof}

\begin{lemma}
\label{lemma:tau_structure_on_beta}
For any $a \in \acat$, $\beta(a)$ is a potential $q$-stage in the sense of \cref{def:general_potential_l_stage}.
\end{lemma}

\begin{proof}
By \cref{lemma:beta_on_homotopy_groups}, $\beta(a)$ has abstractly the right homotopy groups to be a potential $q$-stage, we just have to verify that for any $0 \leq k < q$, the maps 
\[
\pi_{k} \beta(a)[-1] \simeq \pi_{k+1} \beta(a)
\]
induced by $\tau$ are isomorphisms. When $a = i$ is injective, this follows from the fact that 
\[
\beta(i) \colonequals C\tau^{q+1} \otimes \nu(\beta^{inj}(i)) \simeq \nu(\beta^{inj}(i))_{\leq q}
\]
and the general case follows by the naturality of the spectral sequence used in the proof of \cref{lemma:beta_on_homotopy_groups}.
\end{proof}

\begin{proposition}
The Bousfield functor uniquely extends uniquely to a right exact functor
\[
\beta^{*}\colon \dcat^{b}(\acat) \rightarrow \Mod_{C\tau^{q+1}}(\dcat^{\omega}(\ccat))
\]
of prestable $\infty$-categories. 
\end{proposition}

\begin{proof}
The perfect prestable Freyd envelope $A_{\infty}^{\omega}(\acat)$ is freely generated by representables under finite colimits, so that $\beta$ uniquely extends to a right exact functor $\beta^{*}$ defined on $A_{\infty}^{\omega}(\acat)$, as we observed in \cref{remark:universal_property_of_perfect_freyd_envelope}.

We claim that $\beta^{*}$ factors uniquely through the sheafification $L\colon A_{\infty}^{\omega}(\acat) \rightarrow \dcat^{b}(\acat)$. The maps inverted by $L$ are those whose cofibre $C$ has the property that $L(C) = 0$, and since $\beta^{*}$ is right exact it is enough to show that it annihilates objects with this property.  

Any $C \in A_{\infty}^{\omega}(\acat$) can be represented as a geometric realization (in fact, a finite skeleton) of a simplicial object $a_{\bullet}$ of $\acat$ corresponding under the Dold-Kan equivalence to a bounded chain complex by \cref{lemma:any_perfect_presheaf_over_abelian_category_rep_by_a_bounded_chain_complex}, and we will have $L(C) = 0$ precisely when this chain complex is acyclic. Since $\beta^{*}$ is right exact it commutes with forming finite skeleta of simplicial objects and we have $|\beta(a_{\bullet})| \simeq \beta^{*}C$, resulting in a realization spectral sequence
\[
H_{t} (\pi_{k} \beta(a_{\bullet})) \Rightarrow \pi_{k+t} \beta^{*}C.
\]
By \cref{lemma:beta_on_homotopy_groups}, we have an isomorphism 
\[
H_{t} (\pi_{k} \beta(a_{\bullet})) \simeq H_{t} (a_{\bullet}[-k])
\]
and since $a_{\bullet}$ was assumed to be acyclic, we deduce that the second page of this spectral sequence vanishes, so that $\beta^{*}C = 0$ as needed. Thus, we deduce that $\beta^{*}$ factors uniquely through a unique right exact functor out of $\dcat^{b}(\acat)$. 
\end{proof}

\begin{lemma}
\label{lemma:beta_upper_star_a_left_adjoint}
The functor $\beta^{*}\colon \dcat^{b}(\acat) \rightarrow \Mod_{C\tau^{q+1}}(\dcat^{\omega}(\ccat))$ is a left adjoint. 
\end{lemma}

\begin{proof}
Since $\dcat^{b}(\acat)$ is by definition an $\infty$-category of perfect sheaves on $\acat$, observe that if $\beta_{*}$ was the right adjoint, then we would have 
\[
(\beta_{*} X)(a) \simeq \Map_{\dcat^{b}(\acat)}(a, \beta_{*}X) \simeq \Map_{C\tau^{q+1} \otimes -}(\beta(a), X).
\]
It follows that for the right adjoint to exist, it is sufficient for the right hand side to define a perfect sheaf for any $X \in \Mod_{C\tau^{q+1}}(\dcat^{\omega}(\ccat))$. 

By \cref{corollary:finite_ctaun_modules_are_bounded}, any $C\tau^{q+1}$-module whose underlying presheaf is perfect is bounded. Thus, it can be obtained by iteratively taking fibres from sheaves with a single non-zero homotopy group, and which thus admit a structure of a $C\tau$-module. 

We can thus assume that $X$ is of the latter kind. Then, since $\beta(a)$ is a potential $q$-stage for any $a \in \acat$ by \cref{lemma:tau_structure_on_beta},  \cref{lemma:ctau_tensor_over_ctaun_is_discrete_fora_potential_n_stage} implies that 
\[
\Map_{C\tau^{q+1}}(\beta(a), X) \simeq \Map_{C\tau}(\beta(a)_{\leq 0}, X).
\]
Using the equivalence $ \Mod_{C\tau}(\dcat^{\omega}(\ccat)) \simeq \dcat^{b}(\acat)$ of \cref{theorem:finite_ctau_modules_same_as_derived_category} and that $\pi_{0} \beta(a) \simeq a$, we can rewrite the above as 
\begin{equation}
\label{equation:equivalence_between_ctau_modules_and_db_showing_up_through_bousfield_adjunction}
(\beta_{*} X)(a) \simeq \Map_{C\tau}(a, X) \simeq \Map_{\dcat^{b}(\acat)}(a, Y) \simeq Y(a),
\end{equation}
where $Y \in \dcat^{b}(\acat)$ is a perfect sheaf on $\acat$ corresponding to $X$ under the above equivalence of $\infty$-categories. Thus, $\beta_{*}X$ is also a perfect sheaf, as needed. 
\end{proof}

\begin{definition}
\label{definition:bousfield_adjunction}
We will denote the right adjoint whose existence is asserted by \cref{lemma:beta_upper_star_a_left_adjoint} by $\beta_{*}$ and call the resulting adjunction
\[
\beta^{*} \dashv \beta_{*}\colon \dcat^{b}(\acat) \leftrightarrows \Mod_{C\tau^{q+1}}(\dcat^{\omega}(\ccat))
\]
the \emph{Bousfield adjunction}. We will refer to the associated monad $\beta_{*} \beta^{*}$ on $\dcat^{b}(\acat)$ as the \emph{Bousfield monad}. 
\end{definition}

\begin{remark}
\label{remark:composite_of_right_adjoint_bousfield_with_forgetful_from_ctau_modules_an_equivalence}
Observe that the proof of \cref{lemma:beta_upper_star_a_left_adjoint}, namely equivalence (\ref{equation:equivalence_between_ctau_modules_and_db_showing_up_through_bousfield_adjunction}), shows that $\beta_{*}$ has the property that the composite
\[
\Mod_{C\tau}(\dcat^{\omega}(\ccat)) \rightarrow \Mod_{C\tau^{q+1}}(\dcat^{\omega}(\ccat))  \rightarrow \dcat^{b}(\acat)
\]
is an equivalence, which we can identify with the equivalence of \cref{theorem:finite_ctau_modules_same_as_derived_category}. Note that this holds despite these equivalences being induced by different functors, namely the Bousfield splitting $\beta$ and the homology functor $H\colon \ccat \rightarrow \acat$. These functors induce the same equivalence on the hearts and equivalences of derived $\infty$-categories are determined by what they do on the heart and so are somewhat rigid. 
\end{remark}

\begin{lemma}
\label{lemma:bousfield_right_adjoint_exact_and_commutes_with_homotopy_groups}
The right adjoint $\beta_{*}\colon \Mod_{C\tau^{q+1}}(\dcat^{\omega}(\ccat)) \rightarrow \dcat^{b}(\acat)$ is exact and induces an equivalence on the hearts. Thus, it commutes with $\acat$-valued homotopy groups. 
\end{lemma}

\begin{proof}
It is clear that $\beta_{*}$ induces an equivalence on the hearts, since both the forgetful functor $\Mod_{C\tau}(\dcat^{\omega}(\ccat)) \rightarrow  \Mod_{C\tau^{q+1}}(\dcat^{\omega}(\ccat))$ and its composite with $\beta_{*}$ do, the latter by \cref{remark:composite_of_right_adjoint_bousfield_with_forgetful_from_ctau_modules_an_equivalence}. We have to verify that it is exact. 

We first claim that $\beta_{*}$ preserves $1$-connected objects. Every such object can be obtained by iteratively taking fibres of maps into at least $2$-connected objects with homotopy concentrated in a single degree, which thus admit a structure of a $C\tau$-module. These fibres will be preserved by $\beta^{*}$ which is left exact. As the composite 
\[
\Mod_{C\tau}(\dcat^{\omega}(\ccat)) \rightarrow \Mod_{C\tau^{q+1}}(\dcat^{\omega}(\ccat))  \rightarrow \dcat^{b}(\acat)
\]
is an equivalence and the first functor is exact, we deduce that such $2$-connected objects with homotopy in a single degree are taken to $2$-connected objects. Thus, the fibres will be at least $1$-connected, as needed. 

As for any $X \in \dcat^{\omega}(\ccat))$ we have a cofibre sequence 
\[
X \rightarrow X_{\leq 0} \rightarrow \Sigma X_{\geq 1},
\]
where the last object is $1$-connected. It follows that 
\[
\beta_{*} X \rightarrow \beta_{*} X_{\leq 0} \rightarrow \beta_{*} \Sigma X_{\geq 1},
\]
is a fibre sequence where the last object is also $1$-connected, and the long exact sequence of homotopy groups implies that $\pi_{0} \beta_{*} X \simeq \pi_{0} \beta_{*} X_{\leq 0}$. 

Since $X_{\leq 0}$ is in the heart and $\beta_{*}$ is an equivalence on the hearts, we deduce that it commutes with taking $\pi_{0}$. As cofibre sequences can be characterized by fibre sequences where the second map is $\pi_{0}$-epimorphism, we deduce that $\beta_{*}$ is left exact, as needed. 
\end{proof}

\begin{proposition}
\label{proposition:bousfield_adjunction_monadic}
The Bousfield adjunction is monadic; that is, it induces an equivalence
\[
\Mod_{C\tau^{q+1}}(\dcat^{\omega}(\ccat)) \simeq \Mod_{\beta_{*} \beta^{*}}(\dcat^{b}(\acat)).
\]
\end{proposition}

\begin{proof}
This is the same as the proof of  \cref{theorem:finite_ctau_modules_same_as_derived_category}, but we briefly recall the argument. Since both prestable $\infty$-categories in question are bounded in the sense that every object is $k$-truncated for some $k$, it is enough to verify that 
\[
\tau_{\leq k} \Mod_{C\tau^{q+1}}(\dcat^{\omega}(\ccat)) \rightarrow \tau_{\leq k} \Mod_{\beta_{*} \beta^{*}}(\dcat^{b}(\acat))
\]
is an equivalence of $\infty$-categories for any $k \geq 0$. To each of these one can apply Barr-Beck Lurie criterion \cite{higher_algebra}[4.7.3.5] as in the proof of \cref{lemma:bounded_derived_cat_monadic_over_derived_of_ccat}, since $\beta_{*}$ is exact and commutes with homotopy groups by 
\cref{lemma:bousfield_right_adjoint_exact_and_commutes_with_homotopy_groups}.
\end{proof}

\subsection{The truncated thread monad} 
\label{subsection:the_truncated_thread_monad}

In \cref{proposition:bousfield_adjunction_monadic} we have shown that the structure of $\Mod_{C\tau^{q+1}}(\dcat^{\omega}(\ccat))$, which can be thought of as an approximation the the derived $\infty$-category of $\ccat$ and hence the latter itself, is completely described by the Bousfield monad $\beta_{*} \beta^{*}$ on $\dcat^{b}(\acat)$ of \cref{definition:bousfield_adjunction}. Our goal in this section is to give an explicit description of this monad and identify it with one of purely algebraic origin. 

Observe that since $\beta_{*}$ preserves $\acat$-valued homotopy groups, for any $a \in \acat$ we have 
\[
\beta_{*} \beta^{*}(a) \simeq a \oplus \Sigma a[-1] \oplus \ldots \Sigma^{q} a[-q].
\]
It follows by right exactness of the monad that we necessarily must have, for any $X \in \dcat^{b}(\acat)$, that 
\begin{equation}
\label{equation:bousfield_monad_as_an_endofunctor}
\beta_{*} \beta^{*} X \simeq X \oplus \Sigma X[-1] \oplus \ldots \oplus \Sigma^{q} X[-q].
\end{equation}
Thus, the structure of $\beta_{*} \beta^{*}$ as an endofunctor is quite simple, and the interesting part is identifying the monad. 

The hint to finding the right approach is given by \cref{lemma:tau_structure_on_beta}, which asserts that the homotopy groups of $\beta^{*}(a)$ for any $a \in \acat$ are connected together by $\tau$. Thus, at least at the level of homotopy groups, the Bousfield monad behaves like a ``free $\mathbb{Z}[\tau]/\tau^{q+1}$-module'' monad. 

In the presence of monoidal structure on $\ccat$ this argument can be made precise, and this is the approach taken in \cite{pstragowski_chromatic_homotopy_algebraic}. Morally, this is still true in the general case. Instead of working with algebras, we will show that $\beta_{*} \beta^{*}$ can be identified with an appropriate truncation of a free monad generated by the endofunctor $\Sigma X[-1]$. 

\begin{notation}
Similarly to our approach to sheafification taken in \cref{subsection:bounded_derived_cat_of_abelian_cat}, it will be convenient to at least a priori work in a larger $\infty$-category in which our constructions exist. We will denote by 
\[
\Ind(\dcat^{b}(\acat))
\]
the $\Ind$-completion of the bounded derived $\infty$-category of $\acat$. This is, by definition, the free $\infty$-category generated by $\dcat^{b}(\acat)$ under filtered colimits. 
\end{notation}

\begin{remark}
Since $\dcat^{b}(\acat)$ has finite colimits, $\Ind(\dcat^{b}(\acat))$ has all small colimits and is generated under them by $\acat$. One can show that it is again prestable \cite{antieau2019k}[2.13].
\end{remark}

\begin{notation}
We will denote by $S\colon \Ind(\dcat^{b}(\acat)) \rightarrow \Ind(\dcat^{b}(\acat))$ the unique cocontinuous extension of the endofunctor of $\dcat^{b}(\acat)$ given by $S(X) \colonequals \Sigma X[-1]$.
\end{notation}

\begin{definition}
\label{definition:thread_monad}
We will say a monad $T$ on $\Ind(\dcat^{b}(\acat))$ together with a natural transformation $S \rightarrow T$ of endofunctors is a \emph{thread monad} if it is the free monad generated by $S$; that is, if for any other monad $M$ the induced morphism
\[
\Map_{\mathrm{Mon}}(T, M) \rightarrow \Map_{\mathrm{End}}(S, M) 
\]
of mapping spaces in, respectively, monads and endofunctors of $\Ind(\dcat^{b}(\acat))$, is an equivalence.
\end{definition}
Note that it is clear from the definition that if a thread monad exists, then it is unique up to a contractible space of choices.

\begin{proposition}
\label{proposition:formula_for_thread_monad}
A thread monad $T$ exists and as an endofunctor it is given by the formula 
\[
T(X) \colonequals X \oplus S(X) \oplus S^{2}(X) \oplus \ldots \simeq X \oplus \Sigma X[-1] \oplus \Sigma^{2} X[-2] \oplus \ldots,
\]
an infinite direct sum.
\end{proposition}

\begin{proof}
Since $\Ind(\dcat^{b}(\acat))$ has all colimits, and $S$ preserves these, this is \cite{gepner2017infty}[4.2.8, 4.3.2].
\end{proof}

\begin{remark}
\label{remark:composition_in_free_monad}
Note that the monad structure on the thread monad, at least at the level of homotopy categories, is the expected one: the map $T^{2}(X) \rightarrow T(X)$ restricts on the direct summand 
\[
\Sigma^{m}(\Sigma^{n}X[-n])[-m] \simeq \Sigma^{n+m}X[-n-m]
\]
to the inclusion of the corresponding summand of $T(X)$, see \cite{gepner2017infty}[4.3.8].
\end{remark}

\begin{remark}
\label{remark:direct_sum_decomposition_of_thread_monad_is_canonical}
As a consequence of \cref{remark:composition_in_free_monad}, we see that the direct sum decomposition of $T$ given in \cref{proposition:formula_for_thread_monad} is canonical. Namely, the inclusion of the $k$-th summand can be identified with the composite
\[
S^{k} \rightarrow T^{k} \rightarrow T
\]
of $k$-fold composition of the canonical morphism $S \rightarrow T$ and the monad multiplication.
\end{remark}
The thread monad is somewhat reminiscent of the formula (\ref{equation:bousfield_monad_as_an_endofunctor}) for the Bousfield monad, except in the latter case there are only finitely many summands, the last one being $\Sigma^{q} X[-q]$. Thus, to obtain an equivalent monad, we need to consider an appropriate truncation of the thread monad. 

Unfortunately, truncating while preserving an algebra structure can be difficult in the $\infty$-categorical setting. One method is an appropriate localization of the monoidal $\infty$-category in question, this is the route we will take. 

To be able to appropriately truncate the thread monad using localization, we need to define an appropriate subcategory of $\mathrm{End}(\Ind(\dcat^{b}(\acat))$ which at the same time 
\begin{enumerate}
    \item contains the thread monad and the desired truncation,
    \item is closed under composition, so that it inherits a monoidal structure and 
    \item is simple enough so that we can easily verify that the needed localization exists. 
\end{enumerate}
We will use the following subcategory, which should be treated as a technical device and not an important object. 

\begin{definition}
We say an endomorphism $W: \Ind(\dcat^{b}(\acat)) \rightarrow \Ind(\dcat^{b}(\acat))$ is a \emph{block-sum} if there exist natural equivalence
\[
W \simeq \bigoplus_{i \in I} \Sigma^{n_{i}} F_{i},
\]
for some countable index set $I$, self-equivalences $F_{i}$ of $\Ind(\dcat^{b}(\acat))$ and $n_{i} \geq 0$, such that for any $q \geq 0$, all but finitely many $n_{i}$ satisfy $n_{i} > q$. 
\end{definition}

\begin{example}
The underlying endomorphism of the thread monad is block-sum, by \cref{proposition:formula_for_thread_monad}, and so is the one underlying the Bousfield monad, by \cref{equation:bousfield_monad_as_an_endofunctor}.
\end{example}

\begin{remark}
Observe that any block-sum endofunctor is cocontinuous; in particular, it commutes with suspensions and direct sums. It follows that block-sum endofunctors are stable under composition.
\end{remark}

\begin{definition}
\label{definition:truncated_block_sum_endofunctor}
We say a block-sum endofunctor $W$ is \emph{q-truncated} if it can be written as a finite sum
\[
W \simeq \bigoplus \Sigma^{n_{i}} F_{i}
\]
where $n_{i} \leq q$.
\end{definition}

\begin{remark}
\label{remark:q_truncated_block_sum_a_finite_product}
Observe that any block-sum endofunctor which is $q$-truncated for some $q \geq 0$ is a \emph{finite} direct sum, which we can thus identify with a finite product in the $\infty$-category of additive endofunctors. 
\end{remark}

\begin{lemma}
\label{lemma:q_truncated_endofunctors_form_a_localization}
Consider the $\infty$-category $\mathrm{End}^{\Box}$ of block-sum endomorphisms of $\Ind(\dcat^{b}(\acat))$. Then, the inclusion of the subcategory of $q$-truncated blocks-sum endomorphisms admits a left adjoint $L_{\leq q}: \mathrm{End}^{\Box} \rightarrow \mathrm{End}^{\Box}_{\leq q}$. 
\end{lemma}

\begin{proof}
We construct $L_{\leq q}$ directly. If $W \simeq \bigoplus \Sigma^{n_{i}} F_{i}$, we define 
\[
L_{\leq q} W \colonequals \bigoplus _{n_{i} \leq q} F_{i},
\]
note that this $q$-truncated. We have to show that the obvious projection $W \rightarrow L_{\leq q}W$ induces an equivalence
\[
\Map_{\mathrm{End}}(L_{\leq q} W, R) \simeq \Map_{\mathrm{End}}(W, R)
\]
for any other $q$-truncated $R$. 

As both mapping spaces can be written as products indexed using the direct sum decomposition of $W$ and the product decomposition of $R$ of \cref{remark:q_truncated_block_sum_a_finite_product}, it is enough to verify that if $F_{1}, F_{2}$ are self-equivalences of $\Ind(\dcat^{b}(\acat))$, then whenever $n_{1} > n_{2}$ we have 
\[
\Map_{\mathrm{End}}(\Sigma^{n_{1}}F_{1}, \Sigma^{n_{2}} F_{2}) \simeq \Map_{\mathrm{End}}(F_{1}, \Omega^{n_{1}-n_{2}} F_{2}) \simeq 0.
\]
To see this, observe that the since $F_{1}$ is cocontinuous, it is a left Kan extension of its restriction to $\acat$, which generates the $\mathrm{Ind}$-completion under colimits. Thus, it is enough to verify that a restriction of any natural transformation $F_{1} \rightarrow \Omega^{n_1 -n_2} F_{2}$ to $\acat$ vanishes. This is clear, since $F_{2}$ is an equivalence, so $F_{2}(a)$ is discrete for any $a \in \acat$ and thus $\Omega^{n_{1} - n_{2}}F_2(a)$ vanishes. 
\end{proof}

\begin{remark}
\label{remark:block_sum_truncation_is_a_direct_summand}
The proof of \cref{lemma:q_truncated_endofunctors_form_a_localization} shows that an arbitrary block-sum endofunctor $W$ can be written as a direct sum $W \simeq L_{\leq q}W \oplus \Sigma^{q+1} W^{\prime}$ for some other block-sum $W^{\prime}$. 
\end{remark}

\begin{proposition}
\label{proposition:k_truncation_on_endofunctors_compatible_with_monoidal_structure}
The localization $L_{\leq q}$ is compatible with the composition monoidal structure on endofunctors; that is, there is a unique monoidal structure on $\mathrm{End}_{\leq q}^{\Box}$ such that $L_{\leq q}$ can be promoted to a monoidal functor. 
\end{proposition}

\begin{proof}
By \cite{higher_algebra}[2.2.1.9], we have to show that $L_{\leq q}$-equivalences, that is, those maps which are inverted by $L_{\leq q}$, are stable under monoidal multiplication with an arbitrary object; that is, under both post- and pre-composition. It is enough to do so for the the $q$-truncation maps $W \rightarrow L_{\leq q}W$, where $W$ is a block-sum endofunctor. 

This follows from \cref{remark:block_sum_truncation_is_a_direct_summand}, as $W \simeq L_{\leq q}W \oplus \Sigma^{q+1} W^{\prime}$, and the latter sum will stay a $q+1$-fold suspension under either post- or pre-composition, so that it will be $L_{\leq q}$-acyclic.
\end{proof}

\begin{corollary}
\label{corollary:proposition:k_truncation_on_endofunctors_compatible_with_monoidal_structure}
Let $M$ be a monad on $\Ind(\dcat^{b}(\acat))$ whose underlying endofunctor is block-sum. Then, $L_{\leq q} M$ acquires a unique structure of a monad such that $M \rightarrow L_{\leq q} M$ is a morphism of monads.
\end{corollary}

\begin{proof}
Observe that we can identify $M$ with an associative algebra object of $\mathrm{End}^{\Box}$. As we have shown in \cref{proposition:k_truncation_on_endofunctors_compatible_with_monoidal_structure}, $L_{\leq q}\colon \mathrm{End}^{\Box} \rightarrow \mathrm{End}_{\leq q}^{\Box}$ is monoidal and it follows formally that the right adjoint given by inclusion is lax monoidal. Thus, $L_{\leq q}$ considered as an endofunctor of $\mathrm{End}^{\Box}$ takes algebras to algebras, as needed. 
\end{proof}
We are now ready to define our candidate for an algebraic model for the Bousfield monad. 

\begin{definition}
\label{definition:q_thread_monad}
Let $q \geq 0$. The \emph{$q$-thread monad} is the monad 
\[
T_{q} \colonequals L_{\leq q}T
\]
on $\Ind(\dcat^{b}(\acat))$, where $T$ is the thread monad of \cref{definition:thread_monad}
\end{definition}

\begin{remark}
\label{remark:q_thread_monad_as_an_endofunctor}
Observe that as an endofunctor of $\Ind(\dcat^{b}(\acat))$, the $q$-thread monad is given by 
\[
T_{q}X \simeq X \oplus S(X) \oplus \ldots \oplus S^{q}(X) \simeq X \oplus \Sigma X[-1] \oplus \ldots \Sigma^{q} X[-q].
\]
This is immediate from the description of the thread monad given in \cref{proposition:formula_for_thread_monad} and the explicit description of $q$-truncation appearing in the proof of \cref{lemma:q_truncated_endofunctors_form_a_localization}.
\end{remark}

\begin{remark}
\label{remark:direct_sum_decomposition_of_q_thread_monad_canonical}
As in the non-truncated case considered in \cref{remark:direct_sum_decomposition_of_thread_monad_is_canonical}, the direct sum decomposition of the $q$-thread monad given in \cref{remark:q_thread_monad_as_an_endofunctor} is canonical: the inclusion of the $k$-th direct summand, where now $0 \leq k \leq q$, can be identified with the composite
\[
S^{k} \rightarrow T_{q}^{k} \rightarrow T_{q}
\]
where the second map is monad multiplication.
\end{remark}

\begin{remark}
\label{remark:universal_property_of_q_thread_monad}
As a consequence of its construction as a localization and the universal property of the thread monad, the $q$-thread monad enjoys its own universal property. Namely, for any monad $M$ whose underlying endofunctor is $q$-truncated block-sum, to give a morphism of monads $T_{q} \rightarrow M$ is the same as to give a natural transformation $\Sigma(-)[-1] \rightarrow M$ of endofunctors.  
\end{remark}

Note that as a consequence of \cref{remark:q_thread_monad_as_an_endofunctor} and \cref{equation:bousfield_monad_as_an_endofunctor}, the Bousfield monad $\beta_{*} \beta^{*}$ and the $q$-thread monad have same underlying endofunctors of $\dcat^{b}(\acat)$. We want to compare them as monads, using the universal property of \cref{remark:universal_property_of_q_thread_monad}. To do so, we will need an appropriate natural transformation into the underlying endofunctor of the Bousfield monad. 

\begin{definition}
\label{definition:thread_generator}
We define a natural transformation of endofunctors of $\dcat^{b}(\acat)$ of signature
\[
t\colon S(X) \simeq \Sigma X[-1] \rightarrow  \beta_{*} \beta^{*} X.
\]
as the adjoint of the natural transformation
\[
\tau\colon \beta^{*}(\Sigma X[-1]) \simeq \Sigma (\beta^{*} X)[-1] \rightarrow \beta^{*} X.
\]
\end{definition}

\begin{notation}
\label{notation:composites_of_thread_generator}
For any $k \geq 0$, let us write $t^{k}$ for the composite 
\[
S^{k}(X) \simeq \Sigma^{k}X[-k] \rightarrow (\beta_{*} \beta^{*})^{k}X \rightarrow \beta_{*}\beta^{*}X,
\]
where the second arrow is the monad multiplication. Note that each $t^{k}$ is adjoint to $\tau^{k}$; in particular, $t^{0}$ is the unit of the Bousfield monad.
\end{notation}

\begin{proposition}
\label{proposition:direct_sum_decomposition_of_bousfield_monad_using_ts}
For any $X \in \dcat^{b}(\acat)$, the direct sum
\[
(t^{0}, t^{1}, \ldots, t^{q})\colon X \oplus \Sigma X[-1] \oplus \ldots \oplus \Sigma^{q} X[-q] \rightarrow \beta_{*} \beta^{*} X
\]
of the composites $t^{k}$ for $0 \leq k \leq q$ is an equivalence. 
\end{proposition}

\begin{proof}
Both sides are right exact functors, so it is enough to verify that the claim holds for $a \in \acat$. In this case, $\beta^{*}(a)$ is a potential $q$-stage, so that it is $q$-truncated and for any $0 \leq k \leq q$, the map 
\[
\tau^{k}\colon \Sigma^{k} \beta^{*}a [-k] \rightarrow \beta^{*}a
\]
is an isomorphism on $\pi_{k}$. As $\beta_{*}$ commutes with homotopy groups, the same is true for 
\[
\Sigma^{k}(\beta_{*} \beta^{*} a)[-k] \simeq \beta_{*} (\Sigma^{k} \beta^{*}a [-k]) \rightarrow \beta_{*} \beta^{*} a.
\]
As $\beta^{*} \dashv \beta_{*}$ induces an adjoint equivalence between the hearts, the unit of any object is a $\pi_{0}$-isomorphism and thus the composite of the above with the unit 
\[
\Sigma^{k} a[-k] \rightarrow \Sigma^{k}(\beta_{*} \beta^{*} a)[-k]
\]
is also a $\pi_{k}$-isomorphism. After unwrapping the definitions, we see that this is exactly the map $t^{k}$ of \cref{notation:composites_of_thread_generator} and the claim follows. 
\end{proof}

\begin{theorem}
\label{theorem:bousfield_monad_canonically_equivalent_to_q_thread_monad}
The natural transformation of $t\colon \Sigma X[-1] \rightarrow \beta_{*} \beta^{*}$ of endofunctors of $\dcat^{b}(\acat)$ induces a unique equivalence 
\[
T_{q} \rightarrow \beta_{*} \beta^{*}
\]
of monads between the $q$-thread monad and the Bousfield monad.
\end{theorem}

\begin{proof}
Since $\beta_{*} \beta^{*}$ is right exact, it uniquely extends to a cocontinuous endofunctor of $\mathrm{Ind}(\dcat^{b}(\acat))$, and so does the natural transformation $t\colon S(X) \rightarrow \beta_{*} \beta^{*} X$. As a consequence of \cref{equation:bousfield_monad_as_an_endofunctor}, this unique extension of the Bousfield monad is block-sum and $q$-truncated in the sense of \cref{definition:truncated_block_sum_endofunctor}; in fact, its underlying endofunctor is equivalent to that of $T_{q}$. 

From the universal property of the latter of \cref{remark:universal_property_of_q_thread_monad}, we deduce that the natural transformation $S(X) \rightarrow \beta_{*} \beta^{*}X$ induces a unique morphism of monads $T_{q} \rightarrow \beta_{*} \beta^{*}$. We claim the latter is an equivalence.

Notice that by construction, the morphism $T_{q} \rightarrow \beta_{*} \beta^{*}$ has the property that the composite
\[
S^{k} \rightarrow T \rightarrow T_{q} \rightarrow \beta_{*} \beta^{*},
\]
where the first arrow is the inclusion of the $k$-th summand can be identified with the natural transformation $t^{k}$. Thus, the composite of the direct sum
\[
S^{0} \oplus \ldots \oplus S^{q} \rightarrow T_{q} \rightarrow \beta_{*} \beta^{*}
\]
can be identified with the direct sum of the maps $t^{k}$ for $0 \leq k \leq q$. However, this composite is a natural equivalence by \cref{proposition:direct_sum_decomposition_of_bousfield_monad_using_ts}, and the first arrow is a natural equivalence by \cref{remark:direct_sum_decomposition_of_q_thread_monad_canonical}. It follows that $T_{q} \rightarrow \beta_{*} \beta^{*}$ is also a natural equivalence, as needed. 
\end{proof}

\begin{corollary}
\label{corollary:if_we_have_q1_bousfield_functor_then_ctauq1_modules_are_tq_modules}
Suppose that $H\colon \ccat \rightarrow \acat$ is an adapted homology theory such that there exists an additive Bousfield functor 
\[
\beta^{inj}\colon \acat^{inj} \rightarrow h_{q+1} \ccat^{inj} 
\]
with the property that $H(\beta^{inj}(i)) \simeq i$; for example, when $\acat$ admits a splitting of order $(q+1)$. Then, there exists a canonical equivalence
\[
\Mod_{C\tau^{q+1}}(\dcat^{\omega}(\ccat)) \simeq \Mod_{T_{q}}(\dcat^{b}(\acat)).
\]
In particular, the left hand side only depends on $\acat$. 
\end{corollary}

\begin{proof}
This is a combination of \cref{proposition:bousfield_adjunction_monadic} and \cref{theorem:bousfield_monad_canonically_equivalent_to_q_thread_monad}.
\end{proof}

\begin{remark}
\label{remark:if_we_have_q1_bousfield_functor_then_potential_q_stages_are_equivalent}
Suppose that $H_{1}\colon \ccat_{1} \rightarrow \acat$ and $H_{2}\colon \ccat_{2} \rightarrow \acat$ are two adapted homology theories which both admit Bousfield functors valued in the homotopy $(q+1)$-category, so that \cref{corollary:if_we_have_q1_bousfield_functor_then_ctauq1_modules_are_tq_modules} yields a canonical equivalence
\[
\Mod_{C\tau^{q+1}}(\dcat^{\omega}(\ccat_{1})) \simeq \Mod_{T_{q}}(\dcat^{b}(\acat)) \simeq \Mod_{C\tau^{q+1}}(\dcat^{\omega}(\ccat_{2})).
\]
As potential $q$-stages in the sense of \cref{def:general_potential_l_stage} are $q$-truncated and so admit unique $C\tau^{q+1}$-module structures, we can view Goerss-Hopkins $\infty$-categories $\mathcal{M}_{q}(\ccat_{1})$ and $\mathcal{M}_{q}(\ccat_{2})$ as subcategories of, respectively, the left and right hand side. We claim that this composite equivalence respects these subcategories; that is, that it restricts to an equivalence
\[
\mathcal{M}_{q}(\ccat_{1}) \simeq \mathcal{M}_{q}(\ccat_{2}).
\]

To see this, observe that for each of $k = 1, 2$, the forgetful functor 
\[
\dcat^{b}(\acat) \simeq \Mod_{C\tau}(\dcat^{\omega}(\ccat_{k})) \rightarrow \Mod_{C\tau^{q+1}}(\dcat^{\omega}(\ccat_{k})) 
\]
is exact. Thus, by the universal property of $\dcat^{b}(\acat)$ of  \cref{theorem:universal_property_of_bounded_derived_cat_of_an_abelian_cat}, it is uniquely determined by what it does on the heart. It follows that there is an essentially unique commutative diagram 
\[\begin{tikzcd}
	{\Mod_{C\tau}(\dcat^{\omega}(\ccat_{1}))} & {\Mod_{C\tau}(\dcat^{\omega}(\ccat_{2}))} \\
	{\Mod_{C\tau^{q+1}}(\dcat^{\omega}(\ccat_{1}))} & {\Mod_{C\tau^{q+1}}(\dcat^{\omega}(\ccat_{2}))}
	\arrow[from=1-1, to=2-1]
	\arrow["\simeq"', from=1-1, to=1-2]
	\arrow[from=1-2, to=2-2]
	\arrow["\simeq", from=2-1, to=2-2]
\end{tikzcd}\]
compatible with canonical identifications between the hearts and $\acat$, where the vertical arrows are the forgetful functors and the lower one is the equivalence of \cref{corollary:if_we_have_q1_bousfield_functor_then_ctauq1_modules_are_tq_modules}. As potential $q$-stages can be characterized completely in terms of the vertical arrows, as we have shown in \cref{lemma:ctau_tensor_over_ctaun_is_discrete_fora_potential_n_stage}, we deduce that the lower horizontal arrow preserves them, as claimed. 
\end{remark}

\subsection{Proof of Franke's conjecture}
\label{subsection:complete_proof_of_frankes_conjecture}
The goal of this section is to complete the proof of the following result.

\begin{theorem}[Franke's Algebraicity Conjecture]
\label{theorem:frankes_algebraicity}
Suppose that $H\colon \ccat \rightarrow \acat$ is a conservative, adapted homology theory and assume that 
\begin{enumerate}
    \item $\acat$ admits a splitting of order $q+1$,
    \item $\acat$ is of finite cohomological dimension $d$ and 
    \item $q \geq d$. 
\end{enumerate}
Then, we have an equivalence of homotopy $(q+1-d)$-categories 
\[h_{q+1-d}  \dcat^{per}(\acat) \simeq h_{q+1-d}{\ccat}.\]
\end{theorem}

\begin{remark}
Franke's conjecture is also true in the edge case $q = d-1$ if we interpret the homotopy $0$-category as the set of equivalence classes of objects, see \cref{remark:equivalence_of_zero_homotopy_categories_holds_in_edge_case_of_gh_tower} below.
\end{remark}

\begin{remark}
\label{remark:frankes_original_statement_of_conjecture}
Note that the conjecture above is just slightly different than the one appearing in Franke's \cite{franke1996uniqueness}, who worked with derivators rather than $\infty$-categories. We explain how to deduce the classical formulation from ours in \cref{remark:how_do_deduce_frankes_formulation_from_ours} below.
\end{remark}

As we observed in  \cref{remark:if_we_have_q1_bousfield_functor_then_potential_q_stages_are_equivalent}, the existence of a splitting of order $q+1$ of $\acat$ implies that the $\infty$-category of potential $q$-stages does not depend on $\ccat$. In particular, we have a canonical equivalence
\[
\mathcal{M}_{q}(\ccat) \simeq \mathcal{M}_{q}(\dcat^{per}(\acat)).
\]
To deduce Franke's conjecture, we have to show that if $\acat$ is of finite cohomological dimension, these form good approximations of the relevant stable $\infty$-categories. This will be a straightforward consequence of Goerss-Hopkins theory. 

\begin{proposition}
\label{proposition:homotopy_category_of_c_approximated_by_mn_in_finite_global_dimension}
Suppose that $H\colon \ccat\ \rightarrow \acat$ is a conservative adapted homology theory such that $\acat$ is of finite cohomological dimension $d$. Then, for any $n \geq d$, the canonical truncation map 
\[
\ccat \simeq \mathcal{M}_{\infty}(\ccat) \rightarrow \mathcal{M}_{n}(\ccat)
\]
induces an equivalence
\[
h_{n-d+1} \ccat \simeq h_{n-d+1} \mathcal{M}_{n}(\ccat)
\]
of homotopy $(n-d+1)$-categories. 
\end{proposition}

\begin{proof}
First, we have to show that any object of $\mathcal{M}_{n}(\ccat)$ can be lifted to $\mathcal{M}_{\infty}(\ccat)$. As $\acat$ is of finite cohomological dimension and $H$ is conservative, by \cref{proposition:convergence_of_gh_tower} the Goerss-Hopkins tower 
\[
\mathcal{M}_{\infty}(\ccat) \rightarrow \ldots \rightarrow \mathcal{M}_{1}(\ccat) \rightarrow \mathcal{M}_{0}(\ccat)
\]
is a limit diagram of $\infty$-categories. Thus, it is enough to show that we can lift objects along 
\[
(-)_{\leq k}\colon \mathcal{M}_{k+1}(\ccat) \rightarrow \mathcal{M}_{k}(\ccat)
\]
whenever $k \geq n$, as a compatible sequence of such lifts will give a lift from $\mathcal{M}_{n}(\ccat)$ to $\mathcal{M}_{\infty}(\ccat)$. 

By \cref{prop: general obstruction for objects}, an object $X \in \mathcal{M}_{k}(\ccat)$ can be lifted along $(-)_{\leq k}$ if and only if an obstruction class 
\[
\theta_{X} \in \Ext_{\acat}^{k+3, k+1}(\pi_{0}X, \pi_0 X)
\]
vanishes. When $k \geq n \geq d$ is above the cohomological dimension, this is automatic. 

We now claim that for any $k \geq n$, and any $X, Y \in \mathcal{M}_{k+1}(\ccat)$, the induced map 
\begin{equation}
\label{equation:iso_on_homotopy_groups_in_proof_of_stabilization_of_gh_hmn_in_finite_dim}
\pi_{i} \Map_{\mathcal{M}_{k+1}(\ccat)}(X, Y) \rightarrow \pi_{i} \Map_{\mathcal{M}_{k}(\ccat)}(X_{\leq k}, Y_{\leq k})
\end{equation}
is an isomorphism for all $i \leq k-d$. Note that this in particular implies an isomorphism for $i \leq n-d$. To see this, observe that by \cref{proposition:gh_fibre_sequence_for_mapping_spaces} we have a fibre sequence
\[
\xymatrix{ \Map_{\mathcal{M}_{k+1}(\ccat)}(X, Y) \ar[r]^-{(-)_{\leq k}} & \Map_{\mathcal{M}_{k}(\ccat)}(X_{\leq k}, Y_{\leq k}) \ar[r] & \Map_{\dcat^{b}(\acat)}(\pi_0X, \Sigma^{k+2} \pi_0 Y[-k-1]). } 
\]
The homotopy groups of the right hand side are given by $\Ext$-groups, which vanish above the cohomological dimension $d$. Then, the isomorphism of (\ref{equation:iso_on_homotopy_groups_in_proof_of_stabilization_of_gh_hmn_in_finite_dim}) is implied by the long exact sequence of homotopy groups. 

Since $\mathcal{M}_{\infty}(\ccat)$ is a limit, for any $X, Y \in \mathcal{M}_{\infty}(\ccat)$ we have 
\[
\Map_{\mathcal{M}_{\infty}(\ccat)}(X, Y) \simeq \varprojlim \Map_{\mathcal{M}_{k}(\ccat)}(X_{\leq k}, Y_{\leq k}).
\]
Since the homotopy groups of the right hand side stabilize in an increasing range of dimensions, the $\varprojlim^{1}$-term in the Milnor exact sequence vanishes and we deduce that 
\[
\pi_{i} \Map_{\mathcal{M}_{\infty}(\ccat)}(X, Y) \simeq \varprojlim \pi_{i} \Map_{\mathcal{M}_{k}(\ccat)}(X_{\leq k}, Y_{\leq k}).
\]
is an isomorphism for all $i \geq 0$. When $k \geq n$, all transition maps are isomorphisms on $\pi_{i}$ for $i \leq n-d$, and we deduce the limit is already stable in this range. Thus, it follows that \[
\pi_{i} \Map_{\mathcal{M}_{\infty}(\ccat)}(X, Y) \rightarrow \pi_{i} \Map_{\mathcal{M}_{n}(\ccat)}(X_{\leq n}, Y_{\leq n})
\]
is an isomorphism for $i \leq n-d$, as needed. 
\end{proof}

\begin{remark}
\label{remark:equivalence_of_zero_homotopy_categories_holds_in_edge_case_of_gh_tower}
The statement of \cref{proposition:homotopy_category_of_c_approximated_by_mn_in_finite_global_dimension} remains true in the edge case $n = d-1$ if we interpret the homotopy $0$-category $h_{0}$ as the set of equivalence classes of objects.

In this case, we still have vanishing of obstructions for lifts of objects and morphisms along $(-)_{\leq k}\colon \mathcal{M}_{k+1} \rightarrow \mathcal{M}_{k}$ for $k \geq n$, although the lifts of the latter are no longer unique up to homotopy. This is enough for the result, as $(-)_{\leq k}$ reflects equivalences of potential $k+1$-stages, so that any equivalence can be lifted to an equivalence.
\end{remark}

We are now ready to complete the proof of Franke's conjecture. 

\begin{proof}[Proof of \cref{theorem:frankes_algebraicity}:]
Note that $\dcat^{per}(\acat)$ also admits a conservative $\acat$-valued homology theory. Thus, both of the functors 
\[
h_{q-d+1}(\ccat) \rightarrow h_{q-d+1} (\mathcal{M}_{q}(\ccat))
\]
and 
\[
h_{q-d+1}(\dcat^{per}(\acat)) \rightarrow h_{q-d+1} (\mathcal{M}_{q}(\dcat^{per}(\acat)))
\]
obtained from truncation in the Goerss-Hopkins tower by passing to homotopy $(q-d+1)$-categories, are equivalences by \cref{proposition:homotopy_category_of_c_approximated_by_mn_in_finite_global_dimension}.

As $\acat$ is assumed to admit a splitting of order $(q+1)$, the statement of the conjecture follows, as we have a canonical equivalence
\[
\mathcal{M}_{q}(\ccat) \simeq \mathcal{M}_{q}(\dcat^{per}(\acat)),
\]
even before passing to homotopy categories, as verified in \cref{remark:if_we_have_q1_bousfield_functor_then_potential_q_stages_are_equivalent}.
\end{proof}

\begin{remark}
\label{remark:how_do_deduce_frankes_formulation_from_ours}
As we warned the reader previously in  \cref{remark:frankes_original_statement_of_conjecture}, our statement of the conjecture is slightly different than Franke's; however, it implies it as we will now explain. 

In the original manuscript \cite{franke1996uniqueness}[p.56], Franke claims that if $H\colon \ccat \rightarrow \acat$ is a homology theory satisfying the conditions of \cref{theorem:frankes_algebraicity}, then for any category $K$  with finitely many morphisms and whose nerve is of finite dimension $\mathrm{dim}(K) \leq q-d$, we have an equivalence
\[
h_{1} (\dcat^{per}(\acat)^{K}) \simeq h_{1} (\ccat^K)
\]
of homotopy categories of diagrams. Moreover, these equivalences are compatible with change of diagrams and homotopy left and right Kan extensions. 

This statement follows immediately from \cref{theorem:frankes_algebraicity}, as if $\ccat$ is an $\infty$-category, then the projection $\ccat \rightarrow h_{q+1-d}\ccat$ onto the homotopy $(q+1-d)$-category induces an equivalence of homotopy categories 
\[
h_{1} \Fun(K, \ccat) \simeq h_{1} \Fun(K, h_{q+1-d} \ccat)
\]
whenever $\mathrm{dim}(K) < q+1-d$, as one sees by counting connectivity estimates using the cell decomposition of $K$, as proven by Raptis \cite{Raptis}[3.17]. More generally, this argument works for any finite-dimensional simplicial set, not necessarily a nerve of a category with finitely many morphisms.

Alternatively, one can approach the statement about diagrams by studying the induced homology theories $H^{K}\colon \ccat^{K} \rightarrow \acat^{K}$, which we will do in \S\ref{section:applications_to_diagram_categories_and_equivariant_objects}. This strategy has the advantage of being able to establish these equivalences for a more general class of diagrams $K$ than the finite-dimensional ones, perhaps under additional assumptions on $\acat$; for example, we will prove results when $K$ is a finite group, or a direct or inverse sequence diagram. 
\end{remark}

\begin{remark}
More generally than we do, in \cite{franke1996uniqueness}[p.56] Franke allows ``partially adapted'' homology theories $H\colon \ccat \rightarrow \acat$, where every object of $\acat$ embeds into an injective which admits a lift, but not necessarily all injectives do, as we discussed in \cref{remark:adapted_homology_theories_under_other_names_and_comparison_with_franke}. In this case, we have a fully faithful embedding $h_{q+1-d} \ccat \hookrightarrow h_{q+1-d} \dcat(\acat)^{per}$ whose image is controlled by the subgroup of the $K$-group of $\acat$ generated by $H(c)$ for $c \in \ccat$, as in \cite{franke1996uniqueness}[p.54]. 

This can be deduced from \cref{theorem:frankes_algebraicity} by replacing $\ccat$ by its idempotent-completion, to which the given ``partially adapted'' $H$ will uniquely extend, yielding an adapted homology theory. That the image in the idempotent-completion can be described in terms of the $K$-group is shown just as in Franke's work, and we leave the details to the interested reader. 
\end{remark}

\section{Applications of the conjecture}
\label{section:applications_to_diagram_categories_and_equivariant_objects}

In this section we provide examples and applications of the main result. We discuss algebraic models for module spectra, diagram categories and genuine $E(n)$-local $G$-spectra, where $G$ is a finite group of order coprime to $p$.

\subsection{Modules over ring spectra} 
\label{subsection:modules_over_ring_spectra}

Let $R$ be an $\mathbf{E}_1$-algebra in spectra. As we observed in \cref{lemma:homotopy_of_r_module_an_adapted_homology_theory}, the associated homology theory 
\[
\pi_{*}\colon \Mod_{R} \rightarrow \Mod_{R_{*}}
\]
is adapted. In this section, we will see that if $R_{*}$ is sufficiently well-behaved, Franke's algebraicity conjecture applies to the above functor, and discuss examples of ring spectra for which this holds. 

The key is the following straightforward observation.

\begin{lemma}
\label{lemma:sparse_homotopy_groups_of_a_ring_give_splitting_on_abelian_cat_of_modules}
Let $R_{*}$ be a graded ring such that $R_{k} = 0$ unless $(q+1) |  k$; in other words, $R_{*}$ is concentrated in degrees divisible by $q+1$. Then, $\Mod_{R_{*}}$ admits a splitting of order $(q+1)$. 
\end{lemma}

\begin{proof}
If $\phi \in \mathbb{Z}/(q+1)$, we let $(\Mod_{R_{*}})_{\phi}$ be the full subcategory spanned by those graded modules $M_{*}$ such that $M_{k} = 0$ unless $k = \phi \ \mathrm{mod} (q+1)$. Under the given condition on $R_{*}$, any module splits uniquely as a direct sum of submodules with this property. 
\end{proof}

\begin{theorem}[Franke's algebricity conjecture for modules] 
\label{theorem:frankes_algebraicity_for_modules}
Let $R$ be an $\mathbf{E}_1$-algebra in spectra and suppose that the graded homotopy ring $R_*$ satisfies the following: 
\begin{enumerate}
    \item is concentrated in degrees divisible by $(q+1)$ and 
    \item is of graded global dimension $d \leq q$; that is, $\Mod_{R_{*}}$ is of cohomological dimension $d$. 
\end{enumerate}
Then, there exists a canonical equivalence 
\[
h_{q+1-d} (\Mod_{R}) \simeq h_{q+1-d} (\Mod_{HR_*}).
\]
between the homotopy $(q+1-d)$-categories of $R$-modules and modules over the Eilenberg-MacLane spectrum with the same homotopy groups. 
\end{theorem}

\begin{proof} 
This is immediate from \cref{lemma:sparse_homotopy_groups_of_a_ring_give_splitting_on_abelian_cat_of_modules} and \cref{theorem:frankes_algebraicity}, since the latter applies to both $R$ and $HR_{*}$ and thus we have
\[
h_{q+1-d} \Mod_{R} \simeq h_{q+1-d} \dcat^{per}(\Mod_{R_{*}}) \simeq h_{q+1-d} \Mod_{HR_*}.
\]
\end{proof}

\begin{corollary} Let $R$ be an $\mathbf{E}_1$-algebra in spectra and suppose that the graded homotopy ring $R_*$ is concentrated in degrees divisible by $(q+1)$ and is of global dimension $d < q$. Then, the equivalence of homotopy $1$-categories
\[
h_1 (\Mod_R) \simeq h_1 (\Mod_{HR_*})
\]
is triangulated. 
\end{corollary}

\begin{proof}
This follows from \cref{theorem:frankes_algebraicity_for_modules}, since the homotopy $2$-category of a stable $\infty$-category determines the triangulated structure on the homotopy $1$-category \cite{pstragowski_chromatic_homotopy_algebraic}[B.8].
\end{proof}

\begin{remark}
\cref{theorem:frankes_algebraicity_for_modules} significantly strengthens the results of the first author in \cite{patchkoria2017exotic} , which only dealt with the cases where the global dimension satisfies $d \leq 3$.
\end{remark}

\begin{warning} 
\label{warning:algebraicity_produces_exotic_equivalences}
Observe that \cref{theorem:frankes_algebraicity_for_modules} produces a large family of equivalences between homotopy categories of modules over ring spectra. These equivalences can \emph{never} be lifted to an actual equivalence
\[
\Mod_{R} \simeq \Mod_{HR_{*}}
\]
of $\infty$-categories unless $R \simeq HR_{*}$ as $\mathbf{E}_{1}$-ring spectra. In what is common terminology in the literature, these equivalences of homotopy categories are \emph{exotic}. 

To see this, observe that the equivalences we construct commute with the homotopy groups functor valued in $\Mod_{R_{*}}$ by construction. For any ring spectrum, $R$ itself is the only module up to equivalence with homotopy groups $\pi_{*} R \simeq R_{*}$, and so any lift of our equivalence would send free modules of rank one to free modules of rank one. 

As we can recover the $\mathbf{E}_{1}$-structure using endomorphism mapping spectra of the free module of rank one \cite{higher_algebra}[7.1.2.6], we deduce that any lift to an equivalence at the level of stable $\infty$-categories induces an equivalence $R \simeq HR_{*}$ of ring spectra. 
\end{warning}

Let us describe several examples to which our results apply. 

\begin{example}
\label{example:truncated_brown_peterson_spectra} 
Fix a prime $p$ and let $BP\langle n \rangle$ be an $\mathbf{E}_{1}$-form of the truncated Brown-Peterson spectrum in the sense of \cite{hahn2020redshift}, so that in particular we have 
\[
BP \langle n \rangle_{*} \simeq \mathbb{Z}_{(p)}[v_1, \dots, v_n]
\]
where $\vert v_i \vert =2(p^i-1)$. This is a ring of graded global dimension $n+1$, concentrated in degrees divisible by $2p-2$, and thus \cref{theorem:frankes_algebraicity_for_modules} implies that if $2p-2 > n+1$ then we have a canonical equivalence
\[
h_{k} \Mod_{BP \langle n \rangle} \simeq h_{k} \Mod_{H\mathbb{Z}_{(p)}[v_1, \dots, v_n]}
\]
where $k = 2p-2 - (n+1)$. In particular, if $2p-2 > n+2$, then 
we have an equivalence
\[
h_1 \Mod_{BP \langle n \rangle} \simeq h_1 \Mod_{H\mathbb{Z}_{(p)}[v_1, \dots, v_n]}
\]
of homotopy $1$-categories compatible with triangulated structure.
\end{example}

\begin{example}
\label{example:johnson_wilson_spectrum} 
Let $E(n)$ be an $\mathbf{E}_{1}$-form of the Johnson-Wilson spectrum at height $n$, so that 
\[
E(n)_{*} \simeq \mathbb{Z}_{(p)}[v_1, \dots, v_{n-1}, v_n^{\pm 1}]
\]
where again $\vert v_i \vert =2(p^i-1)$. For example, one can take a localization of an $\mathbf{E}_{1}$-form of the truncated Brown-Peterson spectrum of \cref{example:truncated_brown_peterson_spectra}. 

The homotopy groups $E(n)_{*}$ form a ring of graded global dimension $n$, and are concentrated in degrees divisible by $2p-2$. Thus, if $2p-2 > n$, then 
\[
h_{k} \Mod_{E(n)} \simeq h_{k} \Mod_{H\mathbb{Z}_{(p)}[v_1, \dots, v_{n-1}, v_n^{\pm 1}]}
\]
for $k = 2p-2-n$. 

Note that when the given form of $E(n)$ is a localization of $BP\langle n \rangle$, the above equivalence can be deduced from that of \cref{example:truncated_brown_peterson_spectra} by passing to suitable subcategories, but the bound we obtain in the Johnson-Wilson case is sharper. 
\end{example}

\begin{example}
\label{example:p_local_ko}
Consider the connective real K-theory $ko_{(p)}$ localized at an odd prime $p$. This is an $\mathbf{E}_{\infty}$-algebra with homotopy groups given by 
\[
ko_{(p)*}\simeq \mathbb{Z}_{(p)}[v],
\]
where $\vert v \vert = 4$. It is a result of the first author that we have an equivalence 
\begin{equation}
\label{equation:equivalence_of_homotopy_1_categories_for_p_localized_ko}
h_{1} \Mod(ko_{(p)}) \simeq h_1 \Mod_{H \mathbb{Z}_{(p)}[v]}
\end{equation}
of homotopy $1$-categories \cite{patchkoria2012}.

As the ring $ko_{(p)*}$ is of global dimension $2$ and is concentrated in degrees divisible by $4$, \cref{theorem:frankes_algebraicity_for_modules} improves this to an equivalence at the level of homotopy $2$-categories. In particular, the equivalence of (\ref{equation:equivalence_of_homotopy_1_categories_for_p_localized_ko}) is triangulated, a question previously left open by \cite{patchkoria2012}.
\end{example}

\begin{example}[A family of optimal examples]
\label{example:morava_k_theory}
Let $K(n)$ be an $\mathbf{E}_{1}$-form of minimal Morava $K$-theory so that 
\[
K(n)_* \simeq \mathbb{F}_{p}[v_{n}^{\pm 1}],
\]
where $| v_{n} | = 2p^{n}-2$. This is a graded field, and hence of global dimension zero, concentrated in degrees divisible by $2p^{n}-2$ and thus we deduce that 
\[
h_{2p^{n}-2} \Mod_{K(n)} \simeq h_{2p^{n}-2} \Mod_{H \mathbb{F}_p[v_n^{\pm 1}]}.
\]
We claim this is optimal; that is, the above equivalence cannot be improved to one of homotopy $(2p^{n}-1)$-categories. 

To see this, consider the mapping space 
\[
\Map_{h_{2p^{n}-1} \Mod_{K(n)}}(K(n), K(n)) \simeq \tau_{\leq 2p^{n}-2} (\Omega^{\infty} K(n))
\]
The left hand side carries a canonical $\mathbf{E}_{\infty}$-structure as the homotopy $(2p^{n}-1)$-category is additive, which coincides with the structure on the right hand side coming from the spectrum structure on $K(n)$. 

Thus, we can identify the connective spectrum corresponding to  endomorphisms of $K(n)$ in the homotopy $(2p^{n}-1)$-category with the Postnikov truncation
\[
\tau_{\leq 2p^{n}-2} k(n)
\]
of connective Morava $K$-theory. It is well-known that this spectrum is not a product of Eilenberg-MacLane spectra, as the first possibly non-zero $k$-invariant of $k(n)$ is indeed non-trivial, being equal to the Milnor primitive $Q_{n} \in A^{*}$ \cite[IX.7.17]{Rudy}. 

Since the $\infty$-category $h_{2p^{n}-1} \Mod_{H \mathbb{F}_p[v_n^{\pm 1}]}$ is additive and $H \mathbb{Z}$-linear, all mapping spectra therein are Eilenberg-MacLane spectra, thus $\tau_{\leq 2p^{n}-2} k(n)$ cannot arise in this way. We deduce that the homotopy $(2p^{n}-1)$-categories of $\Mod_{K(n)}$ and  $\Mod_{H \mathbb{F}_p[v_n^{\pm 1}]}$ are not equivalent and thus in this case the equivalence constructed in \cref{theorem:frankes_algebraicity_for_modules} is optimal.
\end{example}

\begin{example}
\label{example:connective_morava_k_theory}
Another example is given by the connective covers $k(n)$ of Morava $K$-theories of \cref{example:morava_k_theory}. In this case, the homotopy groups form a ring of dimension one and thus 
\[
h_{2p^{n}-3} \Mod_{k(n)} \simeq h_{2p^{n}-3} \Mod_{H \mathbb{F}_p[v_n]}
\]
On the other hand, the same argument as in \cref{example:morava_k_theory} shows that the homotopy $(2p^{n}-1)$-categories are not equivalent. The comparison of homotopy $(2p^{n}-2)$-categories remains open, but we expect it does not hold, either. 
\end{example}

The following two examples are of ring spectra for which \cref{theorem:frankes_algebraicity_for_modules} does not yield any new results, but we decided to mention them to have an opportunity to highlight some open questions the resolution of which would require a different technique.

\begin{example}
Consider the complex topological $K$-theory spectrum $KU$, with homotopy groups given by $\pi_{*} KU \simeq \mathbb{Z}[u^{\pm 1}]$ where $| u | = 2$. In this case, there exists an equivalence
\begin{equation}
\label{equation:equivalence_with_homotopy_cat_of_ku_modules}
h_1 \Mod_{KU} \simeq h_1 \Mod_{\mathbb{Z}[u^{\pm 1}]}
\end{equation}
of homotopy $1$-categories by a classical result of Wolbert \cite{Wolbert}. This result is immediately recovered by \cref{theorem:frankes_algebraicity_for_modules}, but our methods do not allow us to deduce anything further and it is still an open question whether the equivalence of (\ref{equation:equivalence_with_homotopy_cat_of_ku_modules}) is triangulated, or whether it can be lifted to the level of homotopy $2$-categories. 

Note that (\ref{equation:equivalence_with_homotopy_cat_of_ku_modules}) is known to be triangulated after inverting $2$, by a Galois theoretic argument \cite{patchkoriaKU}. By reducing modulo $2$ and using the argument of  \cref{example:morava_k_theory} we see that it does not lift to an equivalence of homotopy $3$-categories.
\end{example}

\begin{example} 
At any fixed prime, Dugger and Shipley construct a non-formal differential graded algebra which determines a $\mathbf{E}_1$-algebra in spectra $HA$ with homotopy groups
\[
\pi_* HA \simeq \mathbb{F}_p[x^{\pm 1}],
\]
where $| x | = 1$ \cite{DugShiEx}. This ring is a graded field, and Dugger and Shipley use it to show that there exist an equivalence
\[
h_1 \Mod_{HA} \simeq h_1 \Mod_{H \mathbb{F}_p[x^{\pm 1}]}
\]
which is moreover compatible with the triangulated structure. 

On the other hand, they show that $\tau_{\leq 2} \tau_{\geq 0} HA$ and $\tau_{\leq 2}  \tau_{\geq 0} H\mathbb{F}_p[x^{\pm 1}]$ are not equivalent as $\mathbf{E}_1$-algebras and thus a variation on the argument given in \cref{warning:algebraicity_produces_exotic_equivalences} implies that the given equivalence does not lift to the level of homotopy $3$-categories. The case of homotopy $2$-categories, which \cref{theorem:frankes_algebraicity_for_modules} does not resolve, remains open.
\end{example}

\subsection{Chromatic algebraicity} 
\label{subsection:chromatic_algebraicity}

At any fixed prime and finite height $n$, we have Johnson-Wilson homology theory 
\[
E(n)_{*}\colon \spectra_{E(n)} \rightarrow \Comod_{E(n)_{*}E(n)}.
\]
valued in the category of comodules over the Hopf algebroid $E(n)_{*}E(n)$, which presents the moduli stack of formal groups of height at most $n$. The following generalization of the result of second author shows that this functor is a very good approximation to the $E(n)$-local category at large primes. 

\begin{theorem}[Chromatic algebraicity]
\label{theorem:chromatic_algebraicity}
If $2p-2 > n^{2}+n$, there exists a canonical equivalence
\[
h_{k} \spectra_{E(n)} \simeq h_{k} \dcat^{per}(E(n)_{*}E(n))
\]
between the homotopy $k$-categories of $E(n)$-local spectra and differential $E(n)_{*}E(n)$-comodules, where $k = 2p-2-n^{2}-n$.
\end{theorem}

\begin{proof}
We will verify the assumptions of \cref{theorem:frankes_algebraicity}. Observe first that the homology theory $E(n)_{*}$ is conservative, since we take the domain to be given by $E(n)$-local spectra. It is adapted by a result of Devinatz \cite{devinatz1997morava}[1.5].

 Since the Hopf algebroid $E(n)_{*}E(n)$ is concentrated in degrees divisible by $2p-2$, the same argument as in  \cref{lemma:sparse_homotopy_groups_of_a_ring_give_splitting_on_abelian_cat_of_modules} shows that $\Comod_{E(n)_{*}E(n)}$ admits the needed splitting. Finally, the fact that the latter abelian category is of finite cohomological dimension $n^{2}+n$ when $p > n+1$ follows from the chromatic spectral sequence and Morava's work at finite height  \cite{pstragowski_chromatic_homotopy_algebraic}[2.5].
\end{proof}

\begin{remark}[The improved bound]
As we observed above, \cref{theorem:chromatic_algebraicity} is an improvement on the results from the thesis of the second author, where equivalence of homotopy categories as above is proven under the weaker bound $2p-2 > 2(n^{2}+n)$ \cite{pstragowski_chromatic_homotopy_algebraic}. This improvement  comes from a more involved construction of the Bousfield adjunction which appears in \S\ref{subsection:splitting_of_abelian_categories_and_the_bousfield_functor}, which is obtained by first defining it on injectives, then right Kan extending, and finally left Kan extending. 
\end{remark}

\begin{remark}
Informally speaking, in the limit ``$p \to \infty$'', $\spectra_{E(n)}$ and $\dcat(E(n)_{*}E(n))$ are equivalent as symmetric monoidal $\infty$-categories.  This is the asymptotic algebraicity result of Barthel, Schlank and Stapleton \cite{barthel2017chromatic}.
\end{remark}

\subsection{Diagram $\infty$-categories}
\label{subsection:diagram_infty_categories}

In \cref{remark:how_do_deduce_frankes_formulation_from_ours} we observed that equivalences of higher homotopy categories, such as the ones produced by \cref{theorem:frankes_algebraicity}, induce equivalences on homotopy categories of diagrams, under certain restrictions on the dimension of the indexing simplicial set. In this section, we will take a more direct approach to describing these diagram $\infty$-categories by considering induced adapted homology theories. 

To be more precise, if $H\colon \ccat \rightarrow \acat$ is a homology theory and $K$ is a simplicial set, then we have an induced homology theory 
\[
H^{K}\colon \ccat^{K} \rightarrow \acat^{K}
\]
on diagram $\infty$-categories given by
\[
H^{K}(f)(k) = H(f(k)),
\]
where $k \in K$ and $f\colon K \rightarrow \ccat$. If $H$ is adapted, then under certain assumptions on $\acat$ and $K$ this can be shown to be adapted again, allowing our methods to apply to $\ccat^{K}$ directly. 

To avoid analyzing more complicated colimits, we will assume that $K$ is a (nerve of a) category. The first case we will analyze is when $K$ is finite, but since we work in the homotopy coherent setting, we need a stronger version of finiteness than what is usual for categories. 

\begin{definition}
\label{definition:finite_category}
We say a category $K$ is \emph{finite} if its nerve is a finite simplicial set; that is, it has only finitely many non-degenerate cells. The \emph{dimension} $\mathrm{dim}(K)$ is the dimension of its nerve as a simplicial set. 
\end{definition}

\begin{remark}
\label{remark:homotopical_finiteness_of_a_category}
The finiteness of a category $K$ in the sense of \cref{definition:finite_category} is equivalent to asking that 
\begin{enumerate}
    \item $K$ has only finitely many objects $k_{0}, k_{1}, \ldots \in K_{0}$, 
    \item $K$ has only finitely many morphisms $f_{0}, f_{1}, \ldots \in K_{1}$ and 
    \item for any composable chain of non-identity morphisms $k_{0} \rightarrow k_{1} \rightarrow \ldots \rightarrow k_{n}$ we have $n \leq d$, for some $d$ which only depends on $K$. The smallest such $d$ is the dimension of $K$. 
\end{enumerate}
\end{remark}
The key question is whether $\acat^{K}$ has enough injectives and is of finite cohomological dimension. This is always true in the finite case, as the following shows. 

\begin{construction}
\label{construction:injectives_in_finite_diagram_categories}
Let $K$ be a finite category. It is well-known that in this case $\acat^{K}$ has enough injectives, see \cite[Example 2.13]{Wei}, but it will be useful to give an explicit description. 

Given an object $k \in K$, the evaluation 
\[
\Ev_k \colon \acat^K \to \acat
\]
sending $X \in \acat^K$ to the value $X(k) \in \acat$ has a right adjoint $r_k$, explicitly given by the right Kan extension
\[
r_k(a)(k')=\prod_{\Hom_{K}(k',k)} a.
\]
Note that since evaluation is exact, the right adjoint $r_{k}$ sends injectives to injectives. 

Given any $X \in \acat^{K}$, for every object $k \in K$ we can find an injective $i_k$ and an embedding $X(k) \to i_k$. Using the adjunction, we get morphisms of diagrams $X \to r_k(i_k)$ which assemble into a single map 
\[
X \to \prod_{k \in K} r_k(i_k) 
\]
This map is a monomorphism, as it is so levelwise by construction, and the target is injective, as it is a product of injectives. 
\end{construction}

\begin{proposition}
\label{proposition:adapted_homology_remains_adapted_on_finite_diagrams}
Let $H\colon \ccat \rightarrow \acat$ be an adapted homology theory such that $\ccat$ is idempotent-complete and let $K$ be a finite category. Then, the induced homology theory $H^K \colon \ccat^K \to \acat^K$ is also adapted.
\end{proposition}

\begin{proof}
Since every injective in $\acat^K$ embeds into a finite product of injectives of the form $r_{k}(i)$ of \cref{construction:injectives_in_finite_diagram_categories}, where $k \in K$ and $i \in \acat$ is an injective, and $\ccat$ is idempotent complete, by \cref{remark:adapted_homology_theories_under_other_names_and_comparison_with_franke} it suffices to find injective lifts of $r_{k}(i)$. 

Since $H$ is adapted, there exists an injective lift $i_{\ccat} \in \ccat$ of $i$. Let $r^{\ccat}_{k}$ be the right adjoint to the evaluation
\[
\Ev_{k}^{\ccat}\colon \ccat^{K} \rightarrow \ccat, 
\]
which is again given by right Kan extension along $k\colon \Delta^{0} \rightarrow K$. We claim that $r^{\ccat}_{k}(i_{\ccat})$ is the injective lift of $r_{k}(i)$. Again, by the right Kan extension formula we have 
\[
r_{k}^{\ccat}(i_{\ccat})(k^{\prime}) = \prod_{\Hom_{K}(k^{\prime}, k)} i_{\ccat}
\]
and since $H$ commutes with finite products we have $H^{K}(r^{\ccat}_{k}(i_{\ccat})) = r_{k}(i)$, as needed. 

We are left with verifying that
\[
[X, r_{k}^{\ccat}(i_{\ccat})]_{\ccat^{K}} \simeq \Hom_{\acat^{k}}(H^{K}(X), r_{k}(i))
\]
After applying the respective adjunction on both sides, this reduces to
\[
[X(k), i_{\ccat}]_{\ccat} \simeq \Hom_{\acat}(H(X(k)), i)
\]
which is just the adaptedness of $H$.
\end{proof}

\begin{proposition}
\label{proposition:dimension_of_category_of_finite_diagrams}
Let $K$ be a finite category and $\acat$ an abelian category with enough injectives and of cohomological dimension $d$. Then, the diagram category $\acat^{K}$ is of finite cohomological dimension at most $d + \mathrm{dim}(K)$. 
\end{proposition}

\begin{proof}
This is also well-known, see \cite{franke1996uniqueness} or  \cite{Mit1} for the case of posets, but we will provide a proof for the convenience of the reader. Consider the adjunction
\[
\xymatrix{\acat^K \ar@<0.5ex>[r]^-{\Ev} & \prod_{k \in K} \acat \ar@<0.5ex>[l]^-{r}.}
\]
given by evaluation at all objects $k \in K$ simultaneously, and let $T = r \circ \Ev$ be the associated monad. For any $X \in \acat^{K}$, we get the cobar resolution 
\begin{equation}
\label{equation:cobar_resolution_for_functor_category}
X \rightarrow TX \rightrightarrows T^{2}X \triplerightarrow \ldots
\end{equation}
This possesses an extra codegeneracy operator after applying $\Ev$, and hence the corresponding chain complex is a resolution of $\Ev(X)$. As evaluation is exact, we deduce that (\ref{equation:cobar_resolution_for_functor_category}) is a resolution of $X$ in $\acat^{K}$. 

By unwrapping the definitions, we see that levelwise the cobar resolution is given by 
\[
X(k) \rightarrow \prod_{k \to k_1} X(k_1) \rightrightarrows \prod_{k \to k_1 \to k_2} X(k_2) \triplerightarrow 
\]
Through the Dold-Kan correspondence, this has a corresponding normalized cochain complex. In this case, normalization corresponds in degree $n$ to taking the subproduct indexed by those composable arrows $k \rightarrow k_{1} \rightarrow \ldots \rightarrow k_{n}$ in $K$ where all but the first one are not the identity. 

There are no such composable chains above $\mathrm{dim}(K)$ by \cref{remark:homotopical_finiteness_of_a_category}, and we deduce that the normalized cobar resolution is a finite chain complex
\begin{align*}
X \to X_{0} \to X_{1} \to \cdots \to X_{\mathrm{dim}(K)} \rightarrow 0
\end{align*}
where each of $X_{i}$ is a direct summand of a diagram in the image of $T$. Since $r$ is exact, as it is given levelwise by finite products, it takes injective resolutions to injective resolutions and we deduce that anything in the image of $T$ is at most of injective dimension $d$, as all objects of $\acat$ are. 

We deduce that all of the $X_{i}$ are of injective dimension at most $d$. For any $Y \in \acat^{K}$, the cobar resolution induces a spectral sequence
\[
H^{s^{\prime}} (\Ext^{s}_{\acat^{K}}(Y, X_{*})) \Rightarrow \Ext^{s+s^{\prime}}_{\acat^{K}}(Y, X)
\]
The $\Ext$-groups on the left vanish for $s > d$ and the cohomology groups vanish for $s' > \mathrm{dim}(K)$, giving the needed result.
\end{proof}

\begin{theorem}
\label{theorem:algebraicity_for_diagrams_indexed_by_a_finite_category}
Let $H\colon \ccat \rightarrow \acat$ be a conservative, adapted homology theory such that $\acat$ admits a splitting of order $(q+1)$ and is of finite cohomological dimension $d$. Then, for any finite category $K$ we have an equivalence
\[
h_{k}{(\ccat^K)} \simeq h_{k} (\dcat^{per}(\acat^K))
\]
of homotopy $k$-categories, where $k = q+1-d-\mathrm{dim}(K)$. 
\end{theorem}

\begin{proof}
Observe that if $\acat$ admits a splitting of order $q+1$, then so does $\acat^{K}$, with 
\[
(\acat^{K})_{\phi} \colonequals (\acat_{\phi})^{K}
\]
for any $\phi \in \mathbb{Z}/(q+1)$. Thus, the result is a combination of \cref{theorem:frankes_algebraicity}, \cref{proposition:adapted_homology_remains_adapted_on_finite_diagrams} and \cref{proposition:dimension_of_category_of_finite_diagrams}. Note that $\ccat$ is automatically idempotent-complete, as a consequence of \cref{corollary:c_which_admits_conservative_adapted_homology_theory_in_fin_dim_a_is_idempotent_complete}.
\end{proof}

\begin{remark}
Note that the above result implies in particular that, subject to certain finiteness conditions of $K$, the homotopy category of $\ccat^{K}$ does not depend on $\ccat$ but only on $\acat$, which we also deduced previously in \cref{remark:how_do_deduce_frankes_formulation_from_ours} by using connectivity estimates. 

However, the conclusion of \cref{theorem:algebraicity_for_diagrams_indexed_by_a_finite_category} is stronger: it describes these categories as the homotopy category of a derived category of diagrams, rather than diagrams in the derived $\infty$-category. 
\end{remark}

Observe that in \cref{remark:how_do_deduce_frankes_formulation_from_ours}, we did not need to assume that $K$ is a finite category, but only that it is finite-dimensional. This can also be approached using induced homology theories, but requires more assumptions on $\acat$. 

\begin{proposition}
\label{proposition:algebraicity_for_finite_dimensional_diagrams}
Let $H\colon \ccat \rightarrow \acat$ be a conservative, adapted homology theory and $\kappa$ a regular cardinal such that 
\begin{enumerate}
    \item $\acat$ admits a splitting of order $(q+1)$ and is of finite cohomological dimension $d$,
    \item $\acat$ admits $\kappa$-small products and they are exact. 
\end{enumerate}
Let $K$ be a nerve of a category whose set of morphisms is $\kappa$-small and which is finite-dimensional as a simplicial set. Then, we have an equivalence
\[
h_{k} (\dcat^{per}(\acat^K)) \simeq h_{k}(\ccat^K)
\]
of homotopy $k$-categories, where $k = q+1-d-\mathrm{dim}(K)$. 
\end{proposition}

\begin{proof}
One shows that $H^{K}\colon \ccat^{K} \rightarrow \acat^{K}$ is adapted as in the proof of \cref{proposition:adapted_homology_remains_adapted_on_finite_diagrams}, where the products of injectives are now possibly infinite, using that $H$ preserves arbitrary products of injectives by \cref{corollary:an_adapted_homology_theory_commutes_with_producs_of_injectives}. 

To establish that $\acat^{K}$ is of cohomological dimension $d+\mathrm{dim}(K)$, we argue as in \cref{proposition:dimension_of_category_of_finite_diagrams}, where now we need exactness of $\kappa$-small products so that the right adjoint to evaluation is exact. The rest is the same as \cref{theorem:algebraicity_for_diagrams_indexed_by_a_finite_category}.
\end{proof}

\begin{remark}[Why nerve of a category?]
A reader might observe that the equivalence of homotopy categories of diagrams established in \cref{remark:how_do_deduce_frankes_formulation_from_ours} allowed $K$ to be an arbitrary finite-dimensional simplicial set, while in \cref{proposition:algebraicity_for_finite_dimensional_diagrams} we assume that it is a nerve of a category. 

This restriction comes from the study of the right Kan extension appearing in \cref{proposition:adapted_homology_remains_adapted_on_finite_diagrams}. If $K$ is a nerve of a category, then the resulting diagram is levelwise a product, whose homology we can control using \cref{corollary:an_adapted_homology_theory_commutes_with_producs_of_injectives}. If $K$ is an arbitrary simplicial set, we would be forced to study more complicated limits, whose homology would be more difficult to understand.
\end{remark}
In many interesting cases, even the assumption that $K$ is finite-dimensional can be weakened. Let us describe a few examples of this phenomena which we found particularly interesting. In each case, the key step is our ability to control the cohomological dimension of the abelian category $\acat^K$.

\begin{example}[Towers]
\label{example:towers} 
Let $\mathcal{T}$ denote the dual of the poset $(\mathbb{N} , \leq)$ of the natural numbers and let be $\acat$ an abelian category with enough injectives and of finite cohomological dimension $d$. Then, we claim that the category $\acat^{\mathcal{T}}$ of towers has cohomological dimension at most $d+1$. 

The objects of $\acat^{\mathcal{T}}$ are towers
\[\dots \to A_3 \to A_2 \to A_1\]
and one can show that such a tower is injective if and only if each arrow is a split surjection and each object $A_i$ is injective \cite[1.1]{Jann}, which is equivalent to each $A_{i+1} \to A_i$ being injective in the arrow category $\acat^{[1]}$. By \cref{proposition:dimension_of_category_of_finite_diagrams}, the cohomological dimension of $\acat^{[1]}$ is at most $d+1$, and since exactness in diagram categories is  objectwise, we deduce the same is true for  $\acat^{\mathcal{T}}$.

If $H\colon \ccat \rightarrow \acat$ is adapted, then the proof of \cref{proposition:adapted_homology_remains_adapted_on_finite_diagrams} implies that $H^{\mathcal{T}} \colon \ccat^\mathcal{T} \to \acat^\mathcal{T}$ is also adapted. Note that this does not require the existence of infinite products in $\acat$, as for every fixed $m$, $\Hom_{\mathcal{T}}(m,n)$ is always finite and empty for all but finitely many $n$ and thus the relevant right Kan extension involves only finite products. 

We conclude that under the assumptions of \cref{theorem:frankes_algebraicity},
\[
h_{k}{(\ccat^{\mathcal{T}})} \simeq h_{k} (\dcat^{per}(\acat^\mathcal{T}))
\]
where $k = q-d$. 
\end{example}

\begin{example}[Direct sequences]
\label{example:sequences}
Let $\mathcal{N}$ denote the poset $(\mathbb{N}, \leq)$ and $\acat^{\mathcal{N}}$ the category  of sequences
\[
A_1 \to A_2 \to A_3 \to \cdots.
\]
This can be analyzed analogously to the case of towers of \cref{example:towers}, but will require stronger assumptions on $\acat$. 

Suppose that $\acat$ is an abelian category with enough injectives, enough projectives, countable products and of finite cohomological dimension $d$. The dual of the argument in \cref{example:towers}, but using projectives, shows that $\acat^{\mathcal{N}}$ is of finite cohomological dimension at most $d+1$. 

Since $\mathcal{N}$ has only countably many morphisms and $\acat$ has countable products by assumption, the argument of \cref{proposition:adapted_homology_remains_adapted_on_finite_diagrams} shows that $H^{\mathcal{N}}$ is adapted. We conclude that as in the case of towers, under the assumptions of \cref{theorem:frankes_algebraicity} we have
\[
h_{k}{(\ccat^{\mathcal{N}})} \simeq h_{k} (\dcat^{per}(\acat^\mathcal{N}))
\]
where $k = q-d$. For example, these conditions are satisfied for the categories of modules over certain ring spectra as studied in \S\ref{subsection:modules_over_ring_spectra}.
\end{example}

\begin{example}[Group actions]
\label{example:group_actions}
Let $G$ be a finite group. If $\acat$ is an abelian category with enough injectives, let $\acat^{\mathrm{B}G}$ denote the category of $G$-objects. This is a category of diagrams and hence an abelian category with enough injectives. 

Note that the nerve of the classifying category $\mathrm{B}G$ is not finite-dimensional unless $G$ is trivial so that \cref{proposition:dimension_of_category_of_finite_diagrams} cannot be applied to bound the cohomological dimension of the category of $G$-objects. However, if the order $|G|$ is invertible in $\acat$, that is $|G| \cdot \mathrm{id}_{a} \in \Hom_{\acat}(a, a)$ is an isomorphism for every $a \in \acat$, then Maschke's theorem implies that the cohomological dimension of $\acat^{\mathrm{B}G}$ is equal to the cohomological dimension of $\acat$ \cite[3.4]{Mit}.

If $H \colon \ccat \to \acat$ is adapted, then $H^{\mathrm{B}G}\colon \ccat^{\mathrm{B}G} \to \acat^{\mathrm{B}G}$ is also adapted by \cref{proposition:adapted_homology_remains_adapted_on_finite_diagrams}, whose proof only requires that $K$ have finitely many objects and morphisms. We deduce that if $|G|$ is invertible in $\acat$, then under the assumptions of \cref{theorem:frankes_algebraicity} we have 
\[
h_{k}(\ccat^{\mathrm{B}G} ) \simeq h_{k} (\dcat^{per}(\acat^{\mathrm{B}G}))
\]
where $k = q+1-d$. 
\end{example}

\begin{example}[Genuinely equivariant chromatic algebraicity]
One can apply \cref{example:group_actions} to extend the chromatic algebraicity of \cref{theorem:chromatic_algebraicity} to provide an algebraic model for the $\infty$-category of $E(n)$-local genuine $G$-spectra $\spectra_{E(n)}^G$, where $2p-2 > n^{2}+n$ and the order of $G$ is coprime to $p$. 

For brevity, let us fix the height and write $E = E(n)$ and 
\[
\ComodE^{G} \colonequals \Fun_{\Sigma}(A(G), \ComodE)
\]
for the abelian category of $G$-Mackey functors with values in $E_{*}E$-comodules; that is, additive functors on the Burnside category. We claim that there is an equivalence
\[
h_{k} \spectra_{E}^G  \simeq h_{k} \dcat^{per} (\ComodE^{G}), \]
where $k = 2p-2-n^{2}-n$. Note that in the special case of $n = 1$, $p$ odd and $k = 1$, this is a result of Roitzheim and the first author \cite{PatRoi}. 

The equivariant $E$-local $\infty$-category has several equivalent descriptions in the literature. If $\iota^*_G\colon \spectra \rightarrow \spectra^{G}$ denotes the inflation, then by definition, $\spectra_{E}^G$ is the Bousfield localization of the $\infty$-category of genuine $G$-spectra at the inflation $\iota^*_GE$. By the construction of inflation, a $G$-spectrum $X$ belongs to this localization if and only if for any $H \leq G$, the genuine fixed points $X^H$ are $E$-local as a spectrum. 

Equivalently, since $\spectra_{E}$ is equivalent to modules over the $E$-local sphere, one can show that $\spectra_{E}^G$ is equivalent to the $\infty$-category spectral Mackey functors with values in $E$-local spectra in the sense of Barwick \cite[4.8]{PatSanWim}, \cite{Clark}. 

Now suppose that $H \leq G$ is a subgroup, and denote by $W_G(H)$ its Weyl group. Since the order of $G$ is coprime to $p$, a result of Wimmer implies that there is a canonical equivalence
\[
\spectra_{E}^{G} \simeq \prod_{(H) \leq G } \spectra_{E}^{\mathrm{B}W_G(H)}
\]
of symmetric monoidal $\infty$-categories \cite{Wim}[1.1], where on the right hand side we have the product of $\infty$-categories of $E$-local spectra with a ``naive'' group action. On the algebraic side, it is similarly known that we have an equivalence 
\[
\ComodE^{G} \simeq \prod_{(H) \leq G } \ComodE^{\mathrm{B}W_G(H)}, 
\]
see \cite{PatRoi}. As the construction of derived category commutes with products, combining these equivalences with \cref{example:group_actions}, we deduce that for any finite group $G$ of order coprime to $p$ we have 
\[
h_{k} \spectra_{E}^G  \simeq h_{k} \dcat^{per} (\ComodE^{G}), \]
where $k = 2p-2-n^{2}-n$, as claimed above. 
\end{example}

\appendix 

\section{Comodules and quotients of abelian categories} 

In this short appendix we note a few basic properties of the comonadic factorization of an exact left adjoint of abelian categories. We expect these results are well-known, but we could not find a suitable reference. 

The setting is as follows. Let $L \dashv R\colon \acat \rightleftarrows \bcat$ be an adjunction between abelian categories with the left adjoint $L$ exact. Associated to this we have a factorization 
\[
\begin{tikzcd}
	{\acat} && {\bcat} \\
	& {\Comod_{C}(\bcat)}
	\arrow["{L}", from=1-1, to=1-3]
	\arrow["{S}"', from=1-1, to=2-2]
	\arrow["{U}"', from=2-2, to=1-3]
\end{tikzcd},
\]
through the category of comodules for the comonad $C = LR$ over $\bcat$. Here, $S$ is the lift of $L$ observing that $La$ has a canonical $C$-comodule structure given by the unit map 
\[
La \rightarrow LRLa = CLa,
\]
and $U$ is the forgetful functor, which is a left adjoint.

\begin{theorem}
\label{theorem:properties_of_comonadic_factorization_of_an_exact_left_adjoint}
The category $\Comod_{C}(\bcat)$ is abelian with exact forgetful functor into $\bcat$. The lift $S\colon \acat \rightarrow \Comod_{C}(\bcat)$ has a fully faithful right adjoint, hence it presents the category of comodules as the Gabriel quotient by the localizing subcategory $\mathrm{ker}(L)$. 
\end{theorem}

The above result will be proven in steps, namely it is a combination of \cref{proposition:abelianness_of_category_of_coalgebras}, \cref{lemma:in_comonadic_factorization_the_right_adjoint_is_fully_faithful} and \cref{proposition:comodules_form_a_gabriel_quotient} below.

\begin{proposition}
\label{proposition:abelianness_of_category_of_coalgebras}
The category $\Comod_{C}(\bcat)$ is abelian and the forgetful functor $U$ is exact. 
\end{proposition}

\begin{proof}
We claim that for any map $a \rightarrow b$ of $C$-comodules, the kernel and cokernel of that map in $\bcat$ have canonical structure of a comodule; the reader can verify this then makes that comodule into the kernel and cokernel in $\Comod_{C}(\bcat)$. 

Since $C$ is right exact, $C \mathrm{ker}(a \rightarrow b) \simeq \mathrm{ker}(Ca \rightarrow Cb)$, and the structure map for the kernel follows from functoriality. 

In the case of cokernels, observe that since the composite $Ca \rightarrow Cb \rightarrow Ccoker(a \rightarrow b)$ is zero by functoriality, we get an induced map $\mathrm{coker}(Ca \rightarrow Cb) \rightarrow C \mathrm{coker}(a \rightarrow b)$. The comodule structure map for the cokernel is the composite of this map and the map $\mathrm{coker}(a \rightarrow b) \rightarrow \mathrm{coker}(Ca \rightarrow Cb)$ induced by structure maps of $a, b$. 
\end{proof}

\begin{proposition}
\label{proposition:lift_to_comodules_also_left_adjoint}
The functor $S\colon \acat \rightarrow \Comod_{C}(\bcat)$ is also an exact left adjoint, with right adjoint $T\colon \Comod_{C}(\bcat) \rightarrow \acat$.
\end{proposition}

\begin{proof}
It follows from \cref{proposition:abelianness_of_category_of_coalgebras} that $S$ is exact, because $U \circ S \simeq L$ is and $U$ is both exact and conservative. Thus, we only have to check that it is left adjoint. 

To find a right adjoint, we would like to define for each $C$-comodule $c$ an object $Tc \in \acat$ such that\[
\Hom_{C}(Sa, c) \simeq \Hom_{\acat}(a, Tc)
\]
for any $a \in \acat$. Since the underlying $\bcat$-object of $Sa$ is given by $La$, left hand side can be computed as the equalizer of
\[
\Hom_{\bcat}(La, c) \rightrightarrows \Hom_{\bcat}(La, Cc)
\]
which, keeping in mind that $C = LR$, is the same as the equalizer of 
\[
\Hom_{\acat}(a, Rc) \rightrightarrows \Hom_{\acat}(a, RLRc)
\]
It follows that we can define $Tc$ by the equalizer diagram 
\[
Tc \rightarrow Rc \rightrightarrows RLRc.
\]
\end{proof}

\begin{lemma}
\label{lemma:in_comonadic_factorization_the_right_adjoint_is_fully_faithful}
The right adjoint $T\colon \Comod_{C}(\bcat) \rightarrow \acat$ is fully faithful. 
\end{lemma}

\begin{proof}
Using the formula given in the proof of \cref{proposition:lift_to_comodules_also_left_adjoint} and the fact that $S$ is exact, we see that $STc$ is given by the equalizer of
\[
Cc \rightrightarrows C^{2}c. 
\]
This is a split equalizer with limit $c$ for any $C$-comodule.
\end{proof}

\begin{proposition}
\label{proposition:comodules_form_a_gabriel_quotient}
The functor $S\colon \acat \rightarrow \Comod_{C}(\bcat)$ presents its target as the Gabriel quotient of the source by the Serre subcategory
\[
\mathrm{ker}(L) \colonequals \{ a \in \acat \ | \ L(a) = 0 \}.
\]
That is, for any abelian category $\dcat$ and an exact functor $G\colon \acat \rightarrow \dcat$ such that $\mathrm{ker}(L) \subseteq \mathrm{ker}(G)$, there exists an essentially unique factorization of $G$ through $S\colon \acat \rightarrow \Comod_{C}(\bcat)$.
\end{proposition}

\begin{proof}
Since $T$ is fully faithful, for any $c \in \Comod_{C}(\bcat)$, the counit map $STc \rightarrow c$ is an isomorphism. It follows that for any $a \in \acat$, the unit map $a \rightarrow TSa$ has kernel and cokernel contained in $\mathrm{ker}(S) = \mathrm{ker}(L) \subseteq \mathrm{ker}(G)$. Hence, $G(a) \rightarrow G(TSa)$ is an isomorphism, providing the needed factorization of $G$ through $S$. 
\end{proof}
The following observation about injectives will be useful. 

\begin{proposition}
\label{proposition:category_of_comodules_enough_injectives}
If $i \in \bcat$ is injective, then $Ci$ is an injective $C$-comodule. If $\bcat$ has enough injectives, so does $\Comod_{C}(\bcat)$. 
\end{proposition}

\begin{proof}
Since $C\colon \bcat \rightarrow \Comod_{C}(\bcat)$ is a right adjoint to the forgetful functor $U$, which is exact by \cref{proposition:abelianness_of_category_of_coalgebras}, it is clear that it preserves injectives. 

Now suppose that $\bcat$ has enough injectives and let $c$ be a $C$-comodule. The structure morphism $c \rightarrow Cc$ is a morphism into the cofree comodule which is split after applying the forgetful functor (by the counit) and so is monic because the latter is conservative and exact. 

Choose an embedding $c \rightarrow i$ into an injective in $\bcat$. Then, $Cc \rightarrow Ci$ is also monic, and the needed embedding into an injective comodule is given by the composite $c \rightarrow Cc \rightarrow Ci.$
\end{proof}

\bibliographystyle{plain}
\bibliography{general_algebraicity_bibliography}

\begin{thebibliography}{10}

\bibitem{AdamsJIV}
J.~F. Adams.
\newblock On the groups {$J(X)$}. {IV}.
\newblock {\em Topology}, 5:21--71, 1966.

\bibitem{adams1995stable}
John~Frank Adams.
\newblock {\em Stable homotopy and generalised homology}.
\newblock University of Chicago press, 1995.

\bibitem{antieau2019k}
Benjamin Antieau, David Gepner, and Jeremiah Heller.
\newblock K-theoretic obstructions to bounded t-structures.
\newblock {\em Inventiones mathematicae}, 216(1):241--300, 2019.

\bibitem{balderrama2021deformations}
William Balderrama.
\newblock Deformations of homotopy theories via algebraic theories.
\newblock {\em arXiv preprint arXiv:2108.06801}, 2021.

\bibitem{balmer2020nilpotence}
Paul Balmer.
\newblock Nilpotence theorems via homological residue fields.
\newblock {\em Tunisian Journal of Mathematics}, 2(2):359--378, 2020.

\bibitem{balmer2021computing}
Paul Balmer and James Cameron.
\newblock Computing homological residue fields in algebra and topology.
\newblock {\em Proceedings of the American Mathematical Society},
  149(08):3177--3185, 2021.

\bibitem{balmer2019tensor}
Paul Balmer, Henning Krause, and Greg Stevenson.
\newblock Tensor-triangular fields: ruminations.
\newblock {\em Selecta Mathematica}, 25(1):13, 2019.

\bibitem{barkan2023}
Shaul Barkan.
\newblock Chromatic homotopy is monoidally algebraic at large primes.
\newblock {\em arXiv preprint arXiv:2304.14457}, 2023.

\bibitem{barthel2021stratification}
Tobias Barthel, Drew Heard, and Beren Sanders.
\newblock Stratification and the comparison between homological and tensor
  triangular support.
\newblock {\em arXiv preprint arXiv:2106.16011}, 2021.

\bibitem{bss_monochromatic}
Tobias Barthel, Tomer Schlank, and Nathaniel Stapleton.
\newblock Monochromatic homotopy theory is asymptotically algebraic.
\newblock {\em arXiv preprint arXiv:1711.00844}, 2019.

\bibitem{barthel2017chromatic}
Tobias Barthel, Tomer~M. Schlank, and Nathaniel Stapleton.
\newblock Chromatic homotopy theory is asymptotically algebraic.
\newblock {\em Invent. Math.}, 220(3):737--845, 2020.

\bibitem{Clark}
Clark Barwick.
\newblock Spectral {M}ackey functors and equivariant algebraic {$K$}-theory
  ({I}).
\newblock {\em Adv. Math.}, 304:646--727, 2017.

\bibitem{barwick2018exodromy}
Clark Barwick, Saul Glasman, and Peter Haine.
\newblock Exodromy.
\newblock {\em arXiv preprint arXiv:1807.03281v7}, 2020.

\bibitem{beaudry2019tmf}
Agnes Beaudry, Mark Behrens, Prasit Bhattacharya, Dominic Culver, and Zhouli
  Xu.
\newblock On the tmf-resolution of z.
\newblock {\em arXiv preprint arXiv:1909.13379}, 2019.

\bibitem{beaudry20202}
Agnes Beaudry, Mark Behrens, Prasit Bhattacharya, Dominic Culver, and Zhouli
  Xu.
\newblock On the e 2-term of the bo-adams spectral sequence.
\newblock {\em Journal of Topology}, 13(1):356--415, 2020.

\bibitem{beligiannis2000relative}
Apostolos Beligiannis.
\newblock Relative homological algebra and purity in triangulated categories.
\newblock {\em Journal of algebra}, 227(1):268--361, 2000.

\bibitem{BBD}
A.~A. Be\u{\i}linson, J.~Bernstein, and P.~Deligne.
\newblock Faisceaux pervers.
\newblock In {\em Analysis and topology on singular spaces, {I} ({L}uminy,
  1981)}, volume 100 of {\em Ast\'{e}risque}, pages 5--171. Soc. Math. France,
  Paris, 1982.

\bibitem{Biedermann}
Georg Biedermann.
\newblock Interpolation categories for homology theories.
\newblock {\em J. Pure Appl. Algebra}, 208(2):497--530, 2007.

\bibitem{blanc2004realization}
David Blanc, William~G Dwyer, and Paul~G Goerss.
\newblock The realization space of a $\pi$-algebra: a moduli problem in
  algebraic topology.
\newblock {\em Topology}, 43(4):857--892, 2004.

\bibitem{bousfield1979localization}
Aldridge~K Bousfield.
\newblock The localization of spectra with respect to homology.
\newblock {\em Topology}, 18(4):257--281, 1979.

\bibitem{bousfield1985homotopy}
Aldridge~K Bousfield.
\newblock On the homotopy theory of {$K$}-local spectra at an odd prime.
\newblock {\em American Journal of Mathematics}, 107(4):895--932, 1985.

\bibitem{bousfield1990classification}
Aldridge~K Bousfield.
\newblock A classification of {$K$}-local spectra.
\newblock {\em Journal of Pure and Applied Algebra}, 66(2):121--163, 1990.

\bibitem{brantner2017lubin}
Lukas Brantner.
\newblock {\em The Lubin-Tate theory of spectral Lie algebras}.
\newblock PhD thesis, 2017.

\bibitem{bruner2009adams}
Robert~R Bruner.
\newblock An adams spectral sequence primer.
\newblock {\em Lecture notes}, pages 1--78, 2009.

\bibitem{bruner2006h}
Robert~R Bruner, J~Peter May, James~E McClure, and Mark Steinberger.
\newblock {\em H ring spectra and their applications}, volume 1176.
\newblock Springer, 2006.

\bibitem{burklund2021extension}
Robert Burklund.
\newblock An extension in the adams spectral sequence in dimension 54.
\newblock {\em Bulletin of the London Mathematical Society}, 53(2):404--407,
  2021.

\bibitem{burklund2022moore}
Robert Burklund.
\newblock Multiplicative structures on moore spectra.
\newblock {\em arXiv preprint arXiv:2203.14787}, 2022.

\bibitem{burklund2019boundaries}
Robert Burklund, Jeremy Hahn, and Andrew Senger.
\newblock On the boundaries of highly connected, almost closed manifolds.
\newblock {\em arXiv preprint arXiv:1910.14116}, 2019.

\bibitem{burklund2020galois}
Robert Burklund, Jeremy Hahn, and Andrew Senger.
\newblock Galois reconstruction of artin-tate $\mathbb{R}$-motivic spectra.
\newblock {\em arXiv preprint arXiv:2010.10325}, 2020.

\bibitem{burklund2020inertia}
Robert Burklund, Jeremy Hahn, and Andrew Senger.
\newblock Inertia groups in the metastable range.
\newblock {\em arXiv preprint arXiv:2010.09869}, 2020.

\bibitem{burklundpstragowski2023}
Robert Burklund and Piotr Pstr{\k{a}}gowski.
\newblock Quivers and the adams spectral sequence.
\newblock {\em arXiv preprint arXiv:2305.08231}, 2023.

\bibitem{burklund2020high}
Robert Burklund and Andrew Senger.
\newblock On the high-dimensional geography problem.
\newblock {\em arXiv preprint arXiv:2007.05127}, 2020.

\bibitem{cameron2021homological}
James~C Cameron and Greg Stevenson.
\newblock Homological residue fields as comodules over coalgebras.
\newblock {\em arXiv preprint arXiv:2109.10827}, 2021.

\bibitem{Christensen}
J.~Daniel Christensen.
\newblock Ideals in triangulated categories: phantoms, ghosts and skeleta.
\newblock {\em Adv. Math.}, 136(2):284--339, 1998.

\bibitem{devinatz1997morava}
Ethan~S Devinatz.
\newblock Morava modules and {Brown-Comenetz} duality.
\newblock {\em American Journal of Mathematics}, 119(4):741--770, 1997.

\bibitem{DugShiEx}
Daniel Dugger and Brooke Shipley.
\newblock A curious example of triangulated-equivalent model categories which
  are not {Q}uillen equivalent.
\newblock {\em Algebr. Geom. Topol.}, 9(1):135--166, 2009.

\bibitem{dwyer1993e2}
William~G Dwyer, Daniel~M Kan, and Christopher~R Stover.
\newblock An e2 model category structure for pointed simplicial spaces.
\newblock {\em Journal of Pure and Applied Algebra}, 90(2):137--152, 1993.

\bibitem{elmendorf1997rings}
AD~Elmendorf, I~Kriz, MA~Mandell, and JP~May.
\newblock {\em Rings, modules, and algebras in stable homotopy theory},
  volume~47.
\newblock 1997.

\bibitem{franke1996uniqueness}
Jens Franke.
\newblock Uniqueness theorems for certain triangulated categories possessing an
  {Adams} spectral sequence.
\newblock {\em K-theory Preprint Archives}, 139, 1996.

\bibitem{freyd1966representations}
Peter Freyd.
\newblock Representations in abelian categories.
\newblock In {\em Proceedings of the Conference on Categorical Algebra}, pages
  95--120. Springer, 1966.

\bibitem{freyd1966stable}
Peter Freyd.
\newblock Stable homotopy.
\newblock In {\em Proceedings of the conference on categorical algebra}, pages
  121--172. Springer, 1966.

\bibitem{gepner2017infty}
David Gepner, Rune Haugseng, and Joachim Kock.
\newblock $\infty$-operads as analytic monads.
\newblock {\em arXiv preprint arXiv:1712.06469}, 2017.

\bibitem{gheorghe2018c}
Bogdan Gheorghe, Daniel~C Isaksen, Achim Krause, and Nicolas Ricka.
\newblock C-motivic modular forms.
\newblock {\em arXiv preprint arXiv:1810.11050}, 2018.

\bibitem{gheorghe2018special}
Bogdan Gheorghe, Guozhen Wang, and Zhouli Xu.
\newblock The special fiber of the motivic deformation of the stable homotopy
  category is algebraic.
\newblock {\em arXiv preprint arXiv:1809.09290}, 2018.

\bibitem{moduli_problems_for_structured_ring_spectra}
P.~G. Goerss and M.~J. Hopkins.
\newblock Moduli problems for structured ring spectra.
\newblock \url{http://www.math.northwestern.edu/~pgoerss/spectra/obstruct.pdf}.

\bibitem{goerss2014hopkins}
Paul Goerss, Hans-Werner Henn, Mark Mahowald, and Charles Rezk.
\newblock On {Hopkins'} {Picard} groups for the prime 3 and chromatic level 2.
\newblock {\em Journal of Topology}, 8(1):267--294, 2014.

\bibitem{gonzalez2000vanishing}
Jes{\'u}s Gonz{\'a}lez.
\newblock A vanishing line in the bp< 1>-adams spectral sequence.
\newblock {\em Topology}, 39(6):1137--1153, 2000.

\bibitem{GreenleesS1}
J.~P.~C. Greenlees.
\newblock Rational {$S^1$}-equivariant stable homotopy theory.
\newblock {\em Mem. Amer. Math. Soc.}, 138(661):xii+289, 1999.

\bibitem{hahn2020redshift}
Jeremy Hahn and Dylan Wilson.
\newblock Redshift and multiplication for truncated brown-peterson spectra.
\newblock {\em arXiv preprint arXiv:2012.00864}, 2020.

\bibitem{hatcher2005algebraic}
Allen Hatcher.
\newblock {\em Algebraic topology}.
\newblock Cambridge University Press, 2005.

\bibitem{hopkins1999complex}
Michael Hopkins.
\newblock Complex oriented cohomology theories and the language of stacks.
\newblock {\em Course notes}, 1999.

\bibitem{isaksen_stable_stems}
Daniel~C Isaksen.
\newblock Stable stems.
\newblock Memoirs of American Mathematical Society, to appear, 2018.

\bibitem{more_stable_stems}
Daniel~C Isaksen, Gouzhen Wang, and Zhouli Xu.
\newblock More stable stems.
\newblock in preparation.

\bibitem{Jann}
Uwe Jannsen.
\newblock Continuous \'{e}tale cohomology.
\newblock {\em Math. Ann.}, 280(2):207--245, 1988.

\bibitem{krause_2010}
Henning Krause.
\newblock {\em Localization theory for triangulated categories}, page
  161–235.
\newblock London Mathematical Society Lecture Note Series. Cambridge University
  Press, 2010.

\bibitem{levine2015adams}
Marc Levine et~al.
\newblock The {Adams-Novikov} spectral sequence and {Voevodsky}'s slice tower.
\newblock {\em Geometry \& Topology}, 19(5):2691--2740, 2015.

\bibitem{lurie2011derived}
J~Lurie.
\newblock Derived algebraic geometry viii: Quasi-coherent sheaves and tannaka
  duality.
\newblock {\em preprint available online at the author’s webpage}, 2011.

\bibitem{higher_algebra}
Jacob Lurie.
\newblock Higher algebra.
\newblock \url{http://www.math.harvard.edu/~lurie/papers/HA.pdf}.

\bibitem{lurierotation}
Jacob Lurie.
\newblock Rotation invariance in k-theory.
\newblock {\em Available at author’s webpage}.

\bibitem{lurie_spectral_algebraic_geometry}
Jacob Lurie.
\newblock Spectral algebraic geometry.
\newblock \url{http://www.math.harvard.edu/~lurie/papers/SAG-rootfile.pdf}.

\bibitem{lurie2009derived}
Jacob Lurie.
\newblock Derived algebraic geometry v: Structured spaces.
\newblock {\em arXiv preprint arXiv:0905.0459}, 2009.

\bibitem{higher_topos_theory}
Jacob Lurie.
\newblock {\em Higher topos theory}, volume 170 of {\em Annals of Mathematics
  Studies}.
\newblock Princeton University Press, Princeton, NJ, 2009.

\bibitem{lurie2018ultracategories}
Jacob Lurie.
\newblock Ultracategories.
\newblock {\em Preprint available at www. math. harvard. edu/\~{}
  lurie/papers/Conceptual. pdf}, 2018.

\bibitem{lurie_hopkins_brauer_group}
Jacob Lurie and Michael Hopkins.
\newblock On {Brauer} groups of {Lubin-Tate} spectra {I}.
\newblock http://www.math.harvard.edu/~lurie/papers/Brauer.pdf.

\bibitem{mahowald1981bo}
Mark Mahowald.
\newblock bo-resolutions.
\newblock {\em Pacific Journal of Mathematics}, 92(2):365--383, 1981.

\bibitem{mahowald1982image}
Mark Mahowald.
\newblock The image of j in the ehp sequence.
\newblock {\em Annals of Mathematics}, pages 65--112, 1982.

\bibitem{may1970general}
J~Peter May.
\newblock A general algebraic approach to steenrod operations.
\newblock In {\em The Steenrod Algebra and its Applications: a conference to
  celebrate NE Steenrod's sixtieth birthday}, pages 153--231. Springer, 1970.

\bibitem{miller1975some}
Haynes~Robert Miller.
\newblock {\em Some algebraic aspects of the Adams-Novikov spectral sequence.}
\newblock Princeton University, 1975.

\bibitem{Mit}
Barry Mitchell.
\newblock On the dimension of objects and categories. {I}. {M}onoids.
\newblock {\em J. Algebra}, 9:314--340, 1968.

\bibitem{Mit1}
Barry Mitchell.
\newblock On the dimension of objects and categories. {II}. {F}inite ordered
  sets.
\newblock {\em J. Algebra}, 9:341--368, 1968.

\bibitem{moss1968composition}
Robert~MF Moss.
\newblock On the composition pairing of adams spectral sequences.
\newblock {\em Proceedings of the London Mathematical Society}, 3(1):179--192,
  1968.

\bibitem{neeman1996grothendieck}
Amnon Neeman.
\newblock The grothendieck duality theorem via bousfield’s techniques and
  brown representability.
\newblock {\em Journal of the American Mathematical Society}, 9(1):205--236,
  1996.

\bibitem{neeman2001triangulated}
Amnon Neeman.
\newblock {\em Triangulated categories}.
\newblock Number 148. Princeton University Press, 2001.

\bibitem{patchkoria2012}
Irakli {Patchkoria}.
\newblock {On the algebraic classification of module spectra}.
\newblock {\em {Algebr. Geom. Topol.}}, 12(4):2329--2388, 2012.

\bibitem{patchkoria2017exotic}
Irakli Patchkoria.
\newblock On exotic equivalences and a theorem of {Franke}.
\newblock {\em Bulletin of the London Mathematical Society}, 49(6):1085--1099,
  2017.

\bibitem{patchkoriaKU}
Irakli {Patchkoria}.
\newblock {The derived category of complex periodic \(K\)-theory localized at
  an odd prime}.
\newblock {\em {Adv. Math.}}, 309:392--435, 2017.

\bibitem{PatRoi}
Irakli Patchkoria and Constanze Roitzheim.
\newblock Rigidity and exotic models for {$v_1$}-local {$G$}-equivariant stable
  homotopy theory.
\newblock {\em Math. Z.}, 295(1-2):839--875, 2020.

\bibitem{PatSanWim}
Irakli Patchkoria, Beren Sanders, and Christian Wimmer.
\newblock The spectrum of derived mackey functors.
\newblock {\em arXiv:2008.02368v2, to appear in Trans. Amer. Math. Soc.}, 2021.

\bibitem{pstrkagowski2017moduli}
Piotr Pstr{\k{a}}gowski.
\newblock Moduli of {$\Pi$}-algebras.
\newblock {\em arXiv preprint arXiv:1705.05761}, 2017.

\bibitem{pstragowski_chromatic_homotopy_algebraic}
Piotr Pstr{\k{a}}gowski.
\newblock Chromatic homotopy is algebraic when $p > n^{2}+n+1$.
\newblock {\em arXiv preprint arXiv:1810.12250, to appear in Advances in
  Mathematics}, 2018.

\bibitem{pstrkagowski2018synthetic}
Piotr Pstr{\k{a}}gowski.
\newblock Synthetic spectra and the cellular motivic category.
\newblock {\em arXiv preprint arXiv:1803.01804}, 2018.

\bibitem{abstract_gh_theory}
Piotr Pstr{\k{a}}gowski and Paul VanKoughnett.
\newblock Abstract goerss-hopkins theory.
\newblock {\em arXiv preprint arXiv:1904.08881}, 2019.

\bibitem{quillen1969rational}
Daniel Quillen.
\newblock Rational homotopy theory.
\newblock {\em Annals of Mathematics}, pages 205--295, 1969.

\bibitem{Raptis}
George Raptis.
\newblock Higher homotopy categories, higher derivators, and k-theory.
\newblock {\em arXiv:1910.04117}, 2019.

\bibitem{ravenel_complex_cobordism}
Douglas~C. Ravenel.
\newblock {\em Complex cobordism and stable homotopy groups of spheres}, volume
  121 of {\em Pure and Applied Mathematics}.
\newblock Academic Press, Inc., Orlando, FL, 1986.

\bibitem{Roitzheim}
Constanze Roitzheim.
\newblock Rigidity and exotic models for the {$K$}-local stable homotopy
  category.
\newblock {\em Geom. Topol.}, 11:1855--1886, 2007.

\bibitem{Rudy}
Yuli~B. Rudyak.
\newblock {\em On {T}hom spectra, orientability, and cobordism}.
\newblock Springer Monographs in Mathematics. Springer-Verlag, Berlin, 1998.
\newblock With a foreword by Haynes Miller.

\bibitem{schwede2007stable}
Stefan Schwede.
\newblock The stable homotopy category is rigid.
\newblock {\em Annals of Mathematics}, pages 837--863, 2007.

\bibitem{schwede2008algebraic}
Stefan Schwede.
\newblock Algebraic versus topological triangulated categories.
\newblock {\em arXiv preprint arXiv:0807.2592}, 2008.

\bibitem{stacks-project}
The {Stacks project authors}.
\newblock The stacks project.
\newblock \url{https://stacks.math.columbia.edu}, 2021.

\bibitem{Wei}
Charles~A. Weibel.
\newblock {\em An introduction to homological algebra}, volume~38 of {\em
  Cambridge Studies in Advanced Mathematics}.
\newblock Cambridge University Press, Cambridge, 1994.

\bibitem{Wim}
Christian Wimmer.
\newblock A model for genuine equivariant commutative ring spectra away from
  the group order.
\newblock {\em arXiv:1905.12420}, 2019.

\bibitem{Wolbert}
Jerome~J. Wolbert.
\newblock Classifying modules over {$K$}-theory spectra.
\newblock {\em J. Pure Appl. Algebra}, 124(1-3):289--323, 1998.

\end{thebibliography}
\end{document}